\documentclass[a4paper,oneside,11pt]{amsart}
\usepackage[latin1]{inputenc}
\usepackage[T1]{fontenc}
\usepackage[english]{babel}
\usepackage{amssymb}
\usepackage{amsthm}
\usepackage{aliascnt}
\usepackage{enumerate}
\usepackage[usenames,dvipsnames]{xcolor}
\definecolor{darkblue}{rgb}{0.0, 0.0, 0.55}
\usepackage[pagebackref,colorlinks,linkcolor=BrickRed,citecolor=OliveGreen,urlcolor=darkblue,hypertexnames=true]{hyperref}
\usepackage{longtable}
\usepackage{graphicx}

\usepackage{geometry}
\newcommand\N{\mathbb N}

\newcommand\R{\mathbb R}
\newcommand\C{\mathbb C}
\newcommand\OO{\mathcal O}
\renewcommand\S{\mathbb S}

\definecolor{inv}{rgb}{1,1,1}

\DeclareMathOperator\spam{span}

\DeclareMathOperator\inte{int}

\DeclareMathOperator\pz{pz}
\DeclareMathOperator\kz{kz}
\DeclareMathOperator\rk{rk}
\DeclareMathOperator\mconv{mconv}
\DeclareMathOperator\mcc{mcc}
\DeclareMathOperator\im{im}
\DeclareMathOperator\dom{dom}
\DeclareMathOperator\size{size}
\DeclareMathOperator\tr{tr}
\DeclareMathOperator\abex{abex}
\DeclareMathOperator\mex{mext}
\DeclareMathOperator\mexp{mexp}
\DeclareMathOperator\Deh{Deh}
\DeclareMathOperator\Hom{Hom}

\newcommand\ph\varphi
\newcommand\ps\psi
\newcommand\ep\varepsilon
\newcommand\rh\varrho
\newcommand\al\alpha
\newcommand\be\beta
\newcommand\ga\gamma
\newcommand\om\omega
\newcommand\ta\tau
\renewcommand\th\vartheta
\newcommand\de\delta
\newcommand\ze\zeta
\newcommand\ch\chi
\newcommand\et\eta
\newcommand\io\iota
\newcommand\la\lambda
\newcommand\si\sigma
\newcommand\ka\kappa

\newcommand\ov\overline
\newcommand\uv\underline
\DeclareRobustCommand{\SkipTocEntry}[5]{}

\usepackage{bbm}
\newcommand\ii{\mathbbm i}

\newcommand\NC{\C\langle \ov X \rangle}
\newcommand\ml{\mathfrak{L}}
\newcommand\mh{\mathfrak{H}}
\newcommand\mfm{\mathfrak{m}}

\theoremstyle{definition}

\newaliascnt{lemma}{satz}
\newtheorem{lemma}[lemma]{Lemma}
\aliascntresetthe{lemma}

\newaliascnt{theo}{satz}
\newtheorem{theo}[theo]{Theorem}
\aliascntresetthe{theo}

\newaliascnt{cor}{satz}
\newtheorem{cor}[cor]{Corollary}
\aliascntresetthe{cor}

\newaliascnt{rem}{satz}
\newtheorem{rem}[rem]{Remark}
\aliascntresetthe{rem}

\newaliascnt{exar}{satz}
\newtheorem{exar}[exar]{Example}
\aliascntresetthe{exar}

\newaliascnt{definition}{satz}
\newtheorem{definition}[definition]{Definition}
\aliascntresetthe{definition}

\newaliascnt{defprop}{satz}
\newtheorem{defprop}[defprop]{Definition and Proposition}
\aliascntresetthe{defprop}

\newaliascnt{proposition}{satz}
\newtheorem{proposition}[proposition]{Proposition}
\aliascntresetthe{proposition}

\newaliascnt{structure}{satz}
\newtheorem{structure}[structure]{Overview}
\aliascntresetthe{structure}

\newaliascnt{notation}{satz}
\newtheorem{notation}[notation]{Notation}
\aliascntresetthe{notation}

\makeindex


\begin{document}

\newpage
\allowdisplaybreaks
\pagenumbering{arabic}

\title{An introduction to matrix convex sets and free spectrahedra}
\author{Tom-Lukas Kriel}

\begin{abstract}
The purpose of this paper is to give a self-contained overview of the theory of matrix convex sets and free spectrahedra. We will give new proofs and generalizations of key theorems. However we will also introduce various new concepts and results as well. Key contributions of this paper are
\begin{itemize}
\item A new general Krein-Milman theorem that characterizes the smallest operator tuple defining a compact matrix convex set.
\item The introduction and a characterization of matrix exposed points.
\item A (weak) Minkowski theorem in the language of matrix extreme points (with a new proof of the weak Krein-Milman theorem of Webster and Winkler).
\item Simplified/new proofs of the Gleichstellensatz, Helton and McCulloughs characterization of free spectrahedra as closures of matrix convex "free basic open semialgebraic" sets and a characterization of hermitian irreducible free loci of Helton, Klep and Vol$\check{\text{c}}$i$\check{\text{c}}$. 
\end{itemize}
\end{abstract}

\maketitle

\section{Introduction}

\subsection{Overview over the main results}

We advise non-experts who are interested in a introduction to the topic to skip this subsection and continue with the next section "Motivation". In the following we will present the most important new contributions of this work. Since most basic definitions will be well-known to the experts, we will postpone the rigorous introduction of terms and objects to the remainder of the introduction. Let $\S^g:=(\S^g(k))_{k \in \N}:=((S\C^{k \times k})^g)_{k \in \N}$, the set consisting of the collection of $g$-tuples of Hermitian square matrices of uniform size. In this work, all matrix convex sets are supposed to be subsets of $\S^g$.
\begin{itemize}
\item For matrix convex sets various notions of extreme points exist. One aim is to formulate an analogue of the classical Krein-Milman theorem. Due to the weak Krein-Milman theorem of Webster and Winkler it is known that the closure of the convex hull of the matrix extreme points of a compact matrix convex set $K$ equals $K$. However this set of points is not a minimal generator. We show that there exists a $g$-tuple $L$ of hermitian operators on a separable Hilbert space $\mathcal{H}$ such that $K$ equals the matix convex hull/ matricial range of $L$ (i.e. $L$ generates $K$) and up to approximate unitary equivalence $L$ is a direct summand of every other operator valued tuple $H$ generating $K$ (\autoref{secondhalf}). Hence it is justified to call $L$ the smallest generator of $K$, and this gives a satisfactory analogue to the Krein-Milman theorem. The absolute extreme points of $K$ are encoded in $L$ as a generalized direct summand of $L$. 
\item In classical convexity one can separate points from convex sets via linear maps. Matrix convex sets can be separated from points which they do not contain by monic linear pencils. This leads to the natural notion of matrix exposed points of a compact matrix convex set $K$ (as a generalization of exposed points of convex sets) as points $A$ of $K$ that are matrix extreme and for which there exists a monic linear pencil certifying this property via a separation property. We show that the closure of the matrix convex hull of all matrix exposed points of $K$ is again $K$ (an analogue to the classical Straszewicz theorem) and that matrix exposed points can be characterized as matrix extreme points that are also ordinary extreme on the respective level-set (\autoref{grund2} and \autoref{straszewicz}).
\item We give a new and very simple proof of the afore-mentioned weak Krein-Milman theorem by Webster and Winkler and show that the closure is superfluous making it a weak Minkowski theorem (\autoref{krain}).
\item The Gleichstellensatz characterizes the smallest monic linear (matrix) pencil $\ml=I-L_1 X_1 - ... -L_g X_g$ defining a given free spectrahedron $S=\mathcal{D}_\ml:=\{A \in \S^g \ | \ I \otimes I- L_1 \otimes A_1 - ... - L_g \otimes A_g \succeq 0\}$. In fact, this smallest pencil is a direct summand of every other monic linear pencil defining the same spectrahedron. The original proof of Helton, Klep and McCullough uses the concept of the Silov ideal, a complicated object from the theory of operator systems and $C^*$-algebras. We present two new more elementary proofs of the Gleichstellensatz (\autoref{gleichi} and \autoref{gleich2}).  
\item We introduce a new technique to analyze matrix convex sets of the form $S=\{A \in \S^g \ | \ \forall t \in [0,1]: p(tA) \succ 0\}$, where $p$ is a noncommutative polynomial. This technique allows us to give a new variant of the proof of a theorem of Helton and McCullough stating that $S$ is the interior of a free spectrahedron (\autoref{existence}). Another application is that we are able to infer that spectrahedra can be written as intersections of irreducible spectrahedra. We prove that for a monic linear pencil $\ml$ either $\ml$ is reducible or the sequence $(\det_k \ml)_{k \in \N}$ of determinants of $\ml$ restricted to the $k$-th level set $\S^g(k)$ become irreducible for big $k$ (\autoref{ald}).   
\end{itemize}

\subsection{Motivation}

A spectrahedron is a set of the form
\begin{align*}
\left\{(x_1,...,x_g\} \in \R^g \ \middle| \ n \in \N, \ L_0+\sum_{i=1}^g L_i x_i \succeq 0\right\}
\end{align*}
where the $L_i$ are Hermitian matrices of a uniform size. It has turned out that spectrahedra are suitable for numerical calculations, especially because they are convex. For the problem of optimizing a linear functional over a spectrahedron (this task is called semidefinite programming and generalizes linear programming) there exist efficient interior-point algorithms. Therefore in recent years much effort has been made to approximate polynomial optimization problems by semidefinite programs or hierarchies of those (e.g. the $\theta$-number of a graph, Lasserre relaxation; see \cite{BV} and \cite{L}). In Engineering, especially in linear systems and control theory, so-called matrix inequalities appear in pratical problems. The most common approach to solve them is to try to find an equivalent formulation in terms of a semidefinite program (see \cite{BEFB} and \cite{HKM2}). \\[0.2cm]
For those reasons many papers dealing with the question which sets are spectrahedra or at least projections of spectrahedra have been written (see \cite{HN1}, \cite{HN2}, \cite{KS}). Another area of research connected to spectrahedra is the question which hyperbolic polynomials admit a determinantal representation (see \cite{HV}).\\[0.2cm]
Consider now the following linear matrix inequality (LMI)
\begin{align*}
L_0 \otimes I+\sum_{i=1}^g L_i \otimes X_i \succeq 0
\end{align*}
(the left is side is also called a (linear) pencil and in case that $L_0=I$ this pencil is called monic). This expression can be evaluated in the $X_i$ with Hermitian matrices of an arbitrary common size. The solution set \hypertarget{I}{}\begin{align*}
\left(\left\{(X_1,...,X_g\} \in (S\C^{k \times k})^g \ \middle| \ L_0 \otimes I_k+\sum_{i=1}^g L_i \otimes X_i \succeq 0\right\}\right)_{k \in \N}, \ \ \ \text{(I)}
\end{align*}
where $\otimes$ denotes the Kronecker tensor product and $S\C^{k \times k}$ the Hermitian $k \times k$-matrices, is called a free spectrahedron and
fulfills some stronger convexity property, which is called matrix convexity. Matrix convex sets $C$ are closed under matrix convex combinations, i.e.
\begin{align*}
A_1,...,A_r \in C, \sum_{j=1}^r V_j^*V_j =I \implies \sum_{i=1}^r V_j^* A_j V_j \in C,
\end{align*} where the $V_j$ are rectangular matrices of appropriate sizes and $V_j^* A_j V_j:=(V_j^* {A_j}_i V_j \ | \ i \in \{1,...,g\})$. Note that the $A_j$ do not have to be all of the same size. \\[0.2cm]
We see that if we specialize our solution set \hyperlink{I}{(I)} to vectors (tuples of Hermitian $1 \times 1$-matrices), we end up with the (ordinary) spectrahedron from the beginning. Given a spectrahedron defined by a pencil, one can hope that key properties of its geometry are reflected by a particularly nice form of the pencil. Here it turns out that the right framework for following this purpose is the free setting. An example is for instance the following question: Given two spectrahedra with monic defining pencils $\ml_1$ and $\ml_2$, when is the first one contained in the second one? No sharp characterisation for this property is known. A straight-forward certificate for containment would be a sums of squares-representation of $\ml_2$ in terms of $\ml_1$. However this would imply also the containment of the related free spectrahedra. Indeed, for free spectrahedra such a sums-of-squares representation exists if and only if the free spectrahedron given by $\ml_1$ is contained in the one given by $\ml_2$ \cite[Lemma 4.7]{HKM}. \\[0.2cm]
Morally speaking, a spectrahedron can be defined by many different pencils. Evaluating those pencils in tuples of Hermitian matrices and not only vectors makes this hidden structure visible. Therefore a big motivation behind the study of free spectrahedra is to find out more about (ordinary) spectrahedra. However the free theory still needs to evolve in order to provide powerful techniques for the study of ordinary spectrahedra. One example of a succesful result approached in this way is the main result of \cite{HKMS}, which says that one monic pencil $\ml_2$ defining a much bigger spectrahedron than a monic pencil $\ml_1$ defines automatically a bigger free spectrahedron than the pencil $\ml_1$ and hence $\ml_2$ has a sums-of-squares representation with $\ml_1$ as a weight. \\[0.2cm]
Another source of motivation to study matrix convex sets is its connection to the theory of operator systems and completely positive maps. The structure of an operator system is determined up to complete isometry by its matrix range, which is a compact matrix convex set. A further instance where matrix convexity is utilized is the theoretical study of so-called dimension independent polynomial matrix inequalities. Some problems in the theory of linear systems and control are equivalent to solving a polynomial matrix inequality $p(A,Y) \succeq 0$ given by a noncommutative polynomial $p$ with a fixed parameter $Y$. The parameter $Y$ determines the size of the solution matrices $A$. Considering those problems for fixed parameters and transforming them to a problem over commutative variables is not well-behaved in the sense that two different parameters lead to very different and seemingly unrelated problems. The idea is now to stick to the non-commutative setting where the the problem keeps the same structure for different choices of parameters (one says that the problem scales automatically when changing dimensions because noncommutative polynomials behave well with respect to direct sums and compressions). In case that the polynomial matrix inequality defines a convex problem (or at least convex in the variables $A$) the theory of (partial) matrix convexity can be applied (see \cite{HHLM} and \cite{HMPV}).     \\[0.2cm]
The research of matrix convex sets started in the early 1980s. First fundamental results were given in the 1990s, however the number of researchers working in this field was relatively small. Nowadays the subject is increasing and drawing more attention. Matrix convexity has a strong connection to operator theory, in particular to completely positive maps. Other related topics are free real algebraic geometry and free complex analysis.  \\[0.2cm]
One aim of this work is to present a unified introduction to matrix convexity and free spectrahedra. For non-specialists, who want to know more about the topic, there are some obstacles we want to reduce. Some important theorems are not easy to understand because
\begin{itemize}
\item the proofs use advanced results and concepts from operator theory (e.g. Gleichstellensatz \cite{HKM4}) or
\item the theorem is stated in too much generality resulting in a more involved proof and weaker conclusion (e.g. weak free Krein-Milman theorem \cite{WW}) or
\item the proof is long (e.g. Helton-McCulloughs characterization of free spectrahedra as closures of matrix convex free basic open semialgebraic sets \cite{HM}).
\end{itemize}

\subsection{Notation and basic definitions}

\begin{notation} \label{gennot} For this article let \index{g}$g \in \N$ and fix a tuple \index{x@$\ov X$}$\ov X=(X_1,...,X_g)$ of $g$ free noncommutative variables. The letter $i$ will be exclusively denote elements of $\{1,...,g\}$. In this paper all rings shall contain $1$ and ring homomorphisms are required to be unital except of cases where we explicitely use the word "non-unital". If $R$ is a commutative ring, then a polynomial over $R$ in the variables $\ov{X}$ is the formal object 
\begin{align*} f=\sum_{d \in \N_0} \sum_{\al \in {\{1,...,g\}^d}} f_{\al} X_{\al(1)}...X_{\al(d)}
\end{align*} given by a coefficients $f_\al \in R$ where only finitely many $f_{\al}$ are non-trivial. The degree of $f \neq 0$ is the maximal $d$ such that there is $\al \in {\{1,...,g\}^d}$ with $f_\al \neq 0$. Two polynomials are equal if and only if all coefficients are equal. Scalars from $R$ shall commute with the variables $X_1,...,X_g$, i.e. $rX_i=X_ir$ for $i \in \{1,...,g\}$ and $r \in R$. We form the sum and product of two polynomials in the obvious way and denote the polynomials in the variables $\ov X$ over $R$ by \index{C@$\C \langle \ov X \rangle$}$R \langle \ov X \rangle$ ({\bfseries NC polynomials}). This object is also called the {\bfseries free $R$-algebra} over $g$ generators. It has the universal property that for every not necessarily commutative $R$-algebra $B$ which contains $R$ in its center and $b_1,...,b_g \in B$ there is a unique $R$-algebra homomorphism $\varphi: R \langle \ov X \rangle \ \rightarrow B$ such that $\varphi(X_i)=b_i$ for $i \in \{1,...,g\}$. Let \index{F@$\mathcal{F}_g$}$\mathcal{F}_g$ be the free monoid generated by $g$ elements $X_1,...,X_g$. One can identify $\mathcal{F}_g$ as the monomials of $\NC$. The free (complex) algebra $\NC$ admits an involution ${}^*$. For $p \in \NC$ we get $p^*$ by conjugation of the coefficients and reversing the order of the multiplication of variables in each monomial. \\[0.2cm] Contrary to the noncommutative case, if $(Y_1,...,Y_r)$ are commuting variables, we write $R[Y_1,...,Y_r]$ for the polynomial ring over $R$. If $T$ is a ring with a notion of degree and $d \in \N$, then $T_d$ shall denote the elements of $T$ with degree at most $d$. We denote with \index{S@$S\C^{k \times k}$}$S\C^{k \times k}$ the Hermitian $k \times k$ matrices with entries from $\C$. \\[0.2cm] Much of the work will take place in the {\bfseries free space} \index{S@$\S^g$}$\S^g:=(\S^g(k))_{k \in \N}:=((S\C^{k \times k})^g)_{k \in \N}$, the set consisting of the collection of $g$-tuples of Hermitian square matrices of uniform size. The elements of this set will serve as point evaluations for the noncommutative polynomial ring. $\S^g(k)$ will be called the {\bfseries k-th level} (set) of the free space. For $A \in \S^g(k)$ and $B \in \C^{k \times r}$ we set $B^*AB=(B^*A_1B,...,B^*A_gB) \in \S^g(r)$. We define \index{n@$\mid\mid.\mid\mid_{\S^g}$}$||A||=\sqrt{||\sum_{i=1}^g A_i^2||}$ where $||.||$ denotes the operator norm. 
\\[0.2cm]
Consider now a noncommutative matrix polynomial $f \in (\NC)^{\delta \times \ep}$ for some $\ep,\delta \in \N_{0}$. This is nothing else than a matrix with polynomial entries. Equivalently one can also interpret it as a polynomial with matrices as coefficients. The maximum of the degrees of the entries of $f$ will be the degree of $f$. $f^*$ is obtained by transposing $f$ and afterwards applying the involution on $\NC$ to each entry. Write $f=\sum_{\al \in \mathcal{F}^g} f_\al \al$ with matrices $f_\al \in \C^{\delta \times \ep}$.
For a square matrix $A \in (\C^{k \times k})^g$ we define $f(A) \in \C^{k\delta \times k\ep}$ to be $f(A)=\sum_{\al \in \mathcal{F}^g} f_\al \otimes \al(A)$ where the evaluation $\al(A)$ is obtained by replacing every occurance of $X_i$ in $\alpha$ by $A_i$.  Here \index{t@$\otimes$}$\otimes$ denotes the Kronecker product of matrices and $A^0=I_k$ the $k \times k$-identity matrix. We remind the reader of the calculation rule $(a \otimes b)(c \otimes d)=ac \otimes bd$ for matrices of appropriate sizes. 
Another way of defining $f(A)$ is to consider $f$ as a matrix with polynomial entries and substitute $A_i$ for $X_i$ in each entry. If $f^*=f$ and $A \in \S^g(k)$, then $f(A) \in \S\C^{\delta k \times \delta k}$. \\[0.2cm]
All Hilbert spaces in this paper shall be separable and complex. If $\mathcal{H}, \mathcal{K}$ are such Hilbert spaces, we write\index{B@$\mathcal{B}(\mathcal{H},\mathcal{K})$}
\begin{align*}
\mathcal{B}(\mathcal{H},\mathcal{K})=\{A: \mathcal{H} \rightarrow \mathcal{K} \ \text{linear and bounded}\},
\end{align*}$\mathcal{B}(\mathcal{H})=\mathcal{B}(\mathcal{H},\mathcal{H})$ and \index{B@$\mathcal{B}_h(\mathcal{H})$}$\mathcal{B}_h(\mathcal{H})=\{A \in \mathcal{B}(\mathcal{H}) \ | \ A^*=A\}$.
For $A \in \mathcal{B}(\mathcal{H})$ we define $f(A)=\sum_{\al \in \mathcal{F}^g} f_\al \otimes \al(A)$ in the above way as well (we will also allow that the coefficients $f_\al$ take values in some $\mathcal{B}(\mathcal{K})$). We write \index{H@$\mathcal{H}^{(\infty)}$}$\mathcal{H}^{(\infty)}:=\ov{\bigoplus_{n \in \N} \mathcal{H}}$, which becomes a Hilbert space by setting $||\bigoplus_{n \in \N} a_n||=\sqrt{\sum_{n \in \N}||a_n||^2}$. For $B \in \mathcal{B}(\mathcal{H})$ we obtain $B^{(\infty)} \in \mathcal{B}(\mathcal{H}^{(\infty)})$ by declaring $B^{(\infty)}\left(\bigoplus_{n \in \N} a_n\right)=\bigoplus_{n \in \N} Ba_n $. \\[0.2cm] 
For $L \in \S^g(\delta)$ and $C \in S\C^{\delta \times \delta}$, we define $C-L\ov{X}=C-(L_1X_1+...+L_gX_g)$. This expression is a Hermitian matrix polynomial of degree $1$ and {\bfseries size} $\delta$. Matrix polynomials of these form are called \textbf{(linear) pencils} (even though one could make an argument that they should be called affine linear). We say that the pencil is \textbf{monic} if $C=I$. We adopt the following convention: For $L \in \S^g(\delta)$, we denote by $\mathfrak{L}$ the monic pencil $I-L \ov X$. If on the other hand $\mathfrak{L}$ is a monic linear pencil, then $L=(-\mathfrak{L}(e_1)-I,...,-\mathfrak{L}(e_g)-I)$ \index{L@$\mathfrak{L}$} will be its \textbf{truly linear part} (in case we work with two monic pencils we will also use the letters $H$ and $\mathfrak{H}$). \\[0.2cm]
A linear pencil $C-B\ov X$ defines an associated \textbf{(free closed) spectrahedron} \index{D@$\mathcal{D}_{\mathfrak{L}}$}$\mathcal{D}_{C-B \ov X}=\{A \in \S^g \ | \ (C-B\ov X)(A) \succeq 0\}$ where \index{g@$\succeq$}$(C-B \ov X)(A) \succeq 0$ means that $(C- B\ov X)(A)$ is positive semidefinite ($\succ$ stands for "positive definite"). For a separable Hilbert space $\mathcal{H}$ and $B \in \mathcal{B}_h(\mathcal{H})^g, C \in \mathcal{B}_h(\mathcal{H})$ we also set $\mathcal{D}_{C-B \ov X}=\{A \in \S^g \ | \ (C-B\ov X)(A) \succeq 0\}$, however such a set will not be a free spectrahedron in general. We will focus our analysis of free spectrahedra on those which have an interior point $x \in \mathcal{D}_{C-B\ov X}(1)$ in $\R^g$ or equivalently admit a description with a pencil $C-B \ov X$ which is positive definite in this point. Most of the time we will assume that this point is zero. Then one can find a description of this spectrahedron by a monic linear pencil $\mathfrak{L}$ (\autoref{monisieren}). \\[0.2cm] 
Notice that the first level $\mathcal{D}_{B-C\ov{X}}(1)$ coincides with the ordinary spectrahedron defined by $B-C \ov X$ in $\R^g$. All other levels $\mathcal{D}_{B-C \ov X}(k)$ are also ordinary spectrahedra in $\S^g(k)$. For a monic pencil $\mathfrak{L}$ the set $\mathcal{D}_\mathfrak{L}(k)$ can be also determined as the closure of the connected component of the set $\{A \in \S^g(k) \ | \ \det_k \mathfrak{L}(A) \neq 0\}$ around $0$ where $\det_k \mathfrak{L}$ denotes the determinant $\det_k \mathfrak{L}: \S^g(k) \rightarrow \R, A \mapsto \det(\mathfrak{L}(A))$. \\[0.2cm]
Let $k \in \N$. For $\al \in \{1,...,k\}$ let $e_\al \in \C^k$ be the $\al$-th unit vector\index{e@$e_\al$} and set $E_{\al,\al}=e_\al e_\al^*$. For $\al,\beta \in \{1,...,k\}$ with $\al<\beta$ set $E_{\al,\beta}=e_\al e_\beta^*+e_\beta e_\al^*$ and $E_{\beta,\al}=\ii e_\al e_\beta^* - \ii e_\beta e_\al^*$\index{E@$E_{\al,\beta}$}. We define the $g$-tuple of generic $k \times k$ matrices \index{X@$\mathcal{X}$} as \begin{align*}
\mathcal{X}=\sum_{\al,\beta=1}^k E_{\al,\beta} \mathcal{X}_{\al,\beta} \footnotemark 
\end{align*}
where \footnotetext{Our definition of the matrix units $E_{\al,\beta}$ is non-standard. We adapted it in order to be suited for the study of hermitian matrix tuples. The normal matrix units $e_\al e_\beta^*$ are not all hermitian.}the $\mathcal{X}_{\al,\beta}=(\mathcal{X}_{\al,\beta}^1,...,\mathcal{X}_{\al,\beta}^g)$ are $g$-tuples of variables and the tuple \begin{align*}(\mathcal{X}_{\alpha,\beta}^i)_{i \in \{1,...,g\}, \alpha, \beta \in \{1,...,k\}}\end{align*} consists of commuting variables. We surpress the dimension $k$ in this notation. \\[0.2cm] 
For a monic linear pencil $\mathfrak{L}$ the determinant \index{d@$\det_k \mathfrak{L}$}$\det_k \mathfrak{L}(\mathcal{X})$ is an example of a \textbf{real-zero polynomial} (RZ-polynomial). These are defined to be polynomials $p \in \R[Y_1,...,Y_n]$ such that $p(0)\neq 0$ and for all $y \in \R^n \setminus \{0\}$ the univariate polynomial $p(Ty)$ has only real roots. The connected component of $p^{-1}(\R \setminus \{0\})$ around $0$ is automatically convex. We refer the reader to \cite{HV} for more material on RZ-polynomials. \\[0.2cm]
Let $A \in \S^g(k), v \in \C^{\delta k}$ and $C-B \ov X \in S(\NC)^{\delta \times \delta}_1$ be a linear pencil with $(C-B \ov X)(A)v=0$. Write $v=\sum_{\al=1}^\delta e_\al \otimes v_\al$ with $v_\al \in \C^k$. Set \index{M@$M(A,v)_{C-B \ov X}$}$M(A,v)_{C-B \ov X}=\spam(v_1,...,v_\delta)$. \\[0.2cm]
Free spectrahedra are an important example of matrix convex sets (\autoref{mainz}). A
set $S \subseteq \S^g$ is called \textbf{matrix convex} if
\begin{align*}
A_j \in S(k_j), V_j \in \C^{k_j \times s}, \sum_{j=1}^r V_j^* V_j = I_s \implies \sum_{j=1}^r V_j^* A_j V_j \in S(s)
\end{align*}       
We remind the reader that \index{v@$V^*AV$}$V_j^* A_j V_j$ is defined as $(V_j^* (A_j)_1 V_j,...,V_j^* (A_j)_g V_j)$. For a unitary matrix $U \in \C^{k \times k}$ and $A \in S(k)$ this implies that also $U^* A U \in S(k)$. Therefore matrix convex sets are closed under unitary conjugation. For $A,B \in \S^g$ we will write \index{A@$\approx$}$A \approx B$ if there is a unitary matrix $U$ such that $A = U^* B U$. Obviously, each level $S(k)$ of a matrix convex set $S$ is convex. Matrix convex sets are also closed under forming direct sums meaning that $A,B \in S$ implies \index{o@$\oplus$}$A \oplus B:=(A_1 \oplus B_1,...,A_g \oplus B_g) \in S$. If $A \in S(k)$ and $P: \C^k \rightarrow \im(P) \subseteq \C^{k}$ is an orthogonal projection, then $PP^*=I$ and thus $PAP^* \in S(\rk(P))$. In this paper projections are treated as surjective maps; if we want to view them as endomorphisms, we write $\uv{P}$ instead of $P$. \\[0.2cm]
If $0 \in S$, then matrix convexity can be restated as 
\begin{align*}
A_j \in S(k_j), V_j \in \C^{k_j \times s}, \sum_{j=1}^r V_j^* V_j \preceq I_s \implies \sum_{j=1}^r V_j^* A_j V_j \in S(s).
\end{align*}
For $S \subseteq \S^g$ the set \index{m@$\mconv$}$\mconv(S):=\bigcap\{T \subseteq \S^g \ | \ T \text{ matrix convex }, T \supseteq S\}$ is called the \textbf{matrix convex hull} of $S$ and is the smallest matrix convex superset of $S$ in $\S^g$. \\[0.2cm]
There are reasons why we chose our matrix convex sets to be subsets of $\S^g$; however, in principle it would be also natural to allow that $g$-tuples of Hermitian operators on an infinite-dimensional separable Hilbert space are elements of our free space. Both view points have their advantages (The finite-dimensional setting leans itself more to practical computations and the theory of ordinary spectrahedra; if one allows ordinary spectrahedra to be defined by LOI (Linear operator inequality), then every closed convex set would be a spectrahedron defined by a diagonal LOI. For either setting there are theorems which have cleaner statements than in the other setting.). The following notation is reminiscent of this hidden additional structure, which becomes important in some situations. If $\mathcal{H}$ is a separable Hilbert space and $L \in \mathcal{B}_h(\mathcal{H})^g$, then we set 
\begin{align*}
\mconv(L):&=\ov{\mconv\{PLP^* \ | \ P: \mathcal{H} \rightarrow \im(P) \text{ is a finite-dimensional projection }\}} \subseteq \S^g 
\end{align*}
where the closure of a set $S \subseteq \S^g$ is defined as $\ov{S}=(\ov{S(k)})_{k \in \N}$.
(We will prove later that for $\mathcal{H}$ finite-dimensional the closure in the definition of $\mconv(L)$ is superfluous (\autoref{kompakti}), making this definition consistent with the previous one.)
We have seen that matrix convex sets are invariant under unitary conjugation. Therefore when dealing with matrices, we treat them sometimes rather as linear operators than representations of operators with a fixed chosen orthonormal basis. For example we call a matrix $A$ a submatrix (resp. direct summand) of a matrix $B$ if there is a unitary matrix $U$ and matrices $C,D,E$ such that $U^*BU=\begin{pmatrix} A & C \\ D & E\end{pmatrix}$ $\left(\text{resp. }U^*BU=\begin{pmatrix} A & 0\\ 0 & E \end{pmatrix}\right)$. \\[0.2cm]  
For $k \in \N$ we denote by $\mathcal{S}^{k-1} \subseteq \C^k$ the $(k-1)$-dimensional sphere. If $(U,||.||)$ is a normed vector space, $a \in U$, $r>0$, then \index{b@$B_U(a,r)$}$B_U(a,r)=\{x \in U \ | \ ||x-a||<r\}$ is the ball of radius $r$ around $a$. Sometimes we omit the ambient space $U$ in this notation.
\end{notation} 

\subsection{Content of the paper and readers guide}

\begin{structure} \label{struct}
In the rest of the introduction we will cite well-known results about matrix convex sets and free spectrahedra, which are needed for the further exposition. Our methods incorporate techniques from real algebraic geometry, $C^*$-algebras and the theory of completely positive maps. Since most readers will be familiar with the first two branches of mathematics, we will mention quickly what we need from there in the appendix (real closed fields, Tarski transfer principle, Finiteness theorem for semialgebraic classes, infinitesimals and standard part; characterization of $C^*$-algebras as closed subalgebras of the algebra of bounded linear operators on a Hilbert space, basic representation theory, Burnsides theorem). The theory of completely positive maps is helpful for the analysis of matrix convex sets and conceptually very nice, however not completely necessary. In most cases one can get away with the Effros-Winkler separation technique (Chapter 2.1) as a replacement. Some papers combine these two concepts, even though one is enough for most purposes. Therefore complete positivity will be further explained in the appendix and it is up to the reader to read Chapter 2.1 or the appendix or both. As standard references we give \cite{BCR} (real algebraic geometry), \cite{D} ($C^*$-algebras) and \cite{P} (completely positive maps). \\[0.2cm]
Similar to the theory of ordinary convexity, separation techniques are essential and valuable for the analysis of matrix convex sets. The Effros-Winkler separation will be explained in Chapter $2$. We will generalize it in order to also handle non-closed sets (\autoref{effrosnormalexp}; in order to expose matrix extreme points) and even situations where only weak separation is possible (\autoref{exposs} in the appendix). For the latter purpose we will allow separating linear forms to attain values in a real closed extension field of $\R$.\\[0.2cm]
In Chapter 3 we will assign two numbers $\text{pz}(S),\text{kz}(S) \in \N_0 \cup \{\infty\}$ to each matrix convex set $S$. These numbers will indicate whether the whole $S$ is generated by one level $S(k)$ in a maximal/minimal fashion. We analyze how these numbers are affected by taking the polar of $S$ (\autoref{dualo}). As an outcome we will see how the $\text{pz}$-number of a free spectrahedron encodes how small the maximal size of the blocks occuring in the block diagonalization of a pencil description of the given free spectrahedron can be (\autoref{konz} and \autoref{polyh}). 
\\[0.2cm]
In \cite{HM} Helton and McCullough characterized free spectrahedra as "closures of matrix convex free basic open semialgebraic" sets. A new proof and slight generalization of Helton and McCulloughs theorem is given in Chapter 4. The employed techniques will also prove to be fruitful in later chapters. The main results are \autoref{existence} and \autoref{gen}. \\[0.2cm] 
In Chapter 5 we will see that every free spectrahedron $\mathcal{D}_\mathfrak{L}$ can be expressed as an intersection of "irreducible" spectrahedra $\mathcal{D}_{\mathfrak{L_j}}$. Under a natural minimality condition, the pencils $\mathfrak{L_j}$ will be uniquely determined (up to unitary equivalence). This will enable us to decompose the pencil $\mathfrak{L}$ into a direct sum of pencils involving all the $\mathfrak{L_j}$ (\autoref{gleichi}). As a consequence we obtain a new/simplified proof of the Gleichstellensatz \cite[Theorem 3.12]{HKM4} and its generalization \cite[Theorem 1.2]{Z}. \\[0.2cm]
Furthermore we will consider a monic linear pencil $\mathfrak{L}$ and study the system of determinants $\det_k \mathfrak{L}(\mathcal X)$ as polynomials on Hermitian $k \times k$-matrices ($k \in \N$). We will see how a decomposition of $\det_k \mathfrak{L}$ in factors in each level $k$ gives a decomposition of the free locus $\mathcal{Z}(\mathfrak{L})=\{ X \in \S^g \ | \ \ker \mathfrak{L}(X) \neq \{0\} \}$ as a union of "irreducible loci" (\autoref{dedecomp}). This will lead to a characterization of pencils which are "irreducible" and "minimal" (\autoref{minirrchar}). \\[0.2cm]
One can define different notions of extreme points of matrix convex sets. These will be investigated in Chapter 6.  
We will give a new proof of the weak free Krein-Milman theorem from \cite{WW} for compact matrix convex sets $S$ in $\S^g$ and generalize it to a free Minkowski theorem (\autoref{krain}). For sets $S$ with finite $\text{kz}$-number we will strengthen the theorem. We prove that the matrix convex hull of the absolute extreme points of $S$ equals $S$ (\autoref{abso}). This culminates in another proof of the classical Gleichstellensatz (\autoref{gleichcor}). Additionally we introduce matrix exposed points, characterize them (\autoref{grund2}) and prove a free Straszewicz theorem (\autoref{straszewicz}).
Furthermore we give examples that compact matrix convex sets do not need to have absolute extreme points (\autoref{abexex}). Finally we prove a general Gleichstellensatz/strong free Krein-Milman theorem for compact matrix convex sets which shows that every such set $S$ admits a smallest description as the matrix convex hull of an operator tuple $L$ (\autoref{secondhalf}). In the tuple $L$ the absolute extreme points of $S$ and a generalized form of absolute extreme points are encoded. \\[0.2cm]
Afterwards we will take the chance to pause and give some examples of the applications of the accumulated theory up to Chapter 7. References to the examples are given in the earlier chapters. However some techniques are introduced only after the references. \\[0.2cm]
The purpose of Chapter 8 is to analyze the connection between the sequence of determinants $(\det_k \mathfrak{L}(\mathcal X))_{k \in \N}$ of a linear pencil $\mathfrak{L}$ and $\pz(\mathcal{D}_\ml)$. We make some beginning steps towards answering the question when a sequence of RZ polynomials comes from the determinant of a pencil. \\[0.2cm]  
Finally, Chapter 9 contains an analysis of the degrees of the determinants $\det_k \ml$ of a monic linear pencil $\mathfrak{L}$. We show that there is $N \in \N$ and $\ep>0$ such that $\deg(\det_k \mathfrak{L})=k \ep$ for $k \geq N$ (\autoref{degdet}) and that this $N$ cannot be set as $1$ in general (\autoref{gegen}).\\[0.2cm]
At the end, we have included an index for the notation.
\end{structure}

\subsection{More facts on matrix convex sets}

In this section we collect some important facts and techniques for the study of matrix convex sets.

\begin{proposition} \label{mainz}
Let $B-C \ov X$ be a linear pencil. Then $\mathcal{D}_{B-C \ov X}$ is matrix convex.
\end{proposition}

\begin{proof}
Let $r \in \N$, $A_j \in \mathcal{D}_{B-C \ov X}(k_j)$ and $V_j \in C^{k_j \times k}$ for $j \in \{1,...,r\}$ such that $A:=\sum_{j=1}^r V_j^* A_j V_j$ and $I=\sum_{j=1}^r V_j^* V_j$. Then we have $(B-C \ov X)(A)=B \otimes I + \sum_{i=1}^g C_i \otimes (\sum_{j=1}^r V_j^* (A_j)_i V_j)=\sum_{j=1}^r (I \otimes V_j^*) (B \otimes I + \sum_{i=1}^g C_i \otimes (A_j)_i )(I \otimes V_j)=\sum_{j=1}^r (I \otimes V_j^*) (B-C \ov X)(A)(I \otimes V_j) \succeq 0$
\end{proof}

\begin{definition} We say that a set $S \subseteq \S^g$ is bounded/open/closed/compact if the respective level sets $S(k)$ are bounded/open/closed/compact. A matrix convex set $S$ is bounded if and only if $S(1)$ is bounded (\autoref{easy}). We write \index{i@$\inte$}$\inte(S)(k)$ for the interior points of $S(k)$. This hierarchy of sets constitutes $\inte(S)$. Similarly we define the closure $\ov{S}$.
We say that $S$ contains an open neighborhood of $0$ if an $\ep>0$ exists such that
\begin{align*}
B_{\S^g}(0,\ep)=\left\{A \in \S^g \ \middle| \ ||A||<\ep \right\} \subseteq S
\end{align*} (for matrix convex $S$ this is equivalent to the fact that $0$ is an interior point of $S(1)$ (\autoref{easy2})).
We say that a subset $T \subseteq \S^g$ is invariant under \textbf{reducing subspaces} if $A \oplus B \in T$ implies $A \in T$. 
\end{definition}

\begin{proposition} \label{easy} Let $S \subseteq \S^g$ be matrix convex and $r \in [0,\infty)$ such that $S(1) \subseteq B_{\R^g}(0,r)$. Then $S \subseteq B_{\S^g}(0,r\sqrt{g})$ is also bounded.
\end{proposition}

\begin{proof}
Let $A \in S(k)$ such that $||A||^2 \geq r^2g$. Choose $i \in \{1,...,g\}$, $v \in \mathcal{S}^{k-1}$ such that $||A_i^2v|| \geq r^2$. WLOG $v$ is an eigenvector of $A_i$. Then we have $|v^* A_i^2 v| \geq r^2$, $v^*Av \in S(1)$ and $||v^*Av||^2 \geq r^2$.
\end{proof}

\begin{proposition} \label{easy2} \cite[Lemma 4.2]{HKM}
Let $S \subseteq \S^g$ be matrix convex and $r \in (0,\infty)$ such that $\ov{B_{\R^g}(0,r)} \subseteq S(1)$. Then $S$ contains all $A \in \S^g$ with $||A|| \leq \frac{r}{g}$ (i.e. $\ov{B_{\S^g}(0,\frac{r}{g})} \subseteq S$).  
\end{proposition}

\begin{proof} \cite[Lemma 4.2]{HKM}
Let $A \in \S^g(\delta)$ such that $||A|| \leq \frac{r}{g}$. Then each $||g A_i|| \leq r$ and we can find $C_i \in \C^{\delta \times \delta}$ unitary such that $C_i^* g A_i C_i$ is diagonal and the diagonal entries are of norm $\leq r$. Therefore $(0,...,0,C_i^* g A_i C_i,0,...,0) \in S$. $S$ is closed under unitary conjugation; thus $(0,...,0,g A_i,0,...,0) \in S$ and $A=\sum_{i=1}^g \frac{1}{g}(0,...,0,g A_i,0,...,0) \in S$.
\end{proof}

\begin{lemma} \label{inte}
Let $S \subseteq \S^g$ be a matrix convex set and $\varphi: \S^g \rightarrow \R$ be affine linear such that $\varphi(S(1))=\{0\}$. Then we can represent $\varphi$ as the evaluation of a polynomial of degree $1$ and are able to evaluate $\varphi$ in points of $\S^g$. We have $\varphi(S)=\{0\}$.  
\end{lemma}

\begin{proof}
Suppose that $A \in S(k)$ and $\varphi(A) \neq 0$. Choose $v \in \mathcal{S}^{k-1}$ such that $v^* \varphi(A) v \neq 0$. Then $v^*Av \in S(1)$ and $\varphi(v^*Av)=v^*\varphi(A)v \neq 0$.
\end{proof}

\begin{lemma} (Free Caratheodory) \label{carad}
Let $A\in \S^g(m)$, $T \subseteq \S^g$ and $A \in \mconv(T)$.
Then there exist $C_1,...,C_{2gm^2+1} \in T$ and matrices $V_1,...,V_{2gm^2+1}$ of appropriate sizes such that $I_m=\sum_{j=1}^{2gm^2+1} V_j^*V_j$ and $A=\sum_{j=1}^{2gm^2+1} V_j^* C_j V_j$.
\end{lemma}

\begin{proof}
We find $d \in \N$, $C_1,...,C_{d} \in T$ and $V_1,...,V_{d}$ such that $\sum_{j=1}^{d} V_j^*V_j=I_m$ and $\sum_{j=1}^{d} V_j^* C_j V_j=A$. Let $d>2gm^2+1$. We conclude that the matrices $(V_j^*C_jV_j \oplus V_j^*V_j)-(V_1^*C_1V_1\oplus V_1^*V_1)$ $(j \in \{2,...,d\})$ are $\R$-linearly dependent. Choose $\lambda=(\lambda_2,...,\lambda_d) \in \R^{d-1} \setminus \{0\}$ with $\sum_{j=2}^d \lambda_j (V_j^*C_jV_j-V_1^*C_1V_1)=0$, $\sum_{j=2}^d \lambda_j (V_j^*V_j-V_1^*V_1)=0$ and set $\lambda_1=-\sum_{j=2}^d \lambda_j$. Now we have for all $\al \in \R$ that $A=\sum_{j=1}^d (1-\al \lambda_j) (V_j^*C_jV_j)$. For $\al \in\R$ small we can write $A=\sum_{j=1}^d (\sqrt{1-\al\lambda_j}V_j^*X_j\sqrt{1-\al\lambda_j}V_j)$. It is straightforward that we still have $\sum_{j=1}^d (\sqrt{1-\al\lambda_j}V_j^*\sqrt{1-\al\lambda_j}V_j)=I$. Now choose $\al \in \R$ in such a way that all $(1- \al \lambda_j)$ are nonnegative and one is zero. We have reduced one summand in the description of $A$.
\end{proof}

\begin{cor} \label{kompakti} \cite[Proposition 2.5]{DDSS}
Let $L \in \S^g$. Then $\mconv(L)$ is compact. \hfill\qedsymbol
\end{cor}

The following is a nice characterization of matrix convex sets, which is seemingly easier to verify in practice because it does not feature complicated matrix convex combinations, which can be difficult to calculate with.

\begin{lemma} \label{konv} \cite[Lemma 2.3]{HKM3} Let $S \subseteq \S^g$. Then the following is equivalent:
\begin{enumerate}[(a)]
\item $S$ is matrix convex.
\item For each level $k \in \N$ the set $S(k)$ is convex. Furthermore $S$ is invariant under unitary tranformations, taking direct sums and reducing subspaces.
\end{enumerate}
\end{lemma}

\begin{proof} \cite[Lemma 2.3]{HKM3}
$(a) \implies (b)$: This is easy. \\[0.2cm]
$(b) \implies (a)$: Suppose $S$ fulfills (b) and $A \in \mconv(S)(\delta)$. Then we find $r \in \N$, $B_j \in S(k_j)$ and matrices $V_j\in \C^{k_j \times \delta}$ for $j \in \{1,...,r\}$ satisfying $A=\sum_{j=1}^r V_j^* B_j V_j$ and $I=\sum_{j=1}^r V_j^* V_j$. We set $k=\sum_{j=1}^r k_j$, $B=\bigoplus_{j=1}^r B_j$, $V=\begin{pmatrix} V_1 \\ \vdots \\ V_r \end{pmatrix}$ and obtain $A=V^*BV$ as well as $V^*V=I$. As $S$ is closed under direct sums, we know $B \in S$. Extend $V$ to a unitary $W: \C^\delta \oplus \C^{k-\delta} \mapsto \C^{k}$. Let $P: \C^{k} \rightarrow \C^{\delta}$ be the projection on the first $\delta$ components. Then $A=P C P^*$ where $C=W^*B W \in S$. Now write $C=\begin{pmatrix} A & E \\ E^* & F \end{pmatrix}$.
\begin{align*}
\begin{pmatrix}
A & 0 \\ 0 & F
\end{pmatrix}
=
\frac{1}{2}\begin{pmatrix}
A & E \\ E^* & F
\end{pmatrix} + \frac{1}{2}\begin{pmatrix}
1 & 0 \\ 0 & -1
\end{pmatrix}\begin{pmatrix}
A & E \\ E^* & F
\end{pmatrix}
\begin{pmatrix}
1 & 0 \\ 0 & -1
\end{pmatrix} \in S
\end{align*}   
Here we used that $S$ is convex and closed under unitary conjugation. Since $S$ is also closed under taking reducing subspaces, we get $A \in S$.
\end{proof}

\begin{rem} \label{determing}
Let $S \subseteq \S^g$ be a matrix convex set. Then the k-th level of $S$ determines all lower levels. Indeed, fix $A \in S(k)$ and let $s <k$. Choose an arbitrary projection $P$ such that $PAP^* \in S(k-s)$. Let $B \in \S^g(s)$. Then we have $B \in S(s) \Longleftrightarrow PAP^* \oplus B \in S(k)$.  
\end{rem}

\begin{proposition} \cite[Proposition 2.1]{HKM4} \label{monisieren}
Let $T-H \ov X \in S(\NC)_1^{\delta \times \delta}$ be a linear pencil and $S=\{A \in \S^g(k) \ | \ [T-H \ov X](A) \succeq 0\}$. Suppose that $0$ is in the interior of $S$. Then there exists a monic linear pencil $\mathfrak{L}$ of size $\delta$ such that $\mathcal{D}_{\mathfrak{L}}=S$.
\end{proposition}

\noindent An important technique in the study of matrix convexity is to use projections to transform statements about a matrix convex set to statements about a certain level set. The following observation is one of the key properties of free spectrahedra.

\begin{lemma} \label{specki} (Projection lemma) \cite[Lemma 2.3]{HKMS}
Let $B \in S\C^{\delta \times \delta},C \in \S^g(\delta)$ and $k \in\N$. Then
$\mathcal{D}_{B-C \ov X}(k)=\{A \in \S^g(k) \ | \ \forall \text{ projections }P: \C^k \rightarrow \im(P): \dim(\im(P)) \leq \delta \implies PAP^* \in \mathcal{D}_{B-C \ov X} \}$. 
\end{lemma}

\begin{proof} \cite[Lemma 2.3]{HKMS}
$\mathcal{D}_{B-C \ov X}$ is matrix convex. Thus if $A \in \S^g(k)$ and $P: \C^k \rightarrow \im(P)$ is a projection, then $PAP^* \in \mathcal{D}_{B-C \ov X}$. Now let $A \in \S^g(k)$ and $A \notin \mathcal{D}_{B-C \ov X}$. Choose $v \in \mathcal{S}^{\delta k}$ such that $v^*(B-C \ov X)(A)v<0$ and write $v=\sum_{\al=1}^\delta e_\al \otimes v_\al$ with $v_\al \in \C^k$. Let $P: \C^k \rightarrow \im(P)$ be the projection onto $\spam\{v_1,...,v_\delta\}$. Then we calculate $v^*(B-C \ov X)(PAP^*)v=[(I \otimes P^*)v]^*(B-C \ov X)(A)(I \otimes P^*)v=v^*(B-C \ov X)(A)v<0$.
\end{proof}

\section{Separation techniques for matrix convex sets}

\noindent An important technique in classical convexity consists in the separation of a point from a convex set that does not not contain that point. The following theorem transfers the Hahn-Banach separation theorem (bipolar theorem) to the matrix convex setting. Version (a) was originally stated for bounded sets $S$ however the assumption is not used in the proof. It shows that the role of linear functionals in the classical convex setting in $\R^g$ can be taken over by linear matrix inequalities in the free setting. 

\begin{lemma} \label{trenn} \cite[Proposition 6.4]{HM}, \cite[Theorem 5.4]{EW} \label{abc} (Effros-Winkler theorem, separation of matrix convex sets and points via pencils)
\begin{enumerate}[(a)]
\item Let $S\subseteq \S^g$ be a closed matrix convex set with $0 \in S$. Let $Y \notin S$. Then we can find a monic linear pencil $\mathfrak{L}$ of the same size of $Y$ which separates $S$ from $Y$ in the sense that $\mathfrak{L}(A) \succeq 0$ for all $A \in S$ and $\mathfrak{L}(Y)$ is not positive semidefinite.
\item Let $S\subseteq \S^g$ be an open matrix convex set with $0 \in S$. Let $Y \in \partial S$. Then we can find a monic linear pencil $\mathfrak{L}$ of the same size of $Y$ which separates $S$ from $Y$ in the sense that $\mathfrak{L}(A) \succ 0$ for all $A \in S$ and $\ker \mathfrak{L}(Y) \neq \{0\}$.
\end{enumerate}
\end{lemma}

\noindent There are two ways to prove this result. First, one can interpret matrix convex combinations as images of completely positive maps, an object originating from the theory of $C^*$-algebras. Hence one is able to use the power of this better developed area of mathematics to deal with matrix convex sets. The second option is to apply the Effros-Winkler separation technique which aims to translate separating linear forms into separating pencils. \\[0.2cm]
In this paper we present both ways. Material on completely positive maps, its connections to matrix convexity and one proof of \autoref{trenn} are located in the appendix. We will need some of those results also later for the study of absolute extreme points. The Effros-Winkler separation technique is explained in the rest of this chapter starting after \autoref{notzero}. We will need some variations of this concept also later for the study of matrix exposed points. \\[0.2cm]
It is possible to understand most of the paper having read only one of those chapters. For readers unfamiliar with completely positive maps we recommend to read the chapter about the Effros-Winkler separation technique and to skim over the chapter in the appendix quickly. For readers familiar with completely positive maps we recommend to read this chapter until \autoref{notzero} and the chapter in the appendix while only skimming over the chapter of the Effros-Winkler separation technique. \\[0.2cm] We can reformulate the separation theorem by defining the notion of a polar of a matrix convex set and prove a bipolar theorem.

\begin{definition} \label{polar}
Let $S\subseteq \S^g$ be a matrix convex set. Then we define the \textbf{free polar} of $S$ to be
\begin{align*}
S^\circ=\left\{ H \in \S^g \ \middle| \ \forall A \in S: I \otimes I - \sum_{i=1}^g H_i \otimes A_i \succeq 0\right\}=\{H \in \S^g \ | \ S \subseteq \mathcal{D}_{\mathfrak{H}}\}.
\end{align*} $S^\circ$ is a matrix convex set because $S^\circ=\bigcap_{L \in S} \mathcal{D}_\ml$.
\end{definition}

\begin{theo} \label{simpolara}\cite[Proposition 4.3]{HKM} (Bipolar theorem)
Let $S \subseteq \S^g$. Then $S^{\circ \circ}=\ov{\mconv(S \cup \{0\})}$.
\end{theo}

\begin{proposition} \label{simpolarb} \cite[Proposition 4.3]{HKM}
Let $S \subseteq \S^g$ be a matrix convex set. Then $0 \in \inte(S^\circ) \Longleftrightarrow S$ is bounded.
\end{proposition}

\begin{proof} \cite[Proposition 4.3]{HKM}
$"\Longleftarrow"$: Let $0 \in \inte(S^\circ)$. Hence there is $r>0$ such that $B(0,r) \subseteq S^\circ(1)$. In particular for $i \in \{1,...,g\}$ we have $\pm r e_i \in S^\circ$. Let $A \in \S^g$. Then $(1 \pm r e_i \ov X)(A)=I \pm r A_i \succeq 0$. Therefore $r^2 A_i^2 \preceq I$ and $\sum_{i=1}^g A_i^2 \preceq \frac{g}{r^2}$ \\[0.2cm]
$"\Longrightarrow"$: Let $S$ be bounded. Fix $r \in [0,\infty)$ such that $||A||^2 \leq r^2$ for all $A \in S$. Now let $L \in \S^g$ with $||L|| \leq \frac{1}{g r}$. We have that $(L_i \otimes A_i)^2 \preceq ||L||^2 r^2 \preceq \frac{1}{g^2}$ and hence $||L_i \otimes A_i|| \leq \frac{1}{g}$ for all $i \in \{1,...,g\}$. Thus $\ml(A) \succeq 0$.
\end{proof}

\begin{rem} \label{simpolarbb}
Let $S \subseteq \S^g$ be matrix convex and closed with $0 \in S$. Then with the help of the bipolar theorem we get [$0 \in \inte(S) \Longleftrightarrow S^\circ$ is bounded] as a corollary of \autoref{simpolarb}. $\Longrightarrow$ can be proved without usage of the bipolar theorem by basically copying the second part of the proof of \autoref{simpolarb}; we need this fact for the proof of \autoref{simpolara} in the appendix.   
\end{rem}

\begin{proposition} \label{simpolarc}
Let $S \subseteq \S^g$ be a matrix convex set. Then $S^\circ(\delta)=\mconv(S(\delta))^\circ(\delta)$
\end{proposition}

\begin{proof}
Clearly we have $S^\circ(\delta) \subseteq \mconv(S(\delta))^\circ(\delta)$. Now let $H \in \mconv(S(\delta))^\circ(\delta)$ and assume there is some $A \in S(k)$ and $v \in \C^{\delta k}$ such that $v^*\mh(A)v <0$. If we write $v=\sum_{\al=1}^\delta e_\al \otimes v_\al$ with $v_\al \in \C^k$ and define $P: \C^k \rightarrow \im(P)$ to be the projection onto $M(A,v)_\mh=\spam(v_1,...,v_\delta)$, then $v^* \mh(PAP^*) v <0$, however $PAP^* \in \mconv(S(\delta))$.
\end{proof}

\begin{rem} \label{notzero}
In case that $0 \notin S$ in \autoref{trenn}, the statements are still true if one deletes the word "monic". The reason is that we can shift $S$ and $Y$ in such a way that $0 \in S$.
\end{rem}

\subsection{A closer look at the Effros-Winkler separation technique}

In this section we want to present and refine the separation technique of Effros and Winkler \cite{WW}. The aim is to give a certificate that a point $A \in \S^g(\delta)$ is not contained in a closed matrix convex set $S$ by the means of a linear matrix inequality. The idea is the following: One applies the usual Hahn-Banach separation to separate $A$ from $S(\delta)$ with the help of an affine-linear function $\ell$. Now one wants to translate this to a linear matrix inequality. By compressing elements of $S$ to elements of $S(\delta)$ one is able to apply $\ell$. Up to the constant part of $\ell$ this procedure translates $\ell$ into a homogeneous linear pencil $L \ov X$ with $L \in \S^g(\delta)$. However translating the constant part is difficult and non-constructive. \\[0.2cm]
Therefore, to obtain the Effros-Winkler result for matrix cones is easier (see \cite[Lemma 2.1]{FNT} and \autoref{mcone}). So in order to separate $A$ from $S$, it is convenient to homogenize the problem, separate the resulting matrix cone $S'$ from the point $(I,A)$ and after dehomogenize (cf. \autoref{homo}). However it is not clear how one can homogenize $\ell$. Therefore this approach has an even more non-constructive component. \\[0.2cm]
For our later results we have to strengthen the Effros-Winkler separation to be able to cover also situations where the matrix convex set is not closed (strong separation for non-closed sets) or only weak separation is possible (for instance separating a non-exposed extreme point $x$ of a convex set $S$ from $S \setminus \{x\}$). The strong separation theorem for non-closed sets will be later needed to analyze matrix exposed points (cf. \autoref{grund2}). \\[0.2cm] Since the proof of the weak separation theorem is a bit technical and a slightly weaker statement can be recovered by using a homogenization trick, we have moved it to the appendix. We will use the weak separation to give a proof of the free Minkowski theorem in the appendix. 

\begin{definition}
We set \index{T@$\mathcal{T}_\delta$}$\mathcal{T}_\delta=\{ T \in S\C^{\delta \times \delta} \ | \ T \succeq 0, \tr(T)=1\}$.
\end{definition}

\begin{proposition} \label{effri1} \cite[Lemma 6.2]{HM}
Let $S \subseteq \S^g$ be matrix convex, $0 \in S$ and $A \in \S^g(\delta) \setminus S$. Suppose there is a linear $\varphi: \S^g(\delta) \rightarrow \R$ such that $\varphi(A)>1$ and $\varphi(B) \leq 1$ for all $B \in S(\delta)$. For $B \in S(k)$ and a contraction $V \in \C^{k\times \delta}$ define $f_{B,V}: S\C^{\delta \times \delta} \rightarrow \R, T \mapsto \text{tr}(V \ov{T} V^*)-\varphi(V^* B V)$.
Then the set $\mathcal{F}=\{f_{B,V} \ | \ k \in \N, \ B \in S(k), V\in \C^{k \times \delta}, V^*V \preceq I\}$ is convex. For $f_{B,V} \in \mathcal{F}$ there is $T \in \mathcal{T}_\delta$ such $f_{B,V}(T) \geq 0$.
\end{proposition}

\begin{proof}
\cite[Lemma 6.2]{HM} Let $f_{B,V} \in \mathcal{F}$. Choose $w \in \mathcal{S}^{\delta-1}$ such that $||Vw||=||V||$ and set $T=\ov{w}\ \ov{w}^*$. Then $T$ is positive semidefinite and $\tr(T)=\tr(\ov{w}^*\ov{w})=1$. We have $f_{B,V}(T)=\text{tr}(V ww^* V^*)-\varphi(V^* B V)= \text{tr}((V w)^* V w)-\varphi(V^* B V)=||V||^2-\varphi(V^* B V)=||V||^2\left(1-\varphi\left(\frac{V^*}{||V||}B\frac{V}{||V||}\right)\right) \geq 0$. \\[0.2cm]
Let $r \in \N$ and $B_j \in S(k_j)$, $V_j \in \C^{k_j \times \delta}$ contractions, $\lambda_j \in [0,1]$ for $j \in \{1,...,r\}$ such that $\sum_{j=1}^r \lambda_j=1$. Define 
\begin{align*}
B=\bigoplus_{j=1}^r B_j \text{  and  } V^*=\begin{pmatrix}\sqrt{\lambda_1}V_1^* & \hdots & \sqrt{\lambda_r}V_r^*\end{pmatrix}
\end{align*}
Then $B \in S$, $V^*BV=\sum_{j=1}^r \lambda_j V_j^* B_j V_j$ and $V^* V=\sum_{j=1}^r \lambda_j V_j^* V_j \preceq I$ and for $T \in S\C^{\delta \times \delta}$ we have $\text{tr}(V \ov{T} V^*)=\text{tr}(\ov{T} V^* V)=\sum_{j=1}^r \lambda_j \text{tr}(\ov{T} V_j^* V_j)=\sum_{j=1}^r \lambda_j \text{tr}(V_j \ov{T} V_j^*)$. This shows that $f_{B,V}=\sum_{j=1}^r \lambda_j f_{B_j,V_j}$ and that $\mathcal{F}$ is convex.
\end{proof}

\begin{lemma} \label{effri2} \cite[Lemma 6.1]{HM}
Suppose $\mathcal{F}$ is a convex set of affine-linear mappings $f: S\C^{\delta \times \delta} \rightarrow \R$ and that there is a concave function $\Psi: \mathcal{F} \rightarrow \R$ such that for all $f \in \mathcal{F}$ there is $T \in \mathcal{T}_\delta$ such that $f(T) \geq \Psi(f)$. Then there is $T \in \mathcal{T}_\delta$ such that $f(T) \geq \Psi(f)$ for all $f \in \mathcal{F}$.
\end{lemma}

\begin{proof} \cite[Lemma 6.1]{HM}
$\mathcal{T}_\delta$ is compact. Thus for $f \in \mathcal{F}$ the set $\{T \in \mathcal{T}_\delta \ | \ f(T) \geq \Psi(f)\}$ is also compact. In order to show $\bigcap_{f \in \mathcal{F}}\{T \in \mathcal{T}_\delta \ | \ f(T) \geq \Psi(f)\}\neq \emptyset$, it is enough to show that every finite intersection of those sets is non-empty. \\[0.2cm]
So let $m \in \N$, $f_1,...,f_m \in \mathcal{F}$. We have to show that $\bigcap_{j=1}^m\{T \in \mathcal{T}_\delta \ | \ f_j(T) \geq \Psi(f_j)\}\neq \emptyset$.
We set 
\begin{align*}
F: \mathcal{T}_\delta \rightarrow \R^m: T \mapsto (f_1(T),...,f_m(T)). 
\end{align*}
We want to show $F(\mathcal{T}_\delta) \cap \prod_{j=1}^m[\Psi(f_j),\infty) \neq \emptyset$. Assume the opposite is true. As $F(\mathcal{T}_\delta)$ is convex and compact, we can choose an affine linear $h=\sum_{j=1}^m h_j X_j+h_0$ and $t \in (-\infty,0)$ such that $h(\prod_{j=1}^m[\Psi(f_j),\infty)) \subseteq [0,\infty)$ and $h(F(\mathcal{T}_\delta)) \subseteq (-\infty,t]$. For each $\lambda>0$ and $k\in \{1,...,m\}$ we have $0 \leq h((\Psi(f_1),...,\Psi(f_m))+\lambda e_k)=h_0 + \sum_{j=1}^m h_j \Psi(f_j) + \lambda h_k$ which implies $h_k \geq 0$. Without loss of generality we can suppose that $\sum_{j=1}^m h_j=1$. Now set $f=\sum_{j=1}^m h_j f_j \in \mathcal{F}$. \\[0.2cm] We know $0 \leq h(\Psi(f_1),...,\Psi(f_m))=\sum_{j=1}^m h_j \Psi(f_j)+h_0$. For every $T \in \mathcal{T}_\delta$ we calculate $f(T)=\sum_{j=1}^m h_j f_j (T)=h(F(T))-h_0 \leq h(F(T)) + \sum_{j=1}^m h_j \Psi(f_j) < \sum_{j=1}^m h_j \Psi(f_j) \leq \Psi(f)$, a contradiction.  
\end{proof}

\begin{cor} \label{effrosnormal} (Effros-Winkler - strong separation) \cite[Proposition 6.4]{HM} \\
Let $S \subseteq \S^g$ be matrix convex, $0 \in S$ and $A \in \S^g(\delta) \setminus S$. Suppose there is a linear $\varphi: \S^g(\delta) \rightarrow \R$ such that $\varphi(A)>1$ and $\varphi(B)\leq 1$ for all $B \in S(\delta)$. Then there exists $T \in S\C^{\delta \times \delta}$ and $H \in \S^g(\delta)$ such that:
\begin{align*}
\text{For all } B \in S: [T-H\ov X](B) \succeq 0, \ [T-H\ov X](A) \nsucceq 0
\end{align*}
\end{cor}

\begin{proof} \cite[Proposition 6.4]{HM}
We apply \autoref{effri1} and \autoref{effri2} with $\mathcal{F}$ defined like in \autoref{effri1} and $\Psi:=0$. We conclude that there is $T \in \mathcal{T}_\delta$ such that $f(T) \geq0$ for all $f \in \mathcal{F}$. Extend $\varphi$ to a $\C$-linear functional $\varphi: (\C^{\delta \times \delta})^g \rightarrow \C$. By the Riesz representation theorem we can find matrices $H_1,...,H_g \in \C^{\delta \times \delta}$ such that $\varphi(C)=\sum_{i=1}^g \tr(\ov{H_i}^* C_i)$. It is easy to see that the $\ov{H_i}$ have to be Hermitian. Hence $H_i$ is Hermitian and $H_i^T=\ov{H_i}$. \\[0.2cm]
Now let $B \in \S^g(k)$ and $v \in \C^{k \delta}$. Write $v=\sum_{\al=1}^\delta e_\al \otimes v_\al$ with $v_\al \in \C^k$. Define the matrix $V=\begin{pmatrix}v_1 & \hdots & v_\delta \end{pmatrix} \in \C^{k \times \delta}$. We calculate (and denote the indices of a matrix as upper case letters)
\begin{align*}
&v^*[T-H\ov X](B)v=v^*(T\otimes I)v - \sum_{i=1}^g v^*(H_i\otimes B_i)v \\
&=\sum_{\al=1, \beta=1}^\delta \langle e_\al, T e_\beta \rangle \langle v_\al, I v_\beta \rangle - \sum_{i=1}^g \sum_{\al=1,\beta=1}^\delta \langle e_\al, H_i e_\beta \rangle \langle v_\al, B_i v_\beta \rangle \\
&=\sum_{\al=1, \beta=1}^\delta T^{\alpha,\beta} (V^*V)^{\al,\beta} - \sum_{i=1}^g \sum_{\al=1,\beta=1}^\delta H_i^{\al,\beta} (V^*B_iV)^{\al,\beta} \\
&=\tr (T^TV^*V) - \sum_{i=1}^g \tr(H_i^T V^*B_iV)=\tr (V\ov{T}V^*) - \sum_{i=1}^g \tr(\ov{H_i}V^*B_iV)=f_{B,V}(T)
\end{align*}
If $v \in \mathcal{S}^{k \delta -1}$ and $w \in \mathcal{S}^{\delta-1}$ we have $||Vw|| \leq \sum_{\al=1}^\delta |w_\alpha| \ ||v_\al|| \leq 1$ by the Cauchy-Schwarz inequality. This means that $V$ is a contraction.
So in case that $B \in S$ we have $v^*[T-H\ov X](B)v \geq 0$ because of $V^*BV \in S$. On the other hand let $e=\sum_{\al=1}^{\delta} e_\al \otimes e_\al$ and set $V=I$ (i.e. $v_\al=e_\al$). Then $e^*[T-H\ov X](A)e=\tr(\ov{T}) -\varphi(A)<0$.
\end{proof} 

\begin{proof} (of \autoref{trenn})
(a) By the Hahn-Banach separation theorem from the theory of ordinary convexity we can find $\varphi: \S^g(\delta) \rightarrow \R$ linear such that $\varphi(Y)>1$ and $\varphi(S(\delta)) \subseteq (-\infty,1]$. Now we apply \autoref{effrosnormal} to construct a pencil $T-H \ov X$ such that for all $B \in S$ we have $[T-H\ov X](B) \succeq 0$ and $\ [T-H\ov X](Y) \nsucceq 0$. There is $\ep>0$ such that $[\ep I+T-H\ov X](Y)\nsucceq 0$. Now let $D=\sqrt{\ep I+T}$. Then $I-D^{-1}HD^{-1} \ov X$ is the desired pencil.
\end{proof}

\begin{proof} (of \autoref{simpolara})
It is easy to see that the polar is always matrix convex, closed and contains $0$. Due to $S \subseteq S^{\circ \circ}$ this shows one direction. On the other hand suppose $A \in \S^g(\delta)$ such that $A \notin \ov{\mconv(S \cup \{0\})}=:T$. $T$ is also matrix convex and the Effros-Winkler theorem says that there is $L \in \S^g(\delta)$ such that $T \subseteq \mathcal{D}_{\mathfrak L}$ however $A \notin \mathcal{D}_{\mathfrak L}$. Hence $L \in S^{\circ}$ and $A \notin S^{\circ \circ}$.
\end{proof}

The next corollary is very important for the study of free spectrahedra.

\begin{cor} \label{polari} \cite[Theorem 4.6, Proposition 4.9]{HKM}
Let $L \in \S^g$. Then we have $\mathcal{D}_{\mathfrak{L}}^\circ=\{H \in \S^g \ | \ \mathcal{D}_\ml \subseteq \mathcal{D}_\mh\}=\mconv(L,0)$ and $\mconv(L)^\circ=\mathcal{D}_{\mathfrak{L}}$. If $\mathcal{D}_{\mathfrak{L}}$ is even bounded, then $\mconv(L,0)=\mconv(L)$.
\end{cor}

\begin{proof}
First claim: Let $H \in \mconv(L,0)$. This means there are matrices $V_1,...,V_m$ such that $\sum_{j=1}^m V_j^* V_j \preceq I$ and $\sum_{j=1}^m V_j^* L V_j = H$. Let $A \in \mathcal{D}_\ml$. Then $\mh(A)=(I-\sum_{j=1}^m V_j^* V_j) \otimes I + \sum_{j=1}^m (V_j^* \otimes I) \ml(A) (V_j \otimes I) \succeq 0$. Hence $\mconv(L,0) \subseteq \mathcal{D}_\ml^\circ$. \\[0.2cm] If $A \in \S^g(\delta) \setminus \mconv(L,0)$, we know from \autoref{abc} that there exists $H \in \S^g(\delta)$ such that $0 \preceq \mh(B) \approx (I-B \ov X)(H)$ for all $B \in \mconv(L,0)$ and $0 \npreceq \mh(A)\approx (I-A\ov X)(H)$. Thus we have $H \in \mathcal{D}_\ml$ and $A \notin \mathcal{D}_\ml^\circ$. This shows $\mconv(L,0) \supseteq \mathcal{D}_\ml^\circ$. \\[0.2cm]
Second claim: By bipolarity we gain $\mathcal{D}_\ml=\mathcal{D}_\ml^{\circ \circ}=\mconv(L,0)^\circ=\mconv(L)^\circ$. \\[0.2cm]
Third claim: Now suppose $\mathcal{D}_{\mathfrak{L}}$ is bounded. We show that $0 \in \mconv(L)$. We present the proof of \cite[Proposition 4.2]{HKM5}. Let $\delta=\size(L)$. In order to show $0 \in \mconv(L)$ set $S=\text{conv}(\{v^*Lv \ | \ v \in \mathcal{S}^{\delta-1}\})$. We have to verify $0 \in S$. Assume $0 \notin S$. Then there exists a linear functional $\varphi: \R^{g} \rightarrow \R$ such that $\varphi(S) \subseteq \R_{\leq 0}$. Set $x_i=\varphi(e_i)$ for $i \in \{1,...,g\}$. Let $v \in \mathcal{S}^{\delta-1}$, $r>0$. Then we have $v^*\ml(r x)v=1-r\left(\sum_{i=1}^g x_i (v^* L v)_i\right)=1-r\varphi(v^*Lv)>0$, which contradicts the boundedness of $\mathcal{D}_\ml$.  
\end{proof}

We have seen how to separate a point from a closed matrix convex set with the help of a linear pencil. An exposed extreme point of a closed ordinary convex set $S \subseteq \R^n$ is a point $A \in S$ which can be separated from $S \setminus \{A\}$ strictly with an affine-linear function $\psi$ (meaning $\psi(A) > \psi(B)$ for all $B \in S \setminus \{A\}$). However the set $S \setminus \{A\}$ is only closed if $S=\{A\}$. To characterize matrix exposed points of matrix convex sets, we need some separation theorem for non-closed matrix convex sets. 

\begin{definition}
Let $R$ be a real closed extension field of $\R$ and let $a,b \in R_{>0}$. We write \index{g@$>>$}$a>>b$ if $a>Nb$ for all $N \in \N$. For $b \in R$ with an $N \in \N$ such that $-N < b < N$ we denote by \index{st}$\text{st}(b) \in \R$ the standard part of $b$ (a concrete definition is given in the appendix). \\[0.2cm] On the algebraic closure $C=R[\ii]$ we have an involution $^*: C \rightarrow C, a+b \ii \mapsto a-b \ii \ (a,b \in R)$, which leaves $R$ invariant. For $A \in C^{r \times k}$ we obtain $A^*$ by transposing $A$ and afterwards applying the involution to each entry. If $A=A^*$, we call $A$ Hermitian and define \index{S@$SC^{k \times k}$}$SC^{k \times k}=\{ A \in C^{k \times k} \ | \ A^*=A\}$. We remark that this notion depends not only on $C$ but also the choice of the real closed field $R$, however this should not cause ambiguity anywhere. We set \index{T@$\mathcal{T}_{R,\delta}$}$\mathcal{T}_{R,\delta}=\{ T \in SC^{\delta \times \delta} \ | \ T \succeq 0, \tr(T)=1\}$ where $T \succeq 0$ means $v^*Tv \geq 0$ for all $v \in C^{\delta}$. For $A \in SC^{\delta \times \delta}$ we write \index{G@$\succeq_\R$}$A \succeq_\R 0$ if $v^*Av \geq 0$ for all $v \in \C^{\delta}$ and $A \succ_\R 0$ if $v^*Av > 0$ for all $v \in \C^{\delta} \setminus \{0\}$. 
\end{definition}

\begin{lemma} \label{vitali}
Let $\mathcal{R}$ be a real closed extension field of $\R$ and $v \in \mathcal{R}^g$. Then there exist $r \in \{1,...,g\}$, $\lambda_1,...,\lambda_{r} \in R$ and $a_i^j \in \R$ such that $v_i=\sum_{j=1}^{r} \lambda_j a_i^j$ for all $i \in \{1,...,g\}$ and $\lambda_1>>...>>\lambda_r>0$. 
\end{lemma}

\begin{proof}
Without loss of generality $v_1>...>v_g>0$. Set $\lambda_1=v_1$. Then write $(0,w)=v-\lambda_1 (1,....,\text{st}\left(\frac{v_g}{v_1}\right))$. For every entry $w_j$ of $w$ we have $\lambda_1>>w_j$. Now continue inductively with $w$.
\end{proof}

\begin{proposition} \label{effri1exp}
Let $S \subseteq \S^g$ be matrix convex, $0 \in S$ and $A \in \S^g(\delta) \setminus S$. Suppose there is a linear $\varphi: \S^g(\delta) \rightarrow \R$ such that $\varphi(A)=1$ and $\varphi(B) < 1$ for all $B \in S(\delta)$. For $B \in S(k)$ and a contraction $V \in \C^{k\times \delta}$ define $f_{B,V}: S\C^{\delta \times \delta} \rightarrow \R, T \mapsto \text{tr}(V \ov{T} V^*)-\varphi(V^* B V)$.
Then the set $\mathcal{F}=\{f_{B,V} \ | \ k \in \N, \ B \in S(k), V\in \C^{k \times \delta}, V^*V \preceq I\}$ is convex. For $f_{B,V} \in \mathcal{F}$ there is $T \in \mathcal{T}_\delta$ such $f_{B,V}(T) > 0$.
\end{proposition}

\begin{proof}
This is basically the same proof as the one of \autoref{effri1}.
\end{proof}

\begin{lemma} \label{effri2exp}
Suppose $\mathcal{F}$ is a convex set of affine-linear mappings $f: S\C^{\delta \times \delta} \rightarrow \R$ such that for all $f \in \mathcal{F}$ there is $T \in \mathcal{T}_\delta$ such that $f(T) > 0$. Then there exists a real closed field extension $R$ of $\R$ and $T \in \mathcal{T}_{R,\delta}$ such that $f(T) > 0$ for all $f \in \mathcal{F}$ ($f$ has a unique extension to an $R$-affine linear map $SR[\ii]^{\delta \times \delta} \rightarrow \R$).
\end{lemma}

\begin{proof}
For $f \in \mathcal{F}$ consider the $\R$-semialgebraic classes
\begin{align*}\{(R,T) \ | \ R \text{ real closed extension field of }\R, \ T \in \mathcal{T}_{R,\delta}, f(T) > 0\}.
\end{align*} 
\autoref{finot} tells us that in order to prove 
\begin{align*}
\bigcap_{f \in \mathcal{F}} \{(R,T) \ | \ R\text{ real closed extension field of }\R, \ T \in \mathcal{T}_{R,\delta}, f(T) >0\}\neq \emptyset
\end{align*}
it is enough to show that every finite intersection of those sets is non-empty.
That the latter is the case is basically the same proof as in \autoref{effri2}.
\end{proof}

\begin{cor} \label{effrosnormalexp} (Effros-Winkler separation for non-closed sets)
Let $S \subseteq \S^g$ be matrix convex, $0 \in S$ and $A \in \S^g(\delta) \setminus S$. Suppose there is a linear $\varphi: \S^g(\delta) \rightarrow \R$ such that $\varphi(A)=1$ and $\varphi(B)<1$ for all $B \in S(\delta)$. Then there exists $T \in S\C^{\delta \times \delta}$ and $H \in \S^g(\delta)$ such that:
\begin{align*}
\text{For all } B \in S: [T-H\ov X](B) \succ 0, \ \ker[T-H\ov X](A) \neq \{0\}
\end{align*}
\end{cor}

\begin{proof} 
The same proof as the one of \autoref{effrosnormal} gives us a real closed extension field $R$ of $\R$ with algebraic closure $C=R[\ii]$, $T \in SC^{\delta \times \delta}$ and $H \in \S^g(\delta)$ such that $T \succeq 0$, $\tr(T)=1$ and
\begin{align*}
\text{for all } B \in S: [T-H\ov X](B) \succ_\R 0, \ e^*[T-H\ov X](A)e=0,
\end{align*} where $e=\sum_{\al=1}^\delta e_\al \otimes e_\al$.
Since $T \succeq 0$, we can find $D \in SC^{\delta \times \delta}$ such that $T=D^2$. With \autoref{vitali} we find $r \in \N$ and $\lambda_1>>...>>\lambda_r>0 \in R$, $D_j \in S\C^{\delta \times \delta}$ such that $D=\sum_{j=1}^r \lambda_j D_j$. In case that $r=1$, we are done. So suppose that $r\geq 2$. As $\tr(D)=1$, it is easy to see that we can choose $\lambda_1=1$. Set $E=\sum_{j=1}^{r-1} \lambda_j D_j$. We want to replace $T=D^2$ by $E^2+D_r^2$. \\[0.2cm] Indeed let $B \in S$. Applying the standard part, we see that $[D_1^2-H \ov X](B) \succeq_\R 0$. Now let $v \in \C^{\delta^2} \setminus \{0\}$ and assume $v^*[E^2+D_r^2 -H \ov X](B)v \leq 0$. We know $v^*[D_1^2-H\ov X](B)v \geq 0$ and $v^*[D_r^2](B)v \geq 0$. The sum of those two terms is $\text{st}(v^*[E^2+D_r^2 -H \ov X](B)v)\leq 0$, thus $v^* (D_r^2 \otimes I) v=v^*[D_r^2](B)v = 0$. However this means that $(D_r \otimes I)v=0$. We conclude that $v^*[E^2+D_r^2 -H \ov X](B) v=v^*[T -H \ov X](B) v > 0$, a contradiction. We have shown $[E^2+D_r^2 -H \ov X](B) \succ_\R 0$ for all $B \in S$.\\[0.2cm]
On the other hand we have $e^*[T-H\ov X](A)e=0$. Clearly $\R \lambda_r^2 \cap \spam_\R(\{\lambda_j \lambda_h \ | \ j,h \in \{1,...,r-1\}\} \cup \{\lambda_j \lambda_r \ | \ j \in \{1,...,r-1\}\})=\{0\}$ which means $0=e^*[D_r^2](A)e=0$. Again we infer that $(D_r \otimes I)e=0$ and $e^*[E^2+D_r^2 -H \ov X](A) e=e^*[T - H \ov X](A) e =0$. We have shown that we can replace $T$ by $E^2+D_r^2$. In the same way we continue and show that we can replace $E^2+D_r^2$ by $(\sum_{j=1}^{r-2} \lambda_j D_j)^2+D_r^2+D_{r-1}^2$ and so on. Inductively we deduce that $\sum_{j=1}^r D_j^2-H \ov X$ is the desired pencil.  
\end{proof}

\begin{proof} (of \autoref{trenn} (b))
From the theory of convexity we know that there exists $\varphi: \S^g(\delta) \rightarrow \R$ linear such that $\varphi(B)<1$ for all $B \in S(\delta)$ and $\varphi(Y)=1$. Now \autoref{effrosnormalexp} tells us that there exists $T \in S\C^{\delta \times \delta}$ and $H \in \S^g(\delta)$ such that:
\begin{align*}
\text{For all } B \in S: [T-H\ov X](B) \succ 0, \ \ker([T-H\ov X](A)) \neq \{0\}
\end{align*}
\autoref{monisieren} tells us that we can replace $T-H \ov X$ by a monic linear pencil $\ml$.
\end{proof}

\section{Projection number and convexity number}

\noindent The following definition is a generalization of the concept of the minimal and maximal matrix convex set generated by a subset of $\R^g$, which was covered in \cite[Chapter 4]{DDSS} and in \cite{FNT}. The corresponding results up to \autoref{duali} are present in \cite{DDSS}. 

\begin{definition} \label{nullst}
\index{pz}\index{kz}\index{p@$\pz_m$}\index{k@$\kz_m$}Let \marginpar{[\autoref{ex3}]} $T=(T_n)_{n \in \N} \subseteq \mathbb{S}^g$ be closed. Define for $m \in \N$ the sets 
\begin{flalign*}
\pz_m(T)&:=\left(\left\{ B \in \S^g(k) \ | \right.\right. \\
 & \left.\left. \ \text{For all projections } P: \C^{k} \rightarrow \im P \text{ of rank } \leq m : PBP^* \in T(\rk(P))\right\}\right)_{k \in \N} \\
\kz_m(T)&:=\mconv{(T(1),...,T(m))}
\end{flalign*}
For matrix convex $T$ we set
\begin{flalign*}
\pz(T) &:= \inf\{ m \in \N \ | \ \pz_m(T)=T \} \in \N \cup \{\infty\} \ \ \ \ \text{the \textbf{projection number} of } T\\
\kz(T) &:= \inf\{ m \in \N \ | \ \kz_m(T)=T \} \in \N \cup \{\infty\} \ \ \ \ \text{the \textbf{convexity number} of } T
\end{flalign*}
\end{definition}

\begin{proposition} \label{erst}
Let $T \subseteq \mathbb{S}^g$ be a closed matrix convex set. Then $\kz_m(T)$ is the smallest matrix convex set $S$ with $S(j)=T(j)$ for $j \in \{1,...,m\}$. On the other hand, $S=\pz_m(T)$ is the largest set satisfying this. If $0 \in T$, then $\pz_m(T)=\bigcap_{L \in \S^g(m), T(m) \subseteq \mathcal{D}_{\mathfrak{L}}}\mathcal{D}_{\mathfrak{L}}$.
\end{proposition}

\begin{proof}
The claim concerning $\kz_m(T)$ is true by definition. \\[0.2cm] For the other we suppose WLOG that $0 \in T$ (otherwise consider a shifted version of $T$; we remind the reader that $T(1)$ is non-empty because $T$ is closed with respect to compressions). We show that $\pz_m(T)=\bigcap_{L \in \S^g(m), T(m) \subseteq \mathcal{D}_{\mathfrak{L}}}\mathcal{D}_{\mathfrak{L}}$. Let $A \in \pz_m(T)$. If we had $L \in \S^g(m)$ with $\mathfrak{L}(A) \nsucceq 0$, then \autoref{specki} would give us a projection $P$ to an at most $m$-dimensional subspace with $\mathfrak{L}(PAP^*) \nsucceq 0$ as well. Since $PAP^* \in T(\rk(P))$, we would conclude $T(\rk(P)) \nsubseteq \mathcal{D}_{\mathfrak{L}}$. \\[0.2cm]
Conversely let $A \notin \pz_m(T)$ and take a projection $P$ with $PAP^* \notin T$. By the Effros-Winkler separation method \autoref{abc} we find $L \in \S^g(m)$ such that $\mathfrak{L}$ separates $PAP^*$ from $T$ (i.e. $T \subseteq \mathcal{D}_\ml$ and $PAP^* \notin \mathcal{D}_\ml$). We can interpret $PAP^*$ a submatrix of $A$ (up to unitary equivalence). Since $\mathfrak{L}(PAP^*)$ is a submatrix of $\mathfrak{L}(A)$ and $\mathfrak{L}(PAP^*) \nsucceq 0$, also $\mathfrak{L}(A) \nsucceq 0$. \\[0.2cm]
This shows that $\pz_m(T)$ is matrix convex as an intersection of matrix convex sets. From the Effros-Winkler separation it is clear that that it contains $T(j)=\pz_m(T)(j)$ for $j \leq m$. That it is the largest possible set is due to the fact that for all $A \in \mathbb{S}^g(k)$ and every projection $P: \C^k \rightarrow \im(P)$ we have $PAP^* \in \mconv(A)$. 
\end{proof}

\begin{cor} (cf. \cite[Lemma 7.3]{HM}) \label{haupt2}
Let $L \in \S^g(k)$. Then $\pz(\mathcal{D}_{\mathfrak L}) \leq k$.
\end{cor}

\begin{proof}
This is an immediate corollary of \autoref{specki}.
\end{proof}

\begin{proposition} \label{orka}
Let $T \subseteq \S^g$, $m \in \N$ and suppose that for all $B \in T$, $k \in \N$ and projections $P: \C^{k} \rightarrow \im(P)$ of rank at most $m$ we have $PBP^* \in T$. Then $\mconv(T)(m)=\mconv(T(m))(m)$.
\end{proposition}

\begin{proof}
Let $A \in \mconv(T)(m)$ and $V_j \in \C^{k_j \times m}$, $B_j \in T(k_j)$ with $\sum_{j=1}^r V_j^* V_j=I$ and $\sum_{j=1}^r V_j^* B_j V_j=A$. Let $P_j$ be the projection of $\C^{k_j}$ onto $\im(V_j)$. Then we have 
\begin{align*}
&A=\sum_{j=1}^r V_j^* B_j V_j = \sum_{j=1}^r V_j^* P_j^* (P_j B_j P_j^*) P_j V_j \in \mconv(T(m)). \qedhere
\end{align*} 
\end{proof}

\begin{cor} \label{unint}
Let $T \subseteq \S^g$ be a closed matrix convex set. Then we have 
\begin{align*}
\pz_m(T)=\{A \in \S^g \ | \ (\mconv(A))(m) \subseteq T\}
\end{align*}
\end{cor}

\begin{proof}
This follows from \autoref{orka}.
\end{proof}

\begin{lemma} \label{duali}
Let $T \subseteq \S^g$ be closed and matrix convex with $0 \in T$. Then $\kz_m(T)^\circ = \pz_m(T^\circ)$ and $\pz_m(T)^\circ = \ov{\kz_m(T^\circ)}$. If $0 \in \inte(T)$ we also have $\pz_m(T)^\circ = \kz_m(T^\circ)$. 
\end{lemma}

\begin{proof}
\begin{align*}
&\pz_m(T^\circ) \overset{(a)}{=} \left(\bigcap_{L \in \S^g(m), T^\circ \subseteq \mathcal{D}_{\mathfrak L}}\mathcal{D}_{\mathfrak L}\right)=\left(\bigcap_{L \in T^{\circ \circ}(m)}\mathcal{D}_{\mathfrak L}\right)=\left(\bigcap_{L \in T(m)}\mathcal{D}_{\mathfrak L}\right) \\&\overset{(b)}{=}
 \left(\bigcap_{L \in T(m)} \mconv(L)^\circ\right) = \left(\bigcup_{L \in T(m)} \mconv(L)\right)^\circ = \kz_m(T)^\circ
\end{align*}
where the equalities follow from these principles:
\begin{enumerate}[(a)]
\item We know that  
\begin{align*}
\bigcap_{L \in \S^g(m), T^\circ \subseteq \mathcal{D}_{\mathfrak L}}\mathcal{D}_{\mathfrak L} = \bigcap_{L \in \S^g(m), T^\circ(m) \subseteq \mathcal{D}_{\mathfrak L}}\mathcal{D}_{\mathfrak L}=\pz_m(T^\circ)
\end{align*}
where we used \autoref{specki} and the description obtained in \autoref{erst}.
\item This is \autoref{polari}.
\end{enumerate}
For the other claim we calculate $\pz_m(T)^\circ=\pz_m(T^{\circ\circ})^\circ=\kz_m(T^\circ)^{\circ\circ}=\ov{\kz_m(T^\circ)}$.
In case that $0 \in \inte(T)$ we know from \autoref{simpolarb} that $T^\circ(m)$ is compact. Thus the theorem of Caratheodory \autoref{carad} implies that $\kz_m(T^\circ)$ is closed.
\end{proof}

\begin{cor} \label{dualo}
Let $T \subseteq \S^g$ be matrix convex and closed with $0 \in T$. Then $\kz(T
) \geq \pz(T^\circ)$ and $\pz(T) \leq \kz(T^\circ)$. In case that $0 \in \inte(T)$, we have $\pz(T)=\kz(T^\circ)$. In case that $T$ is bounded, we have $\kz(T)=\pz(T^\circ)$.
\end{cor}

\begin{proof}
Let $S \subseteq \S^g$ be matrix convex and closed with $0 \in S$. Suppose that $\kz_m(S)=S$. Then we know $S^\circ=\kz_m(S)^\circ=pz_m(S^\circ)$. This shows $\kz(T
) \geq \pz(T^\circ)$ and $\pz(T) \leq \kz(T^\circ)$. Now in case that $0 \in \inte{S}$ and $\pz_m(S)=S$, then $S^\circ=\pz_m(S)^\circ=\kz_m(S^\circ)$. As a consequence we get $\kz(T^\circ) \leq \pz(T)$ in case that $0 \in \inte{T}$. \\[0.2cm]
In case that $T$ is bounded, we know that $0 \in \inte{T^\circ}$ and therefore $\kz(T)=\kz(T^{\circ \circ}) \leq \pz(T^\circ)$.
\end{proof}

\noindent The next result characterizes how small the maximal size of the blocks occuring in the block diagonalization of a pencil description of a given free spectrahedron can be. We encourage the reader to compare this result with \cite[Theorem 2.3]{FNT}.

\begin{proposition} \label{konz}
Let $T=\mathcal{D}_{\mathfrak{L}}$ be a free spectrahedron given by a monic linear pencil $\ml$. Then $\pz(T) \leq k$ if and only if $T$ can be written as a finite intersection of spectrahedra defined by monic pencils of size $k$.
\end{proposition}

\begin{proof}
"$\Longleftarrow$": Suppose $T$ is a finite intersection of spectrahedra defined by monic linear pencils of size $k$. Then \autoref{haupt2} tells us that each of these spectrahedra has projection number at most $k$. This means that $\text{pz}(\mathcal{D}_{\mathfrak{L}}) \leq k$. \\[0.2cm]
"$\Longrightarrow$": Let $m:=\pz(T)\leq k$. Since $0 \in \inte{T}$, \autoref{dualo} tells us $\kz(\mconv(L,0)) =\kz(T^\circ)=\pz(T) = m \leq k$. Therefore there exist $H_1,...,H_r \in \mconv(L, 0)(k)$ such that $L \oplus 0 \in \mconv(H_1,...,H_r)$. This means $\mconv(H_1,...,H_r)=\mconv(L, 0)$. We conclude $T=\mathcal{D}_{I-\bigoplus_{j=1}^r H_j \ov X}=\bigcap_{j=1}^{r} \mathcal{D}_{\mathfrak{H}_j}$ (\autoref{polari}).
\end{proof}

\begin{cor} \label{chars}
Let $T=\mathcal{D}_{\mathfrak L}\subseteq \S^g$ be a free spectrahedron given by a monic linear pencil $\ml$. Then $\pz(T)$ is the least number $k$ for which we find finitely many spectrahedra each defined by a monic pencil of size at most $k$ whose intersection is $T$.
\end{cor}

\noindent The following corollary was already proven in \cite[Theorem 3.2]{FNT}, however that proof did not use polarity.

\begin{cor} \label{polyh} \cite[Theorem 3.2]{FNT} Let $S \subseteq \R^g=\S^g(1)$ be convex with $0 \in \inte(S)$. Then $\pz_1(S)$ is a free spectrahedron if and only if $S$ is a polyhedron. 
\end{cor}

\begin{proof} Let $\pz_1(S)$ be a spectrahedron. \autoref{chars} yields that $\pz_1(S)$ is a intersection of spectrahedra defined by $1 \times 1$ matrices. So $S$ is a polyhedron. \\[0.2cm]
If $S$ is a polyhedron, then we find $f_1,...,f_r \in \R[\ov X]_1$ with $S=\{x \in \R^g \ | \ 1+f_1(x) \geq 0,...,1+f_r \geq 0\}$. The free spectrahedron defined by the monic pencil $\ml$ with the $1+f_j$ on the diagonal satisfies $\mathcal{D}_\ml(1)=S$. Of course we also have $\pz(\mathcal{D}_\ml)=1$ (\autoref{haupt2}); thus $\mathcal{D}_\ml=\pz_1(S)$.   
\end{proof}

\begin{rem} (cf. \cite[Proposition 3.5]{DDSS})
Let $T \subseteq \S^g$ be matrix convex, compact and $\kz(T)=m< \infty$. Then there exists a dense subset $\{L_n \ | \ n \in \N\}$ of $T(m)$ such that $\mconv(\bigoplus_{n \in \N} L_n)=T$ (cf. \autoref{closca} for more details). By taking the polar (\autoref{simpolarbb} and \autoref{polariapp}) we get the following statement: \\[0.2cm] 
Let $S \subseteq \S^g$ be matrix convex, closed and with $0 \in \inte{S}$ as well as $\pz(S)=m<\infty$. Then there exists a dense subset $\{L_n \ | \ n \in \N\}$ of $S^\circ(m)$ such that $\mathcal{D}_{\bigoplus_{n \in \N} \ml_n}=S$.
\end{rem}

\begin{proposition} \label{arviskz}
Let \marginpar{[\autoref{ex1}]} $L \in \S^g(\delta)$ such that $\mathcal{D}_\ml$ is bounded. Let $\mathcal{T}$ be the operator system defined by the $L_i$ and let $m \in \N$. Then $\text{kz}(\mathcal{D}_\mathfrak L) \leq m$ if and only if: \\[0.2cm]
For all finite-dimensional $C^*$-algebras $\mathcal{B}$ and all unital $m$-positive maps $\varphi: \mathcal{T} \rightarrow \mathcal{B}$ the map $\varphi$ is already completely positive. 
\end{proposition}

\begin{proof}
"$\Longrightarrow$": Let $\text{kz}(\mathcal{D}_\mathfrak L) \leq m$. Fix a finite-dimensional $C^*$-algebra $\mathcal{B}$ and suppose that $\varphi: \mathcal{T} \rightarrow \mathcal{B}$ is $m$-positive. Let $\mathfrak H=I-\varphi(L_1)X_1-...-\varphi(L_g)X_g$. \autoref{rumpf2} tells us that $\mathcal{D}_\mathfrak L (m) \subseteq \mathcal{D}_\mathfrak H (m)$. Since $\text{kz}(\mathcal{D}_\mathfrak L) \leq m$, we conclude $\mathcal{D}_\mathfrak L \subseteq \mathcal{D}_\mathfrak H$. Again \autoref{rumpf2} tells us that $\varphi$ is completely positive. \\[0.2cm]
"$\Longleftarrow$": Assume the right hand side holds and that $\text{kz}(\mathcal{D}_\mathfrak L) > m$. Since $\kz_m(\mathcal{D}_\ml)$ is closed, the Effros-Winkler separation \autoref{trenn} (a) tells us that we can find a monic linear pencil $\mathfrak H$, $s \in \N$ and $B \in \mathcal{D}_\mathfrak L(m+s)$ such that $\mathcal{D}_\mathfrak L(m) \subseteq \mathcal{D}_\mathfrak H(m)$ and $B \notin \mathcal{D}_\mathfrak H$. Let $\mathcal{R}$ be the operator system defined by the $H_i$. \autoref{rumpf2} implies that the linear map $\varphi: \mathcal{T} \rightarrow \mathcal{R}$ given by $L_i \mapsto H_i$ and $I \mapsto I$ is $m$-positive. By assumption it is also completely positive so $B \in \mathcal{D}_\mathfrak L \subseteq \mathcal{D}_\mathfrak H$.   
\end{proof}

\section{Helton-McCulloughs characterisation of free spectrahedra}

\subsection{Infinite-dimensional projection lemma and Nash manifolds.}

\begin{definition}
Let $p \in S(\NC)^{\delta \times \delta}$ be a Hermitian matrix polynomial. For $B \in \S^g(r)$ with $p(B) \succ 0$ we define for all $k \in \N$ the set \index{e@$\mathcal{E}_p(B)$}$\mathcal{E}_p(B)(k)$ as the connected component of \index{i@$\mathcal{I}_p$}$\mathcal{I}_p(kr):=\{ A \in (S\C^{kr \times kr})^g | \ p(A) \text{ is invertible} \}$ containing $I_k \otimes B$ (we have $p(I_k \otimes B) \succ 0$ automatically, see \autoref{canoni}). Denote its closure by \index{d@$\mathcal{D}_p(B)$}$\mathcal{D}_p(B)(k)$. If $p$ is a monic linear pencil, then $\mathcal{D}_p(0)=\mathcal{D}_p$ (\autoref{kerni}).
\end{definition}

\noindent We want to prove the following:

\begin{theo} \label{existence} \cite[Theorem 1.4]{HM}
Let \marginpar{[\autoref{ex11}]} $\delta \in \N$, $p \in S\NC^{\delta \times\delta}$ of degree $d$, $S:=\mathcal{D}_p(0)$. Suppose that $p(0)=I_\delta$ and that $\mathcal{E}_p(0)$ is matrix convex. Then $S$ is a free spectrahedron. \\[0.2cm]
Furthermore $S$ is a finite intersection of spectrahedra defined by linear matrix inequalities of size $k:=\dim \NC_{\left\lfloor{\frac{d}{2}}\right\rfloor}^{1 \times\delta}$.
\end{theo}

\noindent This result was originally proven by Helton and McCullough and is formulated for bounded $S$ although that assumption is not used in their proof. In \cite{HM} a theory of varieties on $\partial S$ was introduced. Building on that, they constructed a monic linear pencil representing $S$ by applying the Effros-Winkler separation finitely many times. In each step one separates one point $X \in \partial S$ from $S$ with a monic linear pencil $\mathfrak{L}$ such that $\inte{S} \subseteq \inte{\mathcal{D}_\mathfrak{L}}$. By using the Noetherian character of their variety-constructions, Helton and McCullough showed that if the point $X$ is chosen well, then all other points which have the same "vanishing ideals" get also separated and one needs only finitely many steps to separate $\partial S$ from $\inte{S}$. \\[0.2cm]
Our strategy is to merge infinitely many points in $\partial S$ to an operator point (see \autoref{haupt} for the details) and separate this from $\inte{S}$ with a monic linear pencil $\mathfrak{L}$. If the matrix points belong to the same cell of a cell-decomposition of $\partial S(k)$ (more precisely the cell-decomposition of a slightly more complicated set), then $\mathfrak{L}$ separates all these points simultaneously. While modifying the main step of Heltons and McCulloughs argument significantly, we adopt the general framework of their proof; thus one should regard our proof as a new variant of theirs rather than to be completely different. \\[0.2cm]
The size of the LMI needed in \autoref{existence} can be estimated by the product of $k$ and the numbers of cells in our cell-decomposition. However it is hard to estimate the latter number. The proof in \cite{HM} gives a cleaner estimate for the size of the LMI.  

\begin{proposition} \label{kerni} \cite[Lemma 2.1]{HM}
Let $p \in S\NC^{\delta \times\delta}$ such that $p(0) \succ 0$ and $\mathcal{E}_p(0)$ is matrix convex. Then 
\begin{align*}
\partial \mathcal{D}_p(0)=\mathcal{D}_p(0) \setminus \mathcal{E}_p(0)=\{A \in \S^g \ | \ \ker p(A) \neq \{0\}, \forall t \in (0,1): tA \in \mathcal{E}_p(0)\}.
\end{align*}
\end{proposition}

\begin{proof} (cf. \cite[ Lemma 2.1]{HM}) Let $A \in \mathcal{D}_p(0) \setminus \mathcal{E}_p(0)$. By continuity we have $p(A) \succeq 0$. By matrix convexity of $\mathcal{E}_p(0)$ and continuity, we even get that $p(tA) \succeq 0$ and $tA \in \mathcal{D}_p(0)$ for all $t \in (0,1)$. The polynomial function $\psi: \R \rightarrow \R, t \mapsto \det(p(tA))$ fulfills $\psi([0,1]) \subseteq \R_{\geq 0}$ and $\psi(0)>0$. Hence there are only finitely many $t \in (0,1)$ such that $\psi(t)=0$. Let $s \in (0,1)$ and fix $t \in (s,1)$ with $\psi(t) \neq 0$. Then $tA$ lies in $\mathcal{I}_p(0) \cap \ov{\mathcal{E}_p(0)}$. However $\mathcal{E}_p(0)$ is closed in $\mathcal{I}_p(0)$ by definition, thus $tA \in \mathcal{E}_p(0)$ and by convexity also $sA \in \mathcal{E}_p(0)$. Now we see that $A \notin \mathcal{E}_p(0)$ implies that find a non-trivial vector $v \in \ker p(A)$.
\end{proof}

\begin{cor} \label{kerni2}
Let $\ml$ be a monic linear pencil. Then $\inte{\mathcal{D}_\ml}=\{A \in \S^g \ | \ \ml(A) \succ 0\}$ and $\ov{\inte{\mathcal{D}_\ml}}=\mathcal{D}_\ml$. \hfill\qedsymbol
\end{cor}

\begin{definition} Let $p \in S\NC^{\delta \times \delta}$ be a matrix polynomial of degree $d$ with $p(0)=I_\delta$. Let $A \in \S^g(k),v \in \C^{k\delta}$ with $p(A)v=0$. Write $v=\sum_{\al=1}^\delta e_\al \otimes v_\al$ with $v_\al \in \C^{k}$ and $e_\al \in \C^\delta$ the $\al$-th unit vector. Then define \index{M@$M(A,v)_p$}$M(A,v)_p:=\left\{ q(A)v \ \middle| \ q \in \NC_{\lfloor{\frac{d}{2}}\rfloor}^{1 \times \delta}\right\}=\spam\left(\left\{ q(A)v_\al \ \middle| \ q \in \NC_{\lfloor{\frac{d}{2}}\rfloor}, \ \al \in \{1,...,\delta\}\right\}\right)$. If $d=1$ (so $p$ is a linear pencil), notice that the equality $M(A,v)_p=\text{span}\{v_1,...,v_\delta\}$ holds.
\end{definition}

\begin{rem} \label{canoni}
Let $\mathcal{H}$, $\mathcal{K}$ be Hilbert spaces, $A \in \mathcal{B}(\mathcal{H})$, $(B_j)_{j \in \N} \subseteq \mathcal{B}(\mathcal{K})$ and $(e_\al)_{\al \in A}$ be a Hilbert space basis of $\mathcal{H}$ and $w=\sum_{\al \in A} e_\al \otimes \left(\bigoplus_{j \in \N} w_\al^j\right) \in \mathcal{H} \otimes \mathcal{K}^{(\infty)}$. Then we have 
\begin{align*}
&A \otimes \left[\bigoplus_{j \in \N} B_j\right]=U^*\left[\bigoplus_{j \in \N}A \otimes B_j \right] U \text{ and } \\
&\left(A_i \otimes \left[\bigoplus_{j \in \N} B_j\right]\right)w=U^*\left[\bigoplus_{j \in \N}A \otimes B_j \right] U w=U^* \bigoplus_{j \in \N} \left( (A \otimes B_j) \left( \sum_{\al \in A} e_\al \otimes w_\al^j\right)\right)
\end{align*} 
where $U$ is the unitary shuffle operator defined by 
\begin{align*} U: \mathcal{H} \otimes \mathcal{K}^{(\infty)} \rightarrow (\mathcal{H} \otimes \mathcal{K})^{(\infty)}, \sum_{\al \in A} e_\al \otimes \left(\bigoplus_{j \in \N} v_\al^j\right) \mapsto \bigoplus_{j \in \N} \left(\sum_{\al \in A} e_\al \otimes v_\al^j\right).
\end{align*}
We infer that for $p \in \NC^{\delta \times \delta}$ and $(A_n)_{n \in \N} \subseteq \mathcal{B}(K)$ we have $p\left(\bigoplus_{n \in \N} A_n \right) \approx \bigoplus_{n \in \N} p(A_n)$.
\end{rem}

\begin{lemma} (Infinite-dimensional projection lemma) \label{haupt}
Suppose $p \in S\NC^{\delta \times \delta}$ is a matrix polynomial of degree $d$ with $p(0)=I_\delta$. Let $(X_n)_{n \in \N} \subseteq \S^g(h)$ be a bounded sequence and $(v_n)_{n \in \N} \subseteq \mathcal{S}^{h \delta-1}$ such that $p(X_n)v_n=0$. Write $v_n=\sum_{\alpha=1}^\delta e_\alpha \otimes v_{n}^\alpha$ where $e_\alpha \in \C^\delta$ is the $\alpha$-th unit vector. Set $B=\bigoplus_{n \in \N} X_n \in \mathcal{B}((\C^{h})^{(\infty)})$ and $v=\sum_{\alpha=1}^\delta e_\alpha \otimes v^\alpha$ where $v^\alpha:= \bigoplus_{n \in \N} \frac{v_n^\al}{n^2} \in (\C^{h})^{(\infty)}$. 
Then if $P$ denotes the projection onto $M(B,v)_p:=\left\{ q(B)v \ \middle| \ q \in \NC_{\lfloor{\frac{d}{2}}\rfloor}^{1 \times \delta}\right\}$, we have also $\langle v,p(PBP^*)v \rangle=0$. \hypertarget{II}{}\\[0.2cm] Let $k=\dim \NC_{\lfloor{\frac{d}{2}}\rfloor}^{1 \times \delta}$. If $\mathcal{E}_p(0)$ is matrix convex and all $X_n \in \mathcal{D}_p(0)$, then even $p(PBP^*)v=0$. In this case there exist monic $\mathfrak{L} \in \NC_{1}^{k \times k}$, $W \in \NC_{\lfloor{\frac{d}{2}}\rfloor}^{k \times \delta}$ such that $\mathcal{D}_p(0) \subseteq \mathcal{D}_\mathfrak{L}$ and $\mathfrak{L}(B)W(B)v=0$ while $W(B)v \neq 0$. This means $\mathfrak{L}(X_n)W(X_n)v_n=0$ for all $n \in \N$ and $W(X_n)v_n \neq 0$ for one $n \in \N$. (II)
\end{lemma}

\begin{proof} This proof is based on \cite[Lemma 7.3]{HM}.
Let $m \in \NC$ be a monomial of degree at most $d$. Write $m=w_1 x_j w_2$ with $\deg(w_1),\deg(w_2) \leq \lfloor{\frac{d}{2}}\rfloor$. Now for $\alpha,\beta \in \{1,...,\delta\}$ we have $w_1^*(B)v^\alpha, w_2(B)v^\beta \in M(B,v)_p$  by construction. Thus 
\begin{align*}
&\langle v^\beta, m(PBP^*) v^\alpha\rangle=\langle P^*Pw_1(B)^* v^\beta, B_j w_2(B) v^\alpha \rangle \\
&=\langle w_1(B)^* v^\beta, B_j w_2(B) v^\alpha \rangle=\langle v^\beta, m(B) v^\alpha \rangle
\end{align*}  
We write $p=\sum_{\gamma \in \mathcal{F}_g} f_\gamma \gamma$ and calculate 
\begin{align*}
&\langle v, p(PBP^*)v\rangle=\sum_{\gamma \in \mathcal{F}_g} \sum_{\alpha,\beta=1}^\delta \langle e_\beta \otimes v^\beta,  (f_\gamma \otimes \gamma(PBP^*)) (e_\alpha \otimes v^\alpha) \rangle \\ &=\sum_{\gamma \in \mathcal{F}_g} \sum_{\alpha,\beta=1}^\delta \langle e_\beta, f_\gamma e_\alpha \rangle \langle v^\beta, \gamma(PBP^*) v^\alpha \rangle = \sum_{\gamma \in \mathcal{F}_g} \sum_{\alpha,\beta=1}^\delta \langle e_\beta, f_\gamma e_\alpha \rangle \langle v^\beta, \gamma(B) v^\alpha \rangle \\&= \langle v,p(B)v\rangle=0.
\end{align*}
Now suppose that all $X_n \in \mathcal{D}_p(0)$ and that $\mathcal{E}_p(0)$ is matrix convex. For $s \in \N$ consider the partial sum $B_{s}=\bigoplus_{n=1}^s X_n \oplus 0$. Then we see that $B_s \overset{s \rightarrow \infty}{\rightharpoonup} B$ and $PB_sP^* \overset{s \rightarrow \infty}{\rightharpoonup} PBP^* $ (where $\rightharpoonup$ denotes weak convergence). Thus also $p(PB_sP^*) \overset{s \rightarrow \infty}{\rightharpoonup} p(PBP^*)$ and $p(PBP^*) \succeq 0$ due to matrix convexity. We have already seen $\langle v,p(PBP^*)v \rangle=0$. Both facts together mean that $p(PBP^*)v=0$.
By the Effros-Winkler separation \autoref{trenn} (b) there exists monic $\mathfrak{L} \in S\NC_{1}^{k \times k}$ such that $\mathcal{E}_p(0) \subseteq \inte{\mathcal{D}_\mathfrak{L}}$ and $\ker \mathfrak{L}(PBP^*) \neq \{0\}$. Choose a non-trivial kernel vector $w$ and find $W \in \NC_{\lfloor{\frac{d}{2}}\rfloor}^{k \times \delta}$ such that $w=W(B)v$. So we have $0=\mathfrak{L}(PBP^*)W(B)v$. \\[0.2cm]
This translates into $0=\langle W(B)v, \mathfrak{L}(PBP^*)W(B)v \rangle=\langle W(B)v, \mathfrak{L}(B)W(B)v \rangle$. Together with $\mathfrak{L}(B_s) \overset{s \rightarrow \infty}{\rightharpoonup} \mathfrak{L}(B)$ and thus $\mathfrak{L}(B) \succeq 0$, this means $\mathfrak{L}(B)W(B)v=0$. Since NC-matrix polynomials are behaving well with direct sums (see \autoref{canoni}), we get $\mathfrak{L}(X_n)W(X_n)v_n=0$ for all $n \in \N$.
\end{proof}

\begin{rem}
The conclusion \hyperlink{II}{(II)} of \autoref{haupt} remains true even in the case that the sequence $(X_n)_{n \in \N}$ is not bounded.
\end{rem}

\begin{proof}
By applying \autoref{haupt} for a finite number of $X_n$ we get the following: For every fixed number $t \in \N$ there exists $W \in \NC_{\lfloor{\frac{d}{2}}\rfloor}^{k \times \delta}$, $\mathfrak{L} \in \NC_{1}^{k \times k}$ such that $S \subseteq \inte{\mathcal{D}_\mathfrak{L}}$ and $\mathfrak{L}(X_n)W(X_n)v_n =0$ (*) for all $n \in \{1,...,t\}$ as well as $W(X_n)v_n \neq 0$ for one $n$. \\[0.2cm]
Now $\mathfrak{L}W \in \NC_{\lfloor{\frac{d}{2}}\rfloor+1}^{k \times \delta}$ lives in a finite-dimensional subspace. If we enlarge $t$, the subspace of possible solutions to $(*)$ is becoming smaller. Since we find a solution for arbitrary big $t$, there exists a common solution $\mathfrak{L}$, $W$ for all $n \in \N$.   
\end{proof}

\begin{rem} \label{potenzi}
For our new proof of \autoref{existence} we need the concept of a Nash submanifold of $\R^k$. A \textit{Nash submanifold} $M \subseteq \R^k$ (of dimension $r$) is a connected submanifold of $\R^k$ for which for each $x \in M$ there exists a semialgebraic $C^\infty$-diffeomorphism map $\phi: U \rightarrow H$ of the atlas of $M$ from a semialgebraic neighborhood $U$ of $x$ in $\R^k$ to a semialgebraic neighborhood $H$ of $0$ in $\R^k$ with $\varphi(M \cap U)=H \cap (\R^{r} \times \{0\}^{k-r})$. We refer the reader to \cite{BCR} for an introduction to Nash manifolds. A \textit{Nash function} is a function from a Nash manifold into $\R$ whose graph is a Nash manifold. A function $f: \R^n \rightarrow \R^m$ is a Nash function if and only if $f \in \C^\infty$ and the graph of $f$ is semialgebraic.\\[0.2cm]
We need two facts: The first is that each semialgebraic set admits a decomposition into finitely many Nash manifolds $C^\infty$-diffeomorphic to relatively open boxes \cite[Proposition 2.9.10]{BCR}. This is called analytic/$C^\infty$ cell decomposition. The second fact is that a Nash function $f$ defined on an open box has locally a converging power series expansion \cite[Chapter 8.1]{BCR}. In particular the set $f^{-1}(\R \setminus \{0\})$ is dense in the domain of $f$ if $f \neq 0$. \\[0.2cm]
In the following we will identify $\C$ with $\R^2$ and are thereby able to talk about semialgebraic subsets of $\C$.
\end{rem} 

\begin{lemma} \label{bochni}
In the setting of \autoref{existence} $\partial S(k)$ decomposes into a finite disjoint union of Nash manifolds Nash-diffeomorphic to relatively open boxes. Additionally we can achieve that there is a a Nash function $f_C$ on every box which satisfies $f_C(A) \in \mathcal{S}^{k \delta - 1} \cap \ker(p(A))$ 
\end{lemma}

\begin{proof}
As $\partial S(k)$ is a semialgebraic set over $\R$ (\autoref{kerni}), we can make an analytic cell decomposition which gives the first statement. For the second part we choose a semialgebraic function $f_C$ with the desired property and apply the analytic cell decomposition theorem again. 
\end{proof}  

\subsection{Proof of \autoref{existence}}

\begin{proof} (of \autoref{existence})
We write $\partial S(k)$ as in \autoref{bochni}. Then we consider one Nash manifold $C$ with associated Nash function $f_C$. Fix a dense subset $(X_n)_{n \in \N}$ of $C$. We apply \autoref{haupt} which says that there exists $\mathfrak{L}:=\mathfrak{L}_C \in S\NC_{1}^{k \times k}$ monic and $V \in \NC_{\lfloor{\frac{d}{2}}\rfloor}^{k \times \delta}$ such that $\inte{S} \subseteq \inte{\mathcal{D}_\mathfrak{L}}$, $\mathfrak{L}(X_n)V(X_n)f_C(X_n) = 0$ for all $n \in \N$ and $V(X_m)f_C(X_m) \neq 0$ for one $m$. Therefore we conclude that the function $\phi: C \rightarrow \C^{k^2}, X \mapsto V(X) f_C(X)$ is not zero and $\mathfrak{L}(X)V(X)f_C(X) = 0$ for all $X \in C$. Since $C$ is a Nash manifold and $V$ is a matrix polynomial, we know that $V$ is in each component a Nash function on $C$. We can assume without loss of generality that $C$ is a relatively open box. \\[0.2cm]
But then \autoref{potenzi} implies that $\phi^{-1}(\C^{k^2} \setminus \{0\})$ is dense in $\C^{k^2}$ and therefore $\mathfrak{L}(X)$ has non-trivial kernel for all $X \in C$. This means that $\mathfrak{L}$ separates all points of $C$ from $\inte(S)$. Therefore the direct sum of all $\mathfrak{L}_C$ separates all points of $\partial S(k)$ from $\inte(S)$. We have $\pz(S) \leq k$ (indeed, let $A \in \partial S(h)$ and $v\in \mathcal{S}^{\delta h}$ with $p(A)v=0$. One can adapt the first calculations of the proof of \autoref{haupt} to obtain $PAP^* \in \partial S$ where $P: \C^h \rightarrow \im(P)$ is the projection onto $M(A,v)_p$ which proves the claim. Alternatively we can apply \autoref{haupt} to the constant sequences $B=(A)_{n \in \N}$ and $(v)_{n \in \N}$ and obtain: There exist monic $\mathfrak{L} \in \NC_{1}^{k \times k}$, $W \in \NC_{\lfloor{\frac{d}{2}}\rfloor}^{k \times \delta}$ such that $S \subseteq \mathcal{D}_\mathfrak{L}$ and $\mathfrak{L}(A)W(A)v=0$ while $W(A)v \neq 0$. Thus $A \in \partial \mathcal{D}_\ml(h)$ and there exists a projection $P: \C^h \rightarrow \im(P)$ of rank at most $k$ such that $PAP^* \in \partial \mathcal{D}_\ml$ (\autoref{haupt2}). Since $PAP^* \in S$ and $S \subseteq \mathcal{D}_\mathfrak{L}$, we deduce $PAP^* \in \partial S$).\\[0.2cm]
Since $\pz(\mathcal{D}_{\ml_C}) \leq k$ (\autoref{haupt2}) for all $C$ and $\pz(S) \leq k$ and the interior of the spectrahedron defined by the direct sum of the $\mathfrak{L}_C$ agrees with $\inte{S}$ on the level $k$ (because of \autoref{kerni} and \autoref{haupt}), the claim follows.  
\end{proof}

\noindent In \cite[Remark 1.1]{HM} Helton and McCullough say that if we replace $x=0$ in $S=\mathcal{D}_p(x)$ by another point $x \in \R^n$, then $S$ is still a free spectrahedron. This is simply done by a change of coordinates. However if we consider a matrix point $x$, they point towards an extension of the methods of their proof. We show in the following that the situation for an arbitrary matrix point can be reduced to the situation of a scalar point. \\[0.2cm]
The reduction to the scalar case involves to check some extra property, which seems to be not more difficult to check than the usual requirements of the scalar case.   

\begin{cor} \label{gen}
Let $p \in S\NC^{\de \times \de}$ be a matrix polynomial, $A \in \S^g(n)$ and $p(A) \succ 0$. Then the set $\mathcal{E}_p(A)$ constitutes the $n\N$-th levels of the interior of a spectrahedron $\mathcal{D}_{B-C \ov X}$ with $(B-C \ov X)(A) \succ 0$ if and only if $p(A_{11}) \succ 0$, the line between $A$ and $I_n \otimes A_{11}$ is in $\mathcal{I}_p$ and $\mathcal{E}_p(A_{11})$ is matrix convex. 
\end{cor}

\begin{proof}
"$\Longleftarrow$:" The right hand side implies that $\mathcal{D}_p(A_{11})$ equals a spectrahedron $\mathcal{D}_{B-C \ov X}$ such that $(B-C \ov X)(A_{11}) \succ 0$. We have $I_r \otimes A \in \mathcal{E}_p(A_{11})(nr)$ and therefore $\mathcal{E}_p(A)(r) = \mathcal{E}_p(A_{11})(nr)$. From \autoref{kerni} we know that $\inte(\mathcal{D}_p(A_{11}))=\mathcal{E}_p(A_{11})$.\\[0.2cm]
"$\Longrightarrow$:" Suppose $\mathcal{E}_p(A)(r) = \inte(\mathcal{D}_{B-C \ov X}(nr))$ for all $r \in \N$ for a pencil $B-C \ov X$ with $(B-C \ov X)(A) \succ 0$. Since $\inte(\mathcal{D}_{B-C \ov X})$ is matrix convex and $A_{11}=e_1^TAe_1 \in \mconv(A)$, we conclude $A_{11} \in \inte(\mathcal{D}_{B-C \ov X})$. Thus $I_{n} \otimes A_{11} \subseteq \inte{\mathcal{D}_{B-C \ov X}}(n)=\mathcal{E}_p(A)(n)$. Since $p$ behaves well with direct sums, we get $A_{11} \in \mathcal{I}_p$, $p(A_{11}) \succ 0$ and $p(I_{nr} \otimes A_{11}) \succ 0$ for all $r \in \N$. $\inte(\mathcal{D}_{B-C \ov X}(nr))$ is convex so the line between $I_r \otimes A$ and $I_{nr} \otimes A_{11}$ is in $\mathcal{I}_p(A)$. Now we have $\mathcal{E}_p(A_{11})(nr)=\mathcal{E}_p(I_{nr} \otimes A_{11})(1)=\inte{\mathcal{D}_{B-C \ov X}}(nr)$ for all $n \in \N$. $\inte{\mathcal{D}_{B-C \ov X}}$ is matrix convex and $\mathcal{E}_p(A_{11})$ is clearly closed with respect to direct sums and reducing subspaces (the latter one follows from the fact that submatrices of positive definite matrices are again positive definite). We conclude $\mathcal{E}_p(A_{11})=\inte(\mathcal{D}_{B-C \ov X})$ (if we calculate with matrices of size $\neq nk$, we can lift them to a multiple of $k$ by making the direct sum with sufficiently many $A_{11}$). \\[0.2cm]
So we know that $\mathcal{E}_p(A_{11})$ is the interior of a spectrahedron and in particular matrix convex.   
\end{proof}

\section{\texorpdfstring{$\mathcal{D}$-irreducible pencils and the Gleichstellensatz}{D-irreducible pencils and the Gleichstellensatz}}

\subsection{Sequence of determinants of a monic linear pencil} 

\noindent In this subsection fix a monic linear pencil $\mathfrak{L}$ generating a spectrahedron $\mathcal{D}_\mathfrak{L}$ which has on each level the determinant $f_k=\det_k \mathfrak{L}(\mathcal{X}) \in \R[\mathcal{X}_{\alpha,\beta}^1,...,\mathcal{X}_{\alpha,\beta}^g \ | \ \alpha,\beta \in \{1,...,k\}]$.

\begin{lemma} \label{uni} \cite[Lemma 2.1]{HKV}
Let $f \in \R[\mathcal{X}_{\alpha,\beta}^{\ell} \ | \ \alpha,\beta \in \{1,...,k\}, 1 \leq \ell \leq g]$ be  invariant under unitary conjugations. Then also all factors of $f$ are invariant under unitary conjugations. 
\end{lemma}

\begin{proof} (cf. \cite[Lemma 2.1]{HKV})
Let $f$ be monic and $f=p_1\cdot ... \cdot p_m \in \C[\mathcal{X}]$ be the decomposition of $f$ into irreducible polynomials over $\C$ in such a way that associated polynomials in the factorization are equal. Since the set of unitary matrices in $\C^{k \times k}$ is connected, and the map $\C[\mathcal{X}] \times \C^{k \times k} \rightarrow \C[\mathcal{X}], (p,Y) \mapsto (A \mapsto p(Y^*AY))$ is continuous, we conclude that for each unitary $U \in \C^{k \times k}$ and $h \in \{1,...,m\}$ there exists $\lambda_{h}(U) \in \C$ such that $p_h(U A U^*)=\lambda_{h}(U) p_h(A)$ for all $A \in \S^g(k)$. We show that $\lambda(U):=\lambda_{1}(U)=1$ for every unitary $U$. \\[0.2cm]
Let $\text{diag}: \C^k \rightarrow \C^{k \times k}$ be the map that maps a vector $v$ to the diagonal matrix $V$ with $v$ on the diagonal. The map $\lambda$ is multiplicative. Since $I_k$ is in the commutator of $\C^{k \times k}$ we have $\lambda(aI)=1$ for every $a \in \mathcal{S}^0$. For unitary $U,V \in \C^{k \times k}$ we have $\lambda(UV)=\lambda(U)\lambda(V)=\lambda(V)\lambda(U)=\lambda(VU)$. We conclude $\lambda(U^*VU)=\lambda(V)$; thus $\lambda$ is determined by its restriction to diagonal matrices and is constant on unitary orbits. For $a_1,...,a_k \in \mathcal{S}^0$ choose $a \in \mathcal{S}^0$ such that $a^k=\prod_{j=1}^k a_j$. Then we have 
\begin{align*}
&\lambda(\text{diag}(a_1,...,a_k))=\prod_{j=1}^k \lambda(\text{diag}(1,...,1,\underbrace{a_j}_{j\text{-th index}},1,...,1))=\prod_{j=1}^k \lambda(\text{diag}(a_j,1,...,1))\\=&\lambda(\text{diag}(a^k,1,...,1))=\prod_{j=1}^k \lambda(\text{diag}(a,1,...,1))=\prod_{j=1}^k \lambda(\text{diag}(1,...1,\underbrace{a}_{j\text{-th index}},1,...,1))\\=&\lambda(\text{diag}(a,...,a))=1 \qedhere
\end{align*}
\end{proof}

\begin{lemma} (Rule of Descartes for RZ-polynomials) \label{descartes} 
Let $(b_j)_{j \in \{1,...,d\}}$ be a strictly increasing tuple of natural numbers and $0$, $f=\sum_{j=1}^d a_{b_j} X^{b_j} \in \R[X] \setminus \{0\}$ be a real polynomial such that $a_{b_j} \in \R \setminus \{0\}$ for all $j$. We call $|\{j \in \{1,...,d-1\} \ | \ a_{b_j}a_{b_{j+1}}<0 \}|$ the number of sign changes of the coefficient sequence of $f$ and $|\{j \in \{1,...,d-1\} \ | \ (-1)^{b_j+b_{j+1}}a_{b_j}a_{b_{j+1}}<0 \}|$ the number of sign changes of the coefficient sequence of $f(-X)$. For $x \in \R$ we denote with $\mu(f,x) \in \N_{0} \cup \{\infty\}$ the multiplicity of $x$ as a zero of $f$.  \\[0.2cm]
The number of positive roots of $f$ counted with multiplicities is at most the number of sign changes of the coefficient sequence of $f$ and the number of negative roots of $f$ counted with multiplicities is at most the number of sign changes of the coefficient sequence of $f(-X)$. If $f$ is an RZ-polynomial, then in the last statements we even have equality.
\end{lemma}

\begin{proof}
The first statement is well-known (e.g. \cite{B}). If $f$ is an RZ-polynomial we have $\deg(f) \geq |\{j \in \{1,...,d-1\} \ | \ a_{b_j}a_{b_{j+1}}<0 \}|+|\{j \in \{1,...,d-1\} \ | \ (-1)^{b_j+b_{j+1}}a_{b_j}a_{b_{j+1}}<0 \}|+b_1 \geq \sum_{x \in \R_{>0}} \mu(f,x) + \sum_{x \in \R_{<0}} \mu(f,x) + \mu(f,0)= \deg(f)$; hence there must be equality everywhere.
\end{proof}

\begin{cor} \label{descartescor}
Let $(b_j)_{j \in \{1,...,d\}}$ be a strictly increasing tuple of natural numbers and $0$, $f=\sum_{j=1}^d a_{b_j} X^{b_j} \in \R[X] \setminus \{0\}$ be an RZ-polynomial such that $a_{b_j} \in \R \setminus \{0\}$ for all $j$. Then for every $j \in \{1,...,d-1\}$ we have $b_{j+1}-b_j \leq 2$. If $b_{j+1}-b_j$=2, then $a_{b_j} a_{b_{j+1}}<0$.
\end{cor}

\begin{proof}
The sign changes in the coefficient sequences $f$ and $f(-X)$ complement each other well and it is easy to see that the claimed properties of the coefficients of $f$ are necessary in order to have a total of $\deg(f)-b_1$ sign changes in the coefficient sequences of $f$ and $f(-X)$.
\end{proof}

\begin{lemma} \label{rirr}
Let $k \geq 2$. Assume $f_k=\prod_{j=1}^m \alpha_j$ with $\alpha_j \in \R[\mathcal{X}_{\alpha,\beta}^1,...,\mathcal{X}_{\alpha,\beta}^g\ | \ \alpha,\beta \in \{1,...,k\}]$ non-constant polynomials. Then also $f_{k-1}=\prod_{j=1}^m \alpha_j|_{\S^g(k-1)}$ and those polynomials are non-constant. Here we identify $\S^g(k-1)$ as a subset of $\S^g(k)$ by looking at the embedding $\iota: \S^g(k-1) \rightarrow \S^g(k), A \mapsto A \oplus 0$. There are other natural ways to embed $\S^g(k-1)$ into $\S^g(k)$ (i.e. $A \mapsto U^*(A \oplus 0)U$ for a unitary $U \in \C^{k \times k}$), however \autoref{uni} states that they lead to the same results.
\end{lemma}

\begin{proof} 
Write $\alpha_j|_{k-1}:=\alpha_j|_{\S^g(k-1)}$. We only have to show that $\al_1|_{k-1}$ is not constant. So assume $\al_1|_{k-1}=1$. Since $\al_1$ is invariant under unitary tranformations, we know that $\al_1$ is constant on all tuples of Hermitian matrices which share a kernel vector. We write $\al_1=\sum_{s=0}^r p_s$, where the $p_s$ are homogeneous of degree $s$. We have $p_0=1$; thus the $p_s$ with $s \geq 1$ have to vanish on the matrices which share a kernel vector. \\[0.2cm] It is easy to see that a non-trivial linear polynomial cannot vanish on all those matrices (since there span is all $\S^g(k)$). Therefore $p_1=0$. Now we claim that $p_2$ switches its sign or is zero. Suppose $p_2 \neq 0$. If $p_2$ does not change the sign, it is up to sign a sum of squares of linear forms. However with the same argument as before, this would imply $p_2=0$. So $p_2$ switches its sign. \\[0.2cm]
Thus we find an $A \in \S^g(k)$ with $p_2(A)>0$ or $p_2=0$. 
In any case the polynomial function $q: \R \rightarrow \R, t \mapsto \al_1(tA)$ cannot be an RZ-polynomial because of \autoref{descartescor}. 
\end{proof}

\begin{theo} \label{zerl}
There is $M \in \N$ and $N \in \N$ such that each $f_k$ decomposes into a product of $N$ polynomials $g_{j,k} \in \R[\mathcal{X}_{\alpha,\beta}^1,...,\mathcal{X}_{\alpha,\beta}^g \ | \ 1 \leq \alpha,\beta \leq k]$ with $g_{j,k}(0)=1$. For $k \geq M$ they are irreducible. We can achieve that $g_{j,k}|_{\S^g(\ell)}=g_{j,\ell}$ for $1 \leq \ell \leq k$ and $j \in \{1,...,N\}$. 
\end{theo}

\begin{proof}
The preceeding lemma tells us that there are $N,M \in \N$ guaranteeing that $f_k$ decomposes into a product of $N$ irreducible polynomials $g_{1,k},...,g_{N,k} \in \R[\mathcal{X}_{\alpha,\beta}^1,...,\mathcal{X}_{\alpha,\beta}^g \ | \ 1 \leq \alpha,\beta \leq k]$ for $k \geq M$ such that $g_{j,k}(0)=1$. For smaller $k$ of course we have that $f_k$ decomposes into a product of $N_k$ irreducible polynomials $g_{j,k}$ with $N_k \geq N$. Set $N_k=N$ for $k \geq M$. Fix now $k \geq M$ and take $\ell<k$. Write $s=k-\ell$. We define $h_{j,\ell}=g_{j,k}|_{\S^g(\ell)}$ and $p_{j,s}=g_{j,k}|_{\S^g(s)}$. 
Now for $Y \in \S^g(\ell)$, $Z \in \S^g(s)$ we know that
\begin{align*}
&\prod_{j=1}^{N_\ell}g_{j,\ell}(Y) \prod_{j=1}^{N_s}g_{j,s}(Z)=f_\ell(Y)f_s(Z)=f_k(Y \oplus 0) f_k(0 \oplus Z)\\&=f_k(Y \oplus Z)=f_k(Y \oplus 0) f_k(Z \oplus 0)
=\prod_{j=1}^N g_{j,k}(Y \oplus 0) g_{j,k}(Z \oplus 0)=\prod_{j=1}^N h_{j,\ell}(Y) p_{j,s}(Z)
\end{align*}
Now suppose $\ell \geq M$. Then we know that $N_\ell=N$ and from the equation above we get
\begin{align*}
\prod_{j=1}^{N}g_{j,\ell}(Y)=\prod_{j=1}^N h_{j,\ell}(Y)
\end{align*} 
Since the prime factorization is unique and all the polynomials on the left are irreducible, the same factors appear on both sides. So we can in a unique way (up to multiplicities) rearrange the indices in such a way that $g_{j,k}|_{\S^g(\ell)}=g_{j,\ell}$ holds. Repeating this procedure for growing levels $k$ and all $M \leq \ell < k$ gives the desired result.
\end{proof}

\begin{cor} \label{level}
Let $\mathfrak{L}$ be a monic linear pencil and fix the $g_{j,k}$, $N, M$ from \autoref{zerl}. Then for fixed $j \in \{1,...,N\}$ the sequence of closures of the connected components of the sets $\S^g(k) \setminus g_{j,k}^{-1}(0) \subseteq \S^g(k)$ around $0$ form the levels of a closed matrix convex set $D_j$. We have $\bigcap_{j=1}^N D_j=\mathcal{D}_\mathfrak{L}$.
\end{cor}

\begin{proof}
For $j \in \{1,...,n\}$ let $E_j(k)$ be the connected component of the set $\S^g(k) \setminus g_{j,k}^{-1}(0) \subseteq \S^g(k)$ around $0$. We use \autoref{konv} to verify that the $E_j$ are matrix convex. As a factor of a real-zero polynomial every $g_{j,k}$ is a real-zero polynomial. Thus each $E_j(k)$ is ordinary convex. As a determinant $f_k$ is unaffected by unitary conjugations. Therefore \autoref{uni} implies that this is also the case for the $g_{j,k}$. The closedness under taking direct sums and reducing subspaces follows from $g_{j,k_1}(Y) g_{j,k_2}(Z)=g_{j,k_1+k_2}(Y \oplus Z)$ (see the calculation in the proof of \autoref{zerl}). Hence $E_j$ is matrix convex and also its closure $D_j$ is. \\[0.2cm] 
$\mathcal{D}_\mathfrak{L}(k)$ is the closure of the connected component of $(f_k)^{-1}(\R \setminus \{0\})$ around $0$. By convexity this set equals $\{A \in \S^g(k) \ | \ \forall t \in [0,1): f_k(A) > 0\}$ and also $D_j(k)=\{A \in \S^g(k) \ | \ \forall t \in [0,1): g_{j,k}(A) > 0\}$. From this $\bigcap_{j=1}^N D_j=\mathcal{D}_\mathfrak{L}$ follows directly. 
\end{proof}

\begin{lemma} \label{densi}
Let $f \in \R[Y_1,...,Y_n]$ be an RZ-polynomial and $S$ the generated closed convex set. Suppose that there is no other RZ-polynomial of lesser degree defining the same convex set. Then $f$ is regular on a dense subset of $\partial S$.
\end{lemma}

\begin{proof}
This follows from the fact that $\{x \in \R^n \ | \ f(x)=0\}$ has locally dimension $n-1$ around every point $x \in \partial S$ (\cite[Lemma 2.1]{HV}) and \cite[Definition 3.3.4, Theorem 4.5.1 and Proposition 3.3.14]{BCR}.
\end{proof}

\begin{lemma} \label{victor} (cf. \cite[Lemma 2.1]{HV})
Let $f,g \in \R[Y_1,...,Y_n]$, $f$ be an irreducible RZ-polynomial and $C_f$ the connected component of $f^{-1}(\R \setminus \{0\})$ around $0$. Suppose that there is no other RZ-polynomial of lesser degree defining the same convex set $C_f$. Suppose there is a point $x \in \partial C_f$ and $\ep >0$ such that $g=0$ on $\partial C_f \cap B(x,\ep)$. Then $f$ divides $g$.  
\end{lemma}

\begin{proof}
From \cite[Lemma 2.1]{HV} and \cite[Proposition 8.2.2]{BCR} we know that $\partial C_f \cap B(x,\ep)$ has dimension $n-1$ as a semialgebraic set and the $\R$-Zariski closure of $\partial C_f \cap B(x,\ep)$ has dimension $n-1$ as a real variety. Since $Z(f):=\{x \in \R^n \ | \ f(x)=0\}$ is irreducible either $Z(f) \subseteq Z(g)$ or the intersection $Z(f) \cap Z(g)$ has dimension smaller than $n-1$. Therefore $Z(f) \subseteq Z(g)$. Since $(f)$ is real \cite[Theorem 4.5.1]{BCR}, $f$ divides $g$.
\end{proof}

\subsection{Gleichstellensatz} \textcolor{inv}{a} \\
In this chapter and the following we want to analyze the following questions: Given a free spectrahedron $S=\mathcal{D}_{\mathfrak{L}}$, can we find an easier and in some sense minimal description of $S$ as a spectrahedron? How does this minimal description relate to the monic linear pencil $\mathfrak{L}$ and is it unique? We will show that there are in some sense "minimal and indecomposable" pencils which form the building blocks to construct every spectrahedron containing $0$ in its interior and also every monic linear pencil. The decomposition into those building blocks will be unique up to unitary equivalence. \\[0.2cm]
In the literature there are different approaches to deal with these questions. The most obvious way to approach these questions is to look at the spectrahedron $\mathcal{D}_\ml$ directly. We call $L \in \S^g$ irreducible if the $L_1,...,L_g$ admit no joint non-trivial invariant subspace. Helton, Klep and McCullough in \cite{HKM4} call the tuple $L$ minimal if no smaller tuple $H$ describes the spectrahedron $\mathcal{D}_\mathfrak{L}$. They showed that two minimal descriptions of the same bounded spectrahedron are unitary equivalent (Gleichstellensatz). Furthermore they showed that the minimal tuple $H$ defining $\mathcal{D}_\mathfrak{L}$ is a direct summand of $L$. We will give a new proof of these results and generalize them to the unbounded setting (which was established already by Zalar in \cite{Z} by improving the techniques of \cite{HKM4}). Furthermore we will show how these results lead to the existence of the "irreducible" building blocks mentioned above. \\[0.2cm]
An alternative approach of Klep and Vol$\check{\text{c}}$i$\check{\text{c}}$ in \cite{KV} deals with the free locus $\mathcal{Z}(\mathfrak{L}):=\{X \in \S^g \ | \ \ker \mathfrak{L}(X) \neq \{0\} \}$ instead of the spectrahedron $\mathcal{D}_\mathfrak{L}$. This leads to the same building blocks. We look further into this approach in the next subsection.

\begin{definition} 
Let $\mathfrak{L}$ be a monic linear pencil. We call $\mathfrak{L}$ (or $L$) irreducible if $L_1,...,L_g$ admit no joint non-trivial invariant subspace. \\[0.2cm]
We call $\mathfrak{L}$ (or $L$) \textbf{$\mathcal{D}$-irreducible} if for all other monic linear pencils $\mathfrak{H}_1,\mathfrak{H}_2$ the equality $\mathcal{D}_{\mathfrak{H}_1} \cap \mathcal{D}_{\mathfrak{H}_2}=\mathcal{D}_\mathfrak{L}$ implies $\mathcal{D}_\mathfrak{L}=\mathcal{D}_{\mathfrak{H}_i}$ for some $i \in \{1,2\}$. \\[0.2cm]
We call $\mathfrak{L}$ (or $L$) \textbf{$\mathcal{D}$-minimal} if there is no monic linear pencil $\mathfrak{H}$ of smaller size such that $\mathcal{D}_\mathfrak{H}=\mathcal{D}_\mathfrak{L}$.
\end{definition}

For the following we need a special case of the boundary theorem of Arveson. We will outline the elementary proof found by Davidson, adapted to our special situation. 

\begin{lemma} \label{david} (\cite{D2})
Let $L \in \S^g(n)$ irreducible. Suppose that there are $C_j \in \C^{n \times n}$ with $L=\sum_{j=1}^r C_j^* L C_j$ and $I \succeq \sum_{j=1}^r C_j^* C_j$. Then $\sum_{j=1}^r C_j^* C_j =I$ and all $C_j$ are scalar multiples of the identity. 
\end{lemma}

\begin{proof} (\cite{D2})
Consider $C_0=\sqrt{I-\sum_{j=1}^r C_j^*C_j}$, $J=0 \oplus I_r$. Then $\sum_{j=0}^r C_j^* C_j=I_n$. Set $V^*=(C_0^* \ C_1^* \ ... \ C_r^*)$. Then $V$ is an isometry, i.e. $V^*V=I$, $VV^*$ is the projection onto $\im(V)$, and $V^*$ is contractive. We define $\varphi: \C^{n \times n} \rightarrow \C^{n \times n}, A \mapsto \sum_{j=1}^r C_j^* A C_j = V^* (J \otimes A)V$. Denote by $S$ the operator system generated by the $L_i$.\\[0.2cm]
Now let $W$ be a minimal non-zero subspace of $\C^n$ such that $W$ is $C_j$-invariant for all $j \in \{0,...,m\}$. Define $\Gamma=\{ D \in \C^{n \times n} \ | \ \forall j \in \{0,...,r\} \forall w \in W: C_j D w=D C_j w\}=\{ D \in \C^{n \times n} \ | \ \forall w \in W: ((1 \oplus I) \otimes D) V w=V D w\}$. We show that if $D \in S$ and $B \in \Gamma$, then $DB \in \Gamma$. \\[0.2cm]
So take such $D,B$ and set $M=\{ w \in W \ | \ || DB|_W || \ ||w|| = ||DBw|| \}$. Let $w \in M$. Then 
\begin{align*}
&||DBw||^2=||\varphi(D)Bw||^2=||V^* (J \otimes D) VBw||^2=||V^* (J \otimes D)((1 \oplus I) \otimes B)V||^2\\=&||V^* (J \otimes DB) V w||^2 \leq ||(J \otimes DB) V w||^2 =\sum_{j=1}^r ||DB C_j w||^2 \\ 
\leq & \sum_{j=1}^r ||DB|_{W}||^2 \ ||C_j w||^2 \leq ||DB|_{W}||^2 \ ||w||^2= ||DB w||^2
\end{align*}    
Therefore $C_jw \in M$ for $j \in \{1,...,m\}$ and $C_0w=0$. So we verified that $C_jM \subseteq M$ for all $j$. By minimality of $W$ we get $M=W$. We also have
\begin{align*}
||(J \otimes DB)Vw||=||V^* (J \otimes DB)Vw||=||VV^* (J \otimes DB)Vw||
\end{align*} 
where the first equality follows from the calculation above and the second from the fact that $V$ is an isometry. Hence
\begin{align*}
&((1 \oplus I) \otimes DB)Vw=(J \otimes DB)Vw=VV^*(J \otimes D)((1 \oplus I) \otimes B)Vw=VV^*(J \otimes D)VBw\\&=V \varphi(D)Bw=VDBw
\end{align*}
where the second equality comes from the Pythagoras theorem and the first from $C_0w=0$. Thus $DB \in \Gamma$. Of course $I \in \Gamma$ and we conclude that $\Gamma$ contains the algebra $\C^{n \times n}$ generated by the $L_i$. Now let $D \in \C^{n \times n}$, $x \in \C^n$. Choose $w \in W$ and $B \in \C^{n \times n}$ such that $Bw=x$. We know
\begin{align*}
&Dx=DBw=V^*VDBw=V^* ((1 \otimes I) \otimes DB) V w=V^* (J \otimes DB) V w \\ &=V^* (J \otimes D)((I \oplus 1) \otimes B) V w =V^* (J \otimes D)V B w =\varphi(D) x
\end{align*} 
where we used that $DB,B \in \Gamma$. We conclude $\varphi(D)=D$. We have $\varphi(I)=I$. Thus $C_0=0$. Now let $v \in \C^n$ with $||v||=1$. With Cauchy-Schwarz $1 = v^*vv^*v = v^* \varphi(vv^*) v = \sum_{j=1}^r (v^*C_j v)^*(v^* C_j v)=\sum_{j=1}^r| \langle v, C_j v \rangle|^2 \leq \sum_{j=1}^r |\langle C_jv, C_jv \rangle|=v^* \varphi(I) v=1$ and therefore $\spam{v}$ is an invariant subspace of $C_j$. Thus $C_j$ is a multiple of the identity.
\end{proof}

\begin{lemma} \label{ehkm} (\cite[Corollary 3.18]{HKM4} and \cite[Theorem 3.1]{Z})
Let $L \in \S^g(k)$ and $H \in \S^g(h)$ be irreducible such that $\mathcal{D}_{\mathfrak{L}}=\mathcal{D}_{\mathfrak{H}}$. Then $L$ and $H$ are unitarily equivalent. 
\end{lemma}

\begin{proof}
We know that their exist $C_j$ and $D_\ell$ such that $L=\sum_{j=1}^r C_j^* H C_j$ and $H=\sum_{\ell=1}^s D_\ell^* L D_\ell$ with $I_k \succeq \sum_{j=1}^r C_j^* C_j$ and $I_h \succeq \sum_{\ell=1}^s D_\ell^* D_\ell$. Of course we can assume that $C_j^* C_j \neq 0$ for all $j$ and $D_\ell^* D_\ell \neq 0$ for all $\ell$. We see that $L=\sum_{j=1}^r \sum_{\ell=1}^s C_j^* D_\ell^* L D_\ell C_j$ as well as $I_k \succeq \sum_{j=1}^r \sum_{\ell=1}^s C_j^* D_\ell^* I_k D_\ell C_j$. We apply \autoref{david} and find $\lambda_{\ell,j} \in \C$ such that $\lambda_{\ell,j} I=D_\ell C_j$. Additionally $I_k = \sum_{j=1}^r \sum_{\ell=1}^s C_j^* D_\ell^* I_k D_\ell C_j$ which implies $I_k = \sum_{j=1}^r C_j^* C_j$. Fix $j$. Since $0 \neq C_j^* C_j = C_j^* (\sum_{\ell=1}^sD_\ell^* D_\ell )C_j$, we know that $\lambda_{\ell,j} \neq 0$ for one $\ell$. In particular $C_j$ is injective and $k \leq h$. \\[0.2cm]
The same calculation the other way around shows that all $D_\ell$ are also injective, $I_h = \sum_{\ell=1}^s D_\ell^* D_\ell$ as well as $h \leq k$ and thus $h=k$. Since all $D_\ell,C_j$ are injective, we have that $D_\ell C_j=C_j D_\ell$ are a nonzero multiple of the identity. In particular all $C_j$ are linear dependent, so we can find $\lambda_j \neq 0$ and $C$ such that $C_j=\lambda_j C$ for all $j$. We conclude \begin{align*}L=\sum_{j=1}^r |\lambda_j|^2 C^* H C=\left(\sqrt{\sum_{j=1}^r |\lambda_j|^2}C\right)^* H \sqrt{\sum_{j=1}^r |\lambda_j|^2} C.\end{align*} Since $I=\left(\sqrt{\sum_{j=1}^r |\lambda_j|^2}C\right)^* \sqrt{\sum_{j=1}^r |\lambda_j|^2}C$ we conclude that $L$ and $H$ are unitarily equivalent. 
\end{proof}

\begin{lemma} \label{dirr}
Let $\mathfrak{L}$ be a monic linear pencil. Suppose $\det \mathfrak{L}_k = g_{1,k} \cdot ...\cdot g_{N,k}$ is the decomposition of \autoref{level}, where $g_{j,k}$ becomes eventually irreducible for big $k$ and $D_j$ is the induced closed matrix convex set. Let $J \subseteq \{1,...,N\}$ such that $\mathcal{D}_\ml=\bigcap_{j \in J} D_j$ and no matrix convex set on the right sight can be omitted without destroying equality. Then for all $j \in J$ the matrix convex set $D_j$ is in fact a spectrahedron.  
\end{lemma}

\begin{proof}
For a tuple $(K_i)_{i \in I}$ of closed matrix convex sets containing $0$ we have $(\bigcup_{i \in I} K_i)^\circ=\bigcap_{i \in I} K_i^\circ$ and by the bipolarity theorem $\ov{\mconv(\bigcup_{i \in I} K_i)}=((\bigcup_{i \in I} K_i)^\circ)^\circ=(\bigcap_{i \in I} K_i^\circ)^\circ$. \\[0.2cm]
We know from \autoref{polari} that
\begin{align*}
\mconv(L, 0)=(\mathcal{D}_\mathfrak{L})^\circ=\left(\bigcap_{j \in J} D_j\right)^\circ=\ov{\mconv\left(\bigcup_{j \in J} D_j^\circ\right)}=\mconv\left(\bigcup_{j \in J} D_j^\circ\right).
\end{align*} 
where the last equality comes from the fact that $\bigcup_{j \in J} D_j^\circ$ is compact together with the free Caratheodory theorem \autoref{carad} and \autoref{orka}.
Therefore we find $(L_j)_{j \in J} \subseteq \S^g$ with $L_j \in D_j^\circ$ such that $L \in \mconv(\{L_j \ | \ j \in J\})$. From $L_j \in D_j^\circ$ it is clear that $\mathcal{D}_{\mathfrak{L}_j} \supseteq D_j \supseteq \mathcal{D}_\ml$. We also have $\mathcal{D}_\ml=\bigcap_{j \in J} \mathcal{D}_{\ml_j}$. We claim $\mathcal{D}_{\mathfrak{L}_j}=D_j$ for all $j \in J$. \\[0.2cm]
Fix $j \in J$ and choose $M$ such that $g_{j,k}$ is irreducible and $\bigcap_{h \in J \setminus \{j\}} D_h(k)\neq \mathcal{D}_\ml(k)$ for all $k>M$. Then we know $g_{j,k} \ | \ \det_k \mathfrak{L}_j$ from \autoref{victor} which implies the other inclusion $\mathcal{D}_{\mathfrak{L}_j} \subseteq D_j$.
\end{proof}

\begin{proposition} \label{blop}
Let $L \in \S^g(k)$ be $\mathcal{D}$-irreducible and $\mathcal{D}$-minimal. Then $L$ is irreducible. \hfill\qedsymbol
\end{proposition}

\begin{theo} \label{minim}
Let $L_1,...,L_s$ be $\mathcal{D}$-irreducible and $\mathcal{D}$-minimal such that the Zariski closure of each $\partial{D}_{\mathfrak{L}_j}(k)$ is irreducible in the Zariski topology for big $k$. Suppose that 
\begin{align*}
\bigcap_{j=1}^s \mathcal{D}_{\mathfrak{L}_j} \neq \bigcap_{j=1, j\neq \ell}^s \mathcal{D}_{\mathfrak{L}_j}.
\end{align*}
for all $\ell \in \{1,...,s\}$. Suppose $L \in \S^g(\delta)$ defines a spectrahedron $\mathcal{D}_\mathfrak{L}=\bigcap_{j=1}^s \mathcal{D}_{\mathfrak{L}_j}$. Then up to unitary equivalence $L_1 \oplus ... \oplus L_s$ is a direct summand of $L$.
\end{theo}

\begin{proof}
From \autoref{rumpf2} and \autoref{steinfels} we know that we can write 
\begin{align*}
L=\sum_{j=1}^s \sum_{\ell=1}^r V_{j,\ell} L_j V_{j,\ell}^* \text{  and  } I = WW^* + \sum_{j=1}^s \sum_{\ell=1}^r V_{j,\ell} V_{j,\ell}^*.\end{align*}
For $k \in \N$ denote $T_j(k)=\partial \mathcal{D}_{\mathfrak{L}_j}(k) \cap \left(\bigcap_{h \neq j} \inte{\mathcal{D}_{\mathfrak{L}_h}(k)}\right)$. Choose $k_0 \in \N$ in such a way that the Zariski closure of each $\partial \mathcal{D}_{\mathfrak{L}_j}(k)$ is irreducible and $T_j(k) \neq \emptyset$ for $k \geq k_0$. Let $k \geq k_0$. Suppose $X_h \in T_h(k)$ and $v_h \in \C^\delta \otimes \C^{k}$ such that $\mathfrak{L}(X_h)v_h=0$. Write $v_h=\sum_{\alpha=1}^\delta e_\alpha \otimes v_{h,\alpha}$ with $v_{h,\alpha} \in \C^k$. Now since $\mathfrak{L}_j(X_h) \succ 0$ for $j \neq h$, we conclude that $0=(V_{j,\ell}^* \otimes I_\delta) v_h$, $0=(W^* \otimes I_\delta)v_h$. Denote $\delta_h=\dim M(X_h,v_h)_\mathfrak{L}$. \\[0.2cm]
Now choose an orthonormal basis $w_{h,1},...,w_{h,\delta_h}$ of $M(X_h,v_h)_\mathfrak{L}$. Let $U_h(X_h,v_h)=U^h \in \C^{\delta \times \delta_h}$ be the matrix where the $\beta$-th row contains the coefficients of $v_{h,\beta}$ modulo the basis $w_{h,1},...,w_{h,\delta_h}$. Then $U^h$ has rank $\delta_h$. For fixed $\alpha$ we have:
\begin{align*}
e_\al \otimes v_{h,\al} = e_\al \otimes \sum_{\gamma=1}^{\delta_h} U_{\al,\gamma}^h w_{h,\gamma}= \sum_{\gamma=1}^{\delta_h}  U_{\al,\gamma}^h e_\al \otimes w_{h,\gamma}
\end{align*}
 We calculate for $h \neq j$
\begin{align*}
\sum_{\gamma=1}^{\delta_h} \left(\sum_{\ell=1}^r V_{j,\ell}^* \sum_{\alpha=1}^\delta U_{\alpha,\gamma}^h e_\alpha\right) \otimes w_{h,\gamma}
=\sum_{\ell=1}^r (V_{j,\ell}^* \otimes I) \sum_{\alpha=1}^\delta \left( e_\alpha \otimes v_{h,\al} \right)
= \sum_{\ell=1}^r (V_{j,\ell}^* \otimes I) v_h = 0
\end{align*}
Therefore for all $\gamma \in \{1,...,\delta_h\}$ the equality 
\begin{align*} 0=\sum_{\ell=1}^r V_{j,\ell}^* \sum_{\alpha=1}^\delta U_{\alpha,\gamma}^h e_\alpha=\sum_{\ell=1}^r V_{j,\ell}^*U_{\cdot,\gamma}^h
\end{align*}
 holds, so $\sum_{\ell=1}^r V_{j,\ell}^* U^h=0$ for $j \neq h$. Analogeously $W^* U^h=0$. Because of $WW^*+\sum_{j=1}^s\sum_{\ell=1}^r V_{j,\ell}V_{j,\ell}^*=I_\delta$ we have $\sum_{\ell=1}^r V_{h,\ell}V_{h,\ell}^* U^h=U^h$. This means that $\sum_{\ell=1}^r V_{h,\ell}V_{h,\ell}^*$ is the identity on the $\delta_h$-dimensional space $\im(U_h(X_h,v_h))$. Since the eigenspaces of a Hermitian matrix are pairwise orthogonal, we get the following: \\[0.2cm]
Set $S_j=\spam\{ \im(U_j(X,v)) \ | \ k\geq k_0, \ \exists X \in T_j(k), v \in \C^\delta \otimes \C^k: L(X)v=0 \}$ for all $j$, where $U_j$ is defined in the same way as $U_h$. Then $V_{j,\ell}^*|_{S_h}=0$ and $S_h \perp S_j$ for $j \neq h$. We decompose $\C^{\delta}=\bigoplus_{j=1}^s S_j \oplus T$. With respect to this decomposition: Write $Y_{j,\ell}^*:=V_{j,\ell}^*|_{S_j}$, $Z_{j,\ell}^*:=V_{j,\ell}^*|_{T}$
\begin{align*}
L=\begin{pmatrix}
\sum Y_{1,\ell} L_1 Y_{1,\ell}^* & ... & 0 & \sum Y_{1,\ell} L_1 Z_{1,\ell}^* \\
\vdots & \ddots & \vdots & \vdots \\
0 & ... & \sum Y_{s,\ell} L_s Y_{s,\ell}^* & \sum Y_{s,\ell} L_s Z_{s,\ell}^* \\
\sum Z_{1,\ell} L_1 Y_{1,\ell}^* & ... & \sum Z_{s,\ell} L_s Y_{s,\ell}^* & \sum \sum Z_{j,\ell} L_{j} Z_{j,\ell}^*
\end{pmatrix} \\
I=\begin{pmatrix}
\sum Y_{1,\ell} Y_{1,\ell}^* & ... & 0 & \sum Y_{1,\ell} Z_{1,\ell}^* \\
\vdots & \ddots & \vdots & \vdots \\
0 & ... & \sum Y_{s,\ell} Y_{s,\ell}^* & \sum Y_{s,\ell} Z_{s,\ell}^* \\
\sum Z_{1,\ell} Y_{1,\ell}^* & ... & \sum Z_{s,\ell} Y_{s,\ell}^* & \sum \sum Z_{j,\ell} Z_{j,\ell}^*
\end{pmatrix}
\end{align*} 
Fix $j \in \{1,...,s\}$. Let $P$ be the projection onto $S_j$. We want to determine $\mathcal{D}_{P\mathfrak{L}P^*}$. Let $X_j \in T_j(k)$ and $v_j \in \C^{\delta} \otimes \C^k$ with $\ml(X_j)v_j=0$. With the same notation as above ($U^j=U_j(X_j,v_j)$) we see
\begin{align*}
&v_j=\sum_\al e_\al \otimes v_{j,\al} = \sum_\al e_\al \otimes \sum_{\gamma=1}^{\delta_j} U_{\al,\gamma}^j w_{j,\gamma}= \sum_\al \sum_{\gamma=1}^{\delta_j}  U_{\al,\gamma}^j e_\al \otimes w_{j,\gamma}=\sum_{\gamma=1}^{\delta_j}  U_{\cdot,\gamma}^j \otimes w_{j,\gamma} \\
&0=(P \otimes I)0=(P \otimes I) \mathfrak{L}(X_j)v_j=(P \otimes I)\mathfrak{L}(X_j) \left(\sum_{\gamma=1}^{\delta_j}  U_{\cdot,\gamma}^j \otimes w_{j,\gamma}\right) \\
&=(P \otimes I)\mathfrak{L}(X_j)(P^* \otimes I) \left(\sum_{\gamma=1}^{\delta_j}  U_{\cdot,\gamma}^j \otimes w_{j,\gamma}\right)=(P\mathfrak{L}P^*)(X_j)v_j
\end{align*} 
We conclude $\mathcal{D}_{P\mathfrak{L}P^*} \subseteq \mathcal{D}_{\mathfrak{L}_j}$ with \autoref{victor}. The destription of $PLP^*$ in terms of $L_j$ yields that $\mathcal{D}_{P\mathfrak{L}P^*} \supseteq \mathcal{D}_{\mathfrak{L}_j}$; hence we have $\mathcal{D}_{P\mathfrak{L}P^*} = \mathcal{D}_{\mathfrak{L}_j}$. \autoref{ehkm} in connection with \autoref{blop} and the fact that $L_j$ is $\mathcal{D}$-irreducible imply that $L_j$ is a submatrix of $PLP^*$ and $L$. That $L_j$ is even a direct summand is implied by the following easy proposition.
\end{proof}

\begin{proposition}
Let $\begin{pmatrix} L & B \\ B^* & C\end{pmatrix}=\sum_{j=1}^r V_j^* L V_j$ with $\sum_{j=1}^r V_j^* V_j = I$ and $L$ be irreducible. Then $B=0$.
\end{proposition}

\begin{proof}
Write $V_j=\begin{pmatrix} W_j & Z_j \end{pmatrix}$. Then we know $\sum_{j=1}^r\begin{pmatrix} W_j^* W_j & W_j^* Z_j \\ Z_j^* W_j & Z_j^* Z_j\end{pmatrix}=\begin{pmatrix} I & 0 \\ 0 & I\end{pmatrix}$. Because of \autoref{david} and $L=\sum_{j=1}^r W_j^* L W_j$ we know that there exist $\lambda_j \in \C$ such that $W_j = \lambda_j I$. In particular we have $B=\sum_{j=1}^r W_j^* L Z_j=L \sum_{j=1}^r \lambda_j^* Z_j=L \sum_{j=1}^r W_j^* Z_j=0$. 
\end{proof}

\begin{cor} \label{zetdecomp}
Let $\mathfrak{L}$ be a monic pencil. Then $\mathfrak{L}$ is a direct sum of monic linear pencils $\mathfrak{L}_j$ which are simultaneously $\mathcal{D}$-minimal and $\mathcal{D}$-irreducible (and constant pencils $1$) and such that the Zariski closure of $\partial \mathcal{D}_{\mathfrak{L}_j}(k)$ is irreducible for $k$ big enough. 
\end{cor}

\begin{proof}
We prove the claim by induction on $\text{size}(\mathfrak{L})$. The case $\size(\ml)=1$ is clear. Suppose that $\size(\ml)>1$, $\ml$ has no constant direct summand and $\det \mathfrak{L}_k = g_{1,k} \cdot ...\cdot g_{M,k}$ is the decomposition of \autoref{level}, where $g_{j,k}$ becomes eventually irreducible for big $k$ and $D_j$ is the induced closed matrix convex set. Let $J \subseteq \{1,...,M\}$ such that $\mathcal{D}_\mathfrak{L}=\bigcap_{j \in J} D_j$ and no matrix convex set on the right sight can be omitted without destroying equality. Then \autoref{dirr} allows us to write $D_j=\mathcal{D}_{\mathfrak{L}_j}$ with $L_j$ $\mathcal{D}$-minimal. Since the Zariski closure of $\partial \mathcal{D}_{\mathfrak{L}_j}(k)$ is irreducible for big $k$, we know that $L_j$ is also $\mathcal{D}$-irreducible. \autoref{minim} tells us that $\bigoplus_{j \in J} L_j$ is a direct summand of $L$. Thus we can write $L \approx \bigoplus_{j \in J} L_j \oplus H$. Now apply the induction hypothesis to $H$. 
\end{proof}

Finally, we can formulate our version of the Gleichstellensatz. Originally Helton, Klep and McCullough proved the Gleichstellensatz using the Silov ideal of an operator system, whose existence is already a deep result in operator theory. Since they need the Silov boundary only for operator systems living in the set of bounded linear operators of a fixed finite-dimensional Hilbert space, it could be possible that the proof of the existence of the Silov ideal in that case can be carried out in a more elementary way.  

\begin{cor} \label{gleichi} (Gleichstellensatz) (cf. \cite[Corollary 3.18]{HKM4} and \cite[Theorem 3.1]{Z})
Let \marginpar{[\autoref{ex2},\\ \autoref{kreu},\\ \autoref{ex13}]}$\ml$ be a monic linear pencil and $S=\mathcal{D}_\ml$. Then there exist monic linear pencils $\mathfrak{L}_1,...,\mathfrak{L}_r$ satisfying the following: Each $\mathfrak{L}_j$ is $\mathcal{D}$-irreducible and $\mathcal{D}$-minimal. For each $j \in \{1,...,r\}$ and big $k$ the Zariski closure of $\partial \mathcal{D}_{\mathfrak{L}_j}(k) \cap \bigcap_{h \neq j} \inte{\mathcal{D}_{\mathfrak{L}_h}}(k)$ is non-empty and irreducible. $\bigoplus_{j=1}^r \mathfrak{L}_j$ is $\mathcal{D}$-minimal and a direct summand of all monic linear pencils $\mh$ with $S=\mathcal{D}_\mh$. The $\mathfrak{L}_j$ are uniquely determined up to order and unitary equivalence.   
\end{cor}

\begin{proof} 
Combine \autoref{minim} and \autoref{dirr}.
\end{proof}

\subsection{\texorpdfstring{Characterisation of $\mathcal{D}$-irreducible and $\mathcal{D}$-minimal pencils}{Characterisation of Z-irreducible pencils}} 

\begin{definition} 
For a monic linear pencil $\mathfrak{L}$ we define the \textbf{free locus} $\mathcal{Z}(\ml):=\{X \in \S^g \ | \ \ker \ml(X) \neq \{0\} \}$.
We call a monic linear pencil $\ml$ (or $L$) \textbf{$\mathcal{Z}$-irreducible} if for all other monic linear pencils $\ml_1,\ml_2$ the equality $\mathcal{Z}(\ml_1) \cup \mathcal{Z}(\ml_2)=Z(\ml)$ implies $\mathcal{Z}(\ml)=\mathcal{Z}(\ml_j)$ for some $j \in \{1,2\}$. \\[0.2cm] 
We call a monic linear pencil $\ml$ (or $L$) \textbf{$\mathcal{Z}$-minimal} if there is no monic linear pencil $\mh$ of smaller size such that $\mathcal{Z}(\ml)=\mathcal{Z}(\mh)$.\\[0.2cm]
Let $L \in \S^g(k)$. Then $\mathcal{A}_L$ (or $\mathcal{A}_\ml$) shall denote the unital $C^*$-algebra generated by $L_1,...,L_g$.
\end{definition}

\noindent $\mathcal{Z}$-irreducible pencils were introduced by Klep and Vol$\check{\text{c}}$i$\check{\text{c}}$ in \cite{KV} in a more general setting (allowing also non-Hermitian pencils). We recall their main result:

\begin{theo} \cite[Theorem 5.4]{KV}\label{volcicklep} 
Let $L,H \in \S^g$. Then $\mathcal{Z}(\ml) \subseteq \mathcal{Z}(\mh)$ if and only if there is a $*$-homomorphism of $C^*$-algebras $\varphi: \mathcal{A}_H \rightarrow \mathcal{A}_L$ such that $L_i=\varphi(H_i)$ for all $i \in \{1,...,g\}$. In case that $\mathcal{Z}(\ml) = \mathcal{Z}(\mh)$, the map $\varphi$ is even an isomorphism. $L$ is $\mathcal{Z}$-irreducible if and only if $\mathcal{A}_L$ is simple.
\end{theo}

\begin{theo} \label{minirrchar}
Let $\ml$ be a non-constant monic linear pencil of size $\delta$. Then the following is equivalent:
\begin{enumerate}[(a)]
\item $L$ is $\mathcal{Z}$-irreducible and $\mathcal{Z}$-minimal
\item $L$ is $\mathcal{D}$-irreducible and $\mathcal{D}$-minimal
\item $\pz(\mathcal{D}_\ml)=\delta$
\item $L$ is irreducible (i.e. $\mathcal{A}_L=\C^{\delta \times \delta}$).
\item There is $k_0 \in \N$ such that the for all $k \in \N_{\geq k}$ the determinant $\det_k \ml$ is irreducible.
\end{enumerate} 
\end{theo}

\begin{proof}
(a) $\Longleftrightarrow$ (e): This is a direct consequence of \autoref{dedecomp}, which we will prove in the rest of the chapter. \\[0.2cm]
(b) $\implies$ (a): Suppose $L$ is $\mathcal{D}$-irreducible and $\mathcal{D}$-minimal. Obviously $L$ is $\mathcal{Z}$-minimal. Now take monic linear pencils $L_1, L_2$ such that $\mathcal{Z}(\ml_1) \cup \mathcal{Z}(\ml_2) =\mathcal{Z}(\ml)$. If $\mathcal{D}_{\ml_1},\mathcal{D}_{\ml_2} \neq \mathcal{D}_\ml$, then $L$ was not $\mathcal{D}$-irreducible. So suppose $\mathcal{D}_{\ml_1}=\mathcal{D}_\ml$. Then the Gleichstellensatz implies that $L$ is a submatrix of $L_1$. Thus $\mathcal{Z}(\ml_1)=\mathcal{Z}(\ml)$. \\[0.2cm]
(a) $\implies$ (b): Suppose $L$ is $\mathcal{Z}$-irreducible and $\mathcal{Z}$-minimal. We know that $\ml$ decomposes into a direct sum of $\mathcal{D}$-irreducible and $\mathcal{D}$-minimal pencils $\ml_j$. Due to (b) $\implies$ (a) they are also $\mathcal{Z}$-irreducible and $\mathcal{Z}$-minimal. So the decomposition has been trivial. \\[0.2cm]
(b) $\Longleftrightarrow$ (d): This follows directly from \autoref{gleichi}. \\[0.2cm]
(c) $\implies$ (b): Suppose $\pz(\mathcal{D}_\ml)=\delta$. If $L$ is not $\mathcal{D}$-minimal, then there exists a monic linear pencil $\mh$ with size less than $\pz(\mathcal{D}_\mh)$. This contradicts \autoref{haupt2}. So $L$ is $\mathcal{D}$-minimal. Assume $L$ is not $\mathcal{D}$-irreducible. By \autoref{zetdecomp} we know that $L$ is a direct sum of smaller $\mathcal{D}$-irreducible pencils. The $\pz$-number of $L$ is less than the maximum of the $\pz$-number of the summands. So there cannot be more than one summand. \\[0.2cm] 
(b) $\implies$ (c): Suppose now that $L$ is $\mathcal{D}$-irreducible and $\mathcal{D}$-minimal. Assume that $\pz(\mathcal{D}_\ml) < \delta$. \autoref{chars} and \autoref{zetdecomp} allow us to find a monic linear pencil $\mh=\mh_1 \oplus ... \oplus \mh_s$ such that $\mathcal{D}_\mh=\mathcal{D}_\ml$, where each $\mh_j$ has size smaller than $\delta$. Since $L$ is $\mathcal{D}$-irreducible, we can suppose $s=1$. Since $L$ is also $\mathcal{D}$-minimal$, \delta>\text{size}(H_1)$ is a contradiction.    
\end{proof}

\begin{proposition} \label{minand}
A monic linear pencil $L$ is $\mathcal{Z}$-irreducible if and only if for all pencils $L_1$, $L_2$ with $\mathcal{Z}(\ml) \subseteq \mathcal{Z}(\ml_1) \cup \mathcal{Z}(\ml_2)$ we have $\mathcal{Z}(\ml) \subseteq \mathcal{Z}(\ml_j)$ for some $j$. 
\end{proposition}

\begin{proof}
Let $L$ be $\mathcal{Z}$-irreducible. \autoref{volcicklep} says that $\mathcal{A}_L$ is simple. Suppose $\mathcal{Z}(\ml) \subseteq \mathcal{Z}(\ml_1) \cup \mathcal{Z}(\ml_2)$ holds. From \autoref{volcicklep} we know that there is a unital $^*$-homomorphism $\varphi: \mathcal{A}_{L_1} \times \mathcal{A}_{L_2} \rightarrow \mathcal{A}_L$ induced by $((L_1)_i,(L_2)_i) \mapsto L_i$. Since $\mathcal{A}_L$ is simple, $\varphi$ is surjective. The ideals of $\mathcal{A}_{L_1} \times \mathcal{A}_{L_2}$ are direct products of ideals of the $\mathcal{A}_{L_i}$. Now we divide by the kernel of $\varphi$ to see that there are ideals $I_j$ of $\mathcal{A}_{L_j}$ for $j \in \{1,2\}$ such that $\ov{\varphi}: (\mathcal{A}_{L_1} \times \mathcal{A}_{L_2}) /(I_1 \times I_2) \rightarrow \mathcal{A}_L$ is a $C^*$-isomorphism. The domain is isomorphic to $(\mathcal{A}_{L_1}/I_1) \times (\mathcal{A}_{L_2}/I_2)$ and simple.
Thus there is an $j$ such that $\mathcal{A}_{L_j}=I_j$ and $\varphi$ induces also a $C^*$-homomorphism from one of the $\mathcal{A}_{L_k}$ to $\mathcal{A}_L$. Another usage of \autoref{volcicklep} yields the result.
\end{proof}

\noindent The upcoming \autoref{ald} characterizes $\mathcal{Z}$-irreducible linear pencils $L$ as those whose determinants are irreducible polynomial at high levels. In a similar formulation Klep and Vol$\check{\text{c}}$i$\check{\text{c}}$ announced a proof in \cite{KV} for a later paper. Using completely different techniques than ours, they gave a proof in \cite{HKV}. We do not think that our proof generalizes to the non-Hermitian setting in \cite{HKV}.

\begin{lemma} \label{specz}
Let $\ml$ be a monic linear pencil of size $\delta$. Suppose $\det_k \ml = g_{1,k}^{\beta_1} \cdot ...\cdot g_{N,k}^{\beta_N}$ is the decomposition as in \autoref{level}, where $\beta_j \in \N$ denote multiplicities and such that for $k\geq M$ the polynomials $g_{j,k}$ are irreducible and pairwise not associated and $D_j$ is the induced closed matrix convex set. Then each $\pz(D_j) \leq \delta$ and $D_j$ is a spectrahedron for all $j \in \{1,...,N\}$.
\end{lemma}

\begin{proof}
Let $f_k:=\det_k \ml$. We denote by $\partial D_j(k)^{\text{on}}=\{A \in \partial D_j(k) \ | \ \forall h \in \{1,...,N\} \setminus \{j\}: g_{h,k}(B) \neq 0 \}$. \autoref{victor} tells us that those points are dense in $\partial D_j(k)$ for $k > M$ in the Euclidean topology. So fix $j$ and let $X \in \partial D_j(k)^{\text{on}}$ with $k \geq M$, so $g_{j,k}(X)=0$.
\\[0.2cm]
Choose $v=\sum_{\al=1}^\delta e_\al \otimes v_\al \in \C^\delta \otimes \C^k$ with $L(X)v=0$. Let $P$ be the projection of $\C^k$ onto $M(X,v)_L=\text{span}(v_1,...,v_\delta)$. We calculate: 
\begin{align*}
0&=\ml(X)v=\left(I \otimes I + \sum_{j=1}^g L_i \otimes X_i\right) \sum_{\al=1}^\delta (e_\al \otimes v_\al)\\&=\left(I \otimes I + \sum_{j=1}^g L_i \otimes X_i\right) \sum_{\al=1}^\delta (I \otimes P^*)(e_\al \otimes v_\al)
=\ml(XP^*)v
\end{align*} and therefore $\ml(PXP^*)v=(I \otimes P)\ml(XP^*)v=0$. Set $X_\alpha=(1-\alpha) P^*[PXP^*]P + \alpha X$ for $\al \in [0,1]$. The previous calculations mean that $\ml(X_\al)v=0$. Thus $f_k(X_\al)=0$. We conclude that $g_{j,k}(X_\alpha)=0$ for $\alpha$ near $1$ because $g_{h,k}(X) \neq 0$ for $h \neq j$. Since $g_{j,k}$ is a polynomial $g_{j,k}(X_0)=0$. Thus $0=g_{j,k}(P^*[PXP^*]P)=g_{j,k}(PXP^* \oplus 0)=g_{j,\rk(P)}(PXP^*)$. On the other hand we have $X \in D_j$; thus also $PXP^* \in D_j$ by matrix convexity. \\[0.2cm]
We have checked that for every $X \in \partial D_j(k)^{\text{on}}$ with $k \geq M$ there exists a projection $P$ of rank at most $\delta$ such that $PXP^* \in \partial D_j$. Now let $A \in \partial D_j(k)$ with $k \in \N$. Choose $s \in \N$ such that $s+k \geq M$. Then $A'=A \oplus (I_s \otimes 0) \in \partial D_j(s+k)$. Approximate $A'$ by points of $\partial D_j(k+s)^{\text{on}}$. With the help of Bolzano-Weierstra{\ss} we get a projection $P: \C^{k+s} \rightarrow \im(P)$ of rank at most $\delta$ such that $PA'P^* \in \partial D_j$. We have $PA'P^* \in \mconv(A,0)(\delta)=\mconv(\{QAQ^* \ | \  Q: \C^k \rightarrow \im(Q) \text{ projection of rank at most } \delta \} \cup \{0\})$ (\autoref{orka}). Moreover $0 \in \inte(D_j)$, so there is one projection $Q$ of rank at most $\delta$ such that $Q^*AQ \in \partial D_j$. This means $\text{pz}(D_j) \leq \delta$. \\[0.2cm]
Now we imitate our proof that "closures of matrix convex free basic open semialgebraic" sets are spectrahedra to show that each $D_j$ is a spectrahedron. The reason why this is possible is that we did not merely prove $\pz(D_j) \leq \delta$, but also that for a point $X \in \partial D_j(k)^{\text{on}}$ with $\ml(X)v=0$ we have $PAP^* \in \partial D_j$ where $P$ is the projection onto $M(X,v)_\ml$. Let $\ep=\max\{\delta,M\}$. \\[0.2cm]
We work with $\partial D_j(\ep)^{\text{on}}$. We consider one Nash manifold $C$ of $\partial D_j(\ep)^{\text{on}}$ with associated Nash function $f_C$ that maps each $A \in C$ to a non-trivial kernel vector of $\ml(A)$. Fix a dense subset $(X_n)_{n \in \N}$ of $C$ and write $f_C(X_n)=(f_{C,1}(X_n),...,f_{C,\delta}(X_n))$ with $f_{C,j}(X_n) \in \C^\ep$. \\[0.2cm]
We apply the techniques of the proof of \autoref{haupt} and \autoref{existence} and achieve: There exist monic $\mh:=\mh_C \in S\NC_1^{\de \times \de}$, $V \in S\C^{\de \times \de}$ such that $\mh_C(X_n)V(X_n)f_C(X_n)=0$ for all $n \in \N$, $V(X_m)f_C(X_m) \neq 0$ for one $m \in \N$ as well as $\mh(X) \succ 0$ for all $X \in \inte(D_j)$. \\[0.2cm]
Therefore we conclude that the function $\phi: C \rightarrow \C^{\delta\ep}, X \mapsto V(X) f_C(X)$ is not zero and $\mh(X)V(X)f_C(X) = 0$ for all $X \in C$. Since $C$ is a Nash manifold and $V$ is constant, we know that $V$ is in each component a Nash function on $C$. Without loss of generality we can assume that $C$ is a relatively open box. \\[0.2cm]
But then \autoref{potenzi} implies that $\phi^{-1}(\C^{\delta \ep} \setminus \{0\})$ is dense in $C$ and therefore $\mh(X)$ has non-trivial kernel for all $X \in C$. This means that $\mh$ separates all points of $C$ from $\inte(D_j)$. Therefore the direct sum of all $\mh_C$ separates all points of $\partial S_j(\ep)^{\text{on}}$ from $\inte(D_j)$. Due to the Zariski density of $\partial D_j(\ep)^{\text{on}}$ in $\partial D_j(\ep)$ we get that the direct sum of all $\mh_C$ separates all points of $\partial S_j(\ep)$ from $\inte(D_j)$. Since $\pz(D_j) \leq \delta$, $\pz(\mathcal{D}_{\mh_C}) \leq \delta$ and the spectrahedron defined by the direct sum of the $\mh_C$ coincides with $D_j$ on level $\ep$, the claim follows.   
\end{proof}

\begin{theo} \label{detzel} (cf. \cite[Theorem 3.4]{HKV})
Let $\ml$ be a monic linear pencil. Then $L$ is $\mathcal{Z}$-irreducible if and only for big $k$ we have $\det_k \ml=q_k^n$ where $q_k$ is irreducible and $n$ does not depend on $k$. 
\end{theo}

\begin{proof}
"$\Longleftarrow$": This is obvious. \\[0.2cm]
"$\Longrightarrow$": Suppose $\det L_k = g_{1,k}^{\alpha_1} \cdot ...\cdot g_{N,k}^{\alpha_n}$ is the decomposition as in \autoref{level}, where $g_{j,k}$ eventually becomes irreducible for big $j$, $D_j$ is the induced closed matrix convex set and the $g_{j,k}$ are pairwise not associated for all big fixed $k$. Every $D_j$ is a spectrahedron, so write $D_j=\mathcal{D}_{\ml_j}$ with $L_j$ $\mathcal{D}$-minimal. It is clear that the $L_j$ are also $\mathcal{D}$-irreducible. Hence the $L_j$ are also $\mathcal{Z}$-irreducible and $\mathcal{Z}$-minimal. We distinguish two cases:\\[0.2cm]
Case 1: There is no $j$ such that $\mathcal{Z}(\ml_j)=\mathcal{Z}(\ml)$. However $\mathcal{Z}(\ml) \subseteq \bigcup_j \mathcal{Z}(\ml_j)$. This contradicts the fact that $L$ is $\mathcal{Z}$-irreducible (\autoref{minand}). \\[0.2cm]
Case 2: There is an $j$ such $\mathcal{Z}(\ml_j)=\mathcal{Z}(\ml)$. WLOG assume that the vanishing ideal of the Zariski closure of $\partial \mathcal{D}_\ml$ is generated by $g_{1,k}$ (the Zariski closure must be irreducible otherwise we are in case $1$) and that $\mathcal{Z}(\ml_1)=\mathcal{Z}(\ml)$. Since $g_{2,k}$ is not associated to $g_{1,k}$ we have $\mathcal{D}_{\ml_2} \supset \mathcal{D}_{\ml_1}$. \\[0.2cm]
We want that $\mathcal{Z}(\ml_2) \subset \mathcal{Z}(\ml_1)$. If this is not the case, again we apply \autoref{specz} to $\ml_2$ and write $\det_k \ml_2=h_{1,k}^{\beta_1} \cdot ...\cdot h_{M',k}^{\beta_{M'}}$ with associated spectrahedra $\mathcal{D}_{\mathfrak{E}_s}$ $(s \in \{1,...,M'\})$ where $h_{1,k}=g_{2,k}$ and all other $h_{\ell,k}$ are not asssociated to $h_{1,k}$. Then we have $\mathcal{Z}(\ml_2) \subseteq \mathcal{Z}(\ml_1) \cup \bigcup_{s=2}^M \mathcal{Z}(\mathfrak{E}_s)$. $L_2$ is $\mathcal{Z}$-irreducible and $\mathcal{Z}(\ml_2) \nsubseteq \mathcal{Z}(\mathfrak{E}_\ell)$ for $\ell \geq 2$, thus $\mathcal{Z}(\ml_2) \subset \mathcal{Z}(\ml_1)$. \\[0.2cm]
\autoref{volcicklep} says that the mapping rule $(L_1)_i \mapsto (L_2)_i$ defines a unital $^*$-homomorphism $\varphi: \mathcal{A}_{L_1} \rightarrow \mathcal{A}_{L_2}$. However $\mathcal{A}_{L_1}$ and $\mathcal{A}_{L_2}$ are simple so there is some unitary $U$ such that $L_1=U^{-1}L_2U$ (\autoref{ken}). Hence $\mathcal{Z}(L_2)=\mathcal{Z}(L_1)$, a contradiction.
\end{proof}

\begin{cor} \label{ald} (cf. \cite[Theorem 3.4]{HKV})
Let $\ml$ be a monic linear pencil such that $L$ is $\mathcal{D}$-minimal and $\mathcal{D}$-irreducible. Then for $k$ big enough $\det_k \ml$ is irreducible.
\end{cor}

\begin{proof}
We have $\delta:=\size(\ml)=\pz(\mathcal{D}_\ml)$ due to \autoref{minirrchar}. From \autoref{detzel} we know already that $\det_k \ml$ is a power of an irreducible polynomial for large $k$. Hence we are finished if we can show that there is some $X \in \partial \mathcal{D}_\ml (\delta)$ such that the kernel of $\ml(X)$ is one-dimensional. Indeed, then $\det_\delta \ml$ cannot have a double zero in $X$. \\[0.2cm]
Since $\pz(\mathcal{D}_\ml)=\delta$, we find $X \in \partial \mathcal{D}_\ml(\delta)$ such that for each projection $P$ of $\C^\delta$ onto a smaller space $\ml(PXP^*) \succ 0$.
Now assume $v=\sum_{\al=1}^\delta e_\al \otimes v_\al$ and $w=\sum_{\al=1}^\delta e_\al \otimes w_\al$ are linearly indepent members of $\ker \ml(X)$. By construction we know that $\dim M(X,v)_\ml=\dim M(X,w)_\ml=\delta$. Consider the invertible linearly independent matrices $V=(v_1 \ ... \ v_\delta)$, $W=(w_1 \ ... \ w_\delta)$. Then the polynomial $\det (V+tW)$ is not constant and has a root $\lambda \in \C$. Then $v+\lambda w \in \ker \ml(X) \setminus \{0\}$, however $M(X,v + \lambda w)_\ml$ has dimension less than $\delta$. 
\end{proof}

\begin{rem} \label{ald2} 
The preceding proof shows: If $\mathcal{L}=C-B \ov X$ is a linear pencil of size $\delta$, $X \in \mathcal{D}_{\mathcal{L}}(\delta)$ such that $\mathcal{L}(PXP^*) \succ 0$ for all projections $P: \C^{\delta} \rightarrow \im(P)$ of rank at most $\delta-1$, then there exists $v \in \mathcal{S}^{\delta^2-1}$ such that $\dim M(X,v)_{\mathcal{L}}=\delta$ and $\ker(\mathcal{L}(X))=\spam\{v\}$.
\end{rem}

\begin{cor} \label{dedecomp}Let $\ml$ be a monic linear pencil of size $\delta$. Suppose $\det_k \ml = g_{1,k}^{\beta_1} \cdot ...\cdot g_{N,k}^{\beta_N}$ is the decomposition as in \autoref{level}, where $\beta_j \in \N$ denote multiplicities and such that for $k\geq M$ the polynomials $g_{j,k}$ are irreducible and pairwise not associated and $D_j$ is the induced closed matrix convex set. Then for all $j \in \{1,...,N\}$ there is a simultaneously $\mathcal{D}$-irreducible and $\mathcal{D}$-minimal monic linear pencil $\ml_j$ such that $\ml \approx \bigoplus_{j=1}^N (I_{\beta_j}) \otimes \ml_j$.
\end{cor}

\begin{proof}
Write $\ml$ as a direct sum $L=\bigoplus_{j=1}^r L_j$ of $\mathcal{D}$-irreducible and $\mathcal{D}$-minimal pencils $\ml_j$ and observe that $\det_k \ml=\prod_{j=1}^r \det_k \ml_j$.
\end{proof}

\begin{rem} \label{lax}
\autoref{rirr} showed that the number of irreducible factors of the determinant of a monic linear pencil regarded as a commutative polynomial on a fixed level $k$ is decreasing as a function of $k$. We are going to justify why this function is not always constant. \\[0.2cm] 
In \cite{K} P. Br\"{a}nd\'{e}n has given an example of an RZ-polynomial $p$ such that no power of $f$ can be written as the level $1$-determinant of a monic linear pencil but such that an RZ-polynomial $q$ exists such that the connected component of $q^{-1}(\R \setminus \{0\})$ around $0$ is containing the connected component of $p^{-1}(\R \setminus \{0\})$ around $0$ and $pq$ is the level $1$-determinant of a monic linear pencil $\ml$. Set $f_k=\det_k \ml(\mathcal{X})$ \\[0.2cm]
Write $p=p^{1}...p^{s}$ and $q=q^{1}...q^{t}$ as products of irreducible polynomials. Assume now that each $f_k$ was a product of $s+t$ irreducible polynomials. Then we could write $f_k=p^{1}_k...p^{s}_kq^{1}_k...q^{t}_k$ with $p^{j}_k$ and $q^{j}_k$ irreducible such that $p^{i}_1=p^{i}$, $q^{i}_1=q^{i}$ and $p^{j}_k|_{\S^g(\ell)}=p^{j}_\ell$, $q^{j}_k|_{\S^g(\ell)}=q^{j}_\ell$ for $\ell \leq k$. Now \autoref{dedecomp} in connection with \autoref{ald} would imply all the $p^{i}_1$ and $q^{i}_1$ are level $1$-determinants of monic linear pencils. \\[0.2cm]
The generalized Lax-conjecture is equivalent to the fact that for every RZ-polynomial $p$ with associated convex set $C_p$ we can find a RZ-polynomial $q$ as above such that $C_p \subseteq C_q$ and $pq$ can be written as the level $1$ determinant of a monic linear pencil (see \cite{K}). \\[0.2cm]
In case that $p$ is an irreducible RZ-polynomial that cannot be written as the level $1$ determinant of a monic linear pencil and $q$ is a RZ-polynomial of minimal degree such that $pq=\det_1 \ml(\mathcal{X})$ and $C_p \subseteq C_q$, the results from this chapter imply that $\det_k \ml(\mathcal{X})$ is irreducible for big $k$.   
\end{rem}

\section{Notions of extreme points for matrix convex sets}

\noindent In the classical case of a compact convex set $T$ in $\R^n$, the extreme points of $T$ form the "minimal" set of points whose convex hull is $T$ (Minkowski theorem). In the infinite-dimensional setting the Krein-Milman gives a similar statement involving closures. Let $K \subseteq \S^g$ be a compact matrix convex set. We would like to find a similar notion of extreme points which generalizes the Minkowski/Krein-Milman theorem to the matrix convex setting. \\[0.2cm]
Taking ordinary extreme points at every level is clearly too much (for instance when $\text{kz}(K) < \infty$). Also they dont reflect the nature of matrix convex combinations. Farenick, Morenz, Webster and Winkler have developed a theory of matrix extreme points. In this setting the first half of the Krein-Milman theorem holds in the sense that every compact matrix convex set $K$ has enough matrix extreme points to reconstruct itself, how the set of matrix extreme points does not need to the "minimal" set of generators (\cite{WW}, since minimality is an issue we call their theorem "weak" (free) Krein-Milman theorem). The setting in \cite{WW} is more general since we consider only matrix convex sets consisting of $g$-tuples of 
matrices where $g < \infty$ while Webster and Winkler allow tuples of matrices of infinite size.
However for most applications the less general setting is enough. We will give
a new and easier proof of the weak free Krein-Milman in our finite-dimensional setting (in the appendix we explain another variant of that proof). Indeed we will even show the free analogues of the first half of the classical Minkowski theorem and the Straszewicz theorem, which are both stronger than the Krein-Milman theorem. On the downside, the set of the matrix extreme points is only in some very weak sense a minimal generator of $K$. Also the notion of matrix extreme points does not really use the full free setting since one restricts the view only to finitely many levels. \\[0.2cm]
Another notion called absolute extreme point was introduced in \cite{EHKM} by Evert, Helton, Klep and McCullough. This smaller set of points makes use of all the levels; however it is not clear in which cases the related first half of the Krein-Milman-Theorem holds. We will even see examples of compact matrix convex sets which do not have any absolute extreme points. For $K$ being the polar of a spectrahedron, it was proven in \cite{EHKM} that the set of absolute extreme points is finite (up to unitary equivalence) and the smallest set of points generating $K$ as a matrix convex set. We will generalize this theorem to the case where $K$ has finite kz-number and is compact. The main result (general Gleichstellensatz/strong free Krein-Milman) of this chapter will be a perfect analogue of the classical Krein-Milman theorem. It states that there is a smallest operator tuple $L\in \mathcal{B}_h(\mathcal{H})^g$ defining $K$ in the sense of $K=\mconv(L)$. This tuple $L$ features all the absolute extreme points and a generalized version of absolute extreme points. \\[0.2cm]
For this chapter we borrow the notation of matrix extreme and absolute extreme points from \cite{EHKM}.

\subsection{Generating matrix convex compact sets by different kinds of extreme points}

\begin{defprop} \label{extrdef}
Let $K \subseteq \S^g$ be a matrix convex set and $A \in K(\delta)$. We call $A$ an (ordinary) extreme point of $K$ if $A$ is an extreme point of the convex set $K(\delta)$ (where we treat $\S^g(\delta)$ as a real vector space). \\[0.2cm]We call $A$ {\bfseries matrix extreme} if for all matrix convex combinations $A=\sum_{j=1}^r V_j^* B_j V_j$ where $B_j \in K(k_j)$ and the $V_j \in \C^{k_j \times \delta}$ are surjective and $\sum_{j=1}^r V_j^* V_j=I_\delta$, we already have that each $k_j=\delta$ and $A \approx B_j$. One can weaken the hypothesis of $V_j$ surjective to be $k_j \leq \delta$ and $V_j \neq 0$; also (at the same time) one can strengthen the conclusion that there are $U_j \in \C^{\delta \times \delta}$ unitary and $\lambda \in \C^r$ with $||\lambda||_2=1$ such that $B_j=U_j^* A U_j$ and $V_j=\lambda_j U_j^*$. \index{mext}\text{mext}(K) shall denote the set of matrix extreme points of $K$.\\[0.2cm]
We call $A$ {\bfseries absolute extreme} if for all matrix convex combinations $A=\sum_{j=1}^r V_j^* B_j V_j$ where $B_j \in K(k_j)$, $V_j \in \C^{k_j \times \delta} \setminus \{0\}$ and $\sum_{j=1}^r V_j^* V_j=I_\delta$ we already have $k_j \geq \delta$ and there is $C_j \in \S^g(k_j-\delta)$ such that $A \oplus C_j \approx B_j$. In this case there are $U_j \in \C^{k_j \times k_j}$ unitary and $\lambda \in \C^r$ with $||\lambda||_2=1$ such that $B_j=U_j^* (A \oplus C) U_j$ and $V_j=\lambda_j U_j^* P^*$ where $P$ is the projection from $\C^{k_j}$ to $\C^{\delta}$. \index{abex}$\text{abex}(K)$ shall denote the set of absolute extreme points of $K$.  \\[0.2cm] 
We call $A$ {\bfseries matrix exposed} if there is a linear pencil $\mathcal{L}=B-C \ov X$ of size $\delta$ such that $K \subseteq \mathcal{D}_\mathcal{L}$ and $\partial \mathcal{D}_\mathcal{L}(\delta) \cap K(\delta)=\{U^*AU \ | \ U \in \C^{\delta \times \delta} \text{ unitary} \}$. In case that $0 \in \inte{K}$ one can demand that $\mathcal{L}$ is monic. \index{mexp}$\mexp(K)$ shall denote the matrix exposed points of $K$. \\[0.2cm]
Of course absolute extreme points are matrix extreme; matrix extreme points are ordinary extreme.
\end{defprop}

\begin{proof}
Let $A$ be matrix extreme and $A=\sum_{j=1}^r V_j^* B_j V_j$ where $B_j \in K(k_j)$, $k_j \leq \delta$, $V_j \in \C^{k_j \times \delta} \setminus \{0\}$ and $\sum_{j=1}^r V_j^* V_j=I_\delta$. Let $P_j$ be the projection from $\C^{k_j}$ to the range of $V_j$ and $d_j$ the rank of $V_j$. Then $A=\sum_{j=1}^r (P_j V_j)^*(P_j B_j P_j^*) (P_j V_j)$ with $I=\sum_{j=1}^r (P_j V_j)^* (P_j V_j)$ and $P_j V_j \in \C^{d_j \times \delta}$ is surjective. Hence we get $\delta \geq k_j \geq d_j=\delta$ and the $V_j$ have been already bijective. We find $U_j$ unitary such that $B_j=U_j^* A U_j$. Now we have $A=\sum_{j=1}^r (U_j V_j)^*AU_j V_j$ and \autoref{david} implies that the $U_j V_j$ are scalar multiples of the identity (it is clear that a matrix etreme point has to be irreducible; $A=C \oplus D$ with $C,D \neq 0$ would imply $A=P^*CP+(1-P)^*D(1-P)$ where $P$ is the projection on the first coordinates and hence $C$ or $D$ have the same size as $A$). \\[0.2cm]
In a similar way we prove the additional claim for absolute extreme points.
\end{proof}

\begin{rem}
If $A \in \S^g(\delta)$ is an matrix extreme point/absolute extreme/matrix exposed point of a matrix convex set $K$ and $U \in \C^{\delta \times \delta}$ unitary, then also $U^*AU$ is matrix extreme/absolute extreme/matrix exposed. Therefore, if we say things like "$K$ has only finitely many absolute extreme points", we mean that up to unitary equivalence there are only finitely many absolute extreme points of $K$. 
\end{rem}

\begin{exar} \label{cheapexar}
Let $K \subseteq \S^g$ be matrix convex. Then the extreme points/exposed points of $K(1)$ are also matrix extreme/matrix exposed and $\text{mext}(K)(\delta)=\text{mext}(\mconv(K(\delta)))(\delta)$ as well as $\text{mexp}(K)(\delta)=\text{mexp}(\mconv(K(\delta)))(\delta)$.
\end{exar}

\begin{definition} \label{mcone}
We call a non-empty set $T \subseteq \S^{g+1}$ a {\bfseries matrix cone} if
\begin{align*}
A_j \in T(k_j), V_j \in \C^{k_j \times s} \implies \sum_{j=1}^r V_j^* A_j V_j \in T.
\end{align*}  
We call $T$ {\bfseries directed} if for all $(A_0,A_1,...,A_g) \in T(\delta)$ and $v \in \C^{\delta}$ we have $A_0 \succeq 0$ and the equality $v^*A_0v=0$ implies already $A_1v=...=A_gv=0$. If $T$ is a directed matrix cone, then \index{Deh}$\Deh(T)=\{A \in \S^g \ | \ (I,A) \in T\}$ is called the dehomogenization of $T$. For $T \subseteq \S^{g+1}$ we write $\mcc(T)$ for the smallest matrix cone containing $T$.
Let $S \subseteq \S^g$ be matrix convex. We call \index{Hom}$\Hom(S)=\mcc(\{(I_k,A) \ | \ k \in \N, A \in S(k) \})$ the homogenization of $S$. In the following we write $0_s$ for the $(g+1)$-tuple of $s \times s$ zero matrices. \\[0.2cm]
In ordinary convexity the homogenization of a convex set $S$ (i.e. the set $\text{cone}(\{(1,a) \ | \ a \in S\})$) is a trivial-looking process and it is very easy to see how properties of $S$ translate directly to properties of the homogenization of $S$ and vice versa. The homogenization of a matrix convex set is a more complicated process. It is not clear how one should compute the homogenization. We will see that the matrix extreme points of a matrix convex set $S$ are in correspondence to the ordinary extreme rays of the levels of $\Hom(S)$. This fact comes very handy in order to use the theory of ordinary convexity for the analysis of matrix extreme points. However since homogenization is a difficult process, this does not mean that matrix extreme points are easy to understand or to compute.  
\end{definition}

\begin{proposition} \label{homo} \textcolor{inv}{a}
\begin{enumerate}[(a)]
\item The mappings 
\begin{align*}
\Hom: \{S \subseteq \S^g \ | \ S \text{ is matrix convex } \} \rightarrow \{T \subseteq \S^{g+1} \ | \ T \text{ is a directed matrix cone } \} \\
\Deh: \{T \subseteq \S^{g+1} \ | \ T \text{ is a directed matrix cone } \} \rightarrow \{S \subseteq \S^g \ | \ S \text{ is matrix convex } \}
\end{align*}
are inverses of each other.
\item Let $S \subseteq \S^g$ be matrix convex. Then $\Hom(S)=\{A \in \S^{g+1} \ | \ \exists s,k \in \N_{0}, B \in S(k), V \in \C^{k \times k} \text{ invertible }: A \approx [0_s \oplus V^*(I,B)V] \}$.
\item Let $S \subseteq \S^g$ be matrix convex and $s \in \N$. Then $A \in S(\delta)$ is matrix extreme in $S$ if and only if $0_s \oplus (I,A)$ is ordinary extreme in $\Hom(S)(\delta+s)$.
\item (Effros-Winkler separation for matrix cones) Let $T \subseteq \S^g$ be a directed matrix cone $\varphi: \S^{g+1}(\delta) \rightarrow \R$ linear such that $\varphi(T(\delta) \setminus \{0\}) \subseteq \R_{>0}$. Then there exists a linear pencil $H_0X_0+...+H_gX_g$ of size $\delta$ such that $T \setminus \{0\} \subseteq \{ A \in \S^{g+1} \ | \ [H_0X_0+...+H_g X_g](A) \succ 0 \} = \{A \in \S^{g+1} \ | \ \forall V \in \C^{\size(A) \times \delta} \setminus \{0\}: \varphi(V^*AV) > 0\}$ 
\item Let $S \subseteq \S^g$ be matrix convex and $s \in \N$. Then $A \in S(\delta)$ is matrix exposed in $S$ if and only if $0_s \oplus (I,A)$ is exposed in $\Hom(S)(\delta +s)$.
\item Let $S \subseteq \S^g$ be matrix convex and compact. Then $\Hom(S)$ is closed and contains no non-trivial linear subspaces.
\item Let $T \subseteq \S^{g+1}$ be a directed matrix convex cone, $B \in T(\delta)$ and $V \in \C^{\delta \times \delta}$ invertible. Then $B$ is extreme/exposed if and only $V^* B V$ is extreme/exposed. 
\end{enumerate}
\end{proposition}

\begin{proof}
\begin{enumerate}[(a)]
\item Well-definedness of both maps is easy to see. Let $S \subseteq \S^g$ be matrix convex. The inclusion $S \subseteq \Deh(\Hom(S))$ is trivial. So let $(I,A) \in \Hom(S)(k)$. Choose $A_j \in S(k_j)$ and $V_j \in \C^{k_j \times k}$ such that $(I,A)=\sum_{j} V_j^* (I,A_i) V_j$. We conclude $\sum_{j} V_j^* V_j=I$. Hence $A \in S$. \\[0.2cm]
Let $T \subseteq \S^{g+1}$ be a directed matrix cone. Let $(A',B') \in T(\delta)$. Then there exists $k,s \in \N_0$ and $B \in \S^g(s)$, $A \in S\C^{s \times s}$ positive definite such that $(A',B') \approx 0_k \oplus (A,B)$. Hence $(A,B) \in T$. Let $D \in S\C^{s \times s}$ such that $D^2=A^{-1}$. Then we have $(I,DBD) \in T$. Hence $(I,DBD) \in \Hom(\Deh(T))$ and as $\Hom(\Deh(T))$ is a matrix cone, we get $(A',B') \in \Hom(\Deh(T))$. This shows $T \subseteq \Hom(\Deh(T))$. The other inclusion is clear.
\item Let $A \in \Hom(S)(\delta)$. Then there exists $r \in \N$ and $D_j \in S(k_j)$ and $W_j \in \C^{k_j \times \delta}$ such that $A=\sum_{j=1}^r W_j^* (I,D_j) W_j$. Let $P: \C^{\delta} \rightarrow \im(P)$ be the projection onto $\left(\bigcap_{j=1}^r \ker(W_j)\right)^\perp$. Then we have 
\begin{align*}
A \approx 0_{\delta-\text{rk}(P)} \oplus \sum_{j=1}^r P W_j^*(I, D_j) W_j P^*
\end{align*} and $\sum_{j=1}^r PW_j^* W_j P^*$ is positive definite. Hence we can choose some invertible Hermitian $V$ of the same size such that $V^2=\sum_{j=1}^r PW_j^* W_j P^*$. Now we conclude $B:=\sum_{j=1}^r V^{-1}PW_j^* D_j W_j P^*V^{-1} \in S$ and $A \approx [0_{\delta-\text{rk}(P)} \oplus V (I,B) V]$.
\item Let $A \in S(\delta)$ be matrix extreme. Let $(B_j',C_j') \in \Hom(S)(\delta+s)$ and $0_s \oplus (I,A)=\sum_j (B_j',C_j')$. Since $\Hom(S)$ is directed, we see that there are $(B_j,C_j) \in \Hom(S)(\delta)$ such that $(B_j',C_j')=0_s \oplus (B_j,C_j)$. Write $(B_j,C_j)=U_j^*(0_{s_j} \oplus V_j^*(I,A_j) V_j)U_j$ with $A_j \in S(\delta-s_j)$, $V_j \in \C^{(\delta-s_j) \times (\delta-s_j)}$ invertible and $U_j \in \C^{\delta \times \delta}$ unitary. We conclude that 
\begin{align*}(I,A)=\sum_j U_j^* \begin{pmatrix} 0 \\ V_j^* \end{pmatrix} (I,A_j) \begin{pmatrix} 0 & V_j \end{pmatrix}U_j
\end{align*}
Hence there are unitary matrices $W_{j}$ and $\lambda_{j} \in \C$ such that $A_{j}=W_{j}^* A W_{j}$ and $\begin{pmatrix} 0 & V_j \end{pmatrix}U_j=\lambda_{j} W_{j}^*$. Therefore $(B_j,C_j)$ is  a scalar multiple of $(I,A)$ and $s_j=0$. \\[0.2cm]
Let $0_s \oplus (I,A)$ be ordinary extreme in $\Hom(S)(\delta+s)$. Let $B_j \in S(k_j)$ and $V_j \in \C^{k_j \times \delta}$ surjective such that $\sum_j V_j^*V_j=I$ and $\sum_j V_j^* B_j V_j=A_j$. Then 
\begin{align*} 0_s \oplus (I,A)=\sum_{j} \begin{pmatrix} 0 \\ V_j^* \end{pmatrix} (I,B_j) \begin{pmatrix} 0 & V_j \end{pmatrix}. \end{align*} Hence for every $j$ there is $\lambda_j > 0$ such that $(V_j^*V_j, V_j^* B_j V_j)=\lambda_j (I,A)$. Thus $U_j:=\frac{V_j}{\sqrt{\lambda_j}}$ is unitary and $B_j=U_j A U_j^*$.
\item This is an easy variant of the Effros-Winkler separation technique. In this homogeneous setting we have a linear functional which can be easily translated to a linear pencil (contrary to the case in \autoref{effrosnormal} where we had to translate the affine-linear functional $1-\varphi$ with $\varphi$ the linear functional from \autoref{effrosnormal} and translation of the constant part was non-constructive). We sketch the proof: \\[0.2cm]
Extend $\varphi$ to a $\C$-linear functional $\varphi: (\C^{\delta \times \delta})^{g+1} \rightarrow \C$. By the Riesz representation theorem we can find matrices $H_0,H_1,...,H_g \in \C^{\delta \times \delta}$ such that $\varphi(C)=\sum_{i=0}^g \tr(\ov{H_i}^* C_i)$. It is easy to see that the $H_i$ have to be Hermitian. Now $H_0 X_0+...+H_g X_g$ is the required pencil. Let $B \in \S^{g+1}(k)$ and $v \in \C^{k \delta}$. Write $v=\sum_{\al=1}^\delta e_\al \otimes v_\al$ with $v_\al \in \C^k$. Define the matrix $V=\begin{pmatrix}v_1 & \hdots & v_\delta \end{pmatrix} \in \C^{k \times \delta}$. Then
\begin{align*}
&v^*[H_0 X_0+...+H_g X_g](B)v=...= \varphi(V^*BV)
\end{align*}
(see the proof of \autoref{effrosnormal} for more details)
\item Let $0_s \oplus (I,A)$ be exposed in $\Hom(S)(s+\delta)$. Then it is clear that $(I,A)$ is also exposed in $\Hom(S)(\delta)$. Let $\varphi: \S^{g+1}(\delta) \rightarrow \R$ such that $\varphi(I,A)=0$ and $\varphi(\Hom(S)(\delta) \setminus \R(I,A)) \subseteq \R_{>0}$. Claim: $\mcc(\Hom(S)(\delta) \setminus \{(V^*V,V^*AV) \ | \ V \in \C^{\delta \times \delta} \text{ invertible}\})$ is a matrix cone and has empty intersection with $\{(V^*V,V^*AV) \ | \ V \in \C^{\delta \times \delta} \text{ invertible}\}$. Indeed let $(B_j,C_j)\in \Hom(S)(\delta)$, $V_j \in \C^{\delta \times \delta} \setminus \{0\}$ and $V \in \C^{\delta \times \delta}$ invertible such that $(V^*V,V^*AV)=\sum_j V_j^* (B_j,C_j) V_j$. We infer $(I,A)=\sum_j (V^*)^{-1} V_j^* (B_j,C_j) V_j V^{-1}$. Because $(I,A)$ is extreme, we conclude that there are $\lambda_j \in \R_{>0}$ such that $(I,A)=\lambda_j (V^*)^{-1} V_j^* (B_j,C_j) V_j V^{-1} \lambda_j$ for all $j$ with $(V^*)^{-1} V_j^* (B_j,C_j) V_j V^{-1} \neq 0$. Such a $j$ exists and the claim is shown. \\[0.2cm] Now with (d) we find a pencil $H_0 X_0+...+H_g X_g$ of size $\delta$ such that $\mcc(\Hom(S)(\delta) \setminus \{(V^*V,V^*AV) \ | \ V \in \C^{\delta \times \delta} \text{ invertible}\}) \subseteq \{ B \in \S^{g+1} \ | \ [H_0 X_0+...+H_g X_g](B) \succ 0 \}$ and $[H_0 X_0+...+H_g X_g](I,A) \nsucc 0$ . \\[0.2cm]
Now let $A \in S(\delta)$ be matrix exposed in $S$ and $L_0+L \ov X$ a pencil of size $\delta$ exposing $A$. From \autoref{ald2} we know that there is $v = \sum_{\al=1}^\delta e_\al \otimes v_\al$ with $v_\al \in \C^{\delta}$ such that $\ker(L_0+L \ov X)(A)=\spam(v)$ and $\spam(v_1,...,v_\delta)=\C^{\delta}$. We claim that \begin{align*}
\varphi: \S^{g+1}(s+\delta) \rightarrow \R, \begin{pmatrix} D & C \\ C^* & B \end{pmatrix} \mapsto \text{tr}(D_0)+v^*(L_0 X_0+...+L_g X_g)(B)v
\end{align*} exposes $0_s \oplus (I,A)$. Because $T$ is directed, we only have to show that
\begin{align*}
\psi: \S^{g+1}(\delta) \rightarrow \R, B \mapsto v^*(L_0 X_0+...+L_g X_g)(B)v
\end{align*} exposes $(I,A)$ in $\Hom(S)(\delta)$. Now let $C \in S(\delta-t)$, $V \in \C^{(\delta-t) \times (\delta-t)}$ invertible, $U \in \C^{\delta \times \delta}$ unitary and $B=U^*(0_t \oplus V^*(I,C)V)U = U^*\begin{pmatrix} 0 \\ V^*\end{pmatrix} (I,C)  \begin{pmatrix} 0 & V \end{pmatrix}U$. Then we have $\psi(B)=v^* \left(I \otimes  \left(U^*\begin{pmatrix} 0 \\ V^*\end{pmatrix}\right)\right) [L_0X_0 + L\ov X](C) \left(I \otimes \left(\begin{pmatrix} 0 & V \end{pmatrix}U\right)\right) v \geq 0$. In case of $\psi(B)=0$, we conclude $t=0$ and $C \approx A$. WLOG $C=A$ since we can absorb a unitary matrix into $V$. From $\spam(v_1,...,v_\delta)=\C^{\delta}$ we deduce that $VU$ is a scalar multiple of the identity.  
\item The first part follows from (b), the second because $\Hom(S)$ is directed.
\item This follows from $T(\delta)=\{V^*AV \ | \ A \in T(\delta) \}$. \qedhere
\end{enumerate}
\end{proof}

\begin{lemma}
Let $K \subseteq \S^g$ be matrix convex and $A \in K(\delta)$. Then the following is equivalent:
\begin{enumerate}[(a)]
\item $A$ is matrix extreme.
\item $A \notin \mconv(K(\delta) \setminus \{U^*AU \ | \ U \in \C^{\delta \times \delta} \text{ unitary}\})$.
\end{enumerate} 
\end{lemma}

\begin{proof}
$(a) \implies (b)$: This is clear. \\[0.2cm]
$(b) \implies (a)$: Suppose that $A=\sum_{j=1}^r V_j^* B_j V_j$ with $B_j \in K(1),...,K(\delta)$, $V_j \in \C^{\delta \times \delta} \setminus \{0\}$. We have to show that $B_j \approx A$ for all $j \in \{1,...,r\}$. To achieve that we can suppose that $K=\mconv(A,B_1,...,B_r)$ and therefore $K$ is compact. Set $S=\mconv(K(\delta) \setminus \{U^*AU \ | \ U \in \C^{\delta \times \delta} \text{ unitary}\}) \neq K$. Then we know that $\Hom(S) \subseteq \Hom(K)$ and $(I,A) \in \Hom(K) \setminus \Hom(S)$ from \autoref{homo} (a). Also from \autoref{homo} we know that $\Hom(K)$ is closed and contains no non-trivial subspace. Hence the Minkowski theorem for cones tells us that $\Hom(K)(\delta)$ is the convex hull of its extreme rays. As $\Hom(K)(\delta) \neq \Hom(S)(\delta)$ there exists even an extreme ray $C \in \Hom(K)(\delta)$ that is not contained in $\Hom(S)(\delta)$. We have $\Hom(K)(\delta)=\Hom(S)(\delta) \cup \{V^* (I,A) V \ | \ V \in \C^{\delta \times \delta} \text{ invertible}\}$. Hence we know that there is $V \in \C^{\delta \times \delta}$ invertible such that $C=V^*(I,A)V$ is extreme in $\Hom(K)(\delta)$ and not contained in $\Hom(S)(\delta)$. \autoref{homo} tells us that $(I,A)$ is an extreme ray of $K$. Hence $A$ is matrix extreme in $K$ and for all $j$ we get $B_j \approx A$.   \\[0.2cm] $(b) \implies (a)$ also follows easily from the weak separation theorem \autoref{exposs}.
\end{proof}

\begin{lemma} \label{kreini}
Let $K \subseteq \S^g$ be compact and matrix convex with $\text{kz}(K)=\delta < \infty$. Then $K=\mconv(\text{mext}(K))$. 
\end{lemma}

\begin{proof}
From \autoref{homo} we know that $\Hom(K)(\delta)$ is closed and contains no non-trivial linear subspaces. Hence the Minkowski theorem from the theory of ordinary convexity tells us that $\Hom(K)(\delta)$ is the conic hull of its extreme rays. Let $(0_s \oplus V^*V, 0_s \oplus V^*BV) \in \Hom(K)(\delta)$ be extreme with $B \in K$ and $V$ invertible. Then \autoref{homo} tells us that $B$ is matrix extreme in $K$. Hence we see that 
\begin{align*}
\mcc\left(\left\{ (I,A) \ \middle| \ A \in \bigcup_{k=1}^\delta K(k) \text{ matrix extreme }\right\}\right)=\Hom(K).
\end{align*} From \autoref{homo} (a) we conclude that $K=\mconv(\text{mext}(K))$.
\end{proof}

\begin{theo} \label{krain} (Free Minkowski theorem for matrix extreme points - first half)
Let \marginpar{[\autoref{ex6}, \\\autoref{ex14}]} $K \subseteq \S^g$ be compact and matrix convex. Then $K=\mconv(\text{mext}(K))$.
\end{theo}

\begin{proof}
We apply \autoref{kreini} to show that $\kz(K(m))$ is the matrix convex hull of its matrix extreme points. Since the matrix extreme points of $\kz(K(m))$ are exactly the matrix extreme points of $K$ in level at most $m$, this shows the claim. 
\end{proof}

\begin{lemma} \label{charabs} \cite[Theorem 3.10]{EHKM}
Let $K \subseteq \S^g$ be matrix convex and $A \in K(\delta)$. Then $A$ is absolute extreme if and only if $A$ is irreducible and for all $\ep \in \N$, $B \in (\C^{\delta \times \ep})^g$, $C \in \S^g(\ep)$ \begin{align*}\begin{pmatrix} A & B \\ B^*&  C\end{pmatrix} \in K(\delta+\ep)\end{align*} implies $B=0$. One can weaken the condition to $\ep=1$.
\end{lemma}

\begin{proof} \cite[Theorem 3.10]{EHKM}
We prove the last part first: If $B \in (\C^{\delta \times \ep})^g \setminus \{0\}$ and $C \in \S^g(\ep)$ such that $\begin{pmatrix} A & B \\ B^*&  C\end{pmatrix} \in K(\delta+\ep)$, then we can choose $v \in \mathcal{S}^{\ep-1}$ such that $Bv \neq 0$ and have $\begin{pmatrix} A & (Bv) \\ (vB)^*&  v^*Cv\end{pmatrix} \in K(\delta+1)$ \\[0.2cm]
"$\Longrightarrow$": Let $A$ be absolute extreme. We have already seen in the proof of \autoref{extrdef} that $A$ is irreducible. Suppose $E:=\begin{pmatrix} A & b \\ b^*&  c\end{pmatrix} \in K(\delta+1)$. Then we know that there is unitary $U$ such that $U^*\begin{pmatrix} A & b \\ b^*&  c\end{pmatrix}U=\begin{pmatrix} A & 0 \\ 0 & z\end{pmatrix}=:F$ for some $z \in \R^g$. Fix $i \in \{1,...,g\}$. The characteristic polynomials of $F_i$ and $E_i$ agree for each $i \in \{1,...,g\}$. We can suppose $A_i$ is diagonal with entries $\lambda_1,...,\lambda_\delta$. Now $\chi(F_1)=\prod_{j=1}^\delta (\lambda_j-X)(z_i-X)$, $\chi(E_1)=\prod_{j=1}^\delta (\lambda_j-X)(c_i-X)-\sum_{h=1}^\delta |(b_i)_h|^2 \prod_{j \neq h}(\lambda_j-X)$. Hence
\begin{align*}
\prod_{j=1}^\delta (\lambda_j-X)(c_i-z_i)=\sum_{h=1}^\delta |(b_i)_h|^2 \prod_{j \neq h}(\lambda_j-X)
\end{align*} 
In order that the degrees on both sides agree, both sides have to equal zero; in particular we obtain $b_i=0$.  \\[0.2cm]
"$\Longleftarrow$": Suppose the right side holds. Let $A=\sum_{j=1}^r V_j^* B_j V_j$ where $B_j \in K(k_j)$, $V_j \in \C^{k_j \times \delta} \setminus \{0\}$ and $\sum_{j=1}^r V_j^* V_j=I_\delta$. WLOG let the $B_j$ be irreducible (otherwise split $V_j^*B_jV_j$ into several summands and discard those which for which the matrix weights on the left and right side are zero). Define $V^*=(V_1^* \ ... \ V_r^*)$, $C=\bigoplus_{j=1}^r B_j \in K$. Then $V$ is an isometry and $V^*CV=A$. $VV^*$ is a projection and the requirements say that $0=(1-VV^*)CV$. Now $CV=(1-VV^*)CV+VV^*CV=VA$. We also get $V^* C=A V^*$ by applying the involution. We conclude $C V V^*=V A V^*=V V^* C$ . This means
\begin{align*}
\begin{pmatrix}
B_1 V_1 V_1^* & \hdots & B_1 V_1 V_m^* \\
\vdots & \vdots & \vdots \\
B_m V_m V_1^* & \hdots & B_m V_m V_m^*
\end{pmatrix} = \begin{pmatrix}
V_1 V_1^* B_1 & \hdots & V_1 V_m^*B_m \\
\vdots & \vdots & \vdots \\
V_m V_1^* B_1 & \hdots & V_m V_m^*B_m
\end{pmatrix}
\end{align*} 
The diagonal entries together with the fact that the $B_j$ are irreducible and \autoref{burn} tell us that there is $\lambda_j \in (0,1]$ such that $V_j V_j^*=\lambda_j I$. In particular $V_j^*$ is injective and $k_j \leq \delta$. Since $CV=VA$, we have $B_j \frac{1}{\sqrt{\lambda_j}} V_j = \frac{1}{\sqrt{\lambda_j}} V_j A$ for all $j$. Thus $\frac{1}{\sqrt{\lambda_j}} V_j A \frac{1}{\sqrt{\lambda_j}} V_j^* = B_j$ and $B_j \in \mconv(A)$. We conclude $B_j \approx A$ because $A$ is matrix extreme in $\mconv(A)$ (\autoref{david}).  
\end{proof}

\begin{definition}
Let $K \subseteq \C^n$ be a convex set. We say that $K$ is balanced if for all $x \in K$ and $\lambda \in \C$ with $|\lambda| \leq 1$ already $\lambda x \in K$. 
\end{definition}

The next lemma analyzed when we can dilate a matrix extreme point to another matrix extreme points. Using slightly different constructions for the proof, similar version have appeared earlier (\cite[Lemma 2.3]{DK}) and later (\cite[Section 2.2]{EH}).

\begin{lemma} \label{matexdil}
Let \marginpar{[\autoref{ex10}]}$K \subseteq \S^g$ be a compact matrix convex set. Let $A \in K(m)$ be a matrix extreme point of $K$. Then either $A$ is absolute extreme or it dilates to a matrix extreme point
\begin{align*}
B=\begin{pmatrix}
A & b \\
b^* & c
\end{pmatrix}
\end{align*}
where $b \in (\C^m)^g \setminus \{0\}$ and $c \in \R^g$. 
\end{lemma}

\begin{proof} This proof is inspired by \cite[Proposition 5.1]{M}.
Assume that $A$ is not absolute extreme. Consider the compact convex set \begin{align*} F:=\left\{b \in (\C^m)^g \ \middle| \ \exists c \in \R^g: \ \begin{pmatrix} A & b \\ b^* & c \end{pmatrix} \in K(m+1)\right\} \end{align*} which is balanced due to
\begin{align*}
\begin{pmatrix}
I & 0 \\ 0 & \lambda^* 
\end{pmatrix}
\begin{pmatrix}
A & b \\ b^* & c 
\end{pmatrix}
\begin{pmatrix}
I & 0 \\ 0 & \lambda 
\end{pmatrix}=\begin{pmatrix}
A & \lambda b \\ (\lambda b)^* & c 
\end{pmatrix} \text{ for } \lambda \in \C \text{ with } |\lambda|=1.
\end{align*} Choose an extreme point $b \neq 0$ of $F$. Afterwards choose an extreme point $c$ of the set $E=\left\{c \in \R^g \ \middle| \ \begin{pmatrix} A & b \\ b^* & c \end{pmatrix} \in K(m+1)\right\}$ and set $B=\begin{pmatrix}
A & b \\
b^* & c
\end{pmatrix}$. \\[0.2cm]
Assume now $B=\sum_{j=1}^r V_j^* H_j V_j$ where the $V_j \neq 0$, $H_j \in \bigcup_{h=1}^{m+1}K(h)$ and $\sum_{j}^r V_j^* V_j = I$. WLOG we can suppose $H_j \in K(m+1)$ (Justification: In the case that $K(1)$ contains only one element $a$, it is easy to see that $a$ is the only matrix extreme/absolute extreme point of $K$; thus the theorem is true in this setting. In the case that some $s_j:=m+1-\size(H_j)>0$ and $K(1)$ contains infintely many elements, we can find $a_j \in K(1)$ such that $B \not\approx H_j':=H_j \oplus \bigoplus_{\al=1}^{s_j} a_j$ because the $B_i$ have only finitely many eigenvalues. Now we replace $H_j$ in the matrix convex combination by $Q_jH_j'Q_j^*$ where $Q_j: \C^{m+1} \rightarrow \C^{\size{H_j}} \oplus \{0\}^{s_j}$ is the canonical projection).\\[0.2cm] Let $P: \C^{m+1} \rightarrow \C^m, (x_1,...,x_{m+1}) \mapsto (x_1,...,x_m)$ be the projection on $\C^m$. Then we can choose unitaries $U_j$ such that $U_j V_j P^*=\begin{pmatrix}
W_j \\
0
\end{pmatrix}$. We calculate 
\begin{align*}
&A=PBP^*=\sum_{j=1}^r (W_j^* \ 0) U_jH_jU_j^* \begin{pmatrix}
W_j \\
0
\end{pmatrix}=\sum_{j=1}^r W_j^* (P U_j H_j U_j^* P^*) W_j.
\end{align*}
Define $J=\{j \in \{1,...,r\} \ | \ W_j \neq 0\}$. Therefore we conclude for all $j \in J$ that there are $\lambda_j \in \C \setminus \{0\}$ such that $P U_j H_j U_j^* P^*=Z_j^* A Z_j$ and $W_j=\lambda_j Z_j^*$ where the $Z_j \in \C^{m \times m}$ are unitary and $\sum_{j \in J} \lambda_j \lambda_j^*=1$. We deduce that $\mathcal{U}_j U_j H_j U_j^* \mathcal{U}_j^* = \begin{pmatrix} A & s_j \\ s_j^* & t_j \end{pmatrix}$ with $\mathcal{U}_j:=Z_j \oplus 1$ for some $s_j,t_j$. Define $H=\{1,...,r\} \setminus J$. For $h \in H$ write $H_h=
\begin{pmatrix}
B_h & s_h \\
s_h^* & t_h
\end{pmatrix}$. Now we calculate
\begin{align*}
&B=\sum_{j \in J} V_j^* U_j^* \mathcal{U}_j^* [\mathcal{U}_j U_j H_j U_j^* \mathcal{U}_j^*] \mathcal{U}_j U_j V_j + \sum_{h \in H} V_h^* U_h^* [U_h H_h U_h^*] U_h V_h \\
&= \sum_{j \in J} V_j^* U_j^* \mathcal{U}_j^* \begin{pmatrix} A & s_j \\ s_j^* & t_j \end{pmatrix} \mathcal{U}_j U_j V_j
+\sum_{h \in H} V_h^* U_h^* \begin{pmatrix} B_h & s_h \\ s_h^* & t_h \end{pmatrix} U_h V_h, \\
&\text{where   } 
\mathcal{U}_j U_j V_j P^*= \mathcal{U}_j \begin{pmatrix}
W_j \\
0
\end{pmatrix}=\begin{pmatrix}
Z_j W_j\\
0
\end{pmatrix}=\begin{pmatrix} \lambda_j I \\ 0\end{pmatrix} \text{ for } j \in J.
\end{align*}
We have written $B$ as 
\begin{align*}
&I=\sum_{j \in J} \begin{pmatrix}
\lambda_j^* I & 0 \\
w_j^* & y_j^*
\end{pmatrix}
\begin{pmatrix}
\lambda_j I & w_j \\
0 & y_j
\end{pmatrix}+\sum_{h \in H} \begin{pmatrix}
0 & 0 \\
w_h^* & y_h^*
\end{pmatrix}
\begin{pmatrix}
0 & w_h \\
0 & y_h
\end{pmatrix}
\\&=
\sum_{j \in J}
\begin{pmatrix}
\lambda_j^* \lambda_j & \lambda_j^* w_j \\
\lambda_j w_j^* & w_j^* w_j + y_j^*y_j
\end{pmatrix} + \sum_{h \in H}
\begin{pmatrix}
0 & 0 \\
0 & w_h^* w_h + y_h^*y_h
\end{pmatrix}\\
&B=\sum_{j \in J} \begin{pmatrix}
\lambda_j^* I & 0 \\
w_j^* & y_j^*
\end{pmatrix}
\begin{pmatrix}
A & s_j \\
s_j^* & t_j
\end{pmatrix}
\begin{pmatrix}
\lambda_j I & w_j \\
0 & y_j
\end{pmatrix}+\sum_{h \in H} \begin{pmatrix}
0 & 0 \\
w_h^* & y_h^*
\end{pmatrix}
\begin{pmatrix}
B_h & s_h \\
s_h^* & t_h
\end{pmatrix}
\begin{pmatrix}
0 & w_h \\
0 & y_h
\end{pmatrix}\\
&=\sum_{j \in J}
\begin{pmatrix}
 \lambda_j^* \lambda_j A & \lambda_j^* A w_j + \lambda_j^* s_j y_j \\
 \lambda_j w_j^* A + \lambda_j y_j^* s_j^* & w_j^*A w_j + w_j^*s_jy_j + y_j^*s_j^*w_j + y_j^* t_j y_j
\end{pmatrix}+\sum_{h \in H}\begin{pmatrix}
0 & 0 \\
0 & q_H
\end{pmatrix}
\\&=\sum_{j \in J}\begin{pmatrix}
\lambda_j^*\lambda_j A & \lambda_j^* s_j y_j \\
\lambda_j y_j s_j^* & w_j^*A w_j + w_j^*s_jy_j + y_j^*s_j^*w_j + y_j^* t_j y_j
\end{pmatrix} +\sum_{h \in H} \begin{pmatrix}
0 & 0 \\
0 & q_H
\end{pmatrix}\\
& \text{where }q_h=w_hB_h w_h^* + w_h^*s_hy_h + y_h^*s_h^*w_h + y_h^* t y_h.
\end{align*}
Because $F$ is balanced, we have $\frac{y_j}{|y_j|}\frac{\lambda_j^*}{|\lambda_j|} s_j \in F$. Now the Cauchy-Schwarz inequality shows that $b=\sum_{j \in J} \lambda_j^* y_j s_j=\sum_{j \in J} (|\lambda_j| |y_j|) \frac{y_j}{|y_j|} \frac{\lambda_j^*}{|\lambda_j|} s_j$ is a convex combination of the $\frac{y_j}{|y_j|}\frac{\lambda_j^*}{|\lambda_j|}s_j$ and $0 \in F$. $b$ is extreme and we conclude $b=\frac{y_j}{|y_j|}\frac{\lambda_j^*}{|\lambda_j|} s_j$ for all $j \in J$. The Cauchy-Schwarz inequality implies $|\lambda_j| = |y_j|$. Thus $\sum_{j \in J} y_j^*y_j=1$ and $w_\ell=0$ for all $\ell \in \{1,...,r\}$, $y_h=0$ for $h \in H$. We deduce
\begin{align*}
\begin{pmatrix}
A & s_j \\ s_j^* & t_j
\end{pmatrix}=\begin{pmatrix}
A & b \frac{|y_j|}{y_j}\frac{|\lambda_j|}{\lambda_j^*} \\ b^* \frac{|y_j|}{y_j^*}\frac{|\lambda_j|}{\lambda_j} & t_j
\end{pmatrix}=
\begin{pmatrix}
I & 0 \\
0 & \frac{|y_j|}{y_j^*}\frac{|\lambda_j|}{\lambda_j}
\end{pmatrix}
\begin{pmatrix}
A & b \\
b^* & t_j
\end{pmatrix}
\begin{pmatrix}
I & 0 \\
0 & \frac{|y_j|}{y_j}\frac{|\lambda_j|}{\lambda_j^*}
\end{pmatrix}
\end{align*} We see $t_j \in E$ and $\sum_{j \in J} y_j^* y_j t_j=t$ for all $j \in J$. Hence $t=t_j$ by the choice of $t$.
\end{proof}

\begin{cor} \label{abso} (Absolute Minkowski theorem for sets with finite $\text{kz}$-number)
Let \marginpar{[\autoref{tracesalgebra}]}$K$ be a compact matrix convex set and $\text{kz}(K)<\infty$. Then $\mconv(\abex{K})=K$. If $S \subseteq \S^g$ is a set with only irreducible elements and $\mconv(S)=K$, then $S$ contains every element of $\abex(K)$ (up to unitary equivalence). 
\end{cor}

\begin{proof}
We have $\mconv(\mex(K))=K$ due to the Minkowski theorem for matrix extreme points. Now let $A \in \mex(K)$. If $A$ is absolute extreme, $A \in \abex{K}$. Otherwise there exists $B \in \mex(K)$ and a non-trivial projection $P$ such that $A=PBP^*$. However it is clear that each matrix extreme point has size at most $\text{kz}(K)$. This means that every matrix point dilates to an absolute extreme point. Hence $\mex(K) \subseteq \mconv(\abex(K))$. The second part is clear from the definition of absolute extreme points.
\end{proof}

The following corollary was obtained in \cite{EHKM} by applying the Gleichstellensatz. We are able to conclude it from \autoref{abso}. This results in another way to prove the Gleichstellensatz (see the following proof).

\begin{cor} \label{gleich2} \cite[Theorem 1.2]{EHKM} \label{gleichcor} Let $L \in \S^g$. Then $S=\mathcal{D}_{\ml}^\circ=\mconv(L,0)$ has only finitely many absolute extreme points. Their matrix convex hull equals $S$.
\end{cor}

\begin{proof}
By \autoref{abso} and \autoref{dualo} we can choose $A_1,...,A_k \in \abex{S}$ which are pairwise not unitary equivalent and fulfill $L \oplus 0 \in \mconv(A_1,...,A_k)$. Thus $\{A_1,...,A_k\}$ equals $\abex(S)$ (up to unitary conjugations) by \autoref{abso}. For every $j \in \{1,...,k\}$ we have $A_j \in S$ and hence $A_j=0$ or $A_j \in T:=\mconv(L)$. Let $r=k$ if $0$ is not an absolute extreme point of $S$, and $r=k-1$ as wells as WLOG $A_k=0$ if $0$ is an absolute extreme point of $S$. \\[0.2cm] Since $A_1$ is absolute extreme in $T$, we know that $A_1 \oplus B_1 \approx L$ for a tuple $B_1$. Now $A_2 \in \mconv(B_1,A_1)$ and the fact that $A_2$ is absolute extreme implies that $B_1 \approx A_2 \oplus B_2$ for a tuple $B_2$. Inductively we see that up to unitary equivalence $H:=\bigoplus_{j=1}^r A_j$ is a direct summand of $L$. We observe $\mathcal{D}_{\ml}=\mathcal{D}_{\mh}$. \\[0.2cm]
We have seen the following: If $S$ is the polar of a free spectrahedron $\mathcal{D}_\ml$ and $\mh$ is a monic linear pencil such that $\mathcal{D}_\mh=\mathcal{D}_\ml$, then $S$ has only finitely many absolute extreme points and the direct sum $B$ of the absolute extreme points that do not equal $0$ is a direct summand of $H$ and defines the same free spectrahedron as $H$. Hence $I-B \ov X$ is the minimal pencil defining $\mathcal{D}_\ml$ (this is the statement of the Gleichstellensatz in the version of Helton, Klep and McCullough; note that $B$ does not depend on the pencil defining our spectrahedron).
\end{proof}

\begin{cor}
Let $S \subseteq \S^g$ matrix convex, closed and $0 \in S$. Then $S$ is the polar of a free spectrahedron defined by a monic linear pencil if and only if $S$ has only finitely many absolute extreme points and the matrix convex hull of those is again $S$.
\end{cor}

\begin{rem} 
The above construction yields also another proof of Arvesons boundary theorem for finite-dimensional operator systems in a matrix algebra (\autoref{david}). Indeed if $L \in \S^g(\delta)$ is irreducible, we have seen $\mconv(L,0)=\mconv(\abex(\mconv(L,0)))$ and it is clear that only $0$ and $L$ can be absolute extreme in $\mconv(L)$. Hence $L$ is absolute extreme. Together with \autoref{extrdef} this proves the claim if we can show that there is no unitary $U \in \C^{\delta \times \delta} \setminus \{\lambda I \ | \ \lambda \in \C\}$ such that $U^*LU=L$. If $U^*LU=L$, then the $L_i$ and $U$ commute. By Burnsides theorem \autoref{burn} the algebra generated by the $L_i$ equals $\C^{\delta \times \delta}$ and so $U$ is in its center. This means that $U$ is a scalar multiple of the identity. 
\end{rem}

\begin{proposition} \label{ragi}
Let $K \subseteq \R^g$ convex, compact. Then $A \in K$ is an exposed extreme point of $K$ if and only if for all $B \in K$ there is $\varphi: \R^g \rightarrow \R$ linear such that $\varphi(A)<\varphi(B)$ and $\varphi(A) \leq \varphi(C)$ for all $C \in K$. 
\end{proposition}

\begin{rem}\label{canv}
Let $L,H \in \S^g(\delta)$ and $v \in \mathcal{S}^{\delta^2}$ with $\ml(H) \succeq 0$, $\ker \ml(H)=\spam(v)$ and $\dim M(H,v)_\ml=\delta$. \\[0.2cm]
Write $v=\sum_{\al}e_\al \otimes v_\al=\sum_{\al}w_\al \otimes e_\al$. Set $w=\sum_{\al}e_\al \otimes w_\al$. Then $\ml(H) \approx \mh(L)$, so $\mh(L) \succeq 0$ and $\ker \mh(L)=\spam(w)$. We have $(v_1 \ ... \ v_\delta)=(w_1 \ ... \ w_\delta)^*$ and hence $\dim M(L,w)_\mh=\delta$.
\end{rem}

\begin{proposition} \label{grund}\textcolor{inv}{a}
\begin{enumerate}[(a)]
\item Let $K \subseteq \S^g$ be matrix convex. Every matrix exposed point of $K$ is matrix extreme. 
\item Let $H \in \S^g(\delta)$ be irreducible. Then $H$ is exposed in $(\mconv (H,0))(\delta)$.
\end{enumerate}
\end{proposition}

\begin{proof}
(a) Suppose $C-D \ov X$ is a linear pencil exposing $A$ and $A=\sum_j V_j B_j V_j^*$ such that $I = \sum_j V_j V_j^*$, every $V_j^*$ is surjective and $B_j \in K$. Assume $B_1$ is not unitarily equivalent to $A$. Since $(\mathcal{D}_{C-D \ov X} \cap K)(\delta-1)=\emptyset$ we know that there is $v \in \C^\delta \otimes \C^\delta$ such that $\dim M(A,v)_{C-D \ov X}=\delta$ and $\spam(v)=\ker(C-D \ov X)(A)$(\autoref{ald2}). Therefore $(I \otimes V_1^*)v=0$ implies $V_1=0$.  \\[0.2cm]
(b) Denote $K=(\mconv (H,0))(\delta)$. Due to \autoref{ragi} it is enough to show that for every $L \in K \setminus \{H\}$ we find $h: \S^g(\delta) \rightarrow \R$ affine linear such that $H \in \ker(h)$, $h(L)>0$ as well as $h(C) \geq 0$ for all $C \in K$. So fix $L \in K \setminus \{H\}$. \\[0.2cm]
Case $1$: If $L$ is not unitarily equivalent to $H$, then $\mathcal{D}_\ml \neq \mathcal{D}_\mh$. Since $L \in \mconv(H,0)$, we have $\mathcal{D}_\ml \supseteq \mathcal{D}_\mh$. Thus $\mathcal{D}_\ml \nsubseteq \mathcal{D}_\mh$ and we choose $Z \in \partial \mathcal{D}_\mh \cap \inte(\mathcal{D}_\ml)$ and $v \in \mathcal{S}^{\delta^2}$ such that $v^*\mh(Z)v=0$. Then set $h(C):=v^*(I-C \ov{X})(Z)v$ for $C \in \S^g(\delta)$ to obtain a separating affine linear function. \\[0.2cm]
Case $2$: Now assume that $U \neq I$ is unitary and $L=U^*HU$. Since $\pz(\mathcal{D}_\mh)=\delta$, we can choose $Z \in \partial\mathcal{D}_\mh$ and $v \in \C^{\delta \times \delta} \setminus \{0\}$ such that $\dim M(Z,v)_{\mh}=\delta$ and $\dim \ker (\mh)(Z)=1$ (\autoref{minirrchar} and \autoref{ald2}). We apply \autoref{canv} to get $w \in \mathcal{S}^{\delta^2}$ such that $\ker(I-Z \ov X)(H)=\spam(w)$ and $\dim M(H,w)_{I-Z \ov X}=\delta$. Now it is clear that $w^* (I- Z\ov X)(L)w > 0$ because $w$ is not an invariant subspace of $I \otimes U$. This leads to a separating hyperplane as in the case before. 
\end{proof}

\begin{lemma} \label{expos}
Let $A \in \S^g(\delta)$ be a matrix extreme point of a matrix convex set $K \subseteq \S^g$ with $0 \in K$. Suppose there is linear $\varphi: \S^g(\delta) \rightarrow \R$ such that $\varphi(A)=1$ and $\varphi(B)<1$ for all $B \in K(\delta) \setminus \{A\}$. Then $A$ is matrix exposed in $K$.
\end{lemma}

\begin{proof}
Since $A$ is matrix extreme, we know for the matrix convex set $S:=\mconv(K(\delta) \setminus \{U^* A U \ | \ U \in \C^{\delta \times \delta} \text{ unitary}\})$ that $A \notin S$. Now the result follows directly from the separation theorem for non-closed sets \autoref{effrosnormalexp}.
\end{proof}

\begin{theo} \label{grund2} (Characterization of matrix exposed points)
Let $K \subseteq \S^g$ be a matrix convex set and $A \in K(\delta)$ matrix extreme. Then $A$ is exposed in $K(\delta)$ if and only if $A$ is matrix exposed in $K$.
\end{theo}

\begin{proof}
Let $A$ be exposed. Then \autoref{expos} guarantees that $A$ is matrix exposed (We can shift $K$ to ensure $0 \in K$. If $A \neq 0$, we can apply \autoref{expos}. Otherwise the result follows from \autoref{cheapexar}.). \\[0.2cm]
Let $A$ be matrix exposed and $B-C \ov X$ be a linear pencil certifying this. Find $v \in \C^{\delta} \otimes \C^{\delta}$ such that $\ker (B-C \ov X)(A)=\spam(v)$ and $\dim M(A,v)_{B-C \ov X}=\delta$. Then $D \mapsto v^*(B-C \ov X)(D)v$ separates $A$ from all points from $K(\delta)$.
\end{proof}

\begin{cor} \label{kra}
\begin{enumerate}[(a)]
\item If $B \in \S^g$ irreducible, then $B$ is matrix exposed in $\mconv(B,0)$.
\item If $\ml$ is a monic linear pencil, then every matrix extreme point of $\mathcal{D}_\ml$ is matrix exposed. 
\end{enumerate}
\end{cor}

\begin{proof}
(a) $B$ is matrix extreme in $\mconv(B,0)$ due to \autoref{david}. Combine now \autoref{grund} and \autoref{grund2}. \\[0.2cm]
(b) It is known that every ordinary spectrahedron has only exposed faces. So every matrix extreme $A \in \mathcal{D}_\ml(\delta)$ is exposed in $\mathcal{D}_\ml(\delta)$ and due to \autoref{grund2} matrix exposed.
\end{proof}

We obtain a free Straszewicz theorem with the same homogenization trick already used in the proof of the Minkowski theorem.

\begin{cor} \label{straszewicz} (Free Straszewicz theorem)
Let $S \subseteq \S^g$ be matrix convex and compact. Then $S=\ov{\mconv(\mexp(S))}$. 
\end{cor} 

\begin{proof}
The proof runs completely analogeously to the proof of \autoref{kreini} using the classical Straszewicz theorem instead of the classical Minkowski theorem. 
\end{proof}

\subsection{Smallest defining tuples of matrix convex sets, "infinite-dimensional" absolute extreme points and the general Gleichstellensatz}

\begin{rem}
In the remainder of the chapter we need the results from the chapter about completely positive maps in the appendix. We remind the reader of \autoref{closca} and \autoref{operator}. Let $K \subseteq \S^g$ matrix convex and $\mathcal{H}$ be a separable Hilbert space. Then $K$ is compact if and only if their exists $L \in \mathcal{B}_h(\mathcal{H})^g$ such that $\mconv(L)=K$. We have \begin{align*}
K=\{ B \in \S^g(k) \ | \ \exists \varphi: \spam(I,L_1,...,L_g) \rightarrow \C^{k \times k}: \varphi(L)=B, \ \varphi \text{ is completely positive}\}.
\end{align*}
in that case. We want to analyze closed matrix convex sets and will always suppose they are given in this form. \\[0.2cm]
We have shown that the matrix extreme or matrix exposed points of a matrix convex compact set $K=\mconv(L)$ generate the whole set. Basically by considering the direct sum of all those points, we would end up with another operator tuple whose matrix convex hull is $K$. However matrix extreme/matrix exposed points are not necessarily the minimal choice. In the following, we want to show that a kind of "smallest" choice of generating points exists if we allow certain kinds of infinite-dimensional absolute extreme points. The direct sum of these points will appear as a generalized direct summand of all $H \in \mathcal{B}_h(\mathcal{H})^g$ with $\mconv(H)=K$. We need the concept of approximate unitary equivalence and a characterization from Hadwin and Larson. \\[0.2cm]
In our definitions of $\S^g$ we have excluded points of $\mathcal{B}_h(\mathcal{M})^g$ for $\mathcal{M}$ a separable infinite-dimensional Hilbert space. The next proposition shows that if we had allowed this points (in the sense that the matrix convex hull of $L$ was defined as 
\begin{align*}
\{ &B \in \mathcal{B}_h(\mathcal{M})^g \ | \ \exists \varphi: \spam(I,L_1,...,L_g) \rightarrow \mathcal{B}(\mathcal{M}): \varphi(L)=B, \ \varphi \text{ is completely positive and} \\
&\mathcal{M} \text{ is a separable Hilbert space}\}),
\end{align*}
a compact matrix convex set would still be determined by all its tuples of matrices. Therefore if $K=\mconv(L)$ with $L$ an operator tuple, we can think of $L$ as a generalized element of $K$.      
\end{rem}

\begin{proposition} \label{gelihelpcontra}
Let $\mathcal{H}$ be a separable Hilbert space and $L,H \in \mathcal{B}_h(\mathcal{H})^g$ and $\mathcal{S}$ be the operator system generated by the $L_i$. Then $\mconv(H) \subseteq \mconv(L,0)$ if and only if there is a completely positive map $\ph: \mathcal{S} \rightarrow \mathcal{B}(\mathcal{H})$ such that $\varphi(L)=H$.
\end{proposition}

\begin{proof}
Suppose $\mconv(H) \subseteq \mconv(L,0)$. By taking the polar we see 
\begin{align*}
\mathcal{D}_L = \mconv(L,0)^\circ \subseteq \mconv(H)^\circ=\mathcal{D}_H.
\end{align*}
Now the claim follows from \autoref{polariapp} \\[0.2cm]
The other direction is obtained by applying \autoref{operator}.
\end{proof}

\noindent In 1969 Arveson introduced the notion of boundary representations of operator systems in order to understand the structure of completely positive maps and complete isometries. We will see that these are a variant of absolute extreme points (possibly infinite-dimensional).

\begin{definition}
Let $\mathcal{S} \subseteq \mathcal{B}(\mathcal{H})$ be a concrete operator system. Then a completely positive map $\varphi: \mathcal{S} \rightarrow B(\mathcal{K})$ where $\mathcal{K}$ is a Hilbert space is called \textbf{boundary representation} for $S$ if $\varphi$ has only one completely positive extension $\widetilde{\varphi}: C^*(\mathcal{S}) \rightarrow B(\mathcal{K})$ and $\widetilde{\varphi}$ is an irreducible representation.
\end{definition}

\begin{lemma} \label{arviequi} \cite[Proposition 2.4]{A1}
Let $\mathcal{S} \subseteq \mathcal{B}(\mathcal{H})$ be a concrete operator system and $\varphi: \mathcal{S} \rightarrow \mathcal{B}(\mathcal{H}_1)$ completely positive. Then the following is equivalent:
\begin{enumerate}[(a)]
\item If $\psi: \mathcal{S} \rightarrow \mathcal{B}(\mathcal{H}_1 \oplus \mathcal{H}_2)$ is a completely positive map, $P$ denotes the projection from $\mathcal{H}_1 \oplus \mathcal{H}_2$ onto $\mathcal{H}_1$ and $P\psi(A)P^*=\varphi(A)$ for all $A \in \mathcal{S}$, then $\psi=\varphi \oplus \rho$ for some completely positive $\rho: \mathcal{S} \rightarrow \mathcal{B}(\mathcal{H}_2)$ 
\item There is only one completely positive extension of $\varphi$ to $C^*(S)$ and this extension is a representation (in this case we say $\varphi$ has the {\bfseries unique extension property}).  
\end{enumerate}
\end{lemma}

\begin{proof} \cite[Proposition 2.4]{A1}
$(a) \implies (b)$: Suppose $\varphi$ fulfills (a). We show that each completely positive map $\widetilde{\varphi}: C^*(\mathcal{S})\rightarrow \mathcal{B}(\mathcal{H}_1)$ extending $\varphi$ is a representation which also proves the uniqueness (existence is due to the Arveson extension theorem \autoref{arvesons}). By Stinesprings representation theorem we can find a Stinespring representation $\psi: \C^*(\mathcal{S}) \rightarrow \mathcal{B}(\mathcal{H}_1 \oplus \mathcal{H}_2)$ such that $\widetilde{\varphi}(A)=P\psi(A)P^*$ for all $A \in C^*(\mathcal{S})$ (\autoref{stinesprings}). Let $Q$ be the projection from $\mathcal{H}_1 \oplus \mathcal{H}_2$ onto $\mathcal{H}_2$. Using (a) we know that we can write $\psi(A)=\begin{pmatrix} P \psi(A) P^* & 0 \\ 0 & Q\psi(A)Q^* \end{pmatrix}$ for all $A \in \mathcal{S}$.
As $\psi$ ist a representation, this equality is even true for all $A \in C^*(\mathcal{S})$, and $P\psi P^*$ is also a representation. \\[0.2cm]
$(b) \implies (a)$: Suppose $\varphi$ fulfills (b) and $\psi: \mathcal{S} \rightarrow \mathcal{B}(\mathcal{H}_1 \oplus \mathcal{H}_2)$ is a completely positive map and $P\psi(A)P^*=\varphi(A)$ for all $A \in \mathcal{S}$. By Arvesons extension theorem \autoref{arvesons} we can extend $\psi$ to a completely positive map $\psi: C^*(\mathcal{S}) \rightarrow \mathcal{B}(\mathcal{H}_1 \oplus \mathcal{H}_2)$. From (b) we know that $P\psi P^*$ is a representation. For $A \in C^*(\mathcal{S})$ we have \begin{align*}
P \psi(A)^* P^* P \psi(A) P^*=P \psi(A^*) P^* P \psi(A) P^*=P \psi(A^*A) P^* \succeq P \psi(A)^*\psi(A) P^*
\end{align*} by the Schwarz inequality \autoref{schwatz}. This means $0 \preceq P \psi(A^*) [1-P^*P]\psi(A) P^*=P \psi(A^*) [1-P^*P]^* [1-P^*P] \psi(A) P^*$ which implies $[1-P^*P]\psi(A) P^*=0$. Thus $\psi$ is of the form $\psi=P\psi P^* \oplus \rho$.
\end{proof}

\begin{cor}
Let $\mathcal{H}$ be a separable Hilbert space, $L \in \mathcal{B}_h(\mathcal{H})^g$, $\mathcal{S}$ the operator system generated by $L_i$ and $A \in \S^g(\delta)$. Then $A$ is an absolute extreme point of $\mconv(L)$ if and only if the unital linear map $\varphi: S \rightarrow \C^{\delta \times \delta}$ defined by $\varphi(L)=A$ is a well-defined boundary representation.
\end{cor}

\begin{proof} This is true due to \autoref{operator}, \autoref{arviequi} and \autoref{charabs}.
\end{proof}

Consider a compact matrix convex set set $K \subseteq \S^g$ and write $K=\mconv(L)$ with $L \in \mathcal{B}_h(\mathcal{H})^g$. We see that a boundary representation of the operator system generated by the $L_i$ is basically a (possibly) infinite-dimensional absolute extreme point of $K$. These boundary representations do not depend on the choice of the generator $L$ (cf. \cite[Theorem 5.1]{DDSS}).\\[0.2cm] 
In order to answer the question of Arveson whether there are "enough boundary representations" for an operator system $\mathcal{S}$, Davidson and Kennedy showed that every pure completely positive map $\varphi: \mathcal{S} \rightarrow \mathcal{B}(\mathcal{H})$ where $\mathcal{H}$ is separable can be dilated to a boundary representation $\varphi$ acting again on a separable Hilbert space ("pure" can be interpreted as $\varphi$ defining a matrix extreme point). In particular the matrix convex hull of all boundary representation of the operator system generated by $L$ is again $K$. We will use some of the ideas of Davidson and Kennedy to prove the general Krein-Milman theorem/Gleichstellensatz at the end of this chapter, which shows how to extract from all boundary representations a smallest subset which generates $K$. \\[0.2cm] In the following we will show that there are not necessarily enough absolute extreme points to generate $S$, so one needs also the "infinite-dimensional" absolute extreme points to recover the complete $S$. There are even compact matrix convex sets which do not have any absolute extreme points. Considering "infinite-dimensional" absolute extreme points might also lead to other disadvantages (see \autoref{schrott}).

\begin{exar} \label{abexex}
Let $\mathcal{H}$ be a separable Hilbert space and $T_1,T_2 \in B(\mathcal{H})$ be two isometries defining the Cuntz algebra $\mathcal{A}$. Set $L=(T_1+T_1^*,\ii(T_1-T_1^*),T_2+T_2^*,\ii(T_2-T_2^*))$. Then 
there is no finite-dimensional representation of $\mathcal{A}$ because $\mathcal{A}$ is simple. Therefore the operator system $S$ defined by $L$ has also no non-trivial finite-dimensional boundary representations. Hence $\abex(\mconv(L))=\emptyset$.  
\end{exar}

The following result by Evert is the first example of a class of compact matrix convex sets having no absolute extreme points. Everts proof uses techniques from [WW] and a different description of $\mconv(L)$ than the one from \autoref{operator}, the latter one being probably the reason why it is rather long. For a concrete example see \cite[Section 5]{E}.

\begin{cor} \cite[Theorem 1.2]{E}
Let $L \in \mathcal{B}_h(\mathcal{H})^g$ be a tuple of compact operators. Suppose
that $L$ has no finite-dimensional reducing subspace. Then $\mconv(L)$ has no absolute extreme points.
\end{cor}

\begin{proof} (Sketch)
Let $\mathcal{A} = C^*(L) \subseteq \mathcal{B}(\mathcal{H})$ and $\mathcal{B}$ the non-unital $*$-algebra generated by $L$. $\mathcal{B}$ is a non-unital $C^*$-algebra of compact operators. WLOG of generality we can suppose that $\ov{C^*(L)\mathcal{H}}=\mathcal{H}$. Let $\psi: \mathcal{A} \rightarrow \mathcal{B}(\mathcal{K})$ be a non-trivial representation. Then $\psi|_\mathcal{B}$ is also a representation of non-unital $*$-algebras. \cite[Theorem 16.11 and Theorem 16.6]{C} say that $\psi|_\mathcal{B}$ decomposes into a direct sum of irreducible representations $\psi_j$ and every $\psi_j$ is unitarily equivalent to a subrepresentation of the identity representation on $\{A \in \mathcal{B}(\mathcal{H}) \ | \ A \text{ is compact}\}=C^*(L)$. Therefore the $\psi_j$ are infinite-dimensional and $\psi$ is infinite-dimensional. Thus there are no finite-dimensional boundary representations for $\mathcal{A}$.
\end{proof}

\begin{definition}
Let $\mathcal{H}$ be a separable Hilbert space and $(\mathcal{H}_n)_{n \in \N}$ a sequence of closed finite-dimensional subspaces of $\mathcal{H}$ such that $\mathcal{H}_n \subseteq \mathcal{H}_m$ for $n \leq m$ and $\ov{\bigcup_{n \in \N} \mathcal{H}_n}=\mathcal{H}$. If we are given operators $A_n: \mathcal{H}_n \rightarrow \mathcal{H}_n$ such that $A_m|_{\mathcal{H}_n}=A_n$ for all $n \leq m$ and there exists $C \in \N$ such that $||A_n|| \leq C$ for all $n \in \N$, we call $A \in \mathcal{B}_h(\mathcal{H})^g$ defined by $A|_{\mathcal{H}_n}=A_n$ for all $n \in \N$ the limit of $(A_n)_{n \in \N}$. \\[0.2cm]
In case that $S=\mconv(L) \subseteq \S^g$ is a compact matrix convex set as well as all $A_n$ are matrix extreme points of $S$ and $A$ corresponds to a completely positive map that is a boundary representation of $\mathcal{S}_L$ (the operator system defined by the $L_i$) we say that $A$ is an accesible absolute boundary point of $S$. From \autoref{gelihelpcontra} it is clear that $\mconv(A) \subseteq \mconv(S)$ in this case. We see that all absolute extreme points of $S$ are accesible absolute boundary points.  
\end{definition}

\begin{lemma} \label{irrautomatic}
Let $S \subseteq \S^g$ be compact matrix convex and $A \in \mathcal{B}_h(\mathcal{H})^g$ the limit of a sequence $(A_n)_{n \in \N}$ of matrix extreme points of $S$. Then $A$ is irreducible. 
\end{lemma}

\begin{proof}
Suppose $A=W^*(B \oplus C)W$ with $W$ unitary and $B \in \mathcal{B}_h(\mathcal{H}_B)^g, C \in \mathcal{B}_h(\mathcal{H}_C)^g$ defined on a Hilbert space of dimension at least $1$. Write $W=\begin{pmatrix}W_A \\ W_B\end{pmatrix}$. For $n \in \N$ let $P_n: \mathcal{H} \rightarrow \mathcal{H}_n$ be the canonical projection and $Q_{B,n}: \mathcal{H}_B \rightarrow \im(W_B P_n^*)$ the projection onto $\im(W_B P_n^*)$ ($Q_{C,n}$ is defined in the same manner). We obtain \begin{align*}&A_n=P_n A P_n^*=P_n W_B^* B W_B P_n^* + P_n W_C^* C W_C P_n^*\\=&[P_n W_B^*Q_{B,n}^*] Q_{B,n} B Q_{B,n}^* [Q_{B,n} W_B P_n^*] + [P_n W_C^*Q_{C,n}^*] Q_{C,n} C Q_{C,n}^* [Q_{C,n} W_C P_n^*] \end{align*} 
Since the $A_n$ are matrix extreme, we conclude that there are $\lambda_{B,n},\lambda_{C,n} \in [0,1]$ such that $\lambda_{B,n}+\lambda_{C,n}=1$ as well as $\lambda_{B,n} I=[P_n W_B^*Q_{B,n}^*] [Q_{B,n} W_B P_n^*]=P_n W_B^* W_B P_n^*$ and $Q_{B,n} B Q_{B,n}^*$ is unitary equivalent to $A_n$ in case that $\lambda_{B,n} \neq 0$ (and similar for $C$). Taking $v \in \mathcal{H}_1$ of norm $1$ we evaluate $\lambda_{B,n}=v^*\lambda_{B,n}Iv=v^* P_n W_B^* W_B P_n^* v = v^* W_B^* W_B v$. Hence $\lambda_{B,n}=:\lambda_B$ is not depending on $n$. Remembering $\mathcal{H}=\ov{\bigcup_{n \in \N} \mathcal{H}_n}$, we deduce $W_B^* W_B=\lambda_B I$. On the other hand we know $W_B W_B^*=I$ because $W$ is unitary. So $\lambda_B W_B^*=W_B^* W_B W_B^*=W_B^*$. Therefore $\lambda_B=1$ and $\lambda_C=1$, a contradiction.
\end{proof}

\begin{lemma} \label{gelihelp2}
Let $S=\mconv(L \oplus H) \subseteq \S^g$ be matrix convex and compact and $A$ an accesible absolute boundary point of $S$. Then $A \in \mconv(L)$ or $A \in \mconv(H)$.
\end{lemma}

\begin{proof}
We write $A \in \mathcal{B}(\mathcal{H})$ as the limit of a sequence $(A_n)_{n \in \N}$ of matrix extreme points of $S$. We have $\mconv(L \oplus H)=\mconv(\mconv(L) \cup \mconv(H))$. Hence we know that $A_n \in \mconv(L)$ or $A_n \in \mconv(H)$ for all $n \in \N$. WLOG suppose that all $A_n \in \mconv(L)$. Then $A \in \mconv(L)$.
\end{proof}

\begin{rem}
Let $S \subseteq \S^g$ be a compact matrix convex set and $A \in S(\delta)$. By compactness there exists $b \in (\C^\delta)^g$ and $c \in \R^g$ such that $\begin{pmatrix}A & b \\ b^* & c \end{pmatrix} \in S$ and $||(b_1)_1||$ is maximal under this property. If $d \in (\C^\delta)^g, e \in \C^g, f \in \R^g$ such that $\begin{pmatrix} A & b & d \\
b^* & c & e \\
d^*& e^*& f \end{pmatrix} \in S$, then necessarily $(d_1)_1=0$. The reason is that that for the function \begin{align*} 
p: \{ v \in \spam(e_{\delta+1},e_{\delta+2}) \ | \ ||v||=1\} \rightarrow \R, \begin{pmatrix} v_1 \\ v_2 \end{pmatrix} \mapsto \left|\left\langle \begin{pmatrix} v_1 \\ v_2 \end{pmatrix}, \begin{pmatrix} (b_1)_1 \\ (d_1)_1 \end{pmatrix} \right\rangle\right|
\end{align*} scalar multiples of $\begin{pmatrix} (b_1)_1 \\ (d_1)_1 \end{pmatrix}$ are the unique maximizers of $p$ in case that $p \neq 0$.      
\end{rem}

\begin{lemma} \label{exiacc}
Let $S \subseteq \S^g$ be a compact matrix convex set and $A \in S$ matrix extreme. Suppose $A$ does not dilate to an absolute extreme point of $S$. Then $A$ dilates to an accesible absolute boundary point $D \in \mathcal{B}_h(\mathcal{H})^g$.   
\end{lemma}

\begin{proof}
Let $\delta=\size(A)$. We want to dilate $A$ to a matrix extreme point $B=\begin{pmatrix}
A & h \\
h^* & g
\end{pmatrix}$ of $S$ of size $\delta + g \delta$ with the property that if $\begin{pmatrix} A & h & d \\
h^* & g & e \\
d^* & e^* & f 
\end{pmatrix} \in S$, then $d=0$. We call such $B$ a desired dilation of $A$. Set $A^{0,0}=A$. \\[0.2cm]
If $A^{j,k} \in S(\delta+j\delta + k)$ is already defined for $j \in \{0,...,g-1\}$ and $k \in \{0,...,\delta-1\}$, then we set $Z_{j,k}:=\left\{b \in (\C^{\delta+j\delta + k})^g \ | \ \begin{pmatrix} A^{j,k} &b \\
b^*& c
\end{pmatrix} \in S\right\}$ and choose an extreme point $b$ of $Z_{j,k}$ in such a way that the norm of $(b_{j+1})_{k+1} \in \C^{g}$ is maximal and that $b$ is extreme in $Z_{j,k}$. After choose extreme $c \in Y_j:=\left\{c \in \C^g \ | \ \begin{pmatrix} A_j& b \\
b^* &c
\end{pmatrix} \in S\right\}$. Set $A^{j+1,k}=\begin{pmatrix} A^{j,k} &b \\
b^*& c
\end{pmatrix} \in S$. \\[0.2cm]
If $A^{j,k} \in S(\delta+j\delta + k)$ is already defined for $j=g$ and $k \in \{0,...,\delta-1\}$, then we set $Z_{j,k}:=\left\{b \in (\C^{\delta+j\delta + k})^g \ | \ \begin{pmatrix} A^{j,k} &b \\
b^*& c
\end{pmatrix} \in S\right\}$ and choose an extreme point $b$ of $Z_{j,k}$ in such a way that the norm of $(b_{1})_{k+1} \in \C^{g}$ is maximal and that $b$ is extreme in $Z_{j,k}$. After choose extreme $c \in Y_j:=\left\{c \in \C^g \ | \ \begin{pmatrix} A_j& b \\
b^* &c
\end{pmatrix} \in S\right\}$. Set $A^{0,k+1}=\begin{pmatrix} A^{j,k} &b \\
b^*& c
\end{pmatrix} \in S$. \\[0.2cm]
By the first statements of the proof of \autoref{matexdil} we conclude that all the $A^{j,k}$ are matrix extreme in $S$ and in case that $\begin{pmatrix} A^{j,k}& d \\
d^*& e
\end{pmatrix} \in S$, then $(d_1)_1=(d_2)_1=...=(d_g)_1=(d_1)_2=...=(d_j)_k=0$. Now set $B=A^{g,\delta}$.\\ Now continue the same procedure by dilating $B$ to a desired dilation $C$ of $S$ and so on. By construction the limit of the sequence of these elements of $\S^g$ constitutes the accesible absolute boundary point (use \autoref{arviequi} and \autoref{irrautomatic}). 
\end{proof}

We have arrived at another type of free Krein-Milman theorem (first half) but still minimality is an issue in this version.

\begin{cor}
Let $S \subseteq \S^g$ be a compact matrix convex set. Then 
\begin{align*}
\ov{\mconv(\{A \ | \ A \text{ accesible absolute boundary point of }S\})}=S.
\end{align*}
\end{cor}

\begin{proof}
Combine \autoref{krain} and \autoref{exiacc}.
\end{proof}

In order to make precise what we mean by a "minimal" generator of a compact matrix convex set, we need the concept of approximate unitary equivalence of representations.

\begin{definition} Let $\mathcal{M}, \mathcal{H}$ be a separable Hilbert spaces
\begin{enumerate}[(a)] 
\item Let $\mathcal{A}$ be a $C^*$-algebra and $\varphi: \mathcal{A} \rightarrow \mathcal{B}(\mathcal{H}), \psi: \mathcal{A} \rightarrow \mathcal{B}(\mathcal{M})$ two linear maps. Then $\varphi$ and $\psi$ are called {\bfseries approximately unitary equivalent} (we write \index{s@$\sim_{a}$}$\varphi \sim_{a} \psi$) if there exists a sequence $(U_n)_{n \in \N}$ of unitary maps $U_n: \mathcal{H} \rightarrow \mathcal{M}$ such that for all $A \in \mathcal{A}$ we have $\lim_{n \rightarrow \infty} \varphi(A) = U_n^*\psi(A)U_n$ in the operator norm topology.
\item Let $A \in \mathcal{B}_h(\mathcal{M})^g, B \in \mathcal{B}_h(\mathcal{H})^g$ and $\mathcal{A}$ the $C^*$-algebra generated by the $A_i$. We write $A \subseteq_{a} B$ if there is a representation $\ph: \mathcal{A} \rightarrow \mathcal{B}_h(\mathcal{H})^g$ with $\ph(A)=B$, a representation $\rho: \mathcal{A} \rightarrow \mathcal{B}(\mathcal{K})$ of $\mathcal{A}$ such that $\rho \sim_{a} \text{id}_\mathcal{A}$ and an isometry $V: \mathcal{H} \rightarrow \mathcal{K}$ such that $\ph(A)=V^* \rho(A) V$ for all $A \in \mathcal{A}$. In this case $\ph$ is a direct summand of $\rho$ (\autoref{Sarason}) and we say that {\bfseries up to approximate unitary equivalence $A$ is a direct summand of $B$}.
\end{enumerate}
\end{definition}

\begin{lemma} \cite[Theorem 1 (3)]{HL} \label{hadi}
Let $\mathcal{H},\mathcal{M}$ be separable Hilbert spaces, $\mathcal{A}$ be a $C^*$-subalgebra of $\mathcal{B}(\mathcal{H})$. Consider a unital completely positive map $\ph: \mathcal{A} \rightarrow \mathcal{B}(\mathcal{M})$. Then the following is equivalent:
\begin{enumerate}[(a)]
\item There is a representation $\rho: \mathcal{A} \rightarrow \mathcal{B}(\mathcal{H})$ of $\mathcal{A}$ such that $\rho \sim_{a} \text{id}_\mathcal{A}$ and an isometry $V: \mathcal{M} \rightarrow \mathcal{H}$ such that $\ph(A)=V^* \rho(A) V$ for all $A \in \mathcal{A}$.
\item There exists a sequence $(V_n: \mathcal{M} \rightarrow \mathcal{H})_{n \in \N}$ of isometries such that $V_n^*AV_n \rightarrow \ph(A)$ for $n \rightarrow \infty$ in the weak operator topology for all $A \in \mathcal{A}$.
\end{enumerate}
\end{lemma}

\begin{lemma} \label{hadi2} \cite[Theorem 2.5, Theorem 5.1]{H}
Let $\mathcal{H}, \mathcal{M}$ be separable Hilbert spaces, $\mathcal{A} \subseteq \mathcal{B}(\mathcal{H})$ be a $C^*$-algebra and $\ph: \mathcal{A} \rightarrow \mathcal{B}(\mathcal{M})$ be a representation. 
\begin{enumerate}[(a)]
\item $\ph \sim_{a} \text{id}_\mathcal{A}$ if and only if $\ph$ is rank-preserving, i.e. for all $A \in \mathcal{A}$ the Hilbert space dimension of $\ov{\im(\ph(A))}$ is equal to the Hilbert space dimension of $\ov{(\im(A))}$. 
\item There exists a representation $\rho$ of $\mathcal{A}$ such that $\ph \oplus \rho \sim_{a} \text{id}_\mathcal{A}$ if and only if $\ph$ is rank-nonincreasing, i.e. for all $A \in \mathcal{A}$ the Hilbert space dimension of $\ov{\im(\ph(A))}$ is at most the Hilbert space dimension of $\ov{(\im(A))}$  
\end{enumerate}
\end{lemma}

\begin{lemma} \label{simpliu}
Let $\mathcal{H}, \mathcal{K}$ be a separable Hilbert spaces and $A \in \mathcal{B}_h(\mathcal{H})^g, B \in \mathcal{B}_h(\mathcal{K})^g$. If $A \subseteq_{a} B \subseteq_{a} A$, then there is a representation $\rho: \mathcal{A} \rightarrow \mathcal{B}(\mathcal{K})$ of the $C^*$-algebra $\mathcal{A}$ generated by the $A_i$ such that $\rho \sim_{a} \text{id}_\mathcal{A}$ and $\rho(A)=B$. 
\end{lemma}

\begin{proof}
Suppose that $A \subseteq_{a} B \subseteq_{a} A$ induced by representations $\ph: \mathcal{A} \rightarrow \mathcal{B}_h(\mathcal{K})^g$ with $\ph(A)=B$ and $\psi: \mathcal{B} \rightarrow \mathcal{B}_h(\mathcal{H})^g$ with $\psi(B)=A$, where $\mathcal{B}$ is the $C^*$-algebra generated by the $B_i$. Obviously, we have $\ph=\psi^{-1}$. Due to \autoref{hadi2} (b) we know that for all $C \in \mathcal{A}$ the Hilbert space dimension of $\ov{\im(\ph(C))}$ is at most the Hilbert space dimension of $\ov{(\im(C))}$ and on the other hand that for all $D \in \mathcal{B}$ the Hilbert space dimension of $\ov{\im(\psi(D))}$ is at most the Hilbert space dimension of $\ov{(\im(D))}$. Hence $\ph$ is rank-preserving and unitary equivalent to the identity representation.
\end{proof}

\begin{lemma} \label{gelihelp1}
Let $\mathcal{H}$ be a separable Hilbert space, $L \in \mathcal{B}_h(\mathcal{H})^g$ and $A \in \S^g(\delta)$ be a matrix extreme point of $\mconv(L)$. Then there are is a sequence $(S_n)_{n \in \N}$ of isometries $S_n: \C^{\delta} \rightarrow \mathcal{H}$ such that $\lim_{n \rightarrow \infty} (S_n)^* L S_n = A$ in the operator norm topology. 
\end{lemma}

\begin{proof}
Since $A \in \mconv(L)$, we know that there is a sequence of isometries $(V_n)_{n \in \N} \subseteq \mathcal{B}(\C^\delta,\mathcal{H}^{(\infty)})$ such that $\lim_{n \rightarrow \infty} ||A-V_n^*L^{(\infty)}V_n ||=0$ (\autoref{operator}). Write $V_n=\begin{pmatrix} V_n^1 \\ V_n^2 \\ \vdots \end{pmatrix}$ with $V_n^j \in \mathcal{B}(\C^\delta,\mathcal{H})$. Let $\mathcal{S}_L$ denote the operator system generated by the $L_i$ and consider the completely positive map $\ph_n: \mathcal{S}_L \rightarrow \C^{\delta \times \delta}, C \mapsto V_n^* C^{(\infty)} V_n=\sum_{j=1}^\infty (V_n^j)^* C V_n^j$. Set $r=2g\delta^2+1$. \\[0.2cm]
Now write each $V_n^j=S_n^j T_n^j$ where $S_n^j: \C^\delta \rightarrow \mathcal{H}$ is an isometry and $T_n^j \in \C^{\delta \times \delta}$. Then we have 
\begin{align*}
\varphi_n(L)=\sum_{j=1}^\infty (T_n^j)^* ((S_n^j)^*LS_n^j) T_n^j.
\end{align*} The proof of the free Caratheodory theorem \autoref{carad} applied to the $m$-th partial sum tells us that 
\begin{align*}
\sum_{j=1}^m (T_n^j)^* ((S_n^j)^*LS_n^j) T_n^j=\sum_{j=1}^r (T_{n,m}^j)^* ((S_{n,m}^j)^*LS_{n,m}^j) T_{n,m}^j
\end{align*}
for some isometries $S_{n,m}^j: \C^\delta \rightarrow \mathcal{H}$ and $T_{n,m}^j \in \C^{\delta \times \delta}$ with \hypertarget{III}{}
\begin{align*}
\sum_{j=1}^r (T_{n,m}^j)^* T_{n,m}^j=I-\sum_{j=m+1}^\infty (T_n^j)^* (T_n^j) \ \ \ \text{(III)}.
\end{align*} Since $L$ is bounded and we are only interested in the limit of the $\varphi_n$ we can demand that $0=V_n^{r+2},V_n^{r+3},...$ for all $n \in \N$ (we need the additional summand $V_n^{r+1} L V_n^{r+1}$ because the right side of \hyperlink{III}{(III)} is not zero; however one could get rid of it again with the free Caratheodory technique). By Bolzano-Weierstra{\ss} we can assume that the sequences $((S_n^j)^* L S_n^j)_{n}$, $(T_n^j)_n$ are normconvergent for all $j \in \{1,...,r+1\}$. Hence 
\begin{align*}
A=\sum_{j=1}^{r+1} \lim_{n \rightarrow \infty} (T_n^j)^* [\lim_{n \rightarrow \infty} (S_n^j)^* L S_n^j] \lim_{n \rightarrow \infty} T_n^j.
\end{align*}
Now we invoke the hypothesis that $A$ is matrix extreme to conclude that there is one $j$ such that $\lim_{n \rightarrow \infty} (S_n^j)^* L S_n^j \approx A$ already.
\end{proof}

\begin{theo} \label{geli}
Let $\mathcal{H},\mathcal{K}$ be a separable Hilbert space, $L \in \mathcal{B}_h(\mathcal{H})^g$ and $A \in \mathcal{B}_h(\mathcal{K})^g$ be an accesible absolute boundary point of $\mconv(L)$. Then $A \subseteq_{a}L$.
\end{theo}

\begin{proof}
We want to show first that there is a sequence $(R_n)_{n \in \N}$ of isometries $R_n: \mathcal{K} \rightarrow \mathcal{H}$ such that $\lim_{n \rightarrow \infty} (R_n)^* L R_n \approx A$ in the strong operator topology. If $A$ is an absolute extreme point, this follows from \autoref{gelihelp1}. So suppose that $\mathcal{K}$ is infinite-dimensional. Let $A_n \in \mconv(L)$ be a sequence of matrix extreme points of $\mconv(L)$ of increasing size such that $A_m|_{\C^{\size(A_n)}}=A_n$ and $A$ the limit of $(A_n)_{n \in \N}$. We remind the reader that $\mconv(A)=\ov{\mconv(\{A_n \ | \ n \in \N\})}$ (cf. \autoref{operator}). Again by \autoref{gelihelp1} we know that for each $n \in \N$ there a sequence $(R_n^j)_{j \in \N}$ of isometries $R_n^j: \C^{\size(A_n)} \rightarrow \mathcal{K}$ such that $\lim_{j \rightarrow \infty} (R_n^j)^* L R_n^j \approx A_n$ in the operator norm topology. By choosing $j(n) \in \N$ such that $||A_n-(R_n^{j(n)})^* L R_n^{j(n)}|| \leq \frac{1}{n}$ we see that there is a sequence $(R_n)_{n \in \N}$ of isometries $R_n: \mathcal{K} \rightarrow \mathcal{H}$ such that $A=\lim_{n \rightarrow \infty} R_n^* L R_n$ in the strong operator topology.
 \\[0.2cm] Find now $Q_n$ such that $(R_n \ Q_n)$ is unitary. Let $\mathcal{S}_L$ denote the operator system generated by the $L_i$. We know that the completely positive map $\varphi: \mathcal{S}_L \rightarrow \mathcal{B}(\mathcal{K})$ defined by $\varphi(L)=A$ has the unique extension property. We want to show that $\lim_{n \rightarrow \infty} Q_n^* L R_n=0$ in the strong operator norm topology. Otherwise there exists $x$, a sequence $y_n$ and $C>0$ such that $||y_n^* Q_n^* L R_n x|| \geq C$ for arbitrary big $n \in \N$, WLOG for all $n \in \N$. We conclude that there are completely positive maps $\varphi_n: \mathcal{S}_L \rightarrow \mathcal{B}(\mathcal{K} \oplus \C)$ such that $\psi_n(L)=\begin{pmatrix}
R_n^* L R_n & b_n \\
b_n^* & c_n
\end{pmatrix}$ with $b_n^* x=y_n^* Q_n^* L R_n x$. Now because $\text{CPU}(\mathcal{S}_L,\mathcal{B}(\mathcal{K \oplus \C}))$ is compact with respect to the BW-topology (\autoref{compactbw}) there is a subsequence $(\psi_n)_{n \in \N}$ which converges to a completely positive map $\psi: \mathcal{S}_L \rightarrow \mathcal{B}(\mathcal{K} \oplus \C)$ such that $\psi(L)=\begin{pmatrix}
A & b \\
b^* & c
\end{pmatrix}$ with $b^* x \neq 0$. This contradicts the fact that $\varphi$ has the unique extension property. \\[0.2cm]
 We define the unital completely positive map $\psi: \mathcal{S}_L \rightarrow \C^{\delta \times \delta}, C \mapsto \lim_{n \rightarrow \infty} R_n^* L R_n$ in the strong operator topology. Let $\mathcal{A}$ be the not necessarily closed algebra generated by $L$. We want to extend $\psi$ onto $\mathcal{A}$. Let $s \in \N$, $\al_1,...,\al_s \in \{1,...,g\}$ and consider $B=L_{\al_1} ... L_{\al_s}$. We see that
\begin{align*}
&\lim_{n \rightarrow \infty }R_n^* L_{\al_1} ... L_{\al_s} R_n\\&=\lim_{n \rightarrow \infty }R_n^* L_{\al_1} (Q_n Q_n^* + R_n R_n^*) L_{\al_2} (Q_n Q_n^* + R_n R_n^*) L_{\al_3} ... L_{\al_{s-1}}(Q_n Q_n^* + R_n R_n^*) L_{\al_s} R_n\\&=\lim_{n \rightarrow \infty }(R_n^* L_{\al_1} R_n) ... \lim_{n \rightarrow \infty }(R_n^* L_{\al_s} R_n)
\end{align*}
in the strong operator topology (due to the fact that the product of $(C_nD_n)_{n \in \N}$ of a sequence of uniform bounded operators $C_n$ and a sequence $(D_n)_{n \in \N}$ of operators converging $0$ in the strong operator topology converges strongly to $0$). Hence $\psi: \mathcal{A} \rightarrow \C^{\delta \times \delta}, \lim_{n \rightarrow \infty} R_n^* L R_n$ is even a $*$-homomorphism. By taking limits we see that $\psi: C^*(L) \rightarrow \C^{\delta \times \delta}, C \mapsto 
\lim_{n \rightarrow \infty} R_n^* L R_n$ is a representation of $C^*(L)$. Now the claim follows from \autoref{hadi} and the fact that if we have two representation $\pi,\rho$ on $C^*$-algebras and $\pi$ dilates $\rho$, then $\rho$ is a direct summand of $\pi$ (\autoref{Sarason}).  
\end{proof}

We obtain the following corollary, which can be also interpreted as a general Krein-Milman theorem for compact matrix convex sets $K$ that characterizes the "smallest" defining tuple $L \in \mathcal{B}_h(\mathcal{H})^g$ with respect to $\subseteq_{a}$ with $\mconv(L)=K$. This tuple will be uniquely defined up to approximately equivalent representations.

\begin{theo} (Strong free Krein-Milman for compact matrix convex sets) \label{secondhalf}\\
Let $\mathcal{H}$ be a separable Hilbert space, $L \in \mathcal{B}_h(\mathcal{H})^g$ and $K=\mconv(L)$. Let $T \subseteq \S^g$ be a (up to unitary equivalence) dense at most countable subset of $\text{abex}(\mconv(L))$ such that no two points in $T$ are unitary equivalent. Then there exists a set $P$ of infinite-dimensional accesible absolute boundary points of $\mconv(L)$ with the following properties:
\begin{enumerate} [(a)]
\item $\mconv(\left(\bigoplus_{A \in T} A \right) \oplus \left(\bigoplus_{C \in P} C \right))=K$.
\item For all $H \in \mathcal{B}_h(\mathcal{H})^g$ with $\mconv(H)=L$ there is a representation $\rho$ of $C^*(H)$ such that $\rho \sim_{a} \text{id}_{C^*(H)}$ and an operator tuple $B$ such that $\rho(H)= \left(\bigoplus_{A \in T} A \right) \oplus \left(\bigoplus_{C \in P} C \right) \oplus B$, in other words $\left(\bigoplus_{A \in T} A \right) \oplus \left(\bigoplus_{C \in P} C \right) \subseteq_{a} H$.
\item If $H \in \mathcal{B}_h(\mathcal{H})^g$ is another tuple satisfying (a) and (b), then there is a representation $\rho$ of $C^*(H)$ such that $\rho \sim_{a} \text{id}_{C^*(H)}$ such that $\rho(H)= \left(\bigoplus_{A \in T} A \right) \oplus \left(\bigoplus_{C \in P} C \right)$ 
\end{enumerate} 
In case that $\ov{\mconv(\abex K)}=K$, we can choose $P=\emptyset$.
\end{theo}

\begin{proof} (Sketch)
(c) follows directly from \autoref{simpliu} \\[0.2cm]
Suppose $T$ is infinite. Write $T=\{A_n \ | \ n \in \N\}$ and all the $A_n$ pairwise not unitary equivalent. Set $\delta_n=\size(A_n)$ for $n \in \N$ and let $\mathcal{S}_L \subseteq \mathcal{B}_h(\mathcal{H})$ be the operator system generated by the $L_i$. \autoref{geli} implies that there is a representation $\rho_1: C^*(L) \rightarrow \mathcal{B}(\mathcal{H})$, a Hilbert space $\mathcal{M}_1$ and an operator tuple $B_1 \in \mathcal{B}_h(\mathcal{M}_1)^g$ such that $\C^{\delta_1} \oplus \mathcal{M}_1=\mathcal{H}$, $\rho_1 \sim_{a} \text{id}_{C^*(L)}$ and $\rho_1(L)=A_1 \oplus B_1$. As $A_2$ is absolute extreme and $A_2 \not\approx A_1$, we have $A_2 \notin \mconv(A_1)$. \autoref{gelihelp2} says that $A_2 \in \mconv(B_1)$. Hence by \autoref{geli} there exists a representation $\rho_2: C^*(B_1) \rightarrow \mathcal{B}(\mathcal{M}_1)$, a Hilbert space $\mathcal{M}_2$ and an operator tuple $B_2 \in \mathcal{B}_h(\mathcal{M}_2)^g$ such that $\C^{\delta_2} \oplus \mathcal{M}_2 = \mathcal{M}_1$, $\rho_2 \sim_{a} \text{id}_{C^*(B_1)}$ and $\rho_2(B_1)=A_2 \oplus B_2$. \\[0.2cm]
By continuing in the above way we construct for $n \in \N$ the closed subspaces $\mathcal{M}_n$ of $\mathcal{M}_{n-1}$ with $\mathcal{M}_0=\mathcal{H}$, $\mathcal{M}_{n-1}=\C^{\delta_n} \oplus \mathcal{M}_n$, $B_n \in \mathcal{B}_h(\mathcal{M}_n)^g$, $B_0=L$, representations $\rho_n: C^*(B_{n-1}) \rightarrow \mathcal{B}(\mathcal{M}_{n-1})$, $\rho_n \sim_{a} \text{id}_{C^* (B_{n-1})}$ and $\rho_n(B_{n-1})=A_n \oplus B_n$. \\[0.2cm]
Set $\mathcal{K}=\bigoplus_{n \in \N} \C^{\delta_n} \subseteq \mathcal{H}$. Now for $n \in \N$ we can find a sequence of isometries $(U_{n,m}: \mathcal{M}_{n-1} \rightarrow \mathcal{M}_{n-1})_{m \in \N}$ such that $\rho_n(C)=\lim_{m \rightarrow \infty} U_{n,m}^* C U_{n,m}$ in the operator norm for all $C \in C^*(B_{n-1})$. Let $1_n: \bigoplus_{k=1}^n \C^{\delta_k} \rightarrow \mathcal{H}, D \mapsto D$. For $n,m \in \N$ set $V_{n,m}=\begin{pmatrix} 1_n & 0 \\ 0 & U_{n,m} \end{pmatrix}: \mathcal{H}= \bigoplus_{k=1}^{n-1} \C^{\delta_k} \oplus \mathcal{M}_{n-1} \rightarrow \mathcal{H},\begin{pmatrix} x \\ y \end{pmatrix} \mapsto \begin{pmatrix} x \\ U_{n,m}y \end{pmatrix}$. We define $\rho: \mathcal{S}_L \rightarrow \mathcal{B}(\mathcal{K})$
\begin{align*}
L \mapsto \left(\mathcal{K} \mapsto \mathcal{K}, x \mapsto \lim_{m \rightarrow \infty} V_{m,m}^* ... V_{1,m}^* L V_{1,m} ... V_{m,m}x\right)=\left(\bigoplus_{n \in \N} A_n\middle)\right|_\mathcal{K} 
\end{align*} in the strong operator topology.
Since all the $V_{n,m}$ are isometries, we can take limits and extend $\rho$ to $\rho: \mathcal{S}_L \rightarrow \mathcal{B}(\ov{\mathcal{K}})$
\begin{align*}
L \mapsto \left(\ov{\mathcal{K}} \mapsto \ov{\mathcal{K}}, x \mapsto \lim_{m \rightarrow \infty} V_{m,m}^* ... V_{1,m}^* L V_{1,m} ... V_{m,m}x\right)=\bigoplus_{n \in \N} A_n
\end{align*}
In the same way as in the proof of \autoref{geli} we can extend $\rho$ to a representation $\rho: C^*(L) \rightarrow \mathcal{B}(\ov{\mathcal{K}})$ keeping the same mapping rule (a direct sum of completely positive maps with the unique extension property has again the unique extension property). Now we use \autoref{hadi} in combination with \autoref{Sarason} to obtain: There is a representation $\rho_2$ such that $\rho \oplus \rho_2 \sim_{a} \text{id}_{C^*(L)}$ or in other words $\bigoplus_{A \in T} A \subseteq_{a}L$. The same statement in the case of finite $T$ is technically easier and follows in a similar fashion. \\[0.2cm]
In case that $\ov{\mconv(\abex K)}=K$, we are done. Indeed it is immediate that (a) holds with $P=\emptyset$. Since we have used only $K$ to define $T$ and $P$, but not $L$, (b) holds as well (set $H=L$). So suppose $\ov{\mconv(\abex K)} \neq K$. Choose a dense sequence $(E_m)_{m \in \N}$ of $\text{mext}(K)$. We form $P$ by following the upcoming algorithm:
Set $P_0= \emptyset$. Let $m \in \N$ and $P_m$ be a finite set of accesible absolute boundary points of $K$ already defined. If $E_m \in \mconv \left(\bigoplus_{A \in T} A \right) \oplus \left(\bigoplus_{C \in P_m} C \right)$ set $P_{m+1}=P_m$. So suppose the contrary is the case. By \autoref{exiacc} $E_m$ dilates to an accesible absolute boundary point $C_m$ defined on an infinite-dimensional separable Hilbert space $\mathcal{K}_m$. Set $P_{m+1}=P_m \cup \{C_m\}$. \\[0.2cm]
In the end we set $P=\bigcup_{m \in \N} P_m$. By \autoref{kreini} we have $\mconv(K)=\mconv(\text{mext}(K))$. Therefore (a) is fulfilled. Basically, we repeat now the strategy of the start of the proof. For simplicity of notation we suppose $P$ is infinite and change the numeration of the $C_m$ in such a order-preserving way such that $P=\{C_m \ | \ m \in \N\}$. Inductively we show with the help of \autoref{geli} that for all $m \in \N$ there is a representation $\tau_m: C^*(L) \rightarrow \mathcal{B}(\ov{\mathcal{K}}), L \mapsto \left(\bigoplus_{A \in T} A \right) \oplus \left(\bigoplus_{C \in P_m} C \right)$ and a representation $\sigma_m$ on $C^*(L)$ such that $\tau_m \oplus \sigma_m \sim_{a} \text{id}_{C^*(L)}$. For the induction start and induction step use \autoref{gelihelp2}, which implies $P_1 \in \mconv(\rho_2(L))$ because $P_1 \notin \mconv(\rho(L))$ and $P_{m+1} \in \mconv(\sigma_m(L))$ because $P_{m+1} \notin \mconv(\tau_m(L))$ for $m \in \N$. \\[0.2cm]
Now \autoref{hadi2} says that all $\tau_m$ are ranknonincreasing and therefore also the representation on $C^*(L)$ defined by $L \mapsto \left(\bigoplus_{A \in T} A \right) \oplus \left(\bigoplus_{C \in P} C \right)$ is ranknonincreasing. Again with \autoref{hadi2} $L \subseteq_{a} \left(\bigoplus_{A \in T} A \right) \oplus \left(\bigoplus_{C \in P} C \right)$ follows. Since we did only use $K$ but not the special form of $L$ in our construction, (b) follows by setting $H=L$.
\end{proof}  

\begin{cor}(General Gleichstellensatz)
Let $T$ be a closed matrix convex set with $0 \in \inte(T)$. Then there exists a separable Hilbert space $\mathcal{H}$ and $L \in \mathcal{B}_h(\mathcal{H})^g$ such that $T=\mathcal{D}_\ml$ and for all other $H \in \mathcal{B}_h(\mathcal{K})^g$ with $T=\mathcal{D}_\mh$ we have $L \subseteq_a H$.
\end{cor}

\begin{proof}
We apply \autoref{secondhalf} to find a smallest generator $L \in \mathcal{B}_h(\mathcal{H})^g$ of $T^\circ$. Let $K=\left( \bigcap_{i=1}^g \ker(L_i)\right)^\perp$. Then $L|_K:=(L_1|_K,...,L_g|_K)$ is the desired tuple (cf. \autoref{polariapp}). The restriction onto $K$ is only necessary if $0$ is not contained in the closure of the matrix convex hull of the other matrix extreme points of $T^\circ$.
\end{proof}

\section{Examples}

\begin{exar} \label{ex1}
Let \marginpar{[\begin{footnotesize}\autoref{arviskz}\end{footnotesize}]}$\ml \in S(\NC)_1^{\delta \times \delta}$ be a monic linear pencil defining a bounded spectrahedron $\mathcal{D}_\ml$. We see that $g \leq \delta^2-1$ because the $L_i$ and $I$ have to be linear independent. \autoref{positest2} states that every $\delta$-positive map from $\C^{\delta \times \delta}$ into a $C^*$-algebra $\mathcal{B}$ is completely positive. Now we can find $L_{g+1},...,L_{\delta^2-1}$ such that the $\spam_\C(L_1,...,L_{\delta^2-1})$ contains no non-trivial positive semidefinite matrix. Hence the pencil $H=I-\sum_{j=1}^{\delta^2-1} L_j X_j$ defines a compact spectrahedron $\mathcal{D}_H$ in $\S^{\delta^2-1}$, the operator system defined by the $L_1,...,L_{\delta^2-1}$ equals $\C^{\delta \times \delta}$ and $\text{kz}(\mathcal{D}_H) \leq \delta$ (\autoref{arviskz}). We see $\mathcal{D}_\ml=\mathcal{D}_H \cap \S^g$. 
\end{exar}

\begin{exar} \label{ex11} \cite[Section 2.1.2]{HKM2}
Consider \marginpar{[\begin{small}\autoref{existence}\end{small}]} the set $S=\{ (X,Y) \in \S^2 \ | \ 1-X^4-Y^4 \succeq 0 \}$. It is known that $S(1)$ is convex, however not an ordinary spectrahedron because the polynomial $1-X^4-Y^4$ is not an RZ-polynomial. Thus $S$ is also not a free spectrahedron. $\inte{S}$ is not even matrix convex as otherwise the requirements of \autoref{existence} would be fulfilled.
\end{exar}

\begin{exar} \label{ex2}
Let \marginpar{[\begin{small}\autoref{gleichi}\end{small}]} $S \subseteq \S^g$ be a spectrahedron and $\ml$ a monic pencil of minimal size $k$ defining $S$. Let $H=\{ h \in \R^{g \times g} \ | \ h \text{ invertible and } A \in S \implies hA= (\sum_{j=1}^g h_{i,j} A_j )_{i \in \{1,...,g\}}\in S  \}$ be the group of automorphisms which leave $S$ invariant. Then there exists a group representation $U: H \rightarrow \{U \in \C^{k \times k} \ | \ U \text{ is unitary}\}, h \mapsto U_h$ such that $I-L(h\ov X)=U_h\ml U_h^*$ for $h \in H$.
\end{exar}

\begin{proof}
WLOG we suppose that there is no unitary $U$ such that $ULU^*=L$. This is for instance the case when $\spam(I,L_1,...,L_g)=\C^{k \times k}$. We can can reach this case by introducing new variables if necessary. The Gleichstellensatz tells us that for each $h \in H$ there is a unique unitary $U_h$ such that $(I-L(h\ov X))=U_h\ml U_h^*$. We only have have to show that $h \mapsto U_h$ is multiplicative. \\[0.2cm]
Let $f,h \in H$ and $A \in \S^g$. We have $U_{hf} (L \ov X)(A) U_{hf}^*=(L\ov X)(h fA)=\sum_{i=1}^g L_i \otimes \left(\sum_{j=1}^g (hf)_{i,j} A_j\right)=\sum_{i=1}^g L_i \otimes \left(\sum_{k=1}^g \sum_{j=1}^g h_{i,k} f_{k,j} A_j\right)=\sum_{i=1}^g L_i \otimes \left(\sum_{k=1}^g h_{i,k} (f A)_k\right)=U_h (L \ov X)(fA) U_h^*=U_h U_f (L \ov X)(A) U_f^* U_h^*$. By the uniqueness of $U_{hf}$ we get $U_{hf}=U_{h}U_{f}$.
\end{proof}

\begin{exar} \label{kreu} Let \marginpar{[\begin{small}\autoref{gleichi}\end{small}]} $\ml,\mh$ be $\mathcal{D}$-minimal and $\mathcal{D}$-irreducible monic linear pencils describing two spectrahedra $S,T \subseteq \S^g$. Then the pencil $I-[(L \ov X) \oplus (H \ov Y)]$ is a description of minimal size of $S \times T$. 
\end{exar}

\begin{proof}
We have $S \times T=(S \times \S^g) \cap (\S^g \times T)=\mathcal{D}_{I-L \ov X} \cap \mathcal{D}_{I-H\ov Y}$. Now the Gleichstellensatz \autoref{gleichi} in connection with \autoref{minirrchar} implies that $L \oplus H$ is a direct summand of every other monic linear pencil defining the spectrahedron $S \times T$. 
\end{proof}

If we take a monic linear pencil $\ml$ of size $\delta$ and fix $m \in \N$, then $(\mathcal{D}_\ml(mn))_{n \in \N}$ can be interpreted as a free spectrahedron $S$ in $\S^{gm^2}$ in a natural way defined by a pencil $I-H \ov Y$. The following theorem shows: In this case the spectrahedron $S$ "knows" that it constitutes only some levels of a spectrahedron in $g$ variables and the structure is reflected in the pencil $I-H \ov Y$. 

\begin{exar} \label{ex13}
Let \marginpar{[\begin{small}\autoref{gleichi}\end{small}]}$r,\delta \in \N$ and $\mathcal{L}: \S^g(m) \rightarrow S\C^{r \times r}$ linear. We can write an element of $\S^g(m)$ as the evaluation of the $g$-tuple of generic $(m \times m)$-matrices \hypertarget{IV}{}\begin{align*}
\mathcal{Y}=\sum_{j=1}^m E_{j,j} \mathcal{Y}_{j,j} + \sum_{j,k \in \{1,...,m\} j < k}^m E_{j,k} \mathcal{Y}_{j,k} + E_{k,j}  \mathcal{Y}_{k,j} \ \ \text{(IV)}
\end{align*}
where the $\mathcal{Y}_{j,k}=(\mathcal{Y}_{j,k}^1,...,\mathcal{Y}_{j,k}^g)$ are $g$-tuples of variables. The list of all these variables constitute the tuple $\ov {\mathcal{Y}}=(\mathcal{Y}_{j,k}^i)_{j,k \in \{1,...,m\}, i \in \{1,...,g\}}$ of $gm^2$ variables.
We define the linear pencil $\mathcal{L} \ov{\mathcal{Y}}$ by setting: $\mathcal{L}_{j,k}^i$ is the result of substituting $1$ for $\mathcal{Y}_{j,k}^i$ in $\mathcal{L}(\mathcal{Y})$ and $0$ for every other variable. Set $T=\mathcal{D}_{I-\mathcal{L} \ov{\mathcal{Y}}} \subseteq (\S^{gm^2}(n))_{n \in \N}$. Suppose $I-\mathcal{L} \ov{\mathcal{Y}}$ is the pencil of minimal size in $gm^2$ variables defining $T$. \\[0.2cm]
We identify $(\S^g(m))$ with $\S^{gm^2}(1)$ by associating 
\begin{align*}
\begin{pmatrix} B_{1,1} & \hdots & C_{1,m} + \ii C_{m,1} \\
\vdots &\vdots& \vdots \\
C_{1,m} - \ii C_{m,1} & \hdots & B_{m,m} \end{pmatrix} \text{  with  } (B_{1,1},...,B_{m,m},C_{1,2},C_{2,1}...,C_{m,m-1},C_{m-1,m})
\end{align*} in the same way as we associated $\mathcal{Y}$ with $\ov{\mathcal{Y}}$. By tensoring from the right side with $n \times n$-matrices we also associate $(\S^g(mn))$ with $\S^{gm^2}(n)$ . Let \begin{align*}
S_j=\{A \in \S^g \ | \ (0,...,0,\underbrace{A}_{j,j\text{-th index}},0,...,0) \in T\}.\end{align*} If $S_j(mn)=T(n)$ for all $j \in \{1,...,m\}$ and $n \in \N$, then $S:=S_1$ is a free spectrahedron. For the minimal pencil $I-H \ov X$ defining $S$ we have that 
\begin{align*}
I-\mathcal{L} \ov{\mathcal{Y}} \approx (I-H \ov X)\left(\sum_{j=1}^m E_{j,j} \otimes \mathcal{Y}_{j,j} + \sum_{j<h}^m E_{j,h} \otimes \mathcal{Y}_{j,h} + E_{h,j} \otimes \mathcal{Y}_{h,j}\right)
\end{align*} 
\end{exar}

\begin{proof}
For $j \in \{1,...,m\}$ let $I-\mathcal{L}_{(j,j)}(\ov{\mathcal{Y}}_{j,j})$ be the the pencil we get by substituting every $\ov{\mathcal{Y}}_{h,k}$ where $(h,k) \neq (j,j)$ with $0$ in $I-\mathcal{L} \ov{\mathcal{Y}}$. By definition we have $\mathcal{D}_{I-\mathcal{L}_{(j,j)}\ov{\mathcal{Y}}_{j,j}}=S_j=S \subseteq \S^g$. \\[0.2cm]
Let $I-\mathcal{L}^{\text{diag}}(\ov{\mathcal{Y}}_{1,1},...,\ov{\mathcal{Y}}_{m,m})$ be the the pencil we get by substituting every $\ov{\mathcal{Y}}_{k,j}$ where $k \neq j$ with $0$ in $I-\mathcal{L} \ov{\mathcal{Y}}$. We see that $\mathcal{D}_{I-\mathcal{L}^{\text{diag}}(\ov{\mathcal{Y}}_{1,1},...,\ov{\mathcal{Y}}_{m,m})}=S^m\subseteq \S^{gm^2}$ by identifying the first set with $(S_j(mn))_{n \in \N}$ restricted to block-diagonal tuples. \\[0.2cm]
Let $I-H \ov X$ be a pencil of minimal size $\delta$ describing $S$. \autoref{kreu} tells us that the size of $I-\mathcal{L}^{\text{diag}}(\ov{\mathcal{Y}}_{i,i},...,\ov{\mathcal{Y}}_{m,m})$ is at least $\delta m$. Hence 
\begin{align*}
(I-H \ov X)\left(\sum_{j=1}^m E_{j,j} \otimes \mathcal{Y}_{j,j} + \sum_{j<h}^m E_{j,h} \otimes \mathcal{Y}_{j,h} + E_{h,j} \otimes \mathcal{Y}_{h,j}\right)
\end{align*} 
is a monic linear pencil of size $\delta m$ defining $T$. We have seen that it is of minimal size. Hence the claim follows from the Gleichstellensatz.
\end{proof}

\begin{exar} \label{ex3}
Let \marginpar{[\begin{small}\autoref{nullst}\end{small}]}$S \subseteq \S^g$ be closed matrix convex with $0 \in \inte{S}$. Let $\pz(S)=k< \infty$ and suppose that there is $L \in \S^g$ such that $\mathcal{D}_\ml(k+1)=S(k+1)$. Then $S$ is a spectrahedron.  
\end{exar}

\begin{proof}
Let $T=S^\circ$. Then $\kz(T) = k$ and all matrix extreme points of $T$ have size at most $k$. We have $\mconv(\abex(T))=T$. Since $T(k+1)=\mconv(L\oplus 0)(k+1)$ we know that each absolute extreme point of $T$ is an absolute extreme point of $\mconv(L \oplus 0)$ (\autoref{charabs}) and therefore a direct summand of $L \oplus 0$. Thus $T$ can have only finitely many absolute extreme points $B_1,...,B_m$. Now $S=\mathcal{D}_{I-\bigoplus_j B_j \ov X}$.
\end{proof}

\begin{lemma} (Extreme points of spectrahedra) (cf. \cite[Corollary 3]{RG}) \label{exspec}
Let $L \in \S^g,L_0 \in S\C^{\delta \times \delta}$ such that $\{L_0,...,L_g\}$ is linear independent and $S=\mathcal{D}_{L_0 X_0 + ... + L_g X_g} \subseteq \S^{g+1}, A \in S(\delta) \setminus \{0\}$. Then in the following is equivalent:
\begin{enumerate}[(a)]
\item $A$ is an extreme ray in $S(\delta)$.
\item For every $B \in \S^{g+1}(\delta)$ the inclusion $\ker (L_0 X_0 + ... + L_g X_g)(B) \supseteq \ker (L_0 X_0 + ... + L_g X_g)(A)$ implies that $B=\lambda A$ for some $\lambda \in \R$.
\item For every $B \in S(\delta)$ the inclusion $\ker (L_0 X_0 + ... + L_g X_g)(B) \supseteq \ker (L_0 X_0 + ... + L_g X_g)(A)$ implies that $B=\lambda A$ for some $\lambda \geq 0$.
\end{enumerate}
\end{lemma}

\begin{proof} Set $\mathcal{L}=L_0 X_0 + ... + L_g X_g$. \\[0.2cm]
$(b) \implies (c)$: (c) is an obvious consequence of (b) together with the fact that $\{L_0,...,L_g\}$ is linear independent.\\[0.2cm]
$(c) \implies (a)$: Suppose $A=B+C$ with $B,C \in S(\delta)$. Then $\ker \mathcal{L}(A) \subseteq \ker \mathcal{L}(B)$ and hence there is $\lambda \geq 0$ such that $B= \lambda A$. \\[0.2cm]
$(a) \implies (b)$: Let $B \in \S^{g+1}(\delta)$ and $\ker \mathcal{L}(B) \supseteq \ker \mathcal{L}(A)$. Then we find $\mu>0$ such that $A \pm \mu B \in S(\delta)$. Hence we find $\lambda\geq 0$ such that $A+\mu B=\lambda A$ and $B=\frac{\mu}{\lambda-1}A$.  
\end{proof}

\begin{proposition} \label{cdlabel} \cite[Lemma 2]{AZ}
Let $A,B \in S\C^{k \times k}$ and suppose that $A^2x \in \spam\{x,Ax\}$, $B^2x \in \spam\{x,Bx\}$ for all $x \in \C^k$. Then there is a reducing subspace of dimension $1$ or $2$ of $A$ and $B$.
\end{proposition}

\begin{proof} \cite[Lemma 2]{AZ}
Choose an eigenvector $x \in \C^k \setminus \{0\}$ of $A-B$. Set $H=\spam\{x,Ax\}=\spam\{x,Bx\}$. It is evident that $H$ is an invariant subspace of $A$ and $B$. Since $A$ and $B$ are Hermitian, it is also reducing.
\end{proof}

\begin{exar} \label{ex6}(Free cube) (cf. \cite[Proposition 7.1]{EHKM}) Let \marginpar{[\begin{small}\autoref{krain}\end{small}]} $g \geq 2$, $S=\mathcal{D}_\ml$ where $L_i=e_i e_i^* \otimes \begin{pmatrix} 1 & 0 \\ 0 & -1 \end{pmatrix}$. Then $S=\{A \in \S^g \ | \ \forall i \in \{1,...,g\}: I-A_i^2 \succeq 0\}$. $A \in S$ is absolute extreme if and only if [$A$ is ordinary extreme and irreducible] if and only if [$A_i^2=I$ for every $i \in \{1,...,g\}$ and $A$ is irreducible]. $S$ is the matrix convex hull of its absolute extreme points. \\[0.2cm]
If $g=2$, then $\kz(S)=2$. If $g=3$, then $\kz(S)=\infty$.
\end{exar}

\begin{proof}
Let $A \in S$ and $\lambda \in (-1,1)$ be an eigenvalue of $A_1$ with eigenvector $v$. Then we can find $\mu \in \R$ small such that $A=\frac{1}{2}(A_1+\mu vv^*,A_2,...,A_g) + \frac{1}{2}(A_1-\mu vv^*,A_2,...,A_g) \in S$ is not an extreme point of $S$. \\[0.2cm] On the other hand let $A \in S$ be irreducible such that $A_i^2=I$ for all $i \in \{1,...,g\}$. Suppose $B=\begin{pmatrix} A & b \\ b^* & c \end{pmatrix} \in S$. Then $B_i^2 =\begin{pmatrix} A_i^2 + b_ib_i^* & A_ib_i + b_ic_i \\ b_i^*A_i + c_ib_i^* & b_i^*b_i + c_i^2 \end{pmatrix}\preceq I$, in particular $I \succeq A_i^2+b_ib_i^*=I + b_ib_i^*$. We conclude that $b=0$ and $A$ is absolute extreme. \\[0.2cm]
Now let $g=2$ and $A$ be absolute extreme. The $A_i$ have the eigenvalues $-1$ and $1$. Hence the requirements of the \autoref{cdlabel} are fulfilled and $A$ has a reducing subspace of dimension $1$ or $2$. However $A$ is irreducible, thus $A \in S(1) \cup S(2)$. Now let $g \geq 3$, $m \in \N$ and $n=2m$. We want to construct $A \in S(n)$ absolute extreme. Chandler has shown in \cite[Theorem 1]{CD} that there exist $3$ projections $\uv{P},\uv{Q},\uv{R}: \C^{n} \rightarrow \C^{n}$ (the line under $P$ means that $\uv{P}$ is not necessarily surjective like most other projections in this paper) such that the $C^*$-algebra generated by $\uv{P},\uv{Q},\uv{R}$ equals $\C^{n \times n}$. Burnsides theorem \autoref{burn} implies that $(\uv{P},\uv{Q},\uv{R})$ is irreducible. We see that $(1-2\uv{P},1-2\uv{Q},1-2\uv{R}) \in S$ is irreducible and forms an absolute extreme point of $S$.
\end{proof}

\begin{rem} \label{schrott}
\cite[Theorem 1]{CD} used in the last proof is also valid for infinite-dimensional separable Hilbert spaces $\mathcal{H}$. There exist three projections which generate as a $C^*$-algebra already $\mathcal{B}(\mathcal{H})$. This shows that a matrix convex $S$ set can have an infinite-dimensional boundary representation even though $\abex(S)=S$ already.
\end{rem}

\begin{exar} \label{tracesalgebra}
Let \marginpar{[\begin{small}\autoref{abso}\end{small}]} $L \in \S^g(\delta)$ and suppose that $\spam\{I,L_1,...,L_g\}=\C^{\delta \times \delta}$ and $\mathcal{D}_{\ml}$ is bounded. Then $\mathcal{D}_{\ml}$ is the matrix convex hull of its absolute extreme points and $A \in \S^g(k)$ is absolute extreme if and only if $k=\delta$, $A$ is irreducible and $\dim \ker\ml(A)=k\delta-1$. \\[0.2cm]
The absolute extreme points $A$ of $\mathcal{D}_\ml$ are in bijective correspondence to the one-dimensional subspaces $\spam(v)$ of $\C^{\delta^2}$ given by some $v=\sum_{\al=1}^\delta v_\al \otimes e_\al$ such that $(v_1,...,v_\delta)$ is an orthonormal basis of $\C^\delta$ with respect to some scalar product on $\C^\delta$ depending on $L$. An absolute extreme point $A$ corresponds to $\ker \ml(A)$.  
\end{exar}

\begin{proof}
Because $S:=\mathcal{D}_{\ml}$ is bounded, we know that the $L_i$ are linear independent and $\spam_\R\{L_1,...,L_g\} \cap \{A \in S\C^{\delta \times \delta} \ | \ A \succeq 0\}=\{0\}$. This is a trivial intersection of cones. Therefore there is a non-trivial linear $\varphi: S\C^{\delta \times \delta} \rightarrow \R$ such that $\varphi(\{A \in S\C^{\delta \times \delta} \ | \ A \succeq 0\}) \subseteq [0,\infty)$ and $\varphi(\spam_\R\{L_1,...,L_g\}) \subseteq (-\infty,0]$. We conclude $(\spam_\R\{L_1,...,L_g\})=\ker \varphi$. Since the cone of positive semidefinite matrices is selfdual, we conclude that there is a positive definite matrix $B \in S\C^{\delta \times \delta}$ such that $\varphi(A)=\text{tr}(B^*A)$ for all $A \in \C^{\delta \times \delta}$. \\[0.2cm]
\autoref{arviskz}, \autoref{positest2} and \autoref{abso} imply that $S$ is the matrix convex hull of the absolute extreme points which are elements of $S(1),...,S(\delta)$. We know that $\spam\{L_1,...,L_g\} \otimes \C^{k \times m}=\left\{\sum_{j=1}^k \sum_{h=1}^m e_j e_h^* \otimes D_{j,h} \ | \ D_{j,h} \in \C^{\delta \times \delta}, \varphi(D_{j,h})=0\right\}$ (where $e_j \in \C^k, e_h \in \C^m$). \\[0.2cm]
Let $A \in S(k)$. Then $\ml(A) \neq 0$, whence $\dim \ker \ml(A) \leq k\delta-1$.
Now suppose $A\in \partial \mathcal{D}_{\ml}(k)$ and $k \leq \delta$. 
Consider the map $\psi: [(\ker \ml(A))^\perp]^{1 \times \delta} \rightarrow \C^{k}$
\begin{align*}
\left(\sum_{\al=1}^k v_\al^1 \otimes e_\al,...,\sum_{\al=1}^k v_\al^\delta \otimes e_\al\right) \mapsto \begin{pmatrix} \varphi \left(v_1^1 \ ... \ v_1^\delta\right) \\ \vdots \\ \varphi \left(v_k^1 \ ... \ v_k^\delta\right)\end{pmatrix}.    
\end{align*}
If $\dim \ker \ml(A) < k\delta-1$ or $k<\delta$, the dimension of domain of $\varphi$ is bigger than the dimension of the codomain and therefore $\psi$ admits an element $c\neq 0$ in the kernel. We can write $c=(L \ov X)(b)$ where $b \in (\C^{k \times 1})^g$. Now for $\lambda \in (0,\infty)$
\begin{align*}\ml \left(
\begin{pmatrix}
A & \lambda b \\
\lambda b^* & 0
\end{pmatrix}
\right) \approx \begin{pmatrix}
\ml(A) & \lambda (L \ov X)(b) \\
\lambda (L \ov X)(b)^* & I
\end{pmatrix}=\begin{pmatrix}
\ml(A) & \lambda c \\
\lambda c^* & I
\end{pmatrix}
\end{align*}
The Schur complement tells us that this is matrix is positive semidefinite if and only if $\ml(A)-\lambda^2 c^* c \succeq 0$. This is fulfilled for small $\lambda$ and thus $A$ is not in the Arveson boundary. \\[0.2cm]
If on the other hand $k=\delta$ and $\dim \ker \ml(A) = \delta^2-1$, we write $\ml(A)=\sum_{j,h \in \{1,...,\delta\}} C_{j,h} \otimes F_{j,h}$ where $F_{j,h}=e_je_h^*$ and $C_{j,h} \in \spam{(L_1,...,L_g)}=\ker \varphi$ for $h \neq j$ and $C_{j,j}-I \in \spam{(L_1,...,L_g)}$. Let $\ker \ml(A)=v^\perp$ and write $v=\sum_{\al=1}^\delta v_\al \otimes e_\al$. Since $\ml(A) \succeq 0$, we can scale $v$ such that $\ml(A)=vv^*=\sum_{\al,\beta=1}^\delta v_\al v_\beta^* \otimes F_{\alpha,\beta}$. Since $B$ is positive definite, $\rho: \C^\delta \times \C^\delta\rightarrow \C, (u,v) \mapsto \varphi(vu^*)$ is a scalar product. Scale $B$ in such a way that $\tr(B^*)=1$. Now we see that the vectors $v_\al$ are pairwise orthogonal with respect to $\rho$, thus $\{v_1,...,v_\delta\}$ is an orthonormal basis of $\C^\delta$ with respect to $\rho$. \\[0.2cm]
Consider the map $\psi: [(\ker \ml(A))^\perp]^{1 \times \delta} \rightarrow \C^{\delta},$
\begin{align*}
\left(\lambda_1 \sum_{\al=1}^\delta v_\al \otimes e_\al,...,\lambda_\delta \sum_{\al=1}^\delta v_\al \otimes e_\al\right) \mapsto \begin{pmatrix} \varphi \left(\lambda_1v_1 \ ... \ \lambda_\delta v_1\right) \\ \vdots \\ \varphi \left(\lambda_1 v_\delta \ ... \ \lambda_\delta v_\delta\right)\end{pmatrix}.    
\end{align*}
Suppose $\lambda \in \C^\delta$ such that $\left(\lambda_1 \sum_{\al=1}^\delta v_\al \otimes e_\al,...,\lambda_\delta \sum_{\al=1}^\delta v_\al \otimes e_\al\right) \in \ker \psi$. Then for all $\al \in \{1,...,\delta\}$ we obtain $0=\tr(B^* v_\al \lambda^T)=\tr(\lambda^T B^* v_\al)=\langle \ov{\lambda}, B^* v_\al \rangle$ and hence $\lambda=0$. This means that $\psi$ is injective. \\[0.2cm]
Now let $b \in (\C^{\delta \times 1})^g$ and $d \in \C^g$ such that $\begin{pmatrix} A & b \\
b^* & d\end{pmatrix} \in S$. Due to the Schur complement this implies that $(L \ov X)b \in \ker(\psi)=\{0\}$ and therefore $b=0$; thus $A$ is absolute extreme. \\[0.2cm]
Along the lines of the previous arguments one also shows that a one-dimensional subspace of $\C^{\delta^2}$ coming from an orthonormal basis with respect to $\rho$ induces an absolute extreme point of $S$. 
\end{proof}

\begin{exar} \label{ex10}
We \marginpar{[\begin{small}\autoref{matexdil}\end{small}]}demonstrate how \autoref{matexdil} works in the situation of \autoref{tracesalgebra}. We set $\delta=3$. Define
\begin{align*}
&L=\left( \begin{pmatrix} 2 & 0 & 0 \\ 0 &-1 & 0 \\ 0 & 0 & -1 \end{pmatrix}, \begin{pmatrix} -1 & 0 & 0 \\ 0 & 2 & 0 \\ 0 & 0 & -1 \end{pmatrix}, \begin{pmatrix} 0 & 1 & 0 \\ 1 & 0 & 0 \\ 0 & 0 & 0 \end{pmatrix}, \begin{pmatrix} 0 & \ii & 0 \\ -\ii & 0 & 0 \\ 0 & 0 & 0 \end{pmatrix} \right.\\
&\left. \begin{pmatrix} 0 & 0 & 1 \\ 0 & 0 & 0 \\ 1 & 0 & 0 \end{pmatrix},\begin{pmatrix} 0 & 0 & \ii \\ 0 & 0 & 0 \\ -\ii & 0 & 0 \end{pmatrix},\begin{pmatrix} 0 & 0 & 0 \\ 0 & 0 & 1 \\ 0 & 1 & 0 \end{pmatrix}, \begin{pmatrix} 0 & 0 & 0 \\ 0 & -\ii & 0 \\ 0 & 0 & \ii \end{pmatrix} \right).
\end{align*}This corresponds to $B=I$ and $\varphi=\tr$ in \autoref{tracesalgebra}.
$\ml(e_1)=\begin{pmatrix} 3 & 0 & 0 \\
0 & 0 & 0 \\
0 & 0 & 0 \end{pmatrix}$ has a two-dimensional kernel and $e_1$ is matrix extreme in $\mathcal{D}_\ml$ because $(1,e_1)$ is extreme in $\Hom(\mathcal{D}_\ml)(1)$ (\autoref{exspec}). We want to check if $e_1$ is absolute extreme. If $\begin{pmatrix} e_1 & b^* \\
b & c\end{pmatrix} \in \mathcal{D}_\ml$ with $b \neq 0$, then $\ker \ml(e_1) \subseteq \ker (L \ov X)(b)$. Hence the rows of $(L \ov X)(b)$ have to be elements of $\spam(e_1)$. The possible form of $(L \ov X)(b)$ is $\begin{pmatrix} 0 & 0 & 0 \\ \mu & 0 & 0 \\ \nu & 0 & 0 \end{pmatrix}$ where $\mu,\nu \in \C$ (the trace has to be zero). 
The desired form is \begin{align*}
\ml \left(\begin{pmatrix} e_1 & b^* \\
b & c\end{pmatrix}\right)=\begin{pmatrix}
3 & 0 & 0 & \mu^* & 0 & \nu^* \\
0 & 3-x_5-x_6 & 0 & x_2 & 0 & x_3 \\ 
0 & 0 & 0 & 0 & 0 & 0 \\ 
\mu & x_2^* & 0 & x_5 & 0 & x_4 \\
0 & 0 & 0 & 0 & 0 & 0 \\ 
\nu & x_3^* & 0 & x_4^* & 0 & x_6 \\
\end{pmatrix} \succeq 0
\end{align*}
with some $x_j$. $\Hom(\mathcal{D}_\ml)(2)$ is an ordinary spectrahedron and we know that in order to make $\begin{pmatrix} e_1 & b^* \\
b & c\end{pmatrix}$ matrix extreme in $\mathcal{D}_\ml$, we have to make $\left(I_2,\begin{pmatrix} e_1 & b^* \\
b & c\end{pmatrix}\right)$ extreme in $\Hom(\mathcal{D}_\ml)(2)$. Using the characterization of the extreme points of an ordinary spectrahedron \autoref{exspec}, we see that $\dim\ker \ml \left(\begin{pmatrix} e_1 & b^* \\
b & c\end{pmatrix}\right)=5$ would imply that $\begin{pmatrix} e_1 & b^* \\
b & c\end{pmatrix}$ is matrix extreme. So set $x_2=0,x_3=0,x_5+x_6=3$.  
Hence for $\mu,\nu \in \C$ with $\frac{\mu\mu^*}{3} + \frac{\nu\nu^*}{3}=3$ the choice
\begin{align*}
\ml\left(\begin{pmatrix} e_1 & b^* \\
b & c\end{pmatrix}\right)=\begin{pmatrix}
3 & 0 & 0 & \mu^* & 0 & \nu^* \\
0 & 0 & 0 & 0 & 0 & 0 \\ 
0 & 0 & 0 & 0 & 0 & 0 \\ 
\mu & 0 & 0 & \frac{\mu\mu^*}{3} & 0 & \frac{\mu\nu^*}{3} \\
0 & 0 & 0 & 0 & 0 & 0 \\ 
\nu & 0 & 0 & \frac{\mu^*\nu}{3} & 0 & \frac{\nu\nu^*}{3} \\
\end{pmatrix}
\end{align*}
constitutes an matrix extreme point of $\mathcal{D}_\ml$.
We choose as a solution $\mu=3$ and $\nu=0$ and obtain
\begin{align*}
B:=\begin{pmatrix} e_1 & b^* \\
b & c\end{pmatrix} =\left( \begin{pmatrix} 1 & 0 \\
0 & 0\end{pmatrix}, \begin{pmatrix} 0 & 0 \\
0 & 1\end{pmatrix}, \begin{pmatrix} 0 & \frac{3}{2} \\
\frac{3}{2} & 0\end{pmatrix}, \begin{pmatrix} 0 & -\frac{3}{2}\ii \\
\frac{3}{2}\ii & 0\end{pmatrix},0_2,0_2,0_2,0_2 \right)
\end{align*}
as a new matrix extreme point of $\mathcal{D}_{\ml}$. We want to check if $B$ is extreme and argue similar as above. If $\begin{pmatrix} B & d^* \\
d & e\end{pmatrix} \in \mathcal{D}_\ml$ with $d \neq 0$, then $\ker \ml(B) \subseteq \ker (L \ov X)(d)$. Hence the rows of $(L \ov X)(d)$ have to be elements of $\spam(e_1+e_4)$. The possible form of $(L \ov X)(d)$ is $\begin{pmatrix} 0 & 0 & 0 & 0 & 0 & 0 \\ 0 & 0 & 0 & 0 & 0 & 0 \\ \kappa & 0 & 0 & \kappa & 0 & 0 \end{pmatrix}$ where $\kappa \in \C$ (the traces of two submatrices consisting of column $1,3,5$ resp. $2,4,6$ have to be zero). We maximize $\text{Re}(\kappa)$ in order to make $d$ an "extreme" choice according to the proof of \autoref{matexdil}. Clearly this is the case for $\kappa=3$. The kernel of $\ml\left(\begin{pmatrix} B & d^* \\
d & e\end{pmatrix}\right)$ is then maximal for $\ml(e)=3e_3e_3^T$. We obtain
\begin{align*}
&C:=\begin{pmatrix} B & d^* \\
d & e\end{pmatrix} =\left( \begin{pmatrix} 1 & 0 & 0 \\
0 & 0 & 0 \\ 0 & 0 & -1 \end{pmatrix}, \begin{pmatrix} 0 & 0 & 0\\
0 & 1 & 0 \\ 0 & 0 & -1 \end{pmatrix}, \begin{pmatrix} 0 & \frac{3}{2} & 0 \\
\frac{3}{2} & 0 & 0 \\ 0 & 0 & 0 \end{pmatrix}, \begin{pmatrix} 0 & -\frac{3}{2}\ii & 0 \\
\frac{3}{2}\ii & 0 & 0 \\ 0 & 0 & 0\end{pmatrix}, \right.\\
&\left. \begin{pmatrix} 0 & 0 & \frac{3}{2} \\
0 & 0 & 0 \\ \frac{3}{2} & 0 & 0 \end{pmatrix},\begin{pmatrix} 0 & 0 & -\frac{3}{2}\ii \\
0 & 0 & 0 \\ \frac{3}{2}\ii & 0 & 0 \end{pmatrix},0_3,0_3 \right)
\end{align*}
Now $C$ is absolute extreme. If not, the same methods of above would imply that there is $f$ such that $(L\ov X)(f) \neq 0$ is a $3 \times 9$ matrix with rows in $\spam(e_1+e_5+e_9)$. However the traces of the submatrices built of row $1,4,7$ resp. $2,5,8$ resp. $3,6,9$ must have trace $0$, which is not possible.
\end{exar}

\begin{exar} \label{traces2}
\autoref{tracesalgebra} \marginpar{[\begin{footnotesize}\autoref{positest2}\end{footnotesize}]} in connection with \autoref{positest2} and \autoref{arviskz} implies that for all $\delta \in \N$ there is some $\ep \in \N$ and a $(\delta-1)$-positive map $\varphi_\delta: \C^{\delta \times \delta} \rightarrow \C^{\ep \times \ep}$ which is not completely positive and hence not $\delta$-positive. 
\end{exar}

\begin{exar} \label{ex14} (Computing matrix extreme points of compact free spectrahedra) \marginpar{[\autoref{krain}]}
Let $S=\mathcal{D}_\ml$ be a compact free spectrahedron defined by a monic linear pencil $\ml$. Let $\mathcal{L}=\ml-I$. Then we have $\Hom(S)=\{(A_0,A) \in \S^{g+1} \ | \ A_0 \succeq 0, I \otimes A_0 + ...+L_g A_g \succeq 0\}=:T$. \\
Justification: The inclusion "$\subseteq$" is obvious. $T$ is a matrix cone. We show that it is directed. Let $(A_0,A)=\begin{pmatrix} (0,B) & (0,C) \\ (0,C^*) & (D_0,D)\end{pmatrix}$ with $D_0 \succ 0$ be an element of the right set. By applying a unitary conjugation we conclude that $\begin{pmatrix} \mathcal{L}(B) & \mathcal{L}(C) \\ \mathcal{L}(C^*) & I \otimes D_0 + \mathcal{L}(D)  \end{pmatrix} \succeq 0$. Because $\mathcal{D}_L$ is compact, we know that $\C^{k \times k} \otimes \spam(L_1,...,L_g)$ contains no non-trivial positive semidefinite matrix for all $\delta \in \N$ and the $L_i$ are linearly independent. From $\mathcal{L}(B) \succeq$ we deduce $B=0$. Consequently $\mathcal{L}(C)=0$ and hence $C=0$. This shows that $T$ is directed.\\
Now let $(A_0,A) \in T$. WLOG let $A_0 \succ 0$. Then we have $(I, \sqrt{A_0^{-1}}A\sqrt{A_0^{-1}}) \in T$ and so $\sqrt{A_0^{-1}}A\sqrt{A_0^{-1}} \in S$. This means $(I, \sqrt{A_0^{-1}}A\sqrt{A_0^{-1}}) \in \Hom(S)$ and therefore $(A_0,A) \in \Hom(S)$. \\
Now that we have established $\Hom(S)=T$, suppose that the extreme rays of the ordinary spectrahedron $T(k)$ are given. If $(A_0,A)$ is such an extreme ray we can find unitary $U$ such that $U^* (A_0,A) U=0_s \oplus (B_0,B)$ with $B_0 \succ 0$. Now $\sqrt{B_0^{-1}}B\sqrt{B_0^{-1}}$ is a matrix extreme point of $S$ and all matrix extreme points of size at most $k$ are obtained in this fashion. 
\end{exar}

\section{Projection number and sequences of determinants of monic linear pencils/compatible RZ-polynomials}

\noindent In chapter $5$ we have seen how the sequence $(f_k)_{k \in \N}=(\det_k \ml)_{k \in \N}$ of determinants of a monic linear pencil $\ml$ determines how the decomposition of $\ml$ in simultaneously $\mathcal{D}$-irreducible and $\mathcal{D}$-minimal pencils looks like. In this chapter we want to analyze how $\pz(\mathcal{D}_\ml)$ is encoded in this sequence of determinants. More generally we extend our result also to sequences of RZ-polynomials which satisfy some natural compatibility assumptions and define matrix convex sets as well. \\[0.2cm] If $H \subseteq \C^k$ is a subspace and $P: \C^k \rightarrow H$ is the projection onto $H$, we will set \index{P@$\underline{P}$}$\underline{P}: \C^k \rightarrow \C^k, x \rightarrow Px$. We say that a projection $P$ is minimal with respect to a property (R) if $\im(\uv{P}) \nsupseteq \im(\uv{Q})$ for every other projection $Q$ satisfying (R).

\begin{rem} \label{adjunct}
Let $h: \R \rightarrow \R^{k \times k}$ be a differentiable function. Then for $t \in \R$ we have $(\det h)'(t)=\text{tr} (\text{adj}(h(t)) h'(t))$ where $\text{adj}(h(t))$ denotes the adjugate matrix of $h(t)$.
\end{rem}

\begin{lemma} \label{unipro}
Let $\ml$ be a monic linear pencil of size $\delta$ and $f_k=\det_k \ml(\mathcal{X})$. Suppose $A \in \partial \mathcal{D}_\ml$ such that $\nabla f_k(A) \neq 0$. Then there exists a unique minimal projection $P$ such that $PAP^* \in \partial \mathcal{D}_\ml$.
\end{lemma}

\begin{proof}
Let $P$ be a projection such that $PAP^* \in \partial\mathcal{D}_\ml$. Then we can find a kernel vector $v=\sum_{\al=1}^\delta e_\al \otimes v_\al$ of $\ml(PAP^*)$. Set $w:=\sum_{\al=1}^\delta e_\al \otimes w_\al:=(I_\delta \otimes P^*)v=\sum_{\al=1}^\delta e_\al \otimes P^*v_\al$. Then $w \in \ker \ml(A)$. Of course this kernel vector is unique (\autoref{adjunct}). So $\spam(w_1,...,w_\delta) \subseteq \im(P)$ and the first set is the range of the minimal projection. 
\end{proof}

\begin{lemma} \label{abl}
Let $\ml$ be a monic linear pencil of size $\delta$, $f_k=\det_k \ml(\mathcal{X})$ and $A \in \partial \mathcal{D}_\ml(k)$ such that $\nabla f_k(A) \neq 0$. Then the unique minimal projection $\uv{Q}$ with $\uv{Q}\nabla f_k(A)\uv{Q}=\nabla f_k(A)$ (i.e. the projection onto $\im \nabla f_k(A)$) equals the unique projection $\uv{P}$ given by \autoref{unipro}.
\end{lemma}

\begin{proof}
Because $\nabla f_k(A) \neq 0$ we have $\dim \ker \ml(A)=1$. We know that there is a neighborhood $U$ of $A$ such that $M:=\{B \in \S^g(k) \cap U \ | \ f_k(B)=0\}=\{B \in \S^g(k) \cap U \ | \ \dim\ker \ml(B)=1\}=\partial \mathcal{D}_\ml(k) \cap U$ is a connected $C^1$-manifold of dimension $g k^2-1$.
Let $v \in \ker \ml(A) \setminus \{0\}$. We want to calculate the normal vector $N_A$ of this manifold $M$ in $A$. \\[0.2cm]
We have $0=\ml(A)v$, thus $(L \ov X)(A)v=v$ and $1$ is the biggest eigenvalue of $(L \ov X)(A)$. Consider the function $h:\S^n(k) \times \mathcal{S}^{\delta k}, (B,w) \mapsto \ml(B)w$. Therefore the derivative of $h$ in direction $\mathcal{S}^{\delta k}$ in $(A,v)$ is given by $\ml(A)|_{v ^\perp}$ and the image of this part of the derivative is $v^\perp$. Now if $H \in \S^n(k)$, then the derivative in $(A,v)$ in direction $H$ is given by $(L \ov X)(H)v$ and the component in direction of $v$ by $\langle v, (L \ov X)(H)v \rangle$. In particular we have $\partial_A h(A,v)=v$. This means that $h'(A,v)$ has full rank and there is a neighborhood $V$ of $(A,v)$ such that $T=\{(B,w) \in \S^g(k) \times \mathcal{S}^{\delta k} \ | \ h(B,w)=0\}$ is a $g k^2 -1$-dimensional connected $C^1$-manifold. \\[0.2cm]
Now we know from \cite[Lemma 2.1]{HV} that $U \cap \mathcal{D}_\ml(k)$ has full dimension. Since $\mathcal{D}_\ml(k)$ is convex, the hyperplane $\{A\}+N_A^\perp$ separates $\mathcal{D}_\ml(k)$. Thus up to multiplication by scalar multiples the normal vector $N_A$ is only vector $N$ that can satisfy: If we take some $D$ with $\langle N,D\rangle < 0$, then around $(A,v)$ we can find a unique $C^1$-resolution $\varphi: D^\perp \mapsto \spam{D} \times (\mathcal{S}^{\delta k} \cap B(v,1))$ such that $(A+E+\varphi_1(E),\varphi_2(E)) \in T$ for small $E$ (hence $(a+E+\varphi_1(E)) \in M$) and $\varphi(0)=(0,v)$ (uniqueness of $\varphi_1$ comes from the fact that $\mathcal{D}_\ml(k)$ is convex and $M=\partial \mathcal{D}_\ml(k) \cap U$; uniqueness of $\varphi_2$ is due to the fact that the derivative of $f_k$ does not vanish around $A$).
Now we want to apply the Implicit function theorem to determine $N$ with that property. \\[0.2cm]
We saw already that the image of the derivative of $h$ in $(A,v)$ in direction $\mathcal{S}^{\delta k}$ is $v^\perp$.
For $B \in \S^n(k)$, the component of $v$ of the derivative of $(h,A)$ in direction $B$ is $\langle v, (L \ov X)(H)v \rangle$. Together with the Implicit function theorem this means that the normal vector $N_A$ points into the direction $H \in \S^g(k)$ such that $\langle v, (L \ov X)(G)v \rangle=0$ for all $G \in \S^g(k)$ perpendicular to $H$. Thus the normal direction $N_A$ is given by $\nabla r(A)$ where $r: \S^n(k) \rightarrow \R, B \mapsto v^*\ml(B)v$. \\[0.2cm]
Since the equality $f_k=0$ defines the manifold $M$, we see that $\spam(N_A)=\spam(\nabla f_k(A))=\spam \nabla r(A)$. Now if we write $v=\sum_{\al=1}^\delta e_\al \otimes v_\al$, the proof of \autoref{unipro} tells us that $P$ is the projection onto $\{v_1,...,v_\delta\}$. Since for every $B \in \S^n(k)$ we have $v^*\ml(\uv{P}B\uv{P})v=v^*\ml(B)v$, we get $\uv{P}\nabla f_k(A)\uv{P} = \nabla f_k(A)$. \\[0.2cm]
On the other hand we have $\uv{Q} \nabla r(A) \uv{Q}=\nabla r(A)$. Hence for every $B \in \S^g(k)$ we have 
\begin{align*}
v^* (L \ov X)(B) v=\tr(\nabla r(A) B)=\tr(\uv{Q} \nabla r(A) \uv{Q} B)=\tr(\nabla r(A) \uv{Q} B \uv{Q})=v^* (L \ov X)(\uv{Q}B\uv{Q}) v.
\end{align*} Thus we conclude $0=v^*\ml(A)v= v^*\ml(\uv{Q}A\uv{Q})v$. Since $\mathcal{D}_\ml$ is matrix convex, $\uv{Q}A\uv{Q} \in \mathcal{D}_\ml$. Hence $\uv{Q}A\uv{Q} \in \partial \mathcal{D}_\ml$ and $\im(\uv{P}) \subseteq \im(\uv{Q})$. Due to $\uv{P}\nabla f_k(A)\uv{P} = \nabla f_k(A)$, even $\im(\uv{P}) = \im(\uv{Q})$.  
\end{proof}

\begin{definition}
We call a system of functions $(f_k: \S^g(k) \rightarrow \R)_{k \in \N}$ a \textbf{compatible sequence of RZ-polynomials} if the following holds:
\begin{enumerate}
\item For all $k \in \N$ the function $f_k$ is an RZ-polynomial on $\S^g(k)$ with $f_k(0)=1$.
\item For $A_i \in \S^g(k_j)$ we have $f_k(\bigoplus_j A_j)=\prod_{j} f_{k_j}(A_j)$, where $k=\sum_j k_j$.
\item Every $f_k$ is invariant under unitary similtaries. 
\item There is $k_0 \in \N$ such that $f_k$ is irreducible for $k \geq k_0$.
\end{enumerate}
\end{definition}

\begin{proposition}
Let $(f_k: \S^g(k) \rightarrow \R)_{k \in \N}$ be a compatible sequence of RZ-polynomials. Then the closures of the connected components of $f_k^{-1}(\R \setminus \{0\})$ around $0$ form a closed matrix convex set $S$. 
\end{proposition}

\begin{proof}
The claim follows from \autoref{konv}.
\end{proof}

\begin{cor}
Let $(f_k)_{k \in \N}$ be a compatible sequence of RZ-polynomials, $S$ the generated closed matrix convex set and $m \in \N$. Suppose that for each $A \in \partial S(k)$ there is a projection $P$ of rank at most $m$ such that $\nabla f_k(A)=\uv{P} \nabla f_k(A) \uv{P}$. Then $\text{pz}(S) \leq m$.  
\end{cor}

\begin{proof}
Let $A \in \partial S(k)$. Suppose first that $\nabla f_k(A) \neq 0$. Let $P$ be a projection of rank at most $m$ such that $\nabla f_k(A)=\uv{P} \nabla f_k(A) \uv{P}$. Consider the function $\varphi: \C^2 \rightarrow \C, (s,t) \mapsto f_k(s\uv{P}A\uv{P}+t(A-\uv{P}A\uv{P}))$. Then $\varphi$ is an RZ-polynomial. \cite[Theorem 3.1]{HV} says that there exists $r \in \N$, $B,C \in S\C^{r \times r}$ such that $\varphi(s,t)=\det (I + sB+tC)$. We know that $\varphi'(1,1)(1,1)=\langle\nabla f_k(A),A\rangle \neq 0$ from \cite{HV} (This is true because the connected component of $f_k^{-1}(\R \setminus \{0\})$ around $0$ is convex and the hypersurface $A+(\nabla f_k(A))^\perp$ isolates it. If $\langle\nabla f_k(A),A\rangle=0$, this would contradict $0 \in \inte{f_k^{-1}(\R \setminus \{0\})}$.).  \\[0.2cm]
Now let $\mathcal{L}$ be the pencil $I + BX_1 + CX_2$. $\mathcal{L}(1,1)$ has a unique kernel vector $v$. Apart from that $\varphi'(1,1)(0,1)=0$. Thus the proof of \autoref{abl} shows $v^*Cv=0$. Hence $f_k(\uv{P} A \uv{P})=\varphi(1,0)=0$. \\[0.2cm]
Now for a general $A \in \partial S(k)$ choose $s \geq k$ such that $f_s$ is irreducible. Now approximate $A \oplus 0$ by regular points $B$ of the real variety defined by $f_s$ (\autoref{densi}). We find a projection $P_B$ of rank at most $m$ such that $f_{\rk(P_B)}(P_B^*BP_B)=0$. By taking the limit we find a projection $P$ of rank at most $m$ such that $f_{\rk(P)}(P^*(A \oplus 0)P)=0$. So there exists $C \in \partial S(m) \cap \mconv(A,0)(m)$. However $C \in \mconv(A,0)(m)=\mconv(\{PAP^* \ | \  P \text{ projection of rank at most } m \} \cup \{0\})$ (\autoref{orka}). Moreover $0 \in \inte(S)$, so not all $P^*AP$ can be in $\inte(S)$.
\end{proof}

We do not know whether the converse of the previous corollary is also true.

\section{Sequence of the degrees of the determinants of a monic linear pencil}

\begin{rem} \label{rat} For the purposes of this chapter we need the concept of (noncommutative) rational functions. We will only briefly discuss the properties we need and refer the reader for instance to \cite{R} or \cite{KVV} for a more detailed exposition. \\[0.2cm]
A \textbf{rational expression} is a senseful combination consisting of noncommutative polynomials $\C\langle \ov X \rangle$, functions $+,\cdot,^{-1}$ and brackets. We distinguish rational expressions from the function they represent, so $X_1+X_2$ and $X_2+X_1$ are different rational expressions. If $f$ is a rational expression, then we can evaluate it in a tuple of equally sized square matrices. The domain $\text{dom}(f)$ of $f$ consists of all tuples where all matrix inverses exist when following the arithmetic operations in $f$. On each level $k$ the domain $\dom_k(f)$ of $f$ is a Zariski-open set. \\[0.2cm]  
We call two rational expressions $t,s$ with non-empty domains equivalent if the intersection of their domains $\dom(t) \cap \dom(s)$ is non-empty and for all $A \in \dom(t) \cap \dom(s)$ we have $t(A)=s(A)$. An equivalence class of a rational expression with non-empty domain will be a \textbf{rational function}. Its domain will be the union of all the domains of the rational expressions it represents. \\[0.2cm]
In contrast to noncommutative polynomials, rational expressions can have a strange behaviour. For example there exist non-trivial rational expressions which represent the zero function. Also a noncommutative expression can vanish on level $k$ and be constantly the identity matrix on level $k-1$.  
\end{rem}

\begin{lemma} \label{kal} \cite[Proposition 2.1]{KVV}
If $f $ is a rational expression which does not represent the zero function, then $\det f$ is not the zero function on $\dom(f)$. In this case $\dom(f^{-1})(n)$ is dense in $(\C^{n \times n})^g$ for sufficiently big $n$.
\end{lemma}

\begin{lemma} \label{juri} (J. Vol$\check{\text{c}}$i$\check{\text{c}}$, private communication)
Let $f$ be a rational expression, $k \in \N$ and $\dom_k(f) \neq \emptyset$. Then $\{A \in (S\C^{k \times k})^g \ | \ A \in \dom_k(f)\}$ is a dense subset of $(S\C^{k \times k})^g$ in the Euclidean topology and Zariski open.  
\end{lemma}

\begin{proof}
Let $\mathcal{X}$ be the tuple of generic matrices of size $k$. Since $\dom_k(f)$ is Zariski-open, it is enough to show the following. Let $g \in \C[\mathcal{X}]$ be a polynomial vanishing on all tuples of Hermitian matrices, then $g$ vanishes also on all tuples of complex matrices. However such a $g$ vanishes for all substitutions of the $\mathcal{X}_{\al,\beta}^i$ by real numbers. Therefore $g$ vanishes also for all substitutions of the $\mathcal{X}_{\al,\beta}^i$ by complex numbers and hence $g=0$.
\end{proof}

\begin{proposition} \label{row} \cite[Corollary 2.3]{NT}
Let $L \in \S^g(\delta)$. Then the degree of $p=\det_k \ml(\mathcal{X})$ equals $\max \{ \rk (L \ov X)(B) \ | \ B \in (S\C^{k \times k})^g \}$.
\end{proposition}

\begin{proof}
We consider the polynomial $q=\det(I_{\delta k} \otimes Z- L_1 \otimes \mathcal{X}^1 - ... - L_g \otimes \mathcal{X}^g) \in \C[Z,\mathcal{X}_{\al,\be}^i \ | \ 1 \leq i \leq g, 1 \leq \al,\be \leq k]$. Then we write $q=\sum_{j=0}^{\delta k} Z^j h_j$ where $Z$ does not appear in $h_j$. We know that $\deg(p)=\alpha$ if and only if $Z$ has multiplicity exactly $\delta k-\alpha$ as a factor of $q$. \\[0.2cm]
Now suppose $Z$ has multiplicity $\delta k-\alpha$ as a factor in in $q$. Then $h_{\delta k-\al} \neq 0$. Assume that $h_{\delta k -\al}$ equals $0$ for all substitutions of the $\mathcal{X}_{\al,\beta}^i$ by real numbers. Then $h_{\delta k -\al}$ equals $0$ for all substitutions of the $\mathcal{X}_{\al,\beta}^i$ by complex numbers and thus $h_{\delta k - \al}$=0. Therefore we know that there exists $B \in (S\C^{k \times k})^g$ such that $(L \ov X)(B)$ has kernel of dimension exacly $\delta k-\alpha$. So the rank is $\alpha$. \\[0.2cm]
On the other hand $h_{0}=...=h_{\delta k - \al -1}=0$ and for all $B \in (S\C^{k \times k})^g$ the kernel of $(L \ov X)(B)$ has dimension at least $\delta k-\alpha$, so the rank is at most $\alpha$.
\end{proof}

\begin{rem}($WDW^*$-decomposition) \cite[Section 2.6.2]{HKM2} \label{wdw}\\
Let $A \in S(\C \langle \ov X \rangle)^{\delta \times \delta}$ be a Hermitian matrix polynomial. Then we can find a block-diagonal matrix $D$ which consists of blocks of the form 
\begin{align*}
(p) \ \ \ \text{or} \ \ \ \begin{pmatrix} 0 & p^* \\ p & 0 \end{pmatrix}, 
\end{align*}
where $p$ is a rational function, a lower triangular matrix $W$ with ones on the diagonal and a permutation matrix $Q$ such that $QAQ^T=WDW^*$. The equality is understood entry-wise as an equality of rational functions. The algorithm to calculate such a decomposition is easy to understand:
Let
\begin{align*}
A=\begin{pmatrix} a & b^* & E^* \\ b & c & F^* \\ E & F & G \end{pmatrix}
\end{align*}
with matrices $E,F,G$ of rational functions. In the case $a \neq 0$ or $c \neq 0$ we permute $A$ if necessary to assume $a \neq 0$. Then we have
\begin{align*}
A=\begin{pmatrix} 1 & 0 & 0 \\ ba^{-1} & 1 & 0 \\ Ea^{-1} & 0 & 1 \end{pmatrix} 
\begin{pmatrix} a & 0 & 0 \\ 0 & c-ba^{-1}b^* & F^*-ba^{-1}E^* \\ 0 & F-Ea^{-1}b^* & G-E (a^{-1}E^*) \end{pmatrix}    
\begin{pmatrix} 1 & a^{-1}b^* & a^{-1}E^* \\ 0 & 1 & 0 \\ 0 & 0 & 1 \end{pmatrix}
\end{align*}
Now one can continue with the submatrix appearing when we delete the first row and the first column.
In the case where both $a,c$ are zero we have
\begin{align*}
A=\begin{pmatrix} 0 & b^* & E^* \\ b & 0 & F^* \\ E & F & G \end{pmatrix}
\end{align*}
If $b=0$ and $E^*=0$ we can continue by deleting the first row and column. If $b=0$ and $E^* \neq 0$, we can permute $A$ again to assume $b \neq 0$. So let $b \neq 0$.
Then 
\begin{align*}
A=\begin{pmatrix} 1 & 0 & 0 \\ 0 & 1 & 0 \\ F(b^*)^{-1} & Eb^{-1} & 1 \end{pmatrix} 
\begin{pmatrix} 0 & b^* & 0 \\ b & 0 & 0 \\ 0 & 0 & G-F (b^*)^{-1} E^* -E b^{-1}F^* \end{pmatrix}.    
\begin{pmatrix} 1 & 0 & b^{-1}F^* \\ 0 & 1 & (b^*)^{-1}E^* \\ 0 & 0 & 1 \end{pmatrix}
\end{align*}
\end{rem}

\begin{cor}\label{algo}
Let $\ml$ be a monic linear pencil. The $WDW^*$-decomposition gives an algorithm to determine $\deg (\det_k \ml(\mathcal{X}))$ for almost all $k \in \N$ simultaneously.
\end{cor}

\begin{proof}
\autoref{row} tells us that we have to calculate the maximal rank of $(L \ov X) (B)$ where $B \in (S\C^{k \times k})^g$. Notice that for fixed $k$ the set of $B \in (S\C^{k \times k})^g$ for which $(L \ov X)(B)$ has maximal rank is open in $(S\C^{k \times k})^g$ in the Euclidean topology. Let $WDW^*$ the decomposition of $(L \ov X)$ from \autoref{wdw}. Let $\{d_1,...,d_h\}$ be the entries of $D$ which are not representing the zero function and $\{w_1,...,w_r\}$ be the entries of $W$. \autoref{kal} tells us that the intersection $\mathfrak{D}_k$ of the domains all $(d_\beta)^{-1}$ and $w_\gamma$ and $(\C^{k \times k})^g$ is nonempty if $k$ is big enough. By \autoref{juri} we know that for large $k$ the set $\mathfrak{D}_k \cap (S\C^{k \times k})^g$ is non-empty as well as Zariski open in $(S\C^{k \times k})^g$ and therefore contains some $B \in (S\C^{k \times k})^g$ for which $(L \ov X)(B)$ has maximal rank. Hence for those $k$ we can identify the maximal rank of $(L \ov X)(B)$ where $B \in (S\C^{k \times k})^g$ with the maximal rank of $D(B)$ where $B \in (S\C^{k \times k})^g$.
\end{proof}

\begin{cor} \label{degdet}
Let $\ml$ be a monic linear pencil. Then there is some $b \geq \deg(\det_1 \ml(\mathcal{X}))$ and $N \in \N$ such that for all $n \in \N$ we have $\deg(\det_1 \ml(\mathcal{X}))n \leq \deg(\det_n \ml(\mathcal{X})) \leq bn$ and for $n \geq N$ already $\deg(\det_n \ml(\mathcal{X}))=bn$. 
\end{cor}

\begin{proof}
The second statement for large numbers follows directly from the end of the proof of \autoref{algo}. It remains to show that $\deg(\det_1 \ml(\mathcal{X}))n \leq \deg(\det_n \ml(\mathcal{X})) \leq bn$. Let $\mathcal{Y}_1,...,\mathcal{Y}_n$ be $g$-tuples of generic $1 \times 1$-matrices in pairwise different variables. For $n \in \N$ we have that $\prod_{j=1}^n \det_1 \ml(\mathcal{Y}_j)$ equals $\det_n \ml\left(\bigoplus_{j=1}^n \mathcal{Y}_j\right)$, thus $n\deg(\det_1 \ml(\mathcal{X})) \leq \deg(\det_n \ml(\mathcal{X}))$. Now choose $M \in \N$ with $Mn \geq N$. With the same construction as before we see that the $M$-fold direct potence of $\det_n \ml$ equals $\det_{Mn} \ml$ restricted to $n \times n$-block matrices. Thus $\deg(\det_n \ml(\mathcal{X}))M \leq \deg(det_{Mn} \ml(\mathcal{X}))=bnM$. 
\end{proof}

\begin{cor} \label{gegen}
There is a monic linear pencil $\ml$ and a $k>1$ such that $\deg(\det_k \ml(\mathcal{X})) > k \deg(\det_1 \ml(\mathcal{X}))$.
\end{cor}

\begin{proof}
Set $g=4$ and write $\ov{X}=\{a,x,y,z\}$. We then have for
\begin{small}
\begin{align*}&L \ov X:=
\begin{pmatrix}
a & z & 0 & 0 & 0 & 0 & 0 \\ 
z & 0 & x & 0 & 0 & 0 & 0 \\
0 & x & 0 & y & 0 & 0 & 0 \\
0 & 0 & y & 0 & z & 0 & 0 \\
0 & 0 & 0 & z & 0 & x & 0 \\
0 & 0 & 0 & 0 & x & 0 & y \\
0 & 0 & 0 & 0 & 0 & y & -a
\end{pmatrix}
\end{align*}
\end{small}
\begin{tiny}
\begin{align*}L \ov X = &\begin{pmatrix} 
1 & 0 & 0 & 0 & 0 & 0 & 0 \\ 
v_1 & 1 & 0 & 0 & 0 & 0 & 0 \\
0 & v_2 & 1 & 0 & 0 & 0 & 0 \\
0 & 0 & v_3 & 1 & 0 & 0 & 0 \\
0 & 0 & 0 & v_4 & 1 & 0 & 0 \\
0 & 0 & 0 & 0 & v_5 & 1 & 0 \\
0 & 0 & 0 & 0 & 0 & v_6 & 1
\end{pmatrix}
\begin{pmatrix}
d_1 & 0 & 0 & 0 & 0 & 0 & 0 \\ 
0 & d_2 & 0 & 0 & 0 & 0 & 0 \\
0 & 0 & d_3 & 0 & 0 & 0 & 0 \\
0 & 0 & 0 & d_4 & 0 & 0 & 0 \\
0 & 0 & 0 & 0 & d_5 & 0 & 0 \\
0 & 0 & 0 & 0 & 0 & d_6 & 0 \\
0 & 0 & 0 & 0 & 0 & 0 & d_7
\end{pmatrix}
\begin{pmatrix}
1 & v_1 & 0 & 0 & 0 & 0 & 0 \\ 
0 & 1 & v_2 & 0 & 0 & 0 & 0 \\
0 & 0 & 1 & v_3 & 0 & 0 & 0 \\
0 & 0 & 0 & 1 & v_4 & 0 & 0 \\
0 & 0 & 0 & 0 & 1 & v_5 & 0 \\
0 & 0 & 0 & 0 & 0 & 1 & v_6 \\
0 & 0 & 0 & 0 & 0 & 0 & 1
\end{pmatrix}
\end{align*}
\end{tiny}
with 
\begin{align*}
v=(&za^{-1}, 
-xz^{-1}az^{-1},
yx^{-1}za^{-1}zx^{-1},
-zy^{-1}xz^{-1}az^{-1}xy^{-1}, \\
&xz^{-1}yx^{-1}za^{-1}zx^{-1}yz^{-1},
-yx^{-1}zy^{-1}xz^{-1}az^{-1}xy^{-1}zx^{-1}) \\
d=(&-za^{-1}z,xz^{-1}az^{-1}x,-yx^{-1}za^{-1}zx^{-1}y,zy^{-1}xz^{-1}az^{-1}xy^{-1}z,\\&-xz^{-1}yx^{-1}za^{-1}zx^{-1}yz^{-1}x,-a+yx^{-1}zy^{-1}xz^{-1}az^{-1}xy^{-1}zx^{-1}y)
\end{align*}
One easily sees that $d_7$ is zero on $\dom(d_7)(1)$ and non-trivial on $\dom(d_7)(2)$. Thus we conclude that for the pencil $\ml$ we have $\deg(\det_1 \ml)=6$ and $\deg(\det_2 \ml)=14$.
\end{proof}

\newpage
\setcounter{section}{9}
\renewcommand{\thesection}{\Roman{section}}
\section{Appendix}

\subsection{\texorpdfstring{$C^*$-algebras and representations}{C*-algebra and representations}}

\begin{definition}
A (complex) $C^*$-algebra $\mathcal{A}=(A,+,\cdot,1,||.||,{}^*)$ is a $\C$-algebra (with 1), a submultiplicative norm $||.||$ making $(A,+,||.||)$ a Banach space and an involution $^*$ extending the complex conjugation such that for all $x,y \in \mathcal{A}$ and $\lambda \in \C$:
\begin{align*}
(x^*)^*=x, \ (x+y)^*=x^*+y^*, \ (yx)^*=x^* y^*, \ ||x^*||=||x|| \ \text{and} \ ||x^*x||=||x||^2. 
\end{align*}
\end{definition}

For a Hilbert space $\mathcal{H}$ every closed subalgebra of $\mathcal{B}(\mathcal{H})$ with the adjoint operation as involution and the operator norm as a norm is the standard example of a $C^*$-algebra. A representation of a $C^*$-algebra $\mathcal{A}$ is a (unital) $^*$-homomorphism $\varphi: \mathcal{A} \rightarrow \mathcal{B}(\mathcal{H})$ where $\mathcal{H}$ is a Hilbert space. $\varphi$ is automatically continuous. If $\varphi$ is injective, then $\varphi$ is an isometry. $\varphi$ is called irreducible if it cannot be written as a non-trivial direct sum of other representations or equivalently if $\varphi(\mathcal{A})$ does not admit a non-trivial reducing subspace. For every $C^*$-algebra $\mathcal{A}$ there exists an injective representation map $\varphi$. Therefore $\mathcal{A}$ is isometric to a closed subalgebra of $\mathcal{B}(\mathcal{H})$. Even though one can write every $C^*$-algebra as a concrete subalgebra of the bounded linear operators on a Hilbert space, sometimes it can be advantageous to work with the abstract definition. \\[0.2cm]
We call $x \in \mathcal{A}$ selfadjoint if $x=x^*$ and positive if there exists $y \in \mathcal{A}$ such that $y=xx^*$.

\begin{theo} \label{burn} (Burnsides theorem)
Let $L_1,...,L_g \in \C^{\delta \times \delta}$ and suppose there is some $i$ with $L_i \neq 0$. Then the $C^*$-subalgebra $\mathcal{A}$ of $\C^{\delta \times \delta}$ defined by the $L_1,...,L_g$ equals $\C^{\delta \times \delta}$ if and only if the $L_i$ have no non-trivial common reducing subspace if and only if for all $v,w \in \C^{\delta} \setminus \{0\}$ there exists $A \in \mathcal{A}$ such that $Av=w$.
\end{theo}

\begin{theo}
Let $\mathcal{A}$ be a finite-dimensional $C^*$-algebra (as a $\C$-vector space). Then there exist $\delta_1,...,\delta_m \in \N$ uniquely determined up to permutation such that $\mathcal{A} \cong \bigoplus_{j=1}^m \C^{\delta_j \times \delta_j}$.
\end{theo}

\begin{proposition} \label{Sarason}
Let $\mathcal{A}$ be a $C^*$-algebra and $\phi: \mathcal{A} \rightarrow \mathcal{B}(\mathcal{H}_{1} \oplus \mathcal{H}_{2}), A \mapsto \begin{pmatrix} \psi(A) & \rho_1(A) \\ \rho_2(A) & \eta(A) \end{pmatrix}$ a representation. If $\psi$ is a representation, then $\rho=0$.
\end{proposition}

\begin{proof} For $A,B \in \mathcal{A}$ we have
\begin{align*} &\begin{pmatrix} \psi(A)\psi(B) & \rho_1(AB) \\ \rho_2(AB) & \eta(AB) \end{pmatrix}=\phi(AB)=\phi(A)\phi(B)\\&=\begin{pmatrix} \psi(A)\psi(B) + \rho_1(A)\rho_2(B) & \psi(A)\rho_1(B)+\rho_1(A)\eta(B) \\ \rho_2(A)\psi(B) + \eta(A)\rho_2(B) & \rho_2(A)\rho_1(B)+\eta(A)\eta(B) \end{pmatrix} \end{align*}
For $A$ selfadjoint and $B=A$ we obtain $\rho_2(A)=\rho_1(A)^*$; thus $\rho_1(A)\rho_1(A)^*=0$ and $\rho_1(A)=0$, $\rho_2(A)=0$. Since the linear span of the self-adjoint elements is $\mathcal{A}$, we get $\rho_1=0$, $\rho_2=0$. 
\end{proof}

\begin{lemma} \label{ken} \cite[Lemma III.2.1]{D}
Let $\mathcal{A}=\prod_{j=1}^s\C^{k_j \times k_j}$, $\varphi: \mathcal{A} \rightarrow \mathcal{B}$ be a $^*$-homomorphism into a finite-dimensional $C^*$-algebra $\mathcal{B} \subseteq \C^{m \times m}$. Then there exists a unitary matrix $U \in \C^{m \times m}$ and uniquely determined $h \in \N_0^{\{1,...,s\}}$, $r \in \N_{0}$ such that for all $C=\bigoplus_{j=1}^s C_j \in \mathcal{A}$
\begin{align*}
U\varphi\left(\bigoplus_{j=1}^s C_j\right)U^* = (I_r \otimes 0 ) \oplus \left(\bigoplus_{j=1}^s I_{h(j)} \otimes C_j\right) 
\end{align*}  
\end{lemma}

\subsection{\texorpdfstring{Operator systems, completely positive maps and their connection to matrix convexity}{C*-algebras and completely positive maps}}

For the study of matrix convexity we need the concepts of operator systems and completely positive maps. In particular the Arveson extension theorem and Stinespring representation theorem are important ingredients. We will sketch the proofs, but omit some details and proofs of lemmas. A good exposition of this theory is given in the book of Paulsen \cite{P}. 

\begin{definition}
Let $\mathcal{A} \subseteq \mathcal{B}(\mathcal{H})$ be a $C^*$-algebra. We call $\{a^2 \in \mathcal{A} \ | \ a \in \mathcal{A}, \ a^*=a \}=\{aa^* \in \mathcal{A} \ | \ a \in \mathcal{B}(\mathcal{H})\}$ the psd (positive semidefinite) elements of $\mathcal{A}$. For each $k \in \N$ we can interpret $\mathcal{B}(\bigoplus_{i=1}^k\mathcal{H})$ as $\C^{k \times k} \otimes \mathcal{B}(\mathcal{H})$ and take the operator norm from the former set to make the latter one again a $C^*$-algebra. \\[0.2cm]
A (concrete) operator system $(S,P)$ in $\mathcal{A}$ is a tuple satisfying: There is a linear subspace $\mathcal{S}$ of $\mathcal{A}$ which contains $1$ and is selfadjoint, i.e. $\mathcal{S}^*=\mathcal{S}$ such that $S=\bigcup_{k \in \N} \C^{k \times k} \otimes \mathcal{S}$ and $P=\bigcup_{k \in \N} \{A \in \C^{k \times k} \otimes \mathcal{S} \ | \ A \succeq 0\}$. $P$ is called the set of psd elements of $S$. In the literature the inaccuracy has been established to identify $\mathcal{S}$ with $(S,P)$, which can be dangerous because $P$ depends on the ambient space $\mathcal{A}$. However we will stick to this slightly dangerous convention, reminding the reader to be careful. \\[0.2cm]
If $\mathcal{B}$ is another $C^*$-algebra, then $\varphi: \mathcal{S} \rightarrow \mathcal{B}$ is called (unital) completely positive if $\varphi$ is linear, $\varphi(1)=1$ and for all $k \in \N$ the linear map \index{p@$\varphi_k$}$\varphi_k: \C^{k \times k} \otimes \mathcal{S} \rightarrow \C^{k \times k} \otimes \mathcal{B}$ defined by $B \otimes A \mapsto B \otimes \varphi(A)$ is positive, i.e. it maps positive semidefinite elements (i.e. elements of $P$) to positive semidefinite elements. The condition $1 \in \mathcal{S}$ guarantees that $\mathcal{S}$ is the span of its positive semidefinite elements; thus the condition to be completely positive is not merely a condition speaking about some small part of $\mathcal{S}$. A completely positive map $\varphi$ is automatically continuous and its operator norm equals $1$ (even all $\varphi_k$ have operator norm $1$).
One reason why completely positive maps were introduced is that their structure is more rigid than the one of positive maps. Much more is known about completely positive maps than about positive maps. There is no general procedure to test if a map is positive while one can build a hierarchy of semidefinite programmes to test whether a map is completely positive \cite[Chapter 4]{HKM4}.  
\end{definition}

\begin{proposition} \label{positest2} \cite[Theorem 3.14]{P}
Let \marginpar{[\autoref{traces2}]}$\mathcal{A}$ be a $C^*$-algebra and $\varphi: \C^{n \times n} \rightarrow \mathcal{A}$ linear. Then the following is equivalent:
\begin{enumerate}[(a)]
\item $\varphi$ is completely positive. 
\item $\varphi$ is $n$-positive.
\item The so-called Choi matrix $A=(A_{j,h})_{j,h \in \{1,...,n\}} \in \C^{n \times n} \otimes \mathcal{A}$ defined by $A_{j,h}=\varphi(e_je_h^*)$ is positive semidefinite.
\end{enumerate}
\end{proposition}

\begin{proposition} \label{positest} \cite[Theorem 6.1]{P}
Let $\mathcal{A}$ be a $C^*$-algebra, $\mathcal{S}$ an operator system in $\mathcal{A}$ and $n \in \N$. Let $\varphi: \mathcal{S} \rightarrow \C^{n \times n}$ be a linear map. We define the linear functional $s_\varphi: \C^{n \times n} \otimes \mathcal{S} \rightarrow \C$ by setting $s_\varphi(e_{j}e_h^* \otimes A)=\frac{1}{n} \varphi(A)_{j,h}$. The mapping $\varphi \mapsto s_\varphi$ is an isomorphism between the linear maps $\mathcal{S} \rightarrow \C^{n \times n}$ and the linear maps $\C^{n \times n} \otimes \mathcal{S} \rightarrow \C$.
The following is equivalent:
\begin{enumerate}[(a)]
\item $\varphi$ is completely positive.
\item $\varphi$ is n-positive.
\item $s_\varphi$ is positive and $(n s_\varphi(e_{j}e_h^* \otimes I))_{j,h \in \{1,...,n\}}=I$.
\end{enumerate}
\end{proposition}

\begin{proposition} \label{contract} \cite[Proposition 2.11 and Exercise 2.3]{P}
Let $\mathcal{S}$ be an operator system and $\varphi: \mathcal{S} \rightarrow \C$ be a linear functional. Then $\varphi$ is positive if and only if $\varphi$ is a contraction. 
\end{proposition}

\begin{lemma} \label{vorext} \cite[Theorem 6.2]{P}
Let $\mathcal{A}$ be a $C^*$-algebra, $\mathcal{S}$ an operator system in $\mathcal{A}$ and $n \in \N$. Then every completely positive map $\varphi: \mathcal{S} \rightarrow \C^{n \times n}$ admits a completely positive extension $\varphi: \mathcal{A} \rightarrow \C^{n \times n}$  
\end{lemma}

\begin{proof}
From \autoref{positest} we know that $s_\varphi: \C^{n \times n} \otimes \mathcal{S} \rightarrow \C$ is positive. Hence $s_\varphi$ is a contraction (\autoref{vorext}) and we can extend $s_\varphi$ to a linear contraction $t: \C^{n \times n} \otimes \mathcal{A} \rightarrow \C$. $t$ is again positive. Now we find linear $\psi: \mathcal{A} \rightarrow \C^{n \times n}$ such that $s_\psi=t$. \autoref{positest} says that $\psi$ is completely positive. 
\end{proof}

\begin{definition} \label{bwtop}
Let $\mathcal{S}$ be an operator system and $\mathcal{H}$ a Hilbert space. We define a topology on $\mathcal{B}(\mathcal{S},\mathcal{B}(\mathcal{H}))$ by setting: A net $(\varphi_\lambda)_\lambda$ converges to $\varphi$ if for all $x,y \in \mathcal{H}$ and $A \in \mathcal{S}$ we have
\begin{align*}
\langle \varphi_\lambda(A)x,y \rangle \rightarrow \langle \varphi(A)x,y \rangle
\end{align*} 
The BW-topology can be interpreted as the weak$^*$-topology of a certain Banach space.
\end{definition}

An important fact is that the BW-topology turns the set of completely positive maps into a compact set. The proof appearing in most textbooks shows that the BW-topology can be interpreted as the weak$^*$-topology coming from forming the dual of a certain Banach space (predual). After that one can apply the Banach-Alaoglu-theorem. Instead of that, we will not introduce the predual, but adapt the proof of the Banach-Alaoglu-theorem to this setting.   

\begin{lemma} \cite[Theorem 7.4]{P} \label{compactbw}
Let $\mathcal{S}$ be an operator system and $\mathcal{H}$ a Hilbert space. Then
\begin{align*}
\text{CPU}(\mathcal{S},\mathcal{B}(\mathcal{H})):=\{\varphi: \mathcal{S} \rightarrow \mathcal{B}(\mathcal{H}) \text{ is completely positive} \}  
\end{align*}
is a compact set with respect to the BW-topology.
\end{lemma}

\begin{proof}
Consider the map 
\begin{align*}
\Psi: \text{CPU}(\mathcal{S},\mathcal{B}(\mathcal{H})) &\rightarrow \prod_{A \in \mathcal{S}, \ x,y \in \mathcal{H}} \{a \in \C \ | \ ||a||\leq ||A|| \cdot ||x|| \cdot ||y||\}, \\f &\mapsto [(A,x,y) \mapsto \langle f(A) x,y \rangle]
\end{align*}
 where the left side carries the BW-topology and the right side the product topology. Tykhonovs theorem says that the right side is compact. It is easy to see that $\Psi$ is a homeomorphism onto $\Psi(\text{CPU}(\mathcal{S},\mathcal{B}(\mathcal{H})))$. So it is enough to show that $\Psi(\text{CPU}(\mathcal{S},\mathcal{B}(\mathcal{H})))$ is closed in $\prod_{A \in \mathcal{S}, \ x,y \in \mathcal{H}} \{a \in \C \ | \ ||a||\leq ||A|| \cdot ||x|| \cdot ||y||\}$. The set $\Psi(\text{CPU}(\mathcal{S},\mathcal{B}(\mathcal{H})))$ equals
\begin{align*}
&\bigcap_{A,B \in \mathcal{S}, x,y \in \mathcal{H}, \lambda \in \C} \left\{q \ | \ q(A+\lambda B,x,y) = q(A,x,y)+\lambda q(B,x,y)\right\} \\
&\cap \bigcap_{A\in \mathcal{S}, x,y \in \mathcal{H}} \{q \ | \ q(A^*,x,y)=\ov{q(A,y,x)}\} \\
&\cap \bigcap_{n \in \N, A_{h,j} \in \mathcal{S}, x_h \in \mathcal{H}, \ (1\leq h,j \leq n), (A_{h,j})_{h,j \in \{1,...,n\}} \text{ is psd }} \left\{q \ \middle| \ \sum_{h,j=1}^n q(A_{h,j},x_h,x_j) \geq 0\right\},
\end{align*}
which is a intersection of closed sets.
\end{proof}

\begin{theo} (Arveson extension theorem) (Arveson 1969) \cite[Theorem 7.5]{P} \label{arvesons}
Let $\mathcal{A}$ be a $C^*$-algebra, $\mathcal{S}$ an operator system in $\mathcal{A}$ and $\mathcal{H}$ a separable Hilbert space. Then every completely positive map $\varphi: \mathcal{S} \rightarrow \mathcal{B}(\mathcal{H})$ admits a completely positive extension $\varphi: \mathcal{A} \rightarrow \mathcal{B}(\mathcal{H})$.  
\end{theo}

\begin{proof}
We can find an increasing sequence $(\mathcal{H}_n)_{n \in \N}$ of finite-dimensional subspaces of $\mathcal{H}$ such that $\ov{\bigcup_{n \in \N} \mathcal{H}_n}=\mathcal{H}$. For each $n \in \N$ let $\mathcal{T}_n$ be an orthogonal complement of $\mathcal{H}_n$ in $\mathcal{H}$. Let $P_n$ be the projection from $\mathcal{H}$ onto $\mathcal{H}_n$. Define $\varphi_n: \mathcal{S} \rightarrow \mathcal{B}(\mathcal{H}_n) \bigoplus \mathcal{B}(\mathcal{T}_n), \ A \mapsto [P\varphi(A) P^*] \oplus 0$, which is still completely positive. By \autoref{vorext} we can extend $\varphi_n$ to a completely positive map $\psi_n: \mathcal{A} \rightarrow \mathcal{B}(\mathcal{H})$. Now on $\mathcal{S}$ the sequence $(\psi_n)_{n \in \N}$ converges to $\varphi$ in the BW-topology. Additionally we know that the set of completely positive maps from $\mathcal{A}$ to $\mathcal{B}(\mathcal{H})$ is compact, hence we can assume that $(\psi_n)_{n \in \N}$ converges to a completely positive map $\psi$ which is an extension of $\varphi$.    
\end{proof}

The following result characterizes completely positive maps as dilations of representations (the latter ones are easily seen to be completely positive, remembering that psd elements of a $C^*$-algebra are squares). The proof is a generalization of the GNS-construction. Since positive maps which have range $\C$ are automatically completely positive, the GNS-state-representation theorem can be seen as a special case of the following result. 

\begin{theo} (Stinespring representation theorem) (Stinespring 1955) \cite[Theorem 4.1]{P} \label{stinesprings}
Let $\mathcal{A}$ be a $C^*$-algebra, $\mathcal{H}$ a separable Hilbert space and $\varphi: \mathcal{A} \rightarrow \mathcal{B}(\mathcal{H})$ a completely positive map. Then there exists a Hilbert space $\mathcal{K}$, an isometry $V \in \mathcal{B}(\mathcal{H},\mathcal{K})$ and a representation $\pi: \mathcal{A} \mapsto \mathcal{B}(\mathcal{K})$ such that $\varphi(A)=V^* \pi (A) V$ for every $A \in \mathcal{A}$. This form is called Stinespring representation. \\[0.2cm]
Additionally, one can achieve that $\ov{\pi(\mathcal{A})V\mathcal{H}}=\mathcal{K}$, in which case the representation is called minimal. The minimal Stinespring  representation $(V,\pi,\mathcal{K})$ is uniquely determined up to unitary equivalence and can be obtained by restricting $\pi$ to its reducing subspace $\mathcal{K}'=\ov{\pi(\mathcal{A})V\mathcal{H}}$. 
\end{theo}

The following results relate completely positive maps with matrix convex combinations and are well-known. \autoref{rumpf2} and \autoref{rumpf} are Positivstellens\"atze, which translate the property of containment of two free spectrahedra (i.e. $\mathcal{D}_\ml \subseteq \mathcal{D}_\mh$) into a property regarding their defining pencils (i.e. a sums-of-squares representation of $\mh$ with $\ml$ as weight). 

\begin{cor} \label{steinfels}
Let $n,m \in \N$, $\mathcal{S} \subseteq \C^{n \times n}$ be an operator system. The completely positive maps $\varphi: \mathcal{S} \rightarrow \C^{m \times m}$ are exactly the maps for which there exists an $r \in \N$ and $V_1,...,V_r \in \C^{n \times m}$ such that $\varphi(A)=\sum_{j=1}^r V_j^* A V_j$ for every $A \in \mathcal{S}$. 
\end{cor}

\begin{proof} Let $\varphi$ be completely positive. First extend $\varphi$ to a completely positive map defined on the whole $\C^{n \times n}$. Consider the minimal Stinespring representation $(V,\pi,\mathcal{K})$ of $\varphi$. We have $\pi(\C^{n \times n})V\C^m=\mathcal{K}$. Thus $\mathcal{K}$ is at most of dimension $n^2m$. $\pi$ is unitary equivalent to a direct sum of $r$ identity representations (\autoref{ken}). Thus by absorbing a unitary operator into $V$ we get $\varphi(A)=V^* (A \oplus ... \oplus A) V$ for every $A \in \C^{n \times n}$ and $r \leq mn$. Write $V^*=(V_1^* ... V_r^*)$ with $V_j \in \C^{n \times m}$. Then we have $\sum_{j=1}^r V_j^* V_j = I$ and $\varphi(A)=\sum_{j=1}^r V_j^* A V_j$. \\[0.2cm]
The reverse direction is easy.
\end{proof}

\begin{cor} \label{steini}
Let $S \subseteq \S^g$. Then $\mconv(S)(k)=\{ B \in \S^g(k) \ | \ \exists C \in \S^g: \exists r \in \N, A_1,...,A_r \in S: C=A_1 \oplus ... \oplus A_r, \exists \varphi: \spam(I,C_1,...,C_g) \rightarrow \C^{k \times k}: \varphi(C)=B, \ \varphi \text{ is completely positive}\}$. In other words, $S$ is matrix convex if and only if it is closed with respect to direct sums and taking images of completely positive maps. \hfill\qedsymbol
\end{cor}

\begin{lemma} \label{voiculescu} \cite[Lemma II.5.2]{D}
Let $\mathcal{A} \subseteq \mathcal{B}(\mathcal{H})$ be a $C^*$-algebra containing no other compact operator from $\mathcal{B}(\mathcal{H})$ than $0$ and $\varphi: \mathcal{A} \rightarrow \C^{m \times m}$ be a completely positive map. Then there exists a sequence $(V_n)_{n \in \N}$ of isometries from $\C^m$ to $\mathcal{H}$ such that $\lim_{n \rightarrow \infty} || \varphi(A) - V_n^* A V_n||=0$ for all $A \in \mathcal{A}$.
\end{lemma}

\begin{lemma} \label{operator} \cite[Proposition 3.5]{DDSS}
If $\mathcal{H}$ is a separable Hilbert space and $L \in \mathcal{B}_h(\mathcal{H})^g$, then 
\begin{align*}
&\mconv(L)=\left(\ov{\mconv\{PLP^* \ | \ P: \mathcal{H} \rightarrow \im(P) \text{ is a projection of rank at most }k \}}(k)\right)_{k \in \N} \\&=\{ B \in \S^g(k) \ | \ \exists \varphi: \spam(I,L_1,...,L_g) \rightarrow \C^{k \times k}: \varphi(L)=B, \ \varphi \text{ is completely positive}\} \\
\\&=\{A \ | \ \exists (V_n)_{n \in \N} \subseteq \mathcal{B}(\C^{\size(A)},\mathcal{H}^{(\infty)}): V_n^*V_n=I, \ \lim_{n \rightarrow \infty} ||A-V_n^*L^{(\infty)}V_n ||=0\}.
\end{align*}
$\mconv(L)$ is compact.
\end{lemma}

\begin{proof} (cf. \cite[Proposition 3.5]{DDSS}) Set $L^{(\infty)}=\bigoplus_{n \in \N} L \in \mathcal{B}(\mathcal{H}^\infty)$. Note that the unital linear map 
\begin{align*}
\psi: \spam(I,L^{(\infty)}_1,...,L^{(\infty)}_g) \rightarrow  \spam(I,L_1,...,L_g)
\end{align*} defined by $\psi(L^{(\infty)}_i)=L_i$ is completely positive (even completely isometric).
Let $\varphi: \spam(I,L_1,...,L_g) \rightarrow \C^{k \times k}$ be completely positive and $\varphi(L)=B \in \S^g(\delta)$. Since the composition of completely positive maps is again completely positive, we know that there exists a completely positive map $\rho: \spam(I,L^{(\infty)}_1,...,L^{(\infty)}_g)$ such that $\rho(L^{(\infty)})=B$. Since $C^*(I,L^{(\infty)}_1,...,L^{(\infty)}_g)$ does not contain other compact operators than $0$, \autoref{voiculescu} applies and we find a sequence of isometries $(V_n)_{n \in \N} \subseteq \mathcal{B}(\C^\delta,\mathcal{H}^{(\infty)})$ such that $\lim_{n \rightarrow \infty} ||B-V_n^*L^{(\infty)}V_n ||=0$. \\[0.2cm] 
Now suppose $(V_n)_{n \in \N} \subseteq \mathcal{B}(\C^\delta,\mathcal{H}^{(\infty)})$ is a sequence of isometries satisfying $\lim_{n \rightarrow \infty} ||A-V_n^*L^{(\infty)}V_n ||=0$. Fix an isometry $W \in \mathcal{B}(\C^\delta,\mathcal{H})$. Let $\ep>0$. Choose $n \in \N$ such that $||A-V_n^*L^{(\infty)}V_n || < \frac{\ep}{3}$. As $V_n$ is defined on a finite-dimensional vector space, we know that there is an $M \in \N$ such that $||V_n^* P_m^* P_m L^{(\infty)} P_m^* P_m V_n-V_n^*L^{(\infty)}V_n|| < \frac{\ep}{3}$ for all $m>M$, where $P_m$ denotes the projection of $\mathcal{H}^{(\infty)}$ onto $\mathcal{H}^m$. With the same argument we know that $||I-V_n^* P_m^* P_m V_n|| \rightarrow 0$ for $m \rightarrow \infty$. Hence we can find a big $m >M$ such that $||\sqrt{I-V_n^* P_m^* P_m V_n}W^* L W\sqrt{I-V_n^* P_m^* P_m V_n}|| < \frac{\ep}{3}$. All in all, we obtain $||A- (V_n^* P_m^* L^m P_m V_n +\sqrt{I-V_n^* P_m^* P_m V_n}W^* L W\sqrt{I-V_n^* P_m^* L P_m V_n})||< \ep$. \\[0.2cm] We have proven that for each $\ep>0$ there is $n \in \N$ and an isometry $W \in \mathcal{B}(\C^\delta,\mathcal{H}^n)$ such that $||A-W^*L^nW||<\ep$. For $h \in \{1,...,n\}$ let $Q_h: \mathcal{H}^n \rightarrow \mathcal{H}, (x_1,...,x_n) \mapsto x_h$ and $R_h$ be the projection from $\mathcal{H}$ onto $\im (Q_h W)$. Then we calculate $W^* L^n W=\sum_{h=1}^n W^* Q_h^* L Q_h W=\sum_{h=1}^n W^* Q_h^* R_h^* (R_h L R_h^*) R_h Q_h W \in \mconv(\{PLP^* \ | \ P: \mathcal{H} \rightarrow \im(P) \text{ projection of rank at most }\delta\})$ \\[0.2cm]
Now it is clear that 
\begin{align*}&\mconv\{PLP^* \ | \ P: \mathcal{H} \rightarrow \im(P) \text{ projection on a finite-dimensional space)}\} \\ \subseteq &\{ B \in \S^g(k) \ | \ \exists \varphi: \spam(I,L_1,...,L_g) \rightarrow \C^{k \times k}: \varphi(L)=B, \ \varphi \text{ is completely positive}\}.\end{align*} We only have to show that the latter set is closed. This is the case due to \autoref{compactbw}. \\[0.2cm] It is evident that $\mconv(L)$ is bounded.
\end{proof}

The following Positivstellens\"{a}tze are very important and tell us that for two monic pencils $\ml$ and $\mh$ we have $\mathcal{D}_\ml \subseteq \mathcal{D}_\mh$ if and only if $\mh$ has a sums-of-squares representation with $\ml$ as a weight.

\begin{theo} \cite{HKM}[Theorem 4.6] \label{rumpf2}
Let $\mathcal{H}$ be a separable Hilbert space and $L,H \in \mathcal{B}_h(\mathcal{H})^g$, $n \in \N$ and define the operator systems $\mathcal{S}=\spam(I \oplus 1,L_1 \oplus 0,...,L_g \oplus 0)$, $\mathcal{T}=\spam(1,H_1,...,H_g)$. Consider the linear map $\varphi : \mathcal{S} \rightarrow C^*(\mathcal{T})$ defined by $\varphi(I \oplus 1)=I, \varphi(L_i \oplus 0)=H_i$. Then
\begin{enumerate}[(a)]
\item $\varphi$ is well-defined and $n$-positive if and only if $\mathcal{D}_{\mathfrak{L}}(n) \subseteq \mathcal{D}_{\mathfrak{H}}(n)$
\item $\varphi$ is well-defined and completely positive if and only if $\mathcal{D}_{\mathfrak{L}} \subseteq \mathcal{D}_{\mathfrak{H}}$.
\end{enumerate}
\end{theo}

\begin{proof} \cite{Z}[Theorem 2.5]
(a) Suppose first that $\varphi$ is $n$-positive and well-defined. Now let $A \in \S^g(n)$. If $A \in \mathcal{D}_{\mathfrak{L}}(n)$, then $(I-A \ov X) (L \oplus 0) \approx ((I \oplus 1)-[L \oplus 0] \ov X)(A) \succeq 0$ and $(I-A \ov X) (L \oplus 0) \in \C^{n \times n} \otimes \mathcal{S}$. Hence $0 \preceq\varphi_n((I-A \ov X) (L \oplus 0))=(I-A \ov X)(H) \approx \mathfrak{H}(A)$. \\[0.2cm]
Now suppose that $\mathcal{D}_{\mathfrak{L}}(n) \subseteq \mathcal{D}_{\mathfrak{H}}(n)$. We show well-definedness first. Suppose $\lambda \in \C^n$ such that $\sum_{i=1}^g \lambda_i L_i=0$. Write $\lambda=\mu+ \ii \nu$ with $\mu,\nu \in \R^g$. We know that $\R \mu \subseteq \mathcal{D}_\ml$. If $\sum_{i=1}^g \lambda_i H_i \neq 0$, then we find $r \in \R$ such that $r\lambda \notin \mathcal{D}_\mh$. In the same way we show that $\nu=0$.\\[0.2cm]
Let $B \in \C^{n \times n} \otimes \mathcal{S}$ such that $B \succeq 0$. Choose $A_0 \in \C^{n \times n}$ and $A \in (\C^{n \times n})^g$ such that $B= A_0 \otimes (I \oplus 1) + \sum_{i=1}^g A_i \otimes (L_i \oplus 0)$. In particular $A_0 \otimes 1 + \sum_{i=1}^g A_i \otimes 0 \succeq 0$, whence $A_0 \succeq 0$. We notice that we can demand that $A \in \S^g(n)$ (indeed if $L_1,...,L_g$ is linear independent, then $A \in \S^g(n)$ automatically; otherwise we represent $B$ by a tuple $A$ for which $(L_i \ | \ i \in \{1,...,g\}, A_i \neq 0)$ is linear independent).\\[0.2cm]
Now let $\ep>0$. We find an invertible square matrix $C$ such that $C^*(A_0 + \ep I)C=I$.
From $(I \otimes C)^*[I \otimes (A_0 + \ep I) + \sum_{i=1}^g L_i \otimes A_i ](I \otimes C)\succeq 0$ we conclude that $I \otimes I + \sum_{i=1}^g L_i \otimes {C}^* A_i C \succeq 0$. Thus $C^* A C \in \mathcal{D}_{\mathfrak{L}}(n)$ and by assumption $C^* A C \in \mathcal{D}_{\mathfrak{H}}(n)$. This means $I \otimes I + \sum_{i=1}^g H_i \otimes C^* A_i C \succeq 0$ and consequently $(A_0 + \ep I) \otimes I + \sum_{i=1}^g A_i \otimes H_i \approx I \otimes (A_0 + \ep I) + \sum_{i=1}^g H_i \otimes A_i \succeq 0$. By letting $\ep$ go to $0$ we obtain $\varphi_n(B)=A_0 \otimes I + \sum_{i=1}^g A_i \otimes H_i \succeq 0$ as well. \\[0.2cm]
(b) follows from (a).
\end{proof}

\begin{proposition} \label{bondi} \cite[Lemma 3.6]{HKM4}
Suppose $\mathcal{H}$ is a separable Hilbert space, $L \in \mathcal{B}_h(\mathcal{H})^g$ and that $\mathcal{D}_{\mathfrak{L}}(1)$ is bounded. Then $0 \in \mconv(L)$. 
\end{proposition}

\begin{proof} \cite[Lemma 3.6]{HKM4}
Define the operator systems $\mathcal{S}=\spam(I,L_1,...,L_g)$ and $\mathcal{T}=\spam(1)$. We will show that the map $\varphi: \mathcal{S} \rightarrow \mathcal{T}$ defined by $\varphi(L_i)=0$ and $\varphi(I)=1$ is well-defined and completely positive. Since $\mathcal{D}_\ml$ is bounded, it is clear that that $\{L_i \ | \ i \in \{1,...,g\}\} \cup \{I\}$ is linear independent. Hence $\varphi$ is well-defined. So let $A_0 \in \C^{k \times k}$ and $A \in (\C^{k \times k})^g$ such that $(A_0 \otimes I) + \sum_{i=1}^g A_i \otimes L_i \succeq 0$. We notice that $A \in \S^g$ and $A_0 \in S\C^{k \times k}$ because of the linear independence of $I,L_1,...,L_g$. Assume that $A_0$ was not positive semidefinite. Choose $v \in \mathcal{S}^{k-1}$ with $v^*A_0v <0$. Then $0 \preceq (v^* \otimes I)[(A_0 \otimes I) + \sum_{i=1}^g A_i \otimes L_i](v \otimes I) \approx I \otimes v^*A_0v + \sum_{i=1}^g L_i \otimes v^*A_iv$. This means that $\sum_{i=1}^g L_i \otimes v^*A_iv \succ 0$; therefore $\R_{\leq 0} v^*Av \in \mathcal{D}_{\mathfrak L}(1)$, which contradicts the boundedness. 
\end{proof}

\begin{cor} \cite{HKM4}[Theorem 3.5] \label{rumpf}
Let $L,H \in \S^g$ such that $\mathcal{D}_{\mathfrak L}(1)$ is bounded. Define the operator systems $\mathcal{S}=\spam(I,L_1,...,L_g)$, $\mathcal{T}=\spam(1,H_1,...,H_g)$. Then $\mathcal{D}_{\mathfrak L} \subseteq \mathcal{D}_{\mathfrak H}$ if and only if the linear map $\varphi : \mathcal{S} \rightarrow C^*(\mathcal{T})$ defined by $\varphi(I)=I, \varphi(L_i)=H_i$ is completely positive if and only if there is $r \in \N$ and matrices $V_1,...,V_r$ such that $\sum_{j=1}^r V_j^* L V_j=H$ and $\sum_{j=1}^r V_j^* V_j =I$. \hfill\qedsymbol
\end{cor}

Next, we want to prove the bipolar theorem. In order to do that we determine the polar of a free spectrahedron.

\begin{cor} \label{polariapp} \cite[Theorem 4.6, Proposition 4.9]{HKM}
Let $\mathcal{H}$ be a separable Hilbert space, $L \in \mathcal{B}_h(\mathcal{H})^g$. Then we have $\mathcal{D}_{\mathfrak{L}}^\circ=\mconv(L,0)$ and $\mconv(L)^\circ=\mathcal{D}_{\mathfrak{L}}$. If $\mathcal{D}_{\mathfrak{L}}$ is even bounded, then $\mconv(L,0)=\mconv(L)$.
\end{cor}

\begin{proof}
$\mathcal{D}_{\mathfrak{L}}^\circ=\mconv(L,0)$ can be obtained by looking into \autoref{rumpf2}. \\[0.2cm]
Now let $A \in \mathcal{D}_{\mathfrak{L}}$ and $B \in \mconv(L)(\delta)$. We know that there is a completely positive map $\varphi: \spam(I,L_1,...,L_g) \rightarrow \C^{\delta \times \delta}$ such that $B = \varphi(L)$. Then we have that $0 \preceq \ml(A) \approx (I-A \ov X)(L)$ and $0 \preceq \varphi_{\size(A)} ((I- A\ov X)(L))=(I-A \ov X)(B)$. \\[0.2cm]
Now let $H \in \mconv(L)^\circ$. Then for every projection $P: \mathcal{H} \rightarrow \im(P)$ on a finite-dimensional subspace we have $0 \preceq (I-H \ov X)(PLP^*) \approx (I-PLP^* \ov X)(H)$ and thus $H \in \mathcal{D}_\ml$. \\[0.2cm] The last part is a consequence of \autoref{bondi}.
\end{proof}

\begin{lemma} \label{closca}
Let $S \subseteq \S^g$ be matrix convex. Then $S=\mconv(L)$ for some $L \in \mathcal{B}_h(\mathcal{H})^g$ if and only if $S$ is compact.  
\end{lemma}

\begin{proof}
First claim: Let $S$ be compact. Let $\{A_n \ | \ n \in \N\}$ be a dense subset of $S$. Then we claim $\mconv(\bigoplus_{n \in \N}A_n)=\ov{\mconv(\{A_n \ | \ n \in \N \})}=S$. $\bigoplus_{n \in \N}A_n$ is an element of $\mathcal{B}_h(\mathcal{H})^g$ where $\mathcal{H}=\ov{\bigoplus_{n \in \N} \C^{\size(A_n)}}$ due to compactness of $S$. We use \autoref{operator} and have only to show that for each isometry $V: \C^k \rightarrow \mathcal{H}$ we have $V^* \bigoplus_{n \in \N}A_n V \in S$. However letting $P_m: \mathcal{H} \rightarrow \bigoplus_{n=1}^m \C^{\size(A_n)}$ be the canonical projection, we have $V^* \bigoplus_{n \in \N} A_n V=\lim_{m \rightarrow \infty} V^* P_m^*[P_m \bigoplus_{n \in \N} A_n P_m^*]P_m V=\lim_{m \rightarrow \infty} V^* P_m^*[\bigoplus_{n=1}^m A_n]P_m V \in \ov{S}$.
\end{proof}

\begin{proof} (of \autoref{simpolara})
Let $S \subseteq \S^g$ be matrix convex. Obviously $S^{\circ}=\ov{\mconv(S \cup \{0\})}^\circ$ and hence we can assume that $S$ is closed and contains $0$. Suppose first that $0 \in \inte{S}$. We have to show $S^{\circ \circ}=S$. It is clear that $S \subseteq S^{\circ \circ}$. We have $S=\bigcup_{n \in \N} S_n$ where $S_n:=S \cap \ov{B_{\S^g}(0,n)}$ is matrix convex. We know that $S_n=\mconv(L_n)$ for some $L_n \in \mathcal{B}_h(\mathcal{H})^g$ due to \autoref{closca}. We obtain $S^\circ=:T=:\bigcap_{n \in \N} T_n$ where $(T_n)_{n \in \N}=\mathcal{D}_{I-L_n \ov X}$ is a descending sequence of matrix convex compact sets (\autoref{simpolarbb}).\\[0.2cm] We claim $T^{\circ}=\bigcup_{n \in \N} T_n^\circ$. This proves the theorem as $\bigcup_{n \in \N} T_n^\circ=\bigcup_{n \in \N} \mconv(L_n)=S$ (\autoref{polariapp}). It is immediate that $T_n^\circ \subseteq T^\circ$ for all $n \in \N$. Now assume $A \in T^\circ (\delta)$ and $A \notin \bigcup_{n \in \N} T_n^\circ$. As $A \notin \bigcup_{n \in \N} T_n^\circ=S$ and $S$ is closed, we find $\lambda \in (0,1)$ such that $\lambda A \in T^\circ \setminus S$. Then for each $n \in \N$ we find $B_n \in T_n(\delta)$ such that $(I-\lambda A \ov X)(B_n) \nsucceq 0$. Since $(T_n)_{n \in \N}$ is a descending sequence of compact sets, we can suppose that $B=\lim_{n \rightarrow \infty} B_n$ exists. Let $C \in S$. Then $(I-B_n \ov X)(C) \succeq 0$ for all $n \in \N$ with $C \in B_{\S^g}(0,N)$. Hence $B \in T=S^\circ$. Obviously, $(I-\lambda A \ov X)(B) \nsucc 0$. Together with $\lambda A \in T^\circ$, this means $\ker(I- \lambda A \ov X)(B) \neq \{0\}$. But now $(I- A \ov X)(B) \nsucc (1-\frac{1}{\lambda})I$, which contradicts $A \in T^\circ, B \in T$. \\[0.2cm]
In the case that $0 \notin \inte{S}$, we can find an subspace $U$ of $\R^g$ containing $S(1)$ such that $0$ is in the relative interior of $S(1)$ with respect to $U$. WLOG $U=\R^r \times \{0\}^{g-r}$ with some $r < g$. \autoref{inte} ensures that $S \subseteq \S^r \times \{0\}^{g-r}$. Write $S=S' \times \{0\}^{g-r}$. Denote by $\Box$ the polar in $\S^{r}$. Now we have $S^{\circ \circ}=(S'^\Box \times \S^{g-r})^\circ=S'^{\Box \Box} \times \{0\}^{g-r}=S' \times \{0\}^{g-r}=S$.   
\end{proof}

\begin{proof} (of \autoref{trenn} (a))
Using \autoref{simpolara} and \autoref{simpolarc} we get $H \in S^\circ(\delta)$ such that $\mh(Y) \nsucceq 0$.
\end{proof}

\begin{cor}
Let $S \subseteq \S^g$ be closed and matrix convex. Then $S=\mathcal{D}_\ml$ for some $L \in \mathcal{B}_h(\mathcal{H})^g$ if and only if $0 \in \inte(S)$.
\end{cor}

\begin{proof}
Suppose $0 \in \inte{S}$. Then $S^\circ$ is compact (\autoref{simpolarbb}) and hence we find $L \in \mathcal{B}_h(\mathcal{H})^g$ such that $S^\circ=\mconv(L)$. Now we have $S=S^{\circ\circ}=\mathcal{D}_\ml$ (\autoref{simpolara} and \autoref{closca}).
\end{proof}

\begin{proposition} \label{shu} (Schur complement)
Let $D=\begin{pmatrix} A & B \\ B^* & C \end{pmatrix} \in S\C^{k \times k}$ and $A \succ 0$. Then $D$ is positive semidefinite if and only if $C-B^*A^{-1}B \succeq 0$.
\end{proposition}

\begin{proof}
$\begin{pmatrix} A & B \\ B^* & C \end{pmatrix}=\begin{pmatrix} I & 0 \\ B^*A^{-1} & I \end{pmatrix}\begin{pmatrix} A & 0 \\ 0 & C-B^*A^{-1}B \end{pmatrix}\begin{pmatrix} I & A^{-1}B \\ 0 & I \end{pmatrix}$
\end{proof}

\begin{cor} \label{schwatz} (Schwarz inequality for completely positve maps) \cite[Proposition 3.3]{P}
Let $\mathcal{A},\mathcal{B}$ be $C^*$-algebras and $\varphi: \mathcal{A} \rightarrow \mathcal{B}$ completely positive. For $A \in \mathcal{A}$ we have $\varphi(A^*)\varphi(A) \preceq \varphi(A^*A)$.
\end{cor}

\begin{proof} \cite[Proposition 3.3]{P}
$D=\begin{pmatrix} 1 \\ A^* \end{pmatrix}\begin{pmatrix} 1 & A\end{pmatrix}=\begin{pmatrix} 1 & A \\
A^* & A^*A \end{pmatrix}$ is positive semidefinite. Therefore $0 \preceq \varphi_2(D)=\begin{pmatrix} 1 & \varphi(A) \\
\varphi(A^*) & \varphi(A^*A) \end{pmatrix}$. Now the claim follows from \autoref{shu}.
\end{proof}

\subsection{Real closed fields and semialgebraic sets} \textcolor{inv}{a} \\
A real closed field $R$ is a field such that for $R^2=\{r^2 \ | \ r \in R\}$ we have $R^2 \cup -R^2=R$, $R^2 \cap -R^2=\{0\}$, $R^2+R^2 \subseteq R^2$ and every polynomial $f \in R[X]$ of odd degree has a root in $R$. The convention $a \geq b :\Longleftrightarrow a-b \in R^2$ defines an ordering on $R$. This is the only ordering on $R$. One can also characterize real closed fields as orderable fields $R$ for which $R[\ii]$ is algebraically closed. Every ordered field $(K,\leq)$ admits a real closure $(R,\leq)$ which is a real closed extension field $(R,\leq)$ extending the order of $K$ such that $R|K$ is algebraic. Up to order isomorphism the real closure is uniquely determined. Of course the most important example of a real closed field is the field of real numbers $\R$.

\begin{definition}
Let $R$ be a real closed field and $K \subseteq R$ a real closed subfield. We say that $S \subseteq R^n$ is a $K$-semialgebraic set if there exist $f_h,g_{h,j} \in K[\ov{X}]$ such that $S=\bigcup_{h=1}^r \{x \in R^n \ | \ f_h(x)=0,g_{h,1}(x)>0,...,g_{h,m}(x)>0\}$. If we dont specify $K$, we mean $K=\R$ or $K=R$.
\end{definition}

The sets $\{x \in R^n \ | \ \sum_{i=1}^n x_i^2 \leq \ep\}$ where $\ep \in R_{>0}$ form the basis of a topology on $R^n$. One key fact is that the theory of real closed fields in the language of ordered rings admits quantifier elimination (in particular if $R$ is a real closed field, $F_1,F_2$ are real closed extension fields and $\varphi(x)$ a formula in the language of real closed fields with parameters from $R$ in the free variables $x$, then $\varphi(x)$ is satisfiable over $F_1$ if and only if it is satisfiable over $F_2$). This means that projections of $K$-semialgebraic sets are again $K$-semialgebraic. An important corollary is the following transfer-principle. 

\begin{theo} (Tarski transfer principle)
Let $R$ be a real closed field and $F \supseteq R$ be a real closed extension field of $R$ and $n \in \N_0$. Let $S=\bigcup_{h=1}^r \{x \in R^n \ | \ f_h(x)=0,g_{h,1}(x)>0,...,g_{h,m}(x)>0\}$ be an $R$-semialgebraic set where $f_h,g_{h,j} \in R[\ov{X}]$. Then there is only one $R$-semialgebraic set $S_F \subseteq F^n$ with $S_F \cap R^n=S$ which is called the transfer of $S$ into $F$. We have $S_F=\bigcup_{h=1}^r \{x \in F^n \ | \ f_h(x)=0,g_{h,1}(x)>0,...,g_{h,m}(x)>0\}$.
\end{theo}

The importance of this theorem lies in the fact that one can easily generalize statements, which can be stated using semialgebraic sets, from the real numbers to all real closed extension fields. As an exercise the reader should try to prove: Let $f:\R \rightarrow \R^n$ be a continuous/bijective/monotoneous/... function whose graph is semialgebraic and $R$ a real closed extension field of $\R$. Then $f$ extends uniquely to a function $f_R:R \rightarrow R^n$ whose graph is $\R$-semialgebraic. $f_R$ is continuous/bijective/monotoneous/... \\[0.2cm]
If $S$ is an $R$-semialgebraic set, then 
\begin{align*}
\prod_{F \text{ real closed extension field of }R} S_F
\end{align*} is called a semialgebraic class. \\[0.2cm]
For a real closed extension field $R$ of $\R$ denote by $\mfm_R=\left\{a \in R \ | \ \forall N \in \N: -\frac{1}{N} < a < \frac{1}{N} \right\}$ the set of {\bfseries infinitesimal elements} and by $\OO_R=\left\{a \in R \ | \ \exists N \in \N: -N< a < N\right\}$ the set of {\bfseries finite elements}. A real closed field $R$ is Archimedean if for all $a \in R$ there exists $N \in \N$ such that $a \pm N >0$. Every Archimedean ordered field can be interpreted as a subfield of $\R$.  Therefore every real closed extension field of $\R$ is non-Archimedean and contains infinitesimal and infinite elements. A transcendent extension $L$ of an ordered field $(K,\leq)$ admits an ordering extending $\leq$. Forming the real closure, one sees that there are many non-Archimedean fields. For a real closed field $R$ and $a \in R$ we define $|a|=\max(a,-a)$. For $a,b \in R$ we write $a \sim b$ if there exist $N,M \in \N$ such that $|a| < N|b| < M|a|$. The residue classes $\widetilde{a}$ of this equivalence class are called Archimedean classes. The residue class map $\nu$ is called the canonical valuation on $R$. It fulfills $\nu(a+b) \geq \min\{\widetilde{a}, \widetilde{b}\}$, where $\widetilde{a} \geq \widetilde{b}$ if there is $N \in \N$ such that $|a| \leq |b|N$. \\[0.2cm]
The canonical residue map $\OO_R \rightarrow \OO_R/\mfm_R, a \mapsto \ov{a}^{\mfm_R}$ is a ring homomorphism and for each $a \in \OO_R$ there is a unique $b \in \R$ such that $b \in \ov{a}^{\mfm_R}$. This $b$ is called the standard part of $a$ and denoted by $\text{st}(a)$. $\text{st}: \OO \rightarrow \R$ is a homomorphism of ordered rings meaning that for $a,b \in \OO$ the inequality $a \leq b$ implies $\text{st}(a) \leq \text{st}(b)$.

\begin{theo} (Finiteness theorem for semialgebraic classes) \label{finot}
Let $I$ be an index set, $R$ a real closed field and for every $j \in J$ let $S^j$ be a semialgebraic set. Suppose that $\bigcap_{j \in J} S_\mathcal{R}^j=\emptyset$ for every real closed extension field $\mathcal{R}$ of $R$. Then there exists a finite set $I \subseteq J$ such that $\bigcap_{j \in I} S^j=\emptyset$.  
\end{theo}   

\subsection{Weak separation and a proof of the free Minkowski theorem using pencils}

In the theory of ordinary convexity sometimes strong separation is not possible. For example, a non-exposed point $A$ of a convex set $S \subseteq \R^n$ can not be separated strictly from $S \setminus \{A\}$ by an affine linear function. However it is possible to separate these two sets strongly by a finite hierarchy of affine-linear functions $\varphi_1,...,\varphi_r$ in the sense that $\varphi_1(A)=...=\varphi_r(A)=1$ and for each $B \in S$ there exists $j \in \{1,...,r\}$ such that $\varphi_1(B)=...=\varphi_{j-1}(B)=1$ and $\varphi_j(B)<1$. Equivalently, by working with an infinitesimal number $\ep$ from a real closed extension field $R$ of $\R$, one can form the affine-linear function $\varphi=\sum_{j=1}^r \ep^j \varphi_j$, which separates the two sets strongly. The second view point was introduced in \cite{NT2}. \\[0.2cm]
When going through the idea of the Effros-Winkler translation process, it is not clear how to translate all the $\varphi_j$ into pencils (since there is no concave function $\Psi_j$ in \autoref{effri1} for each $\varphi_j$). However translating $\varphi$ can be achieved with some technical difficulties.  \\[0.2cm]
The aim of this section is to prove the free Minkowski theorem again by using pencils and not the homogenization trick from the proof in Chapter 6. The proof is much more involved, however is has also beautiful aspects in the opinion of the author. Normally, the classical Minkowski theorem is proven by an induction on the dimension and working with the faces of a convex set. However the concept of a "free face" is cumbersome because those sets would not be matrix convex anymore. The proof of \autoref{min2} shows how those "free faces" could look like and how to argue with those non-matrix convex objects. 

\begin{definition}
Let $R$ be a real closed field extension of $\R$ with algebraic closure $C=R[\ii]$ and $S \subseteq \S^g$. Then we set $\mconv_R(S):=\{\sum_{j=1}^r V_j^* B_j V_j \ | \ r,k \in \N, B_j \in S(k_j), V_j \in C^{k_j \times k}, \sum_{j=1}^r V_j^* B_j V_j=I\}$.
\end{definition}     

\begin{proposition} \label{effro0}
Let $H \subseteq \S^g$ be matrix convex, $0 \in H$ and $A \in \S^g(\delta) \setminus H$. Set $r=\delta^2g-1$. If $R$ is a real closed extension field $\R$, $C=R[\ii]$ and $\ep \in R_{>0}$ infinitesimal, then there exist $\varphi_0,...,\varphi_r: \S^g(\delta) \rightarrow \R$ affine linear such that $\varphi:=\sum_{j=0}^r \ep^j \varphi_j$ fulfills: $\varphi(B) \leq \gamma :=1+\ep+...+\ep^r-\ep^{r+1}$ for all $B \in \mconv(H)$ and $\varphi(A)=\gamma+\ep^{r+1}$. \\[0.2cm]
Now extend $\varphi$ to a $R$-linear function on $(SC^{\delta \times \delta})^g$. For all $A_1,...,A_s \in H$ and $W_1,...,W_s \in R^{\size(A_h) \times \delta}$ such that $\sum_{h=1}^s W_h^* W_h \preceq I$ and $C=\sum_{h=1}^s W_h^* A_h W_h$ we have $\varphi(C) \leq \gamma$.
\end{proposition}

\begin{proof}
We construct the $\varphi_j$ inductively. For this purpose we set $H_j=\bigcap_{\al < j} \varphi_\al^{-1}(1) \cap H$, $W_j=\bigcap_{\al < j} \varphi_\al^{-1}(1) \cap \S^g(\delta)$ and we will achieve that $\dim (W_j)=g \delta^2-j-1$ and $A \in W_j$. Now suppose that $\varphi_{0},...,\varphi_{j-1}$ have already been defined. Then $H_j$ is a convex set in $W_j$ which does not contain $A$. Now we can find non-trivial affine-linear $\varphi_j: W_j \rightarrow \R$ such that $\varphi_j(H_j) \subseteq (-\infty,1]$ and $\varphi_j(A)=1$. Now extend $\varphi_j$ to a $R$-linear function $\varphi_j: (SC^{\delta \times \delta})^g \rightarrow R$. This shows the first part. \\[0.2cm]
For the second part fix $A_1,...,A_s \in H$ and $W_1,...,W_s \in R^{\size(A_h) \times \delta}$ satisfying both $\sum_{h=1}^s W_h^* W_h \preceq I$ and $C=\sum_{h=1}^s W_h^* A_h W_h$. Let $j$ denote the first index such that $\text{st}(\varphi_j(C)) \neq 1$ and $r+1$ if no such index exists. We observe $\text{st}(\varphi_k(C))=\varphi_k(\text{st}(C))$ for all $k \in \{1,...,r\}$ and $\text{st}(C) \in H$. Inductively we conclude that $\varphi_k(C) \leq 1$ for all $k \leq j$ due to the Tarski-transfer principle. If $j \neq r+1$, we know that $\text{st}(\varphi_j(C))<1$ which shows the claim. If $j=r+1$, then $\varphi(\text{st}(C))=\gamma+\ep^{r+1}$, a contradiction. 
\end{proof}

\begin{proposition} \label{effro1}
Suppose we are in the situation of \autoref{effro0}. By setting $\psi=\frac{1}{\gamma} \varphi$ we get the linear $R$-functional $\psi: (SC^{\delta \times \delta})^g \rightarrow R$ such that $\psi(B) \leq 1$ for all $B \in (\mconv_R H)(\delta)$ and $\psi(A)>1$. For $B \in H_R(k):=(\mconv_R H)(k)$ and a contraction $V \in C^{\delta \times k}$ define $f_{B,V}: SC^{\delta \times \delta} \rightarrow R, T \mapsto \text{tr}(V \ov{T} V^*)-\psi(V^* B V)$. \\[0.2cm]
Then the set $\mathcal{F}=\{f_{B,V} \ | \ k \in \N, \ B \in H_R(k), V\in C^{\delta \times k}, V^*V \preceq I\}$ is $R$-convex. For every $f \in \mathcal{F}$ there is a $T \in \mathcal{T}_{R,\delta}$ such that $f(T) \geq 0$. 
\end{proposition}

\begin{proof}
This can be obtained by applying the Tarski-transfer to \autoref{effri1}.
\end{proof}

\begin{lemma} \label{effro2}
Suppose we are in the situation of \autoref{effro1}. Then there is a real closed extension field $\mathcal{R}$ of $R$ and $T \in \mathcal{T_\mathcal{R,\delta}}$ such that $f(T) \geq 0$ for all $f \in \mathcal{F}$.
\end{lemma}

\begin{proof}
For $f \in \mathcal{F}$ consider the $R$-semialgebraic classes
\begin{align*}\{(\mathcal{R},T) \ | \ \mathcal{R} \text{ real closed extension field of }R, \ T \in \mathcal{T}_{\mathcal{R},\delta}, f(T) \geq 0\}.
\end{align*} 
\autoref{finot} tells us that in order to prove 
\begin{align*}
\bigcap_{f \in \mathcal{F}} \{(\mathcal{R},T) \ | \ \mathcal{R} \text{ real closed extension field of }R, \ T \in \mathcal{T}_{\mathcal{R},\delta}, f(T) \geq 0\}\neq \emptyset,
\end{align*}
it is enough to show that every finite intersection of those sets is non-empty.
That the latter is the case follows from \autoref{effri2} and the Tarski-transfer principle. (We remark that we cannot use the Tarski-transfer principle directly and avoid the introduction of a new real closed field; the reason is that $H$ does not need to be $\R$-semialgebraic. In our argument it is enough to consider finite subsets of $H$ which are $\R$-semialgebraic.) 
\end{proof}

\begin{cor} \label{exposs} (Effros-Winkler - weak separation)
Suppose we are in the situation of \autoref{effro1}. Then there exist $L_0^j,...,L_g^j \in S\C^{\delta \times \delta}$ for $j \in \{0,...,(g+1)\delta^2-1\}$ such that $L=\sum_{j=0}^{(g+1)\delta^2-1} \ep^j (L_0^j + L_1^j X_1 + ... + L_g^j X_g)$ fulfills: For all $B \in H$: $L(B) \succeq_\R 0$ and $ L(A) \nsucceq_\R 0$. 
\end{cor}

\begin{proof}
We apply \autoref{effro1} and \autoref{effro2} with $\mathcal{F}$ defined like in \autoref{effro1} . We conclude that there is a real closed extension field $\mathcal{R}$ of $R$ and $T \in \mathcal{T}_\mathcal{R,\delta}$ such that $f(T) \geq 0$ for all $f \in \mathcal{F}$. Now one follows the proof of \cite[Proposition 6.4]{HM} to see that there is a linear pencil $L=L_0+L_1 X_1 +...+ L_g X_g \in S(\mathcal{C}\!<\!\ov X\!>\!)_1^{\delta \times \delta}$ such that $L(B) \succeq_\R 0$ (even $L(B) \succeq_R 0$, but maybe $L(B) \nsucceq_\mathcal{R} 0$) for all $B \in H$ and $L(A) \nsucceq_\R 0$. \\[0.2cm] Now due to \autoref{vitali} there exist $\lambda_0,...,\lambda_{2 (g+1)\delta^2-1} \in \mathcal{R}$ and $L_0^j,...,L_g^j \in S\C^{\delta \times \delta}$ such that $\lambda_0>>...>>\lambda_{(g+1)\delta^2-1}>0$ and $L=\sum_{j=0}^{(g+1)\delta^2-1} \lambda_j (L_0^j + L_1^j X_1 + ... + L_g^j X_g)$. Now observe that we can also exchange the $\lambda_j$ by $\ep^j$.
\end{proof}

\begin{definition} \label{defstep}
Let $R$ be a real closed extension field of $\R$ and $C=R[\ii]$ be the algebraic closure of $R$. For $H \in (\S C^{\delta \times \delta})^g$, $A \in C^{\delta \times \delta}$ and $L=A+H \ov X$ we define $\mathcal{D}_L^{\R}=\{C \in \S^g(k) \ | \  L(C) \succeq_\R 0\}$ and $\mathcal{D}_L^{\R,0}=\{C \in \S^g(k) \ | \  L(C) \succeq_\R 0 \ \& \ \exists w \in \C^{k \delta} \setminus \{0\}: w^* L(C) w=0\}$, $\mathcal{D}_L^{\R,\succ}=\mathcal{D}_L^{\R} \setminus \mathcal{D}_L^{\R,0}$. For $B \in \mathcal{D}_L^{\R,0}(k)$ and $v=\sum_{\al=1}^\delta e_\al \otimes v_\al \in \C^{k \delta}$ with $L(B)v=0$ set $M(B,v)_{L,\R}=\spam_\C \{v_\al \ | \ \al \in \{1,...,\delta\}\}$.
\end{definition}

\begin{lemma} (Free Minkowksi theorem) \label{min2}
Let $K \subseteq \S^g$ be compact and matrix convex with $\text{kz}(K)=\delta < \infty$. Then $K=\mconv(\text{mext}(K))$. 
\end{lemma}

\begin{proof}
Without loss of generality suppose that $0 \in K$. We make an induction on $\delta \in \N$: The case $\delta=1$ is just the usual Minkowski theorem since the ordinary extreme points of $K(1)$ equal the matrix extreme points of $\mconv(K(1))$. So let $\delta > 1$ and suppose that the claim has been proven for $\delta-1$. We conclude that $K(\delta -1) \subseteq \mconv(\text{mext}(K))$. Let $r=(g+1)\delta^2-1$ and $R$ be a real closed extension field of $\R$. Take $\ep \in R_{>0}$ infinitesimal. \\[0.2cm]
Let $H=\text{mext}(K)$. Assume that $\mconv(H) \neq K$. Then we can find $A_1 \in K \setminus \mconv(H)$. The induction start tells us that $K(1) \subseteq \mconv{H}$, in particular $0 \in H$. By \autoref{exposs} we find $L^1:=\sum_{k=0}^r \ep^r (M_k^1 + L_k^1 \ov X)$ where $L_0^1,...,L_r^1 \in \S^g(\delta)$, $M_0^1,...,M_r^1 \in S\C^{\delta \times \delta}$ such that $\mconv(H) \subseteq \mathcal{D}_{L^1}^{\R,\succ}$, $A \in \mathcal{D}_{L^1}^{\R,0}$. \\[0.2cm]
Consider the continuous function $\gamma_1: K(\delta) \times \S^{\delta^2-1} \rightarrow \R+\ep \R + ...+ \ep^r \R, (B,v) \mapsto v^* L^1(B) v$. This function attains a minimum $\rho_1=\sum_{k=0}^r \ep^k \lambda_k^1 \leq 0$ where $\lambda^1 \in \R^{r+1}$. Now define $\lambda^1 \circ L^1=L^1 - \rho_1 I_\delta$. We still have $\mconv(H) \subseteq \mathcal{D}_{\lambda^1 \otimes L^1}^{\R,\succ}$ and $Z_1:=\mathcal{D}_{\lambda^1 \otimes L^1}^{\R,0} \cap K \neq \emptyset$. Additionally we achieved that $K \subseteq \mathcal{D}_{\lambda^1 \otimes L^1}^{\R}$.
$Z_1$ is compact. We know that for all $B \in Z_1$ there is $v=\sum_{\al=1}^\delta e_\al \otimes v_\al \in \mathcal{S}^{\delta^2-1}$ such that $\dim M(B,v)_{\lambda^1 \circ L_1,\R}=\dim \spam\{v_1,...,v_\delta\}=\delta$ and $\ker(\lambda^1 \circ L^1)(B) \cap \C^{\delta^2}=\spam(v)$. If the dimension was smaller, we could compress $B$ to $PBP^*$ where $P$ is the projection onto $M(B,v)_{\lambda^1 \circ L^1,\R}$ and the compression would be also not in the closed matrix convex hull of the matrix extreme points (cf. \autoref{ald2}). This contradicts the induction hypothesis. We consider two cases: \\[0.2cm]
Case $1.1$: There is a $E \in Z_1$ such that $\mconv(E) \supseteq Z_1$. Fix $v \in \mathcal{S}^{\delta^2-1}$ such that $\ker(\lambda^1 \circ L^1)(E) \cap \C^{\delta^2}=\spam(v)$ and $\spam\{v_1,...,v_\delta\}=\C^{\delta}$. If $E$ is reducible, then we know that $E=C \oplus D$ and $C,D \in \mconv{H}$ which implies $E \in \mconv{H}$. Therefore $E$ is irreducible. Suppose $E=\sum_{j=1}^r V_j^* A_j V_j$ for some $A_j \in K(1),...,K(\delta)$ with $I=\sum_{j=1}^r V_j^* V_j$ and each $V_j \neq 0$. If one $(\lambda^1 \circ L^1)(A_j) \succ^{\R} 0$, then due to $(I \otimes V_j)v \neq 0$ also $v^*(\lambda^1 \circ L^1)(E)v > 0$, which contradicts the choice of $E$. Thus all $A_j$ are contained in $Z_1 = \mconv(E)$. We saw already that $E$ is matrix extreme in $\mconv(E)$ (\autoref{david}). Hence all $A_j$ are unitarily equivalent to $E$. This means $E \in H$. \\[0.2cm]
Case $1.2$: There is no $B \in Z_1$ such that $\mconv(B) \supseteq Z_1$. Choose $B \in Z_1$ and $C \in Z_1$ such that $C \notin \mconv(B)$. Set $H_1=H \cup \{B\}$. We claim $C \notin \mconv(H_1)$. This is again due to the fact that $(\lambda^1 \circ L^1)(A) \succ^{\R} 0$ for all $A \in H$. So again by \autoref{exposs} we can find $L^2:=\sum_{k=0}^r \ep^r (M_k^2 + L_k^2 \ov X)$ where $L_0^2,...,L_r^2 \in \S^g(\delta)$, $M_0^2,...,M_r^2 \in S\C^{\delta \times \delta}$ such that $\mconv(H_1) \subseteq \mathcal{D}_{L^2}^{\R,\succ}$, $C \in \mathcal{D}_{L^2}^{\R,0}$. The continuous function $\gamma_2: Z_1 \times \S^{\delta^2-1} \rightarrow \R+\ep \R + ... \ep^r \R, (D,v) \mapsto v^* L^2(D) v$ attains a minimum $\rho_2=\sum_{k=0}^r \ep^k \lambda_k^2 \leq 0$ where $\lambda^2 \in \R^{r+1}$. Set $\lambda^2 \circ L^2=L^2 - \rho_2 I_\delta$. We observe that $\mconv(H_1) \subseteq \mathcal{D}_{\lambda^2 \circ L^2}^{\R,\succ}$ and $Z_2:=\mathcal{D}_{\lambda^2 \circ L^2}^{\R,0} \cap Z_1 \neq \emptyset$ is compact as well as $Z_1 \subseteq \mathcal{D}_{\lambda^2 \circ L^2}^{\R}$.\\[0.2cm] 
Again we consider the two cases $1.2$ (there is $B \in Z_2$ such that $\mconv(B) \supseteq Z_2$; in this case we show again that $B$ is matrix extreme) or the case $2.2$ (there is no $B \in Z_2$ such that $\mconv(B) \supseteq Z_2$; in this case we define $L^3$ and $Z_3$ like above and continue the procedure). If at some iteration $m$ the case $m.1$ occurs, we are done. Assume this is not the case. Noticing that  $\ep$ is infinitesimal and algebraically independent over $\R$, we get the following: For $m \in \N$ let $R_m=\{C \in \S^g(\delta) \ | \ \det\left( \sum_{k=0}^r(-\lambda_k^m + M_k^m + L_k^m \ov X)(C)^2\right)=0\} \}$. Then the sequence $\left(\bigcap_{m=1}^n R_m\right)_{n \in \N}$ is a strictly descending sequence of varieties, which is a contradiction to the Hilbert basis theorem.  
\end{proof}

\section*{Acknowledgements}

\noindent First of all I want to thank my supervisor Markus Schweighofer for giving me the possibility of writing a PhD thesis, from which this paper evolved,
under his supervision and helping me with mathematical suggestions, discussions and advice. I always enjoyed working in his group. I am grateful that Markus was always available for giving support and encouraged me to do research. \\[0.2cm]
I am indebted to Victor Vinnikov, who was present in many discussions with Markus Schweighofer and has been something like a second supervisor. His ideas and vision often proved to be very useful and lead me in the right way. \\[0.2cm]
Igor Klep was helping by having inspirational conversations about many different aspects of matrix convex sets with me. I am very grateful that he gave me the possibility to visit him in Ljubljana and to do research together. He made a lot of effort to bring me closer to other people working with matrix convex sets and helped me to make my work known to the mathematical community. \\[0.2cm] 
I also appreciate having fruitful discussions with Tim Netzer, Jurij Vol$\check{\text{c}}$i$\check{\text{c}}$, J.W. Helton and Eric Evert.

\printindex

%

\end{document}